%% file: KurinczukSkodlerackStevens_Endoparameters.tex
\documentclass{amsart}
\usepackage{tensor}
\include{Endo-header}

\def\rob#1{\textcolor{black}{#1}}
\def\daniel#1{\textcolor{black}{#1}}
\def\shauny#1{\textcolor{black}{#1}}
\def\bob#1{\textcolor{black}{#1}}

\usepackage{subfiles}

\title{Endo-parameters for~\lowercase{$p$}-adic classical groups}
\author{Robert Kurinczuk}
\address{Robert Kurinczuk, Department of Mathematics, Imperial College, London, SW7 2AZ, United Kingdom.}
\email{robkurinczuk@gmail.com}
\author{Daniel Skodlerack}
\address{Daniel Skodlerack, Institute of Mathematical Sciences, ShanghaiTech University, Pudong, China.}
\email{dskodlerack@shanghaitech.edu.cn}
\author{Shaun Stevens}
\address{Shaun Stevens, School of Mathematics, University of East Anglia, Norwich, NR4 7TJ, United~Kingdom.}
\email{shaun.stevens@uea.ac.uk}
\subjclass[2010]{22E50; 11F70}

\begin{document}

\begin{abstract}
For a classical group %~$\G^\so$
over a non-archim\-e\-dean local field of odd residual characteristic~$p$, we prove that two cuspidal types, defined over an algebraically closed field~$\CC$ of characteristic different from~$p$, intertwine if and only if they are conjugate. This completes work of the first and third authors %~\cite{St08,RKSS}
who showed that every irreducible cuspidal~$\CC$-representation %over~$\CC$ of~$\G^\so$ 
of a classical group is compactly induced from a cuspidal type. % now giving a classification of irreducible cuspidal representations of classical groups in terms of cuspidal types. 
%A fundamental step in this approach is to construct a list of characters of pro-$p$ subgroups from arithmetic data, called self-dual semisimple characters, which exhaust all smooth representations (every smooth representation of~$\G^\so$ contains a self-dual semisimple character.
%, which is accomplished for all smooth representations of~$\G^\so$, is the 
We generalize Bushnell and Henniart's notion of endo-equivalence %~\cite{BH96}
to semisimple characters of general linear groups and to self-dual semisimple characters of classical groups, and introduce (self-dual) endo-parameters. \shaun{We prove that these parametrize intertwining classes of (self-dual) semisimple characters and conjecture that they are in bijection with wild Langlands parameters, compatibly with the local Langlands correspondence.}
\end{abstract}
\maketitle

\setcounter{tocdepth}{1}
\tableofcontents

\subfile{Endo-intro-rob-rho_quasisplit}

\subfile{Endo-notation}

\subfile{Endo-Witt}
\subfile{Endo-strata}

\subfile{Endo-chars}

\subfile{Endo-ps}

\subfile{Endo-semisimple}

\subfile{Endo-pss}

\subfile{Endo-IIC}

\subfile{Endo-parameters}
\subfile{Endo-one-appendix}

\bibliographystyle{plain}
\bibliography{Endoclass.bib}
\end{document}

%% file: Endo-header.tex
\usepackage{amsmath,amssymb,mathrsfs,euscript,color,mathabx,stmaryrd,mathabx}
\usepackage[colorlinks]{hyperref}
\usepackage[shortlabels]{enumitem}
\usepackage{pdfsync}
\usepackage[backgroundcolor=green!40, linecolor=green!40]{todonotes}
\usepackage[utf8]{inputenc}                                    
\usepackage[T1]{fontenc}

\input{xypic}
\usepackage{tikz-cd}

\usepackage[margin=1.2in]{geometry}

\newcounter{makeconstant}
\newenvironment{makeconstant}%
{\refstepcounter{makeconstant}}%
{}

\newcounter{nonumber}

\theoremstyle{plain}
\newtheorem{theorem}[equation]{Theorem}
\newtheorem{lemma}[equation]{Lemma}
\newtheorem{corollary}[equation]{Corollary}
\newtheorem{proposition}[equation]{Proposition}

\theoremstyle{definition}

\newtheorem{definition}[equation]{Definition}

\newtheorem{remark}[equation]{Remark}
\newtheorem{remarks}[equation]{Remarks}

\newtheorem{notation}[equation]{Notation}

\newtheorem*{remark*}{Remark}

\numberwithin{equation}{section} 
\def\CC{\mathbf{C}}

\def\NN{\mathbb{N}}

\def\ZZ{\mathbb{Z}} 

\def\A{{\rm A}}
\def\B{{\rm B}}
\def\C{{\rm C}}

\def\E{{\rm E}}
\def\F{{\rm F}}
\def\G{{\rm G}}
\def\H{{\mathrm{H}}}
\def\I{{\rm I}}
\def\J{{\rm J}}
\def\K{{\rm K}}
\def\L{{\rm L}}
\def\M{{\rm M}}
\def\N{{\rm N}}
\def\O{{\rm O}}
\def\P{{\rm P}}

\def\R{{\rm R}}
\def\SS{{\rm S}}
\def\T{{\rm T}}
\def\U{{\rm U}}
\def\V{{\rm V}}
\def\W{{\rm W}}
\def\X{{\rm X}}

\def\Z{{\rm Z}}

\def\Cc{\mathscr{C}}

\def\Ee{\mathscr{E}}
\def\Ff{\mathcal{F}}

\def\Hh{\mathcal{H}}

\def\Mm{\mathcal{M}}
\def\Nn{\mathcal{N}}

\def\Qq{\mathscr{Q}}

\def\Ww{\mathcal{W}}

\def\BB{\mathfrak{B}}
\def\fCc{\mathfrak{C}}
\def\HH{\mathfrak{H}}
\def\JJ{\mathfrak{J}}
\def\KK{\mathfrak{K}}

\def\aa{\mathfrak{a}}
\def\bb{\mathfrak{b}}

\def\kk{\mathfrak{k}}

\def\mm{\mathfrak{m}}
\def\nn{\mathfrak{n}}

\def\so{\mathsf{o}}
\def\sw{\mathsf{w}}
\def\an{\mathsf{an}}
\def\even{\mathsf{even}}
\def\odd{\mathsf{odd}}

\def\a{\alpha} 
\def\b{\beta}
\def\d{\delta}
\def\e{\varepsilon}

\def\g{\gamma}
\def\vphi{\varphi}

\def\k{\kappa}
\def\l{\lambda}

\def\o{\mathfrak{o}}
\def\p{\mathfrak{p}}
\def\s{\sigma}
\def\t{\theta}

\def\w{\varpi}
\def\z{\zeta}

\def\La{\Lambda}

\def\Th{\Theta}

\def\({\left(}
\def\){\right)}
\def\>{\geqslant}
\def\<{\leqslant}
\def\le{\leqslant}
\def\ge{\geqslant}

\def\Hom{\operatorname{Hom}}
\def\End{\operatorname{End}}
\def\Aut{\operatorname{Aut}}

\def\GL{\operatorname{GL}}
\def\Gal{\operatorname{Gal}}

\def\ker{\operatorname{ker}}

\def\im{\operatorname{im}}

\def\Tr{\operatorname{Tr}}
\def\id{\operatorname{id}}
\def\Res{\operatorname{Res}}

\def\ind{\operatorname{ind}}

\def\Irr{\operatorname{Irr}}

\def\Art{\operatorname{Art}}

\def\dim{\operatorname{dim}}
\def\diman{\mathrm{dim}_{\mathrm{an}}}
\def\supp{\operatorname{supp}}

\def\diag{\operatorname{diag}}

\def\val{\operatorname{val}}
\def\lcm{\operatorname{lcm}}
\def\gcd{\operatorname{gcd}}

\def\SO{\operatorname{SO}}

\def\SL{\operatorname{SL}}

\def\tG{{\widetilde{\G}}}

\def\tM{{\widetilde{\M}}}
\def\tP{{\widetilde{\P}}}
\def\tU{{\widetilde{\U}}}

\def\ov#1{{\overline{#1}}}
\def\la{\langle}
\def\ra{\rangle}

\definecolor{dorange}{RGB}{255,140,0}
\definecolor{dgreen}{RGB}{0,205,10}

\def\red#1{\textcolor{black}{#1}}
\def\blue#1{\textcolor{black}{#1}}
\def\orange#1{\textcolor{black}{#1}}

\def\gre#1{\textcolor{black}{#1}}

\def\shaun#1{\textcolor{black}{#1}}
\def\shauns#1{\textcolor{black}{#1}}

\def\rob#1{\textcolor{magenta}{#1}}
\def\daniel#1{\textcolor{blue}{#1}}
\def\shauny#1{\textcolor{dorange}{#1}}
\def\bob#1{\textcolor{dgreen}{#1}}

\def\ignore#1{\relax}

\def\bs{\boldsymbol}

\def\bv{{\mathrm v}}
\def\be{{\mathrm e}}
\def\bw{{\mathrm w}}
\def\bdots{\mathinner{\mkern1mu\raise1pt\hbox{.}\mkern2mu\raise4pt\hbox{.}
           \mkern2mu\raise7pt\vbox{\kern7pt\hbox{.}}\mkern1mu}}
\def\into{\hookrightarrow}
\def\presuper#1#2%
  {\mathop{}%
   \mathopen{\vphantom{#2}}^{#1}%
   \kern-\scriptspace%
   #2}
\def\presub#1#2%
  {\mathop{}%
   \mathopen{\vphantom{#2}}_{#1}%
   \kern-\scriptspace%
   #2}

\def\varphican{{\vphi_\b}}
\def\varphicanprime{{\vphi_{\b'}}}
\def\varphigcan{{\vphi_\g}}
\def\varphigcanprime{{\vphi_{\g'}}}

\def\tame{{\mathsf{tame}}}
\def\wt{\w_\tame}
\def\bt{\beta_\tame}

\def\gt{\gamma_\tame}

\def\Ft{\F_\tame}
\def\Fprt{\F'_\tame}

\def\full{{full}}
\def\WT{\mathrm{WT}}
\def\EP{\mathrm{EP}}
\def\fin{\mathrm{fin}}
\def\f{{\mathfrak{f}}} %%%% Or other notation?
\def\ft{{\mathfrak{t}}} %%%% Or other notation?
\def\fz{{\mathfrak{z}}} %%%% Or other notation?
\def\fo{{\scriptstyle{\mathcal{O}}}} %%%% Or other notation?
\def\sfo{{\scriptscriptstyle{\mathcal{O}}}} %%%% Or other notation?
\def\fy{{\mathfrak{h}}} %%%% Or other notation?
\def\fe{{\mathfrak{e}}} %%%% Or other notation?
\def\mfe{{\mathfrak{e}}} 

\def\Falg{\F^{\mathsf{alg}}}
\def\PsF{\operatorname{Ps}(\Falg)}
\def\PsFcl{\operatorname{Ps}_-(\Falg)}

\def\Thip{\Th_{{1}}}
\def\Thim{\Th_{-1}}
\def\Thipm{\Th_{\pm1}}

\def\herm{\operatorname{herm}}
\def\mcH{{\mathcal{H}}}

\def\BBred{\BB_{\operatorname{red}}}

\def\LG{{\tensor*[^L]\G{^\so}}}
\def\hGo{\widehat{\G^\so}}

\def\Ad{\mathrm{Ad}}

\def\Lang{\mathrm{Lang}}
\def\LL{\mathrm{LL}}

\def\WLL{\mathrm{LL}_p}

\def\WP{\mathrm{Wild}}

%% file: Endo-intro-rob-rho_quasisplit.tex
\section{Introduction}\label{sec:intro}
\subsection{}

One approach to study smooth representations of a reductive~$p$-adic group on modules over a commutative ring intrinsically is by restriction to compact open subgroups.  For~$p$-adic general linear groups this has yielded detailed classification results, for example Bushnell and Kutzko's monograph~\cite{BK93} for complex representations and its subsequent partial generalizations to other coefficient rings, see for example~\cite{Vig96,MStype,HelmBC}.% This includes an explicit construction of all cuspidal complex representations building on earlier work of Howe~\cite{Howe}, and a construction of \emph{types} for all Bernstein components.  %Vign\'eras~\cite{Vig96} later showed that the construction of cuspidal representations works well over any algebraically closed field of characteristic different to~$p$.

At the base of the work of Bushnell and Kutzko is an explicit construction of characters of compact open pro-$p$ subgroups of~$p$-adic general linear groups called \emph{simple characters}, constructed from certain arithmetic data.    However, there is a lot of flexibility in the choice of arithmetic data leading to, for example, simple characters contained in isomorphic cuspidal representations.  Moreover, there are functorial relations between simple characters defined on compact open subgroups of different rank~$p$-adic general linear groups.  To control this flexibility and \shaun{encapsulate} these relations, Bushnell and Henniart~\cite{BH96} collected~simple characters into families called \emph{ps-characters} and introduced an equivalence relation on ps-characters \rob{called} \emph{endo-equivalence}.

Endo-classes (endo-equivalence classes of ps-characters) have subsequently been extended to inner forms of general linear groups~\cite{BSS}, and have proved fundamental in understanding fine properties of the local Langlands correspondence~\cite{BHEffective,BH17} and the Jacquet--Langlands correspondence~\cite{SeStJL,Dotto}, as well as in the study of \rob{Galois}-distinguished cuspidal representations~\cite{AKMSS,Secherre} and in Bernstein decompositions of the category of smooth representations over fields of positive characteristic~\cite{SeStblock}.
%
%proved useful in ..... detailed descriptions of the local Langlands correspondence, have been extended to inner forms of general linear groups \c
 
%Their original motivation was to understand \emph{tame lifting} of simple characters, ..  

%%%%%%%%%%%%%%%%%%%%%%%%%%%%%%%%%%%%%%%%%
\subsection{}
Let~$\G^\so$ be a~\emph{$p$-adic classical group}: that is, the group of~$\F_\so$-points of a unitary, symplectic, or special orthogonal group defined over a non-archimedean local field~$\F_\so$ of odd residual characteristic~$p$.  Let~$\CC$ be an algebraically closed field of characteristic different \rob{from}~$p$.   

Building on the work of Bushnell and Kutzko, %and considering~$\G^\so$ as a subgroup of the fixed points of an involution in a general linear group~$\tG$, 
all cuspidal representations of~$\G^\so$ on~$\CC$-vector spaces have been constructed in~\cite{St08,RKSS}, and for complex representations types have been constructed for all Bernstein components~\cite{MiSt}.  Fundamental to this approach are the \emph{self-dual semisimple characters} of compact open pro-$p$ subgroups of~$\G^\so$ constructed in~\cite{St05}.\footnote{\rob{In previous works, including~\cite{St05}, only semisimple characters valued in~$\mathbb{C}$ are considered.  However, the constructions and results also apply to characters valued in~$\CC$ because semisimple characters are (smooth) characters of pro-$p$ groups and~$\CC$ contains a full set of~$p$-power roots of unity.
%
%
%R%The results in~\cite{stevens:02},\cite{St05} and~\cite{SkSt} are for complex semisimple characters. Those still hold for semisimple
%characters with values in~$\CC$, because they are defined on pro-$p$-groups and~$\CC$ contains a full set of~$p$-power roots of unity.%.~$\mu_{p^\infty}(\CC)$.
%R%In fact for most statements it is enough to move between complex and~$\CC$-valued characters via a fixed group isomorphism from~$\mu_{p^\infty}(\mathbb{C})$ to~$\mu_{p^\infty}(\CC)$. 
}}
\shaun{There are functorial relations between the self-dual semisimple characters of different~$p$-adic classical groups, and their definition is cursed by the same flexibility as for general linear groups.}
%{\color{red}Mention there are also functorial relations between different~$p$-adic classical groups, and their definition is cursed with the same flexibility?}
%
%{\color{red}Rewrite paragraph: } To achieve this, first the simple characters for~$p$-adic general linear groups of Bushnell and Kutzko are generalized to semisimple characters for~$p$-adic general linear groups~\cite{BK?,St05}, and then \emph{self-dual semisimple characters} of compact open pro-$p$ subgroups of~$\G^\so$ are constructed~\cite{St05}.  The self-dual semisimple characters then ..... for this approach to the representation theory of~$\G^\so$

In this article, we generalize Bushnell and Henniart's notions of ps-character and endo-equivalence to \emph{self-dual pss-characters} and \emph{endo-equivalence} for~$p$-adic classical groups, and along the way to the semisimple setting of \emph{pss-characters} and \emph{endo-equivalence} for~$p$-adic general linear groups.  We then prove two applications.
\begin{enumerate}\setlength\itemsep{5pt}
\item\label{intro:parti} we complete the classification of isomorphism classes of cuspidal (smooth) representations of~$\G^\so$ on~$\CC$-vector spaces  by conjugacy classes of \emph{cuspidal types}, %~irreducible representations of compact open subgroups 
following the exhaustive constructions of~\cite{St08,RKSS}.  
\item\label{intro:partii} we parametrize the intertwining classes of self-dual semisimple characters by \emph{self-dual endo-parameters}.
%introduce (self-dual) \emph{endo-parameters} which parametrise the intertwining classes of (self-dual) semisimple characters.  
%the theory of which, in its core, contains a generalization of Bushnell and Henniart's theory of endo-equivalence~\cite{BH96}, from simple to semisimple characters of general linear groups, and develops the theory of endo-equivalence for self-dual semisimple characters of classical groups.  
%we classify the intertwining classes of \emph{self-dual semisimple characters} which form the fundamental step of this approach to the smooth dual.
\end{enumerate}

%In previous works, notably~\cite{St05,St08,MiSt,RKSS,SkSt}, the third author and his collaborators have developed an approach to the smooth dual of~$\G^\so$ over an algebraically closed field~$\CC$ of characteristic prime to~$p$ via Bushnell--Kutzko type theory.  

%{\color{red}At the base of the construction.... characters....  Application models for cuspidals....}
%In this article, we accomplish two important parts of this programme: % and complete recent work~\cite{RKSS,SkSt}
%
%%
%%
%%
%%
%%Let~$\G^\so$ be a unitary, symplectic or special orthogonal group over a non-archimedean local field~$\F$ of odd residual characteristic~$p$, let~$\G$ be either~$\G^\so$ in the unitary and symplectic case or the ambient orthogonal group.  
%%
%%
%%For~\ref{intro:partii}, we introduce \emph{the theory of self-dual semisimple endo-parameters}. In its core this generalizes Bushnell and Henniart's theory of endo-class~\cite{BH96}, from simple to semisimple characters of general linear groups, and develops the theory for self-dual semisimple characters of classical groups.  
%
%These results form part of the motivation for our development of endo-parameters for classical groups, together with the use of endo-classes in our proof of~\ref{intro:parti}. {\color{red}tie in}

%%%%%%%%%%%%%%%%%%%%%%%%%%%%%%%%%%%%%%%%%
\subsection{}
We expect self-dual endo-parameters to have a natural interpretation via the local Langlands correspondence in terms of the \emph{restriction} to wild inertia of (extended) Langlands parameters, generalizing Bushnell and Henniart's ramification theorem~\cite[6.1 Theorem]{BHEffective} to classical groups and refining work of the third author with Blondel and Henniart~\cite[Theorem 7.1]{BlHeSt}. See the end of the introduction for a precise conjecture.  The added complexity in endo-parameters for~$p$-adic classical groups in comparison to~$p$-adic general linear groups is explained in the parametrization of~$L$-\shaun{indistinguishable} representations.

Another application of endo-parameters is found in current work of the first and third authors which gives a decomposition of the category of all smooth~$\CC$-representations of~$\G^\so$ by self-dual endo-parameters, refining the decomposition by depth~\cite[II 5.8]{Vig96}.  This decomposition is particularly interesting when~$\CC$ has positive characteristic where a block decomposition of the category is not yet known, but where there has been recent progress in depth zero~\cite{Lanard}.  We expect using endo-parameters that there is a reduction of the block decomposition for~$\G^\so$ to the depth zero block decompositions of twisted Levi subgroups of~$\G^\so$.  This fits %well 
with work of Chinello for general linear groups~\cite{Chinello} %, and assuming our conjecture on endo-parameters and local Langlands, 
and with general predictions of Dat~\cite{Datfunctoriality} .
%
%
%Potential applications of~\ref{intro:partii} include a \emph{local Langlands correspondence for endo-parameters} or \emph{Ramification Theorem} for classical groups (see~\cite[6.1 Theorem]{BHEffective} and, recently,~\cite{BHRam} for the state of the art for~$\tG$), and 

%%%%%%%%%%%%%%%%%%%%%%%%%%%%%%%%%%%%%%%%%
\subsection{}
We now state our results precisely.  Henceforth, all representations are assumed to be smooth. We assume that~$\G^\so$ is not isomorphic to~$\SO(1,1)(\F_\so)\simeq \F_\so^\times$, whose representation theory, in any case, is well-known.

%We begin with~\ref{intro:parti}.  
In~\cite{St08} for complex representations, extended to representations on~$\CC$ vector spaces in~\cite{RKSS}, an explicit list of pairs~$(\J,\l)$, called \emph{cuspidal types}, are constructed consisting of a compact open subgroup~$\J$ of~$\G^\so$ and an irreducible representation~$\l$ of~$\J$ such that the induced representation~$\ind_{\J}^{\G^\so}\l$ is irreducible and cuspidal.  The main results of the cited works~say that every irreducible cuspidal representation~$\pi$ of~$\G^\so$ contains a cuspidal type~$(\J,\l)$, i.e. it is compactly induced~$\pi\simeq \ind_{\J}^{\G^\so} \l$.   In other words,  this gives an explicit model of~$\pi$ in terms of the cuspidal type~$(\J,\l)$.  

There is a precise recipe to construct cuspidal types whence cuspidal representations, however it is a recipe which requires many choices and it is far from clear when the procedure results in isomorphic representations.  We prove the following \emph{intertwining implies conjugacy result}:

\begin{theorem*}[Theorem~\ref{thmIIC}]
For~$i=1,2$, let~$(\J_i,\l_i)$ be cuspidal types and put~$\pi_i\simeq \ind_{\J_i}^{\G^\so} \l_i$.  Then~$\pi_1\simeq \pi_2$ if and only if there exists~$g\in\G^\so$ such that~$\J_1^g=\J_2$ and~$\l_1^g\simeq \l_2$. 
\end{theorem*}
Here~$\l_1^g$ is the representation of~$\J_1^g=g^{-1}\J_1 g$ defined by~$\l_1^g(j)=\l_1(gjg^{-1})$ for all~$j\in \J_1^g$.  
%This completes a classification by conjugacy classes of cuspidal types of the irreducible cuspidal representations of~$\G^\so$.  {\color{red} I guess technically one had a classification before in terms of intertwining classes of cuspidal types.... the point being that it is not really clear when two cuspidal types intertwine.... For conjugacy I guess it is clearer...? Can one say something about the initial data defining the conjugate types? } 
\shaun{This result is not unexpected, by analogy with results for inner forms of general linear groups~\cite{BK93,SeStSupercuspidals}, but the proof for classical groups requires considerably more technical machinery.  % (in particular from~\cite{RKSS} and~\cite{SkSt}). 
A major reason for this added complexity can be interpreted via the local Langlands correspondence:~$L$-packets for classical groups are not singletons.}
%has required a substantial amount of work and relies on the main results of a number of papers; recently,~\cite{RKSS} and~\cite{SkSt}. {\color{red} One of the reasons this result turned out to be more elusive for classical groups is.... contrast with added complexity in endo-parameters}%\red{Mention Hakim--Murnaghan...}% We expect this theorem to find many applications in arithmetic - and be useful whenever detailed analysis of cuspidal representations of~$\G^\so$ is required.

%{\color{green}Shaun: I left Daniel's and my versions of the next paragraph for you to decide what you prefer.  Ideally it needs to emphasise the importance of this paper. }
%
%\daniel{The proof of Theorem~\ref{thmIIC} consists of three major steps. In the first step~\cite{SkSt}, the first and the second author discover the relation between skew semisimple characters which are contained in the same irreducible cuspidal representation. Secondly in~\cite{RKSS} the first and the third author prove intertwining implies conjugacy for cuspidal types under the condition that one can parametrize them by closely related skew semisimple characters. Finally, in the present work, we prove that this condition is always satisfied.}

\rob{A special case of Theorem~\ref{thmIIC}, where the self-dual semisimple characters underlying the cuspidal types are assumed to be closely related, is proved in~\cite{RKSS}.  The proof of Theorem~\ref{thmIIC} combines the %technical
 work of this paper to control the choice in arithmetic data in the construction of cuspidal representations, together with an intertwining implies conjugacy result for semisimple characters of \cite{SkSt}, to show that it is always possible to arrange for this to be the case. }

%%%%%%%%%%%%%%%%%%%%%%%%%%%%%%%%%%%%%%%%%
\subsection{}
In the main theme of this paper, we generalize Bushnell and Henniart's theory of potential simple character and endo-equivalence, originally defined in~\cite{BH96}, to potential semisimple characters and semisimple endo-equivalence for general linear groups, and to self-dual potential semisimple characters and self-dual semisimple endo-equivalence for classical groups.  As well as appearing in an essential way in our proof of Theorem~\ref{thmIIC}, this theory warrants independent study and forms a key part of our parametrization of intertwining classes of (self-dual) semisimple characters via \emph{endo-parameters} which we introduce at the end of the paper.  

%This work then allows one to associate to an irreducible representation of~$\G^\so$ a fine invariant: the endo-parameter of any self-dual semisimple character it contains. {\color{red} contrast with depth?} 

%%%%%%%%%%%%%%%%%%%%%%%%%%%%%%%%%%%%%%%%%
\subsection{}
To proceed further, we need to introduce more notation.  First we realize our classical group as a subgroup of the group of isometries of a signed hermitian form.

Let~$\Sigma=\langle\s\rangle$ denote an abstract finite group of order two.  Let~$\F/\F_\so$ be a quadratic or trivial extension of non-archimedean local fields of odd residual characteristic~$p$, and let~$\ov{\phantom{a}}$ denote the generator of~$\Gal(\F/\F_\so)$. Let~$\V$ be an~$\F$-vector space and~$\tG=\Aut_\F(\V)$. Let~$\e=\pm1$ and~$h:\V\times\V\rightarrow \F$ be an~$\e$-hermitian form defining a unitary, symplectic or orthogonal group~$\G=\U(\V,h)$:
\[\G=\{g\in\tG\mid h(gv,gw)=h(v,w)\text{ for all }v,w\in \V\}=\tG^\Sigma,\]
where~$\s$ acts on~$\tG$ by the inverse of the adjoint anti-involution of~$h$.  

We let~$\G^\so$ denote~$\G$ in the unitary and symplectic cases, and the subgroup of all isometries of determinant one in the orthogonal case; we call~$\G^\so$ a \emph{classical group}.  We fix a non-trivial character~$\psi_\so$ of the additive group of~$\F_\so$, of level one, and put~$\psi=\psi_\so\circ\Tr_{\F/\F_\so}$.  

We consider our sign~$\e$ and extension~$\F/\F_\so$ fixed.  However, we will vary our~$\F$-vector space~$\V$ and~$\e$-hermitian form~$h:\V\times\V\rightarrow\F$.  Indeed, this flexibility will be one of the charms of the theory of endo-equivalence.  We still use the notation~$\tG=\GL_\F(\V)$,~$\G$ for the group of isometries of~$h$, and~$\G^\so$ its classical subgroup.  While the notation~$\tG$ does not specify~$\V$, nor~$\G$ or~$\G^\so$ specify~$h$, we trust this will cause no confusion.  Indeed, non-isometric~$\e$-hermitian forms can define the same isometry group up to isomorphism.

\subsection{}
We now introduce an abstraction of the data underlying the construction of semisimple characters, following Bushnell and Henniart in the simple case~\cite{BH96}. 
Let~$(k,\b)$ be a \emph{semisimple pair}, that is~$\E=\F[\b]$ is a sum~$\bigoplus_{i\in\I}\E_i$ of %pairwise non-isomorphic 
field extensions~$\E_i$ of~$\F$ and~$k$ is an integer satisfying a certain technical bound (see Definition~\ref{def:sesipair}). We call~$\I$ the \emph{indexing set}.   %We write~$\beta=\sum_{i\in\I}\beta_i$ for the decomposition of~$\beta$ in~$\bigoplus_{i\in\I}\E_i$.

We let~$\Qq(k,\b)$ denote the class of all quadruples~$(\V,\vphi,\La,r)$ consisting of an~$\F$-vector space~$\V$, an embedding~$\vphi:\E\rightarrow \End_\F(\V)$, a~$\vphi(\o_\E)$-lattice sequence~$\La$ in~$\V$, and an integer~$r$ closely related to~$k$, see Section~\ref{secPSS}.  

Using work of Bushnell--Kutzko~\cite{BK93} and the third author~\cite{St05}, to~$(\V,\vphi,\La,r)\in \Qq(k,\b)$ we associate a set~$\Cc(\La,r,\vphi(\b))$ of \emph{semisimple characters} (which depend %mildly 
on our initial fixed character~$\psi$) of a compact open subgroup~$\H^{r+1}(\vphi(\b),\La)$ of~$\G$.  In the special case where~$\E$ is a field we call the characters in~$\Cc(\La,r,\vphi(\b))$ \emph{simple characters}.  

Corresponding to the decomposition~$\beta=\sum_{i\in\I}\beta_i$ of~$\beta$ in~$\bigoplus_{i\in\I}\E_i$, we have decompositions~$\V=\bigoplus_{i\in\I}\V^i$ and~$\La=\bigoplus_{i\in\I}\La^i$.  Moreover, there are a natural embedding
\[
\H^{r+1}(\vphi(\beta_i),\La^i)\hookrightarrow \H^{r+1}(\vphi(\beta),\La)\text{ and a map~}\Cc(\La,r,\vphi(\b))\rightarrow \Cc(\La^i,r,\vphi(\b_i)),
\] 
which we write~$\t\mapsto\t_i$.  Thus from a semisimple character~$\t$ we get a collection of simple characters~$\t_i$, for~$i\in\I$, which we call its \emph{simple block restrictions}.

%%%%%%%%%%%%%%%%%%%%%%%%%%%%%%%%%%%%%%%%%
\subsection{}
%An important result on intertwining semisimple characters is the \emph{matching theorem} of the second and third authors~\cite[10.1]{SkSt}, quoted in Theorem~\ref{thm:MatchingForChar}.   
%
Let~$(k,\b)$ and~$(k,\b')$ be semisimple pairs with indexing sets~$\I$ and~$\I'$ respectively,~$(\V,\vphi,\La,r)\in \Qq(k,\b)$ and~$(\V,\vphi',\La',r')\in \Qq(k,\b')$ with~$\La$ and~$\La'$ of the same period as~$\o_\F$-lattice sequences.  

Suppose we have semisimple characters~$\t\in\Cc(\La,r,\vphi(\b))$ and~$\t'\in\Cc(\La',r',\vphi'(\b'))$ which intertwine in~$\tG=\Aut_\F(\V)$.  The \emph{matching theorem} of the second and third authors~\cite[Theorem~10.1]{SkSt}, tells us that there exists a \gre{unique} bijection~$\z:\I\rightarrow \I'$ and~$g\in\tG$ such that~$g\V^i=\V^{\z(i)}$ and~$\presuper{g}\t_i$ and~$\t_{\z(i)}'$ intertwine in~$\Aut_\F(\V^{\z(i)})$.  In this case, we say that~$\t$ intertwines with~$\t'$ with \emph{matching}~$\z$.

%%%%%%%%%%%%%%%%%%%%%%%%%%%%%%%%%%%%%%%%%
\subsection{}
There are natural (bijective) \emph{transfer} maps between the sets of semisimple characters defined by a pair of quadruples in~$\Qq(k,\b)$ (see Lemma~\ref{lemTransfer}).  Following the methodology of Bushnell and Henniart~\cite{BH96}, we collect together the semisimple characters related by transfer, into \emph{pss-characters}: a pss-character~$\Theta$ supported on~$(k,\b)$ is a function from~$\Qq(k,\b)$ to the class of all semisimple characters, \gre{such that~$\Theta(\V,\vphi,\La,r)\in \Cc(\La,r,\vphi(\b))$,} whose values are related by transfer.  

We call a value of a pss-character a \emph{realization} of the pss-character.  Thus, by definition, a pss-character is determined by any one of its realizations.  This definition generalizes the definition of~\cite{BH96} of \emph{ps-characters} - which forms the special case where~$\F[\b]$ is a field.
%{\color{red}is it necessary to write this:} between the sets of characters defined by different quadruples: given~$(\V,\vphi,\La,r),(\V',\vphi',\La',r')\in \Qq(k,\b)$ we have a natural bijection
%\[\tau_{\La,\La',\vphi,\vphi',\b}:\Cc(\La,r,\vphi(\b))\rightarrow \Cc(\La',r',\vphi'(\b)),\]
%which in addition depends mildly on~$r,r'$, but we do not include them in our notation.
 %
%We collect together semisimple characters related by transfer, into \emph{pss-characters} (\emph{potential-semisimple characters}), \emph{supported on a semisimple pair}, 
%
%is a function from the class of all quadruples~$(\V,\vphi,\La,r)$ (defined with respect to the semisimple pair as above) to the class of all semisimple characters whose values are related by transfer.  
%Given a simple character~$\t\in\Cc(\La,r,\b)$, its \emph{degree} is~$[\F[\b]:\F]$, this depends only on~$\t$, and is independent of transfer and intertwining.  
The \emph{degree} of a pss-character~$\Theta$ supported on~$(k,\b)$ is~$\deg(\Theta)=[\F[\b]:\F]$. 

A pss-character~$\Theta$ supported on~$(k,\b)$ with index set~$\I$ gives rise to a collection of ps-characters~$\Theta_i$, for~$i\in\I$, supported on simple pairs~$(k_i,\b_i)$, which we call its \emph{component ps-characters}.  See Lemma~\ref{lemma:pssdecompositions} for more details on this decomposition and Definition~\ref{def:sesipair} for the definition of~$k_i$.

% This depends only on the pss-character and is independent of the choice of~$\b$.  
%The \emph{degree} of a full simple endo-class~$c\in\mathcal{E}$ is the common degree of the values of a ps-character with endo-class~$c$, we denote it by~$\deg(c)$. 
Our next step is to generalize Bushnell and Henniart's notion of endo-equivalence from ps-characters to pss-characters.
Let~$\Theta$ be a pss-character supported on the semisimple pair~$(k,\b)$ and~$\Theta'$ be a pss-character supported on the semisimple pair~$(k',\b')$.  We say that~$\Theta$ and~$\Theta'$ are \emph{endo-equivalent}, written~$\Theta\approx\Theta'$, if 
\begin{enumerate}\setlength\itemsep{5pt}
\item $k=k'$ and~$\deg(\Theta)= \deg(\Theta')$; and
\item there exist realizations on a common~$\F$-vector space~$\V$ which intertwine in~$\tG=\GL_\F(\V)$.
\end{enumerate}
We relate endo-equivalence of pss-characters with endo-equivalence of their component ps-characters:

%\begin{theorem*}[{Part of Theorem~\ref{thm:Endopss}}]
%Let~$\Theta$ and~$\Theta'$ be pss-characters supported on semisimple pairs~$(k,\b)$ and~$(k'\b')$ respectively.  Let~$\I$ and~$\I'$ be the indexing sets of~$(k,\b)$ and~$(k,\b')$ respectively.
%\begin{enumerate}
%\item\label{IntroThzeta} Then~$\Theta$ and~$\Theta'$ are endo-equivalent if and only if there is a bijection~$\z:\I\rightarrow\I'$ such that for all~$i\in\I$ the ps-characters~$\Theta_i$ and~$\Theta'_{\z(i)}$ are endo-equivalent. 
%\item\label{IntroThzeta2} Suppose that~$\Theta$ and~$\Theta'$ are endo-equivalent.  Then any two realizations of~$\Theta$ and~$\Theta'$ on an~$\F$-vector space satisfying~$\dim_\F(\V^i)=\dim_\F(\V^{\z(i)})$ for all~$i\in\I$, where~$\z:\I\rightarrow\I'$ is the map of~\ref{IntroThzeta}, intertwine. 
%\item Endo-equivalence of~pss-characters is an equivalence relation.
%\end{enumerate}
%\end{theorem*}

\begin{theorem*}[{Part of Theorem~\ref{thm:Endopss}}]
Let~$\Th$ and~$\Th'$ be pss-characters supported on semisimple pairs~$(k,\b)$ and~$(k,\b')$ respectively, with index sets~$\I$ and~$\I'$ respectively. 
\begin{enumerate}\setlength\itemsep{5pt}
\item\label{thm:Endopssintro-i} We have~$\Th\approx\Th'$ if and only if there is a bijection~$\z:\I\to\I'$ such that, for all~$i\in\I$, the component ps-characters~$\Th_i$ and~$\Th_{\z(i)}$ are endo-equivalent. Moreover, if~$\Th\approx\Th'$ then the map~$\z$ is uniquely determined.
\item\label{thm:Endopssintro-ii} Suppose that~$\Th\approx\Th'$ and let~$\z:\I\to\I'$ be the bijection of~\ref{thm:Endopssintro-i}. Let~$(\V,\vphi,\La,r)\in\Qq(k,\b)$ and~$(\V,\vphi',\La',r')\in\Qq(k,\b')$.
%\begin{enumerate}\setlength\itemsep{5pt}
%\item\label{thm:Endopss-iia1} For all~$i\in\I$, we have 
%\begin{equation}\label{eqppsFieldDegrees}
%e(\E_i|\F)=e(\E'_{\z(i)}|\F),\ f(\E_i|\F)=f(\E'_{\z(i)}|\F),\ k_\F(\b_i)=k_\F(\b'_{\z(i)}).
%\end{equation}
%\item\label{thm:Endopss-iia2} If%~$(\V,\vphi,\La,r)\in\Qq(k,\b)$ and~$(\V',\vphi',\La',r')\in\Qq(k,\b')$ are such that
%~$e(\La)=e(\La')$ then~$(\V,\vphi,\La,r')\in\Qq(k,\b)$.
%\item
% If~$\V=\V'$ and~$\Th(\V,\vphi,\La,r)$ and~$\Th'(\V,\vphi',\La',r')$ intertwine in~$\tG$ with matching~$\xi$, then~$\xi=\z$.
%\item 
If~$\dim_\F(\V^i)=\dim_\F(\V^{\z(i)})$, for all~$i\in\I$, then~$\Th(\V,\vphi,\La,r)$ and~$\Th'(\V,\vphi',\La',r')$ intertwine in~$\tG$ with matching~$\z$.
%\end{enumerate}
%\item\label{thm:Endopss-iia1} If~$\Th\approx\Th'$ and~$\z:\I\to\I'$ is the bijection given by~\ref{thm:Endopss-i} then, for all~$i\in\I$, we have 
%\begin{equation}\label{eqppsFieldDegrees}
%e(\E_i|\F)=e(\E'_{\z(i)}|\F),\ f(\E_i|\F)=f(\E'_{\z(i)}|\F),\ k_\F(\b_i)=k_\F(\b'_{\z(i)}).
%\end{equation}
\item Endo-equivalence of pss-characters is an equivalence relation.
\end{enumerate}
\end{theorem*}

We call the endo-equivalence classes of pss-characters \emph{semisimple endo-classes}. % and the special case of endo-equivalence classes of ps-characters \emph{simple endo-classes}.  
Given endo-equivalent pss-characters as in Theorem~\ref{thm:Endopss}, we call the map~$\z$ of~\ref{thm:Endopssintro-i} a \emph{matching}.
The condition,~$\dim_\F(\V^i)=\dim_\F(\V^{\z(i)})$ for all~$i\in\I$, in~\ref{thm:Endopssintro-ii} is necessary, %. In constrast to the special case of ps-characters, for pss-characters it is not true that any two realizations on a common vector space intertwine,
 as follows from~\cite[Theorem~10.1]{SkSt}.  In the special case of ps-characters, of course, this condition is automatic. %, and~\ref{thm:Endopssintro-ii} has the slicker formulation: if~$\Theta\approx \Theta'$ then any two realizations of~$\Theta$ and~$\Theta'$ on a common~$\F$-vector space intertwine in~$\tG$.

%%%%%%%%%%%%%%%%%%%%%%%%%%%%%%%%%%%%%%%%%
\subsection{}
\gre{Now we turn to the analogous constructions for our classical group~$\G$, so consider the action of the involution~$\s$ on the data involved.}  Let~$(k,\b)$ be a semisimple pair with indexing set~$\I$ and~$\E=\F[\beta]$.  
\shaun{We call~$(k,\b)$~\emph{self-dual} if the \rob{Galois} involution generating~$\Gal(\F/\F_\so)$ extends to a \rob{Galois} involution on~$\E$ sending~$\b$ to~$-\b$; %~$\ov{\phantom{a}}:\F\rightarrow\F$ extends to an involution~$\ov{\phantom{a}}:\E\rightarrow\E$ sending~$\b$ to~$-\b$, 
in this case we call the~$\F$-algebra~$\E=\F[\b]$ \emph{self-dual} (though %.  Note that really 
self-duality is really a property of the pair~$(\E,\beta)$). In this case, the \rob{Galois} involution induces an action of~$\s$ on the indexing set~$\I$, which decomposes as~$\I=\I_+\cup \I_0\cup\I_-$ with~$\I_0$ the~$\s$-fixed indices,~$\I_+$ a set of representatives for the orbits of size~$2$ and~$\I_-=\s(\I_+)$.}% the other representatives.

Suppose that~$(k,\b)$ is self-dual.  We let~$\Qq_-(k,\b)$ denote the class of all quadruples~$((\V,h),\vphi,\La,r)$ such that~$(\V,\vphi,\La,r)\in\Qq(k,\b)$, the~$\F$-vector space~$\V$ is equipped with an~$\e$-hermitian form~$h:\V\times\V\rightarrow\F$ and~$\vphi,\La$ are \emph{self-dual}.
If~$((\V,h),\vphi,\La,r)\in\Qq_-(k,\b)$ then~$\H^{r+1}(\b,\La)$ is~$\Sigma$-stable and~$\Sigma$ acts on~$\Cc(\La,r,\vphi(\b))$ with fixed points~$\Cc^\Sigma(\La,r,\vphi(\b))$, where as before~$\s$ acts via the inverse of the adjoint anti-involution of~$h$.   We set~$\H^{r+1}_-(\b,\La)=\H^{r+1}(\b,\La)^\Sigma$ and define the set of \emph{self-dual semisimple characters}~$\Cc_-(\La,r,\vphi(\b))$ of~$\H^{r+1}_-(\b,\La)$ by restriction from~$\Cc^\Sigma(\La,r,\vphi(\b))$. 
%
%consisting of an~$\F$-vector space~$\V$, an embedding~$\vphi:\E\rightarrow \End_\F(\V)$, a~$\vphi(\o_\E)$-lattice sequence~$\La$ in~$\V$, and an integer~$r$ closely related to~$k$, see Section~\ref{secPSS}.
%
 %and consider a quadruple~$((\V,h),\vphi,\La,r)$ such that~$(\V,\vphi,\La,r)\in\Qq(k,\b)$, the space~$\V$ equipped with an~$\e$-hermitian form~$h:\V\times\V\rightarrow\F$ and~$\vphi,\La$ satisfy certain duality properties, 
 By the Glauberman correspondence, this restriction is \shaun{injective} and, given~$\t_-\in\Cc_-(\La,r,\vphi(\b))$ we call the unique semisimple character in~$\Cc^\Sigma(\La,r,\vphi(\b))$ whose restriction is~$\t_-$ its \emph{lift}.

%%%%%%%%%%%%%%%%%%%%%%%%%%%%%%%%%%%%%%%%%
\subsection{}
%Given~$((\V,h),\vphi,\La,r),((\V',h'),\vphi',\La',r')\in\Qq_-(k,\b)$ the transfer map~$\Cc(\La,r,\vphi(\b))\rightarrow \Cc(\La',r',\vphi'(\b))$ commutes with the actions of~$\Sigma$ and~$\Sigma'$, the cyclic groups generated by the inverse of the adjoint anti-involutions of~$h$ and~$h'$ respectively, and defines a bijection~$\Cc_-(\La,r,\vphi(\b))\rightarrow \Cc_-(\La',r',\vphi'(\b))$ we again call \emph{transfer}.  See Section .....?
\shaun{Since, for self-dual semisimple pairs, the transfer maps commute with the action of~$\s$,}   
 there are natural (bijective) \emph{transfer} maps between the sets of self-dual semisimple characters defined by a pair of quadruples in~$\Qq_-(k,\b)$ (see Section~\ref{sect:Transferofsesichars}).  Thus we can follow the methodology of Bushnell and Henniart~\cite{BH96} once more.
%
%
%{\color{red}Don't need such explicit def of transfer:}Given two such quadruples~$((\V,h),\vphi,\La,r)$ and~$((\V',h'),\vphi',\La',r')$, the transfer map~$\tau_{\La,\La',\b}:\Cc(\La,r,\vphi(\b))\rightarrow \Cc(\La',r',\vphi'(\b))$,
%commutes with the action of~$\Sigma$ and~$\Sigma'$, the cyclic groups generated by the inverse of the adjoint anti-involutions of~$h$ and~$h'$ respectively, and defines a bijection we again call \emph{transfer}
%\[\tau_{\La,\La',\vphi,\vphi',\b}:\Cc_-(\La,r,\vphi(\b))\rightarrow \Cc_-(\La',r',\vphi'(\b)).\]

A~\emph{self-dual pss-character}, supported on a self-dual semisimple pair~$(k,\b)$, is a function~$\Theta_-$ from~$\Qq_-(k,\b)$ to the class of all self-dual semisimple characters, \gre{such that~$\Theta_-((\V,h),\vphi,\La,r)\in \Cc_-(\La,r,\vphi(\b))$,} whose values are related by transfer.   We call a value of a self-dual pss-character a \emph{self-dual realization} of the self-dual pss-character.  Thus a self-dual pss-character is determined by any one of its self-dual realizations.%, and collects together all the self-dual semisimple characters related to a given one by transfer. 

A pss-character~$\Theta$ supported on the self-dual semisimple pair~$(k,\b)$ is called \emph{$\s$-invariant} if, for any (or equivalently, some)~$((\V,h),\vphi,\La,r)\in\Qq_{-}(k,\b)$ the realization~$\Th(\V,\vphi,\La,r)$ is~$\s$-invariant. %for all (or equivalently one) self-dual realization of~$(k,\b)$ the value is~$\s$-invariant (note that,~$\s$ is varying . 
By the Glauberman correspondence, a self-dual pss-character~$\Theta_-$ comes uniquely from the \emph{restriction} of a~$\s$-invariant pss-character~$\Theta$ (see Section {\color{red}9.3} for the precise statement), which we call its \emph{lift}, and we set~$\deg(\Theta_-)=\deg(\Theta)$.

%A~self-dual pss-character~$\Theta$ supported on~$(k,\beta)$ with index set~$\I=\I_+\cup \I_0\cup\I_-$ gives rise to a collection of self-dual ps-characters~$\Theta_{i,-}$, for~$i\in \I_0$, and ....For all~$i\in\I_0$
%
%
%~$\Theta_-$ supported on a self-dual semisimple pair~$(k,\b)$ there is a unique pss-character~$\Theta$ supported~$(k,\b)$ which has realizations lifting the self-dual realizations of the self-dual pss-character, we call~$\Theta$ the \emph{lift} of~$\Theta_-$.{\color{red}rewrite not quite right}

Let~$\Theta_-$ be a self-dual pss-character supported on the self-dual semisimple pair~$(k,\b)$ and~$\Theta'_-$ be a self-dual pss-character supported on the self-dual semisimple pair~$(k',\b')$.  We say that~$\Theta_-$ and~$\Theta'_-$ are \emph{endo-equivalent} if 
\begin{enumerate}\setlength\itemsep{5pt}
\item $k=k'$ and~$\deg(\Theta_-)=\deg(\Theta'_-)$; and
\item there exist self-dual realizations on a common~$\e$-hermitian~$\F$-space~$(\V,h)$ which intertwine in~$\G=\U(\V,h)$.
\end{enumerate}

%We note that if self-dual realizations on a common~$\e$-hermitian~$\F$-space intertwine in~$\G$ then their lifts intertwine by the same element so we have a matching~$\z:\I\to\I'$ between the index sets~$\I$ and~$\I'$ of~$(k,\b)$ and~$(k,\b')$ respectively.

%%%%%%%%%%%%%%%%%%%%%%%%%%%%%%%%%%%%%%%%%
\subsection{}
Let~$(k,\beta)$ a self-dual simple pair,~$\Theta_-$ a self-dual ps-character supported on~$(k,\beta)$, and
\[
((\V,h),\vphi,\La,r), ((\V,h),\vphi',\La',r')\in\Qq_-(k,\beta).
\]
\rob{It follows from \cite[Theorem 5.2]{SkSt} that}~$\Theta_-((\V,h),\vphi,\La,r)$ intertwines with~$\Theta_-((\V,h),\vphi',\La',r')$ in~$\G$ if and only if~$\vphi$ and~$\vphi'$ are conjugate in~$\G$.  
%
%
 % {\color{blue}Let~$(k,\b)$ be a self-dual simple pair, and~$((\V,h),\vphi,\La,r)$ and~$((\V,h),\vphi',\La',r')$ be self-dual quadruples as above.  Then~$\t\in\Cc_-(\La,r,\vphi(\b))$ intertwines with~$\tau_{\La,\La',\vphi,\vphi',\b}(\t)$ in~$\G$ if and only if~$\vphi$ and~$\vphi'$ are conjugate in~$\G$ {\color{red}REF}.   To understand intertwining of self-dual realizations of pss-characters in general, and develop endo-equivalence, we need a more general notion that conjugacy. }  
However, to develop endo-equivalence of self-dual pss-characters -- where we may be dealing with embeddings of non-isomorphic fields -- we need a more general notion than conjugacy.  With this in mind, in Section~\ref{sec:wittandtransfer} we go back to the start and the theory of~$\e$-hermitian spaces and Witt groups.  
%
%{\color{red} rewrite paragraph: }
%Let~$(\V,h)$ be an~$\varepsilon$-hermitian space and~$\E=\F[\b]$ be a \emph{self-dual field extension}.  An embedding~$\vphi:\E\hookrightarrow \End_\F(\V)$ is called \emph{self-dual} if~$\ov{\vphi(x)}=\vphi(\ov{x})$, for all~$x\in\E$, where on the left~$\ov{\phantom{a}}$ denotes the adjoint anti-involution of~$h$ on~$\End_\F(\V)$. % For self-dual embeddings~$\vphi,\vphi'$.....

%$\e$-hermitian~$\F/\F_\so$-spaces~$(\V,h)$ and~$(\V',h')$ which are isometric. Set~$\A=\End_\F(\V)$ and~$\A'=\End_\F(\V')$, suppose we have self-dual embeddings~$\vphi:\E\into\A$ and~$\vphi':\E'\into\A'$. 
%The pairs~$(\b,\vphi)$ and~$(\b',\vphi')$ are \emph{$(h,h')$-concordant} (or just \emph{concordant}.... 
%Let~$\E'=\F[\b']$ be another self-dual field extension and~$\vphi:\E\hookrightarrow \End_\F(\V)$ and~$\vphi:\E'\hookrightarrow\End_\F(\V)$ self-dual embeddings.
We introduce an equivalence relation, which we call \emph{concordance}, on the set of self-dual embeddings of self-dual field extensions into~$\End_\F(\V)$ (where the embedding is self-dual with respect to a hermitian form on~$\V$), see Definition~\ref{def:concordance}\orange{; more precisely, this is a relation on pairs~$(\b,\vphi)$. This relation generalizes conjugacy: if~$\vphi,\vphi':\F[\beta]\hookrightarrow\End_\F(\V)$ are self-dual embeddings, then~$(\b,\vphi)$ and~$(\b,\vphi')$ are concordant if and only if~$\vphi(\b)$ and~$\vphi'(\b)$ are conjugate in~$\G$, see Remark~\ref{rem:concordance}.}

\subsection{}
%Our development of concordance, together with a study of intertwining self-dual simple characters
We carry concordance through the construction of self-dual simple characters, leading to the following result:

\begin{proposition*}[Proposition~\ref{prop:TiGandGIntertwiningSameNonSympl}]
Let~$\t_-\in\Cc_-(\La,r,\vphi(\b))$ and~$\t'_-\in\Cc_-(\La',r,\vphi'(\b'))$ be self-dual simple characters, and suppose that the periods of~$\La$ and~$\La'$ as sequences of~$\o_\F$-lattices coincide.  Then~$\t_-$ and~$\t'_-$ intertwine in~$\G$ if and only if their lifts intertwine in~$\tG$ and the \orange{pairs~$(\b,\vphi)$ and~$(\b',\vphi')$} are concordant.
\end{proposition*}
In fact, %we show that 
the additional concordance hypothesis is only necessary in the symplectic case when~$\e=-1$ and~$\F=\F_\so$; it is implied by the intertwining of the lifts in all other cases.  %In this case there are two~$\G$-conjugacy classes of self-dual embeddings of a simple element~$\b$.

%%%%%%%%%%%%%%%%%%%%%%%%%%%%%%%%%%%%%%%%%
\subsection{}
%We extend our notion of concordance to semisimple elements. Let~$(k,\b)$ and~$(k,\b')$ be self-dual semisimple pairs with indexing sets~$\I$ and~$\I'$ respectively.  Given a bijection~$\z:\I\rightarrow \I'$  we say that self-dual embeddings
%
In Definition~\ref{def:WittConcordanceSemisimple}, we extend our notion of concordance to \emph{self-dual} embeddings of self-dual~$\F$-algebras. Let~$(k,\b)$ and~$(k,\b')$ be self-dual semisimple pairs with indexing sets~$\I$ and~$\I'$ respectively, and~$\E=\F[\b]$,~$\E'=\F[\b']$.   Let~$(\V,h)$ be an~$\e$-hermitian space and~$\vphi:\E\hookrightarrow \End_\F(\V)$ and~$\vphi':\E'\hookrightarrow \End_\F(\V)$ \emph{self-dual}~$\F$-algebra embeddings.  Suppose we have a bijection~$\z:\I\rightarrow\I'$.  \orange{We say that~$(\b,\vphi)$ and~$(\b',\vphi')$ are \emph{$\z$-concordant} if, for all~$i\in\I_0$, the restrictions of~$(\b,\vphi)$ and~$(\b,\vphi')$ to~$\E_i$ and~$\E'_{\z(i)}$ respectively are concordant.} % field embeddings. %, possibly via an isometry. 

\subsection{}
We can now state our main result on endo-equivalence of self-dual pss-characters:
%\begin{theorem*}[{Theorem~\ref{thmEndoSemisimplev4}}]
%Let~$\Theta_-$ and~$\Theta'_-$ be self-dual pss-characters supported on self-dual semisimple pairs~$(k,\b)$ and~$(k,\b')$ respectively, with index sets~$\I$ and~$\I'$ respectively, and let~$\Theta$ and~$\Theta'$ be their respective lifts.  
%Then, the following assertions are equivalent:
%\begin{enumerate}
% \item The self-dual pss-characters~$\Theta_-$ and~$\Theta'_-$ are endo-equivalent;
% \item The pss-characters~$\Theta$ and~$\Theta'$ are endo-equivalent.
% \item There is a bijection~$\z:\I\rightarrow \I'$, equivariant for the action of the extensions of~$\ov{\phantom{a}}$ to~$\E$ and~$\E'$, such that all pairs of self-dual realizations with~$\z$-concordant embeddings intertwine in~$\G$
% \end{enumerate}
%\end{theorem*}

\begin{theorem*}[{Theorem~\ref{thmEndoSemisimplev4}}]
Let~$\Th_-$ and~$\Th'_-$ be self-dual pss-characters {supported on~$(k,\b)$ and~$(k,\b')$, respectively,} 
and~$\Th$ and~$\Th'$ their respective lifts.
Then, the following assertions are equivalent:
\begin{enumerate}\setlength\itemsep{5pt}
 \item %\label{part1endosemi}
 The self-dual pss-characters~$\Th_-$ and~$\Th'_-$ are endo-equivalent;
 \item %\label{part2endosemi}
 The lifts~$\Th$ and~$\Th'$ are endo-equivalent.
 \item %\label{part3endosemi} %There is a  bijection~$\z:\I\to\I'$ which commutes with~$\s$ such that: for all pairs~$\t,\t'$ of realizations of~$\Th,\Th'$ respectively on a common hermitian space~$(\V,h)$ with~$\z$-concordant embeddings, the characters~$\t,\t'$ intertwine in~$\G=\U(\V,h)$ {\color{red}with matching~$\z$}.
\bob{$\deg(\Th_-)=\deg(\Th'_-)$ and t}\shaun{here is a  bijection~$\z:\I\to\I'$ which commutes with~$\s$ \shauny{ with the following property: if~$((\V,h),\vphi,\La,r)\in\Qq_-(k,\b)$ and~$((\V,h),\vphi',\La',r')\in\Qq(k,\b')$ are such that~$(\vphi,\b)$ and~$(\vphi',\b')$ are~$\z$-concordant and~$\dim_\F\V^i=\dim_\F\V'^{\zeta(i)}$, for~$i\in\I$,} then the realizations~$\Th_-((\V,h),\vphi,\La,r)$ and~$\Th'_-((\V,h),\vphi',\La',r')$ intertwine in~$\G=\U(\V,h)$ with matching~$\z$.}
\end{enumerate}
\end{theorem*}

As a consequence of Theorems~\ref{thm:Endopss} and~\ref{thmEndoSemisimplev4}, we obtain that endo-equivalence of self-dual pss-characters is an equivalence relation.% we call the associated equivalence classes \emph{self-dual semisimple endo-classes}.

%%%%%%%%%%%%%%%%%%%%%%%%%%%%%%%%%%%%%%%%%
%\subsection{}
%An important corollary of Theorem {\color{red}\ref{???}} is a transitivity of~$\G$-intertwining of self-dual semisimple characters statement in Corollary~\ref{SelfdualSemisimpletransendoclass}, which we extend to an analogous statement for~$\G^\so$-intertwining of self-dual semisimple characters in Corollary~\ref{corSOIntEquivRel}. %  While intertwining of characters is obviously a reflexive and symmetric relation in general, it is clearly not necessarily transitive, thus the transitivity statement we obtain is a reflection of the structure in the collection of all (self-dual) semisimple characters.  
%These statements are key to our proof of intertwining implies conjugacy for cuspidal types (Theorem {\color{red}SD}) {\color{red}check now}.

%%%%%%%%%%%%%%%%%%%%%%%%%%%%%%%%%%%%%%%%%
\subsection{}
\shaun{We turn now to the notion of \emph{endo-parameter}.} 
We call a semisimple character \emph{full} if it lies in a set of semisimple characters~$\Cc(\La,0,\b)$, and we call an endo-class \emph{full} if it contains a pss-character supported on a semisimple pair~$(0,\b)$.  
\shaun{Likewise, we call a self-dual pss-character, endo-class or semisimple character~\emph{\full} if the corresponding lift is \full.  Every smooth representation of~\rob{$\tG$} (respectively~$\G^\so$) contains a \full\ (respectively, \full\ self-dual) semisimple character by~\cite[Propositions~7.5,8.5]{Finitude}.}
% This is not a strong restriction: every smooth representation of~$\tG$ contains a full semisimple character by~\cite{Finitude}. 

\gre{We call \gre{full} (self-dual) semisimple characters \emph{endo-equivalent} if they are realizations of endo-equivalent \gre{full} (self-dual) pss-characters.}  By~\cite[Intertwining Theorem]{BHIntertwiningSimple}, full simple characters of~$\tG$ intertwine if and only if they are endo-equivalent.  This not only implies that intertwining of full simple characters is transitive, it also shows that the simple endo-classes of degree~\rob{dividing~$\dim_\F(\V)$} parametrize the intertwining classes of simple characters of~$\tG$.  In the final section we prove a broad generalization of this result to semisimple and self-dual semisimple characters, introducing \emph{endo-parameters} \shaun{to parametrize the intertwining classes}.%, introducing \emph{endo-parameters}.  %The definition of which in the classical case is more involved, % by introducing \emph{endo-parameters}.  {\color{red}say why endo-parameters...}

\rob{First we recall, in the special case of \full\ characters, the transitivity of intertwining statements obtained %\full\ (self-dual) semisimple characters 
from Theorems \ref{thm:Endopss} and~\ref{thmEndoSemisimplev4}:
}% (in the special case of \full\ characters):}
\shaun{
\begin{proposition*}[Corollaries~\ref{cor:equivOfIntertwiningSschar},~\ref{cor:IntertwiningEquivalenceRelG}]
\begin{enumerate}\setlength\itemsep{5pt}
\item Suppose~$\t^{(l)}\in\Cc(\La^{(l)},0,\b^{(l)})$, for~$l=1,2,3$, are semisimple characters such that~$\t^{(1)}$ intertwines with~$\t^{(2)}$, and~$\t^{(2)}$ intertwines with~$\t^{(3)}$, and~$[\F[\beta^{(l)}]:\F]$ is independent of~$l$. Then~$\t^{(1)}$ and~$\t^{(3)}$ intertwine. 
\item Suppose~$\t^{(l)}_-\in\Cc_-(\La^{(l)},0,\b^{(l)})$, for~$l=1,2,3$, are self-dual semisimple characters such that~$\t^{(1)}_-$ intertwines with~$\t^{(2)}_-$ in~$\G$, and~$\t^{(2)}_-$ intertwines with~$\t^{(3)}_-$ in~$\G$, and~$[\F[\beta^{(l)}]:\F]$ is independent of~$l$. Then~$\t^{(1)}_-$ and~$\t^{(3)}_-$ intertwine in~$\G$. 
\end{enumerate}
\end{proposition*}
In Corollary~\ref{cor:IntertwiningEquivalenceRelGo} we prove the analogous transitivity statement for intertwining of self-dual semisimple characters in~$\G^\so$ for special orthogonal groups. This transitivity of intertwining reflects the structure in the collection of semisimple characters.}

%%%%%%%%%%%%%%%%%%%%%%%%%%%%%%%%%%%%%%%%%
\subsection{}
Let~$\Ee$ denote the set of all endo-classes of full ps-characters.  An \emph{endo-parameter} is a function~$\f$ from the set~$\Ee$ to the set~$\NN_0$ of non-negative integers, with finite support. We define the \emph{degree} of an endo-parameter~$\f$ by
\[
\deg(\f):=\sum_{c\in\Ee} \deg(c) \f(c).
\]
Our main theorem on endo-parameters for general linear groups is then:
\begin{theorem*}[Theorem~\ref{thmClassifyConjClassGL}]
The set of intertwining classes of \full\ semisimple characters for~$\tG=\GL_\F(\V)$ is in canonical bijection with the set of endo-parameters~$\f$ of degree~$\dim_\F(\V)$.
\end{theorem*}
\rob{See the statement of Theorem~\ref{thmClassifyConjClassGL} for the description of this map.}
%
%This theorem is essentially a formal consequence of earlier our results on intertwining semisimple characters.%{\color{red}Write more}%, building on earlier work of the second and third authors~\cite{SkSt}.

%%%%%%%%%%%%%%%%%%%%%%%%%%%%%%%%%%%%%%%%%
\ignore{\subsection{}
We call a self-dual pss-character, endo-class or semisimple character~\emph{\full} if the corresponding lift is \full.  Every smooth representation of~$\G^\so$ contains a full self-dual semisimple character by~\cite[Proposition 8.5]{Finitude}.   As a corollary of Theorem~\ref{thmEndoSemisimplev4}, we show that the collection of self-dual semisimple characters enjoys transitivity of~$\G$-intertwining properties, we recall in the special case of full self-dual semisimple characters:
\begin{corollary}[\ref{cor:IntertwiningEquivalenceRelG}]
Suppose~$\t^{(l)}_-\in\Cc_-(\La^{(l)},0,\b^{(l)})$, for~$l=1,2,3$, are self-dual semisimple characters such that~$\t^{(1)}_-$ intertwines with~$\t^{(2)}_-$ in~$\G$, and~$\t^{(2)}_-$ intertwines with~$\t^{(3)}_-$ in~$\G$, and~$[\F[\beta^{(l)}]:\F]$ is independent of~$l$. Then~$\t^{(1)}_-$ and~$\t^{(3)}_-$ intertwine in~$\G$. 
\end{corollary}
In Corollary~\ref{cor:IntertwiningEquivalenceRelGo} we provide the analogue transitivity statement for intertwining of self-dual semisimple characters in~$\G^\so$ for special orthogonal groups.
}

%%%%%%%%%%%%%%%%%%%%%%%%%%%%%%%%%%%%%%%%%
\subsection{}
The definition of endo-parameters for classical groups is more intricate. Let~$(0,\b)$ and~$(0,\b')$ be self-dual \emph{simple} pairs, and~$\Theta_-$ and~$\Theta'_-$ be self-dual ps-characters supported on~$(0,\b)$ and~$(0,\b')$ respectively.  If~\shaun{$\Theta_-$ and~$\Theta'_-$ are endo-equivalent then:}
%~$\Theta_-\approx\Theta'_-$, i.e.~if they define the same element of~$\Ee_-$, then 
\begin{enumerate}\setlength\itemsep{5pt}
\item\label{part1introconps}  the extensions~$\F[\b]/\F$ and~$\F[\b']/\F$ share many arithmetic invariants 
\shaun{-- in particular, by Corollary~\ref{cor:endosimilar}, % and we call 
the extensions are \emph{similar}, in the sense of Definition~\ref{def:similar};}
\item\label{part2introconps}  if~$((\V,h),\vphi,\La,r)\in\Qq_-(0,\b)$ and~$((\V,h),\vphi',\La',r')\in\Qq_-(0,\b')$, then~$\Th_{-}((\V,h),\vphi,\La,r)$ and~$\Th'_{-}((\V,h),\vphi',\La',r')$ intertwine in~$\G=\U(\V,h)$ if and only if~$(\b,\vphi)$ and~$(\b',\vphi')$ are concordant, (Proposition~\ref{prop:endoconcordinter}).
\end{enumerate}
Thus, by~\ref{part2introconps} to parametrize the~$\G$-intertwining class of a self-dual simple character we need to take into account the \emph{concordance class} of the embedding for which it is a realization of a self-dual ps-character and not just the self-dual endo-class of the ps-character. Moreover, by~\ref{part1introconps} we only need consider similar extensions.

%Let~$\Ee_-$ denote the set of all self-dual endo-classes of full self-dual ps-characters.  %The involution~$\ov{\phantom{a}}$ induces an 
\shaun{Our involutions induce an action of~$\Sigma$ on~$\Ee$% we also denote by~$\ov{\phantom{a}}$
, see Definition~\ref{defClassicalInv}, and %, in Section~\ref{subsecOrbitsAndSelfDualEndoClasses}, 
%we can identify~$\Ee_-$ with the set of orbits
we denote by~$\Ee/\Sigma$ the set of orbits. Note that orbits of length one correspond precisely to (the lifts of) endo-classes of self-dual simple ps-characters, but there are also orbits of length two.} % for~$\ov{\phantom{a}}$ acting on~$\Ee$. % For the rest of this introduction, we make the same identification.
%
%
%
%Thus when to parametrize intertwining classes of self-dual simple characters we need to create an invariant corresponding to the embed
%
%Let us collect some invariants of the intertwining classes of full self-dual semisimple characters.
%
%Let~$(0,\b)$ and~$(0,\b')$ be semisimple pairs with indexing sets~$\I$ and~$\I'$ respectively.  Let~$\t_-\in\Cc_-(\La,0,\vphi(\b))$ and~$\t_-'\in\Cc_-(\La',0,\vphi'(\b'))$ be self-dual semisimple characters which intertwine in~$\G$.  
%\begin{enumerate}
%\item They define endo-equivalent self-dual pss-characters.
%\item 
% \item (Theorem~\ref{thm:MatchingForCharForG}) As their lifts intertwine and we have a matching~$\z:\I\rightarrow\I'$ between their indexing sets which moreover commutes with the action of~$\s$.  We show that for~$i\in\I_0$, the spaces~$(\V^i,h\mid_{\V^i})$ and~$(\V^{\z(i)},h\mid_{\V^{\z(i)}})$ are isometric and the embeddings~$\vphi\mid_{\E_i}$ and~$\vphi_{\mid_{\E'_{\z(i)}}}$ are concordant via this isometry.
%\end{enumerate}
%
%
%
%
%
Using the theory of concordance, we attach to an element of~$\fo\in\Ee/\Sigma$ a
\shaun{set~$\WT(\fo)$ of invariants to carry this concordance information, which we call the set of \emph{Witt types} for~$\fo$. When~$\fo$ has cardinality one, so corresponds to the endo-class of a self-dual simple ps-character supported on some~$(0,\b)$ with~$\E=\F[\b]$, the set~$\WT(\fo)$ is in bijection with the Witt group of~$\e$-hermitian forms over~$\E/\E_\so$; on the other hand, when~$\fo$ has cardinality two,~$\WT(\fo)$ is a singleton.} 
%n invariant which carries this concordance information we call a \emph{Witt type}, the set of Witt types which can be paired with a given~$\fo$ is denoted~$\WT(\fo)$.  
A \emph{self-dual endo-parameter}~$\f_-$ is then a section of the map
\begin{align*}
\bigsqcup_{\fo\in\Ee/\Sigma} \WT(\fo)\times \mathbb{N}_0&\rightarrow \Ee/\Sigma,\qquad (w,a)_{\fo}\mapsto \fo
\end{align*}
with finite support.  %We define the degree of~$\f_-$ by
\shaun{Attached to~$\f_-$, we have its \emph{degree}~$\deg(\f_-)$ and an element~$\herm_{\F/\F_\so}(\f_-)$ of the Witt group of~$\e$-hermitian forms over~$\F/\F_\so$ (see Section~\ref{subsec:endoG}).}

We denote by~$\EP(h,\G)$ the set of self-dual endo-parameters with~$\herm_{\F/\F_0}(\f_-)=[h]$ and~$\deg(\f_-)=\dim_\F\V$, and we call it \emph{the set of endo-parameters for~$(h,\G)$}. Note that these depend not only on the isomorphism class of~$\G$, but \gre{on the isometry class of \rob{the} hermitian form}~$h$ too. Our main theorem on endo-parameters for~$\G$ is then:
\begin{theorem*}[Theorem~\ref{thmClassifyConjClassG}]
The set of intertwining classes of~full~self-dual semisimple characters for~$\G$ is in canonical bijection with the set~$\EP(h,\G)$.% of self-dual endo-parameters~$\f_-$ which satisfy~$\deg(\f_-)=\dim_\F\V$ and~$\herm_{\F/\F_0}(\f_-) =[h].$
\end{theorem*}

\rob{See the statement of Theorem~\ref{thmClassifyConjClassG} for the description of this map, which depends on the hermitian form~$h$, not only on its isometry class.}

%%%%%%%%%%%%%%%%%%%%%%%%%%%%%%%%%%%%%%%%%
\subsection{}
\shaun{In the case that~$\G^\so$ is a special orthogonal group, the partition of the set of all self-dual semisimple characters for~$\G$ into~$\G^\so$-intertwining classes is in general finer than the partition into~$\G$-intertwining classes (see Theorem~\ref{thmSOintertwining}).  It is therefore necessary to augment the set of self-dual endo-parameters of Theorem~\ref{thmClassifyConjClassG}: in Section~\ref{sect:sporthendoparameter} we define a set~$\EP(h,\G^\so)$ of endo-parameters for~$(h,\G^\so)$ and prove that it is in canonical bijection with the set of~$\G^\so$-intertwining classes of~\full~self-dual semisimple characters (Corollary~\ref{CorollaryendoparamSO}).}

\ignore{Suppose now that~$\G=\U(\V,h)$ is orthogonal, so that~$\G^\so$ is a special orthogonal group and not equal to~$\G$.  Then the partition of the set of all self-dual semisimple characters for~$\G^\so$ into~$\G^\so$-intertwining classes is in general finer than the partition into~$\G$-intertwining classes (see Theorem~\ref{thmSOintertwining}).  It is therefore necessary to augment the set of self-dual endo-parameters of Theorem~\ref{thmClassifyConjClassG}.
Let~$\f_-$ be a self-dual endo-parameter for~$(h,\G)$.  In Section~\ref{sect:sporthendoparameter}, we associate to~$\f_-$ a set~$\mathcal{H}(\f_-)$, which has cardinality two whenever it is non-trivial and which is always trivial when~$\V$ has odd dimension.  
We define the \emph{set of endo-parameters for}~$(h,\G^\so)$ by
\[
\EP(h,\G^\so):=\{(\f_-,\fy)\mid \f_-\in\EP(h,\G),\ \fy\in\mcH(\f_-)\}.
\]
\begin{theorem*}[Corollary~\ref{CorollaryendoparamSO}]
The set of~$\G^\so$-intertwining classes of~full~self-dual semisimple characters for the special orthogonal subgroup~$\G^\so$ of~$\G=\U(\V,h)$ is in canonical bijection with~$\EP(h,\G^\so)$.
\end{theorem*}
See just after Corollary~\ref{CorollaryendoparamSO} for the description of this map.
}

%%%%%%%%%%%%%%%%%%%%%%%%%%%%%%%%%%%%%%%%%
\subsection{}
%We now conjecture a compatibility of endo-parameters and loc
%
%
%We temporarily assume that~$\F_\so$ has characteristic zero and 
We now conjecture a \rob{Galois}-theoretic interpretation of endo-parameters via the conjectural local Langlands correspondence.  Although our results are for arbitrary classical groups and we expect a similar picture in that situation, we only make a precise conjecture in the case of a quasi-split classical group~$\G$.

Let~$\W_{\F_\so}$ denote the \emph{Weil group} of~$\F_\so$ with \emph{inertia subgroup}~$\I_{\F_\so}$.  Let~$\P_{\F_\so}$ denote the \emph{wild inertia subgroup} of~$\W_{\F_\so}$, that is the pro-$p$ Sylow subgroup of~$\I_{\F_\so}$.  Let~$\W_{\F_\so}'=\W_{\F_\so}\times \SL_2(\mathbb{C})$ denote the \emph{Weil--Deligne group}, and
%We temporarily assume that the algebraic group~$\mathbb{G}^\so$ underlying~$\G^\so$ is quasi-split.  Let~$\hGo$ denote the complex points of the reductive group dual to~$\mathbb{G}^\so$ 
let~$\LG=\hGo\rtimes\W_{\F_{\so}}$ the Langlands dual group of~$\G^\so$ over the complex numbers. 
%\gre{We identify~$\W_{\F_\so}$ and~$\SL_2(\mathbb{C})$ with the subgroups~$\W_{\F_\so}\times 1$ and~$1\times\SL_2(\mathbb{C})$ of~$\W_{\F_\so}'$ respectively, and~$\hGo$ with the subgroup~$\hGo\rtimes 1$ of~$\LG$.}  
For a group~$\H$ we write~$\Z(\H)$ for its centre \shaun{and~$\C_\H(\X)$ for the centralizer in~$\H$ of a subgroup~$\X$ of~$\H$}.  

Let~$(\vrho,\chi_{\vrho})$ be an (extended) \emph{Langlands parameter} for~$\G^\so$.  As these appear in various guises in the literature, we recall one formulation:
 \begin{enumerate}\setlength\itemsep{5pt}
 \item $\vrho:\W_{\F_\so}'\rightarrow \LG$ is a \orange{continuous} homomorphism such that
 \begin{enumerate}
 \item $\vrho(\I_{\F_\so})$ is finite,~$\vrho$ is Frobenius-semisimple, and~$\vrho:{\SL_2(\mathbb{C})}\rightarrow \hGo$ is algebraic, 
 \item the composition~$\W_{\F_\so}\xrightarrow{\vrho}\LG\rightarrow \W_{\F_\so}$ is the identity; 
% \item $\vrho$ is \emph{relevant} for~$\G^\so$ (an empty condition if \gre{the algebraic group underlying}~$\G^\so$ is~$\F_\so$-quasi-split, \gre{cf.~\cite[Definition 1.2.1]{KMSW}}).
 \end{enumerate}\setlength\itemsep{5pt}
  \item \gre{$\chi_{\vrho}$ is an irreducible complex representation of the group}
  \[
 \orange{\mathcal{S}_{\vrho}:=%\pi_0(\C_{\hGo}(\vrho(\W_{\F_\so}'))/ \Z(\hGo)^{\W_{\F_\so}}).}
 \C_{\hGo}(\vrho(\W_{\F_\so}'))/\C_{\hGo}(\vrho(\W_{\F_\so}'))^\circ \Z(\hGo)^{\W_{\F_\so}}.}
  \]  
%   which transforms under~$\Z(\hGo)^{\W_{\F_\so}}$ by a fixed character corresponding, via Kottwitz's isomorphism, to viewing~$\G^\so$ as an inner form of a quasi-split connected reductive group defined over~$\F_\so$.
 %  via Kottwitz's isomorphism to a  for the unique quasi-split }
 \end{enumerate}
%Two (extended) Langlands parameters~$(\vrho,\chi_{\vrho})$ and~$(\vrho',\chi_{\vrho'})$ for~$\G^\so$ are \emph{equivalent} if~$\vrho$ and~$\vrho'$ are conjugate in~$\hGo$ and this conjugation takes~$\chi_{\vrho}$ to~$\chi_{\vrho'}$.  We write~$\Lang(\G^\so)$ for the set of equivalence classes of (extended) Langlands parameters for~$\G^\so$.  
\gre{We write~$\Lang(\G^\so)$ for the set of equivalence classes of (extended) Langlands parameters for~$\G^\so$ under~$\hGo$-conjugacy.}
%It should be possible to remove our quasi-split hypothesis..... , cf.~\cite{KalMingShinWhite}.  

%{\color{red}Cut down a lot, use~$\SL 2$ form:}
The \emph{local Langlands correspondence} for~$\G^\so$ predicts a natural bijection 
\gre{(dependent on %some choices, for example in the case the algebraic group underlying~$\G^\so$ is~$\F_\so$-quasi-split it depends on
 fixing a non-degenerate character of the unipotent radical of a Borel subgroup of~$\G^\so$)}
\[
\LL:\Irr(\G^\so)\rightarrow \Lang(\G^\so),
\]
where~$\Irr(\G^\so)$ denotes the set of isomorphism classes of irreducible smooth representations of~$\G^\so$ on complex vector spaces. %, {\color{red}and this bijection depends on a choice of non-degenerate character of some unipotent group}.   
When~$\F_\so$ has characteristic zero,~$\LL$ is known \gre{for tempered representations} of split classical groups~\cite{Arthur} and quasi-split unitary groups~\cite{Mok}.  \gre{There is also work in progress in a generalization to inner forms of unitary groups~\cite{KMSW}.}  When~$\F_\so$ has positive characteristic, Arthur's results for split classical groups have been extended in some characteristics~\cite{MR3709003}.%, and there is general work of~\cite{???}.
%
% in special cases~\cite{Arthur,Mok,poschar,KalMingShinWhite} {\color{red}rewrite sentence}.

%%%%%%%%%%%%%%%%%%%%%%%%%%%%%%%%%%%%%%%%%
\subsection{}
Let~$\wrho$ be a \emph{wild inertial parameter} for~$\G^\so$, that is a homomorphism~$\wrho:\P_{\F_{\so}}\rightarrow \LG$ which extends to a Langlands parameter~$\vrho:\W_{\F_\so}'\rightarrow \LG$.  Set
\begin{align*}
\C_{\LG}(\wrho)&=\{(g,w)\in\LG \mid (g,w)\wrho(w^{-1}pw)(g,w)^{-1}=\wrho(p),~\text{for all } p\in\P_{\F_\so}\}.
%&\simeq \C_{\hGo}(\wrho)\rtimes_{\Ad_{\vrho}} \W_\F.
\end{align*}
As in~\cite{DHKM}, we notice that~$\C_{\LG}(\wrho)\simeq \C_{\hGo}(\wrho)\rtimes_{\Ad_{\vrho}} \W_{\F_\so}$ which implies that~
\begin{equation}
\label{eq1}
\Z(\hGo)^{\W_{\F_\so}}\leqslant\Z(\C_{\LG}(\wrho))\simeq \Z(\C_{\hGo}(\wrho))^{\vrho(\W_{\F_\so})}\leqslant \C_{\hGo}(\vrho(\W_{\F_\so}')),
\end{equation} 
as the centre of~$\W_{\F_\so}$ is trivial and~$\vrho(\SL_2(\mathbb{C}))\subseteq \C_{\hGo}(\wrho)$.  
\orange{We set
\[
\mathcal{S}_\wrho:=\Z(\C_{\LG}(\wrho))/\Z(\hGo)^{\W_{\F_\so}}.
\]}
By~\rob{(}\ref{eq1}\rob{)}, we thus have a map from representations of~$\mathcal{S}_{\vrho}$ to representations of~\orange{$\mathcal{S}_\wrho$} %\gre{~$\Z(\C_{\LG}(\wrho))$} 
by considering a representation of\gre{~$\mathcal{S}_{\vrho}$} as a representation of~$\C_{\hGo}(\vrho(\W_{\F_\so}'))$ (trivial on\gre{~$\C_{\hGo}(\vrho(\W_{\F_\so}'))^\circ \Z(\hGo)^{\W_{\F_\so}}$}) and restricting to~$\Z(\C_{\LG}(\wrho))$.  

%\begin{lemma}\label{lemma1}
%Let~$\vrho$ be a Langlands parameter extending~$\wrho$.  Then
%\begin{enumerate}
%\item\label{i} $\C_{\LG}(\wrho)\simeq \C_{\hGo}(\wrho)\rtimes_{\Ad_{\vrho}} \W_\F$;\item\label{ii} $\Z(\C_{\LG}(\wrho))\simeq \Z(\C_{\hGo}(\wrho))^{\vrho(\W_\F)}\leqslant \Z(\C_{\hGo}(\vrho))$.
%%\item \label{iii}$\Z(\C_{\LG}(\wrho))$ is a subgroup of~$\Z(\C_{\hGo}(\vrho))$
%\end{enumerate}
%\end{lemma}
%
%\begin{proof}
%{\color{red}remove lemma, summarise quick}
%This is implicit in~\cite{Dat} (see also~\cite{DHKM} for a more general scheme theoretic version of~\ref{i}).  We have an exact sequence
%\begin{equation*}
%1\rightarrow \C_{\hGo}(\wrho)\xrightarrow{\alpha} \C_{\LG}(\wrho)\xrightarrow{\beta} \W_\F
%\end{equation*}
%where $\alpha(g)=(g,1)$ and~$\beta(g,w)=w$.  As~$\vrho$ extends~$\wrho$,
%\[\vrho(w)\wrho(w^{-1}pw)\vrho(w^{-1})=\vrho(w)\vrho(w^{-1}pw)\vrho(w^{-1})=\vrho(p)=\wrho(p),\]
%and~$\vrho(w)\in\C_{\LG}(\wrho)$.  Moreover,~$\beta\circ\vrho=1$ as~$\vrho$ is an~$\L$-parameter, thus the extension~$\vrho$ splits~$\beta$ giving~\ref{i}. The initial isomorphism of~\ref{ii} follows from~\ref{i} as the centre of~$\W_\F$ is trivial, and the containment as~$\C_{\hGo}(\vrho)=\C_{\hGo}(\wrho)^{\vrho(\W_\F)}$.
%\end{proof}
%Let~$\vrho$ be a Langlands parameter for~$\G$ and put
%\[\rS_{\vrho}=\Z(\C_{\hGo}(\vrho))/\Z(\hGo).\]
%It is an abelian~$2$-group.  

An \emph{extended wild inertial parameter} for~$\G^\so$ is a pair~$(\wrho,\chi_{\wrho})$ such that~$\wrho:\P_{\F_\so}\rightarrow \LG$ is a homomorphism and~$\chi_{\wrho}$ is a representation of~\orange{$\mathcal{S}_\wrho$}
%~$\Z(\C_{\LG}(\wrho))$ 
such that there is an extended Langlands parameter~$(\vrho,\chi_{\vrho})$ with~$(\wrho,\chi_{\wrho})=(\vrho\vert_{\P_{\F_\so}},\chi_{\vrho}\vert_{\Z(\C_{\LG}(\wrho))})$. 
%Two extended wild inertial parameters~$(\wrho,\chi_{\wrho})$ and~$(\wrho',\chi_{\wrho'})$ for~$\G^\so$ are \emph{equivalent} if~$\wrho$ and~$\wrho'$ are conjugate in~$\hGo$ and this conjugation takes~$\chi_{\wrho}$ to~$\chi_{\wrho'}$.  We write~$\WP(\G^\so)$ for the set of equivalence classes of extended wild parameters for~$\G^\so$.
\gre{We write~$\WP(\G^\so)$ for the set of equivalence classes of extended wild \rob{inertial} parameters for~$\G^\so$ under~$\hGo$-conjugacy.}
%For any Langlands parameter~$\vrho$ extending~$\wrho$, we can consider representations of~$\mathcal{S}_\vrho$ as representations of~$\C_{\hGo}(\vrho(\W_{\F_\so}'))$ and restrict them to~$\mathcal{S}_{\wrho}$.
%
Thus we have a well-defined restriction map~$\Res:\Lang(\G^\so)\rightarrow\WP(\G^\so)$ given by~$(\vrho,\chi_{\vrho})\mapsto (\vrho\vert_{\P_{\F_\so}},\chi_{\vrho}\vert_{\Z(\C_{\LG}(\wrho))})$.
%\begin{align*}
%\Res:\Lang(\G^\so)&\rightarrow\WP(\G^\so);\\
%(\vrho,\chi_{\vrho})&\mapsto (\wrho=\vrho\vert_{\P_{\F_\so}},\chi_{\vrho}\vert_{\Z(\C_{\LG}(\wrho))}).
%\end{align*}

Let~$\EP(h,\G^\so)$ denote the set of self-dual endo-parameters for~$\G^\so%$ considered as a subgroup of~$
\subseteq\U(\V,h)$.  We have a map~$\vartheta:\Irr(\G^\so)\rightarrow\EP(h,\G^\so)$ which takes~$\pi\in\Irr(\G^\so)$ to the self-dual endo-parameter \rob{attached to the intertwining class of any full} self-dual semisimple character contained in~$\pi$, \gre{we note \rob{that} this map depends on the hermitian form~$h$}.

\begin{conjectureLLC}[Wild local Langlands]\label{conjecture:WLL}
There is a unique bijection
\[
\WLL:\EP(h,\G^\so)\rightarrow \WP(\G^\so)
\]
compatible with the local Langlands correspondence; \shaun{that is,} the following diagram commutes
\begin{equation}\label{Diagramintro}
\xymatrix{
\Irr(\G^\so) \ar[r]^{\LL}\ar[d]_{\vartheta} &\Lang(\G^\so)\ar[d]^{\Res} \\
\EP(h,\G^\so) \ar[r]^{\WLL~~} &\WP(\G^\so)\\
}
\end{equation}
\end{conjectureLLC}

%See~\cite[6.1 Theorem]{BHEffective} for general linear groups {\color{red}only in the simple case a corollary of our main result and BlHeSt gln I think does it for gln}, and 
%
In the special case of \gre{cuspidal representations} of symplectic groups, and assuming~$\LL$, work of the third author with Blondel and Henniart~\cite[Theorem 7.6]{BlHeSt} (together with Theorem~\ref{thmClassifyConjClassG} to define the map~$\vartheta$ as above) shows that if we further project from the set of endo-parameters for~$\G^\so$ by \emph{forgetting their Witt type data} and further project from~$\WP(\G^\so)$ to the set of (non-extended) wild parameters for~$\G^\so$\shaun{, then we get a bijection for which} the resulting diagram commutes.  %The main intrigue of the conjecture is in relating the subtle information contained in the Witt type data in an endo-parameter and the second part of an extended wild inertial parameter.
\subsection{}
In an orthogonal direction to Bushnell and Kutzko's generalization~\cite{BK93} of Howe's construction of cuspidal representations of~$p$-adic general linear groups in the tame case~\cite{Howe}, Yu constructed cuspidal representations of a broad class of~$p$-adic connected reductive groups~$\H$ defined over~$\F_\so$ \cite{Yutame}, a construction which Fintzen recently proved exhausts all cuspidal representations whenever the residual characteristic of~$\F_\so$ does not divide the order of the Weyl group of~$\H$~\cite{Fintzen}.  Hakim and Murnaghan~\cite{HM} considered the flexibility in the data defining Yu's cuspidal representations and developed a \emph{refactorization} procedure to classify isomorphism classes of Yu's cuspidal representations by equivalence classes of these data. It would be interesting to develop \shaun{notions of pss-characters, endo-equivalence, and endo-parameters in this setting of more general groups~$\H$}.% It is tempting to think that there is pss-characters, endo-equivalence, and endo-parameters in this setting too.
%, but at present there are no notions of \emph{pss-character} or \emph{endo-equivalence} in this setting.

\begin{ack}
This work was supported by the Engineering and Physical Sciences Research Council (EP/H00534X/1 and EP/M029719/1), the Heilbronn Institute for Mathematical Research, Imperial College London, ShanghaiTech University, and University of East Anglia.  We thank Colin Bushnell and Marie-France Vign\'eras for their interest in this work, Jean-Fran\c cois Dat for sharing details of his current work, and David Helm and Guy Henniart for helpful conversations.  \rob{Finally, we thank the referees for very helpful comments and suggestions, which have greatly improved the paper.}
\end{ack}

%% file: Endo-notation.tex
\section{Notation}\label{sec:notation}

Let~$\F/\F_\so$ be an extension of locally compact nonarchimedean local fields of odd residual characteristic~$p$, of degree at most two, and denote by~$x\mapsto\ov x$ the generator of~$\Gal(\F/\F_\so)$. %We also write~$\F_-=\{x\in\F\mid \ov x=-x\}$ so that~$\F$ is the direct sum of~$\F_\so$ and~$\F_-$, as an~$\F_\so$-vector space.
For~$\E/\F_\so$ any finite extension, we use the usual notation:~$\o_\E$ its ring of integers,~$\p_E$ its maximal ideal,~$\mathrm{k}_\E$ its residue field,~$\val_\E$ the additive valuation on~$\E$ with image~$\ZZ$. 
\shaun{We also set~$\U^n_\E=1+\p_\E^n$, for~$n\ge 1$.}
\shauns{If~$\E/\L$ is any finite extension of fields, we usually write~$\N_{\E/\L}$ for the norm map and~$\T_{\E/\L}$ for the trace map; if the fields are nonarchimedean local then we write~$e(\E/\L)$ for the ramification index and~$f(\E/\L)$ for the residue degree.}
%There is one excpetion to this: if~$\F/\F_\so$ is an extension as above then we write~$\N_{\F/\F_\so}(\F^\times)$ for the group
%\[
%\N_{\F/\F_\so}(\F^\times)=\{x\ov x\mid x\in\F^\times\};z
%\]
%if~$\F=\F_\so$ then this is in fact the group~$(\F^\times)^2$.

Let~$\CC$ be an algebraically closed field of characteristic~$\ell\neq p$. Throughout, we consider smooth representations of locally compact topological groups on vector spaces over~$\CC$.
%Henceforth, all representations considered are supposed to be smooth and acting on~$\CC$-vector spaces.

Let~$\G$ be a locally compact topological group, and let~$\H$ and~$\H'$ be compact open subgroups of~$\G$. Let~$\rho$ and~$\rho'$ be representations of~$\H$ and~$\H'$ respectively. For~$g\in\G$, we define~$\I_g(\rho,\rho')$ to be the~$\CC$-vector space
\[
\I_g(\rho,\rho')=\Hom_{\presuper{g}\H\cap \H'}(\presuper{g}\rho,\rho'),
\]
where~$\presuper{g}\H=g\H g^{-1}$ and~$\presuper{g}\rho$ is the representation of~$\presuper{g}\H$ defined by~$\presuper{g}\rho(x)=\rho(g^{-1}xg)$ for all~$x\in\presuper{g}\H$.  Moreover, we set
\[
\I_\G(\rho,\rho')=\{g\in\G:\I_g(\rho,\rho')\neq 0\}.
\]
%We say~$\rho$ and~$\rho'$ \emph{intertwine} if~$\I_\G(\rho,\rho')\neq \emptyset$, this is well-defined as~$g\mapsto g^{-1}$ defines a bijection~$\I_\G(\rho,\rho')\rightarrow \I_\G(\rho',\rho)$.\red{This last sentence is only obviously true if the restricted representations are semisimple, which is the case over~$\mathbb{C}$, or if they are characters to start with.}  For~$g\in \G$, 
We say that~$g$ \emph{intertwines $\rho$ with~$\rho'$} if~$\I_g(\rho,\rho')\neq \emptyset$, and that~$\rho$ \emph{intertwines with~$\rho'$ in~$\G$} if~$\I_\G(\rho,\rho')\neq \emptyset$. If~$\CC=\mathbb{C}$ or~$\rho$ and~$\rho'$ are characters, then the definition is symmetric, because then the map~$g\mapsto g^{-1}$ restricts to a bijection from~$\I_\G(\rho,\rho')$ to~$\I_\G(\rho',\rho)$. % if one of the sets is non-empty.  
In this case, we just say that~$\rho$ and~$\rho'$ intertwine in~$\G$. When~$\rho'=\rho$ we abbreviate~$\I_\G(\rho)=\I_\G(\rho,\rho)$.
%~$\I_\G(\rho,\rho')\neq \emptyset$.
%\gre{If the restriction has not always its image in~$\I_\G(\rho',\rho)$, then we should not say that it is generally a bijection. 
%This is the reason why I have changed this small paragraph.}

%
%
%{\color{red}Define intertwining and introduce notation for intertwining spaces}
%{\color{red}Move to Notation I in general}We denote by~$\I(\theta,\theta')$ the set of elements of~$\widetilde{G}$ which intertwine~$\theta$ with~$\theta'$, i.e.~those~$g\in \widetilde{G}$ such that
%\[\Hom_{\presuper{g}H^{r+1}(\beta,\Lambda)\cap H^{r+1}(\beta',\Lambda')}(\presuper{g}\theta,\theta')\neq 0.\]
%We put~$I_G(\theta,\theta')=I(\theta,\theta')\cap G$. 

\orange{Finally, we denote by~$\Sigma=\{1,\s\}$ an abstract group of order two, which will act on various objects.}

%% file: Endo-Witt.tex
\section{Witt groups and transfer}\label{sec:wittandtransfer}

In this section we cover the necessary background for our results from the theory of signed hermitian spaces and introduce a new notion: \emph{concordance} of self-dual embeddings of field extensions.  %%%%%%%%%%%%%%%%%%%%%%%%%%%%%%%%%%%%

%%%%%%%%%%%%%%%%%%%%%%%%%%%%%%%%%%%%
\subsection{Self-dual extensions}\label{subsec:extensions}
%%%%%%%%%%%%%%%%%%%%%%%%%%%%%%%%%%%%
We begin with some basic results on quadratic extensions. For~$\E$ a finite extension of~$\F_\so$, we write
\[
\E^{\even}=\{x\in\E^\times\mid\val_{\E}(x)\text{ is even}\}, \qquad
\E^{\odd}=\{x\in\E^\times\mid\val_{\E}(x)\text{ is odd}\}.
\]

\shaun{
\begin{lemma}\label{lem:sdexts1}
Suppose~$\F/\F_\so$ is quadratic. Then, 
\begin{equation}\label{eqNormFFo}
\N_{\F/\F_\so}(\F^\times)=((\F^\times)^2\cap\F_\so^{\even} )\cup((-(\F^\times)^2)\cap\F_\so^{\odd}).
\end{equation}
In particular: 
\begin{enumerate}\setlength\itemsep{5pt}
\item\label{lem:sdexts1newi} $-1\in(\F^\times)^2$ if and only if~$-1\in \N_{\F/\F_\so}(\F^\times)$;
\item\label{lem:sdexts1i} if~$-1\in(\F^\times)^2$ then~$\N_{\F/\F_\so}(\F^\times)=(\F^\times)^2\cap\F_\so^\times$;
\item\label{lem:sdexts1newiii} if~$\F/\F_\so$ is unramified then~$\N_{\F/\F_\so}(\F^\times)=\F_\so^{\even}$.
\end{enumerate}
\end{lemma}
}

\begin{proof}
\shaun{
The assertions~\ref{lem:sdexts1newi}--\ref{lem:sdexts1newiii} are immediate consequences of~\eqref{eqNormFFo}. We first prove
\begin{equation}\label{eqNormIntegers}
\N_{\F/\F_\so}(\o_\F^\times)=(\o_\F^\times)^2\cap\F_\so^\times.
\end{equation}
Since~$1+\mathfrak{p}_{\F_\so}$ is a subset of~$(\o_{\F_\so}^\times)^2$, by Hensel's lemma, it suffices to show 
\[
\N_{\F/\F_\so}(\o_\F^\times)/(1+\mathfrak{p}_{\F_\so})=((\o_\F^\times)^2\cap\F_\so^\times)/(1+\mathfrak{p}_{\F_\so}).
\]
Writing also~\rob{$\N_{\F/\F_\so}$} for the map on the residue field~$\mathrm{k}_\F$ induced by the norm, this is equivalent to
\[
\N_{\F/\F_\so}(\mathrm{k}_\F^\times)=(\mathrm{k}_\F^\times)^2\cap\mathrm{k}_{\F_\so}^\times.
\]
If~$\F/\F_\so$ is unramified then both sides are equal to~$\mathrm{k}_{\F_\so}^\times$, while if~$\F/\F_\so$ is ramified, then both sides are~$(\mathrm{k}_{\F_\so}^\times)^2$. Thus we have proved~\eqref{eqNormIntegers}.
}

\shaun{
Now both sides of~\eqref{eqNormFFo} are subgroups of~$\F_\so^\times$, containing the subgroup in~\eqref{eqNormIntegers}. If~$\F/\F_\so$ is unramified then~$(-(\F^\times)^2)\cap\F_\so^{\odd}$ is empty and both sides of~\eqref{eqNormFFo} are generated by the square of a \rob{uniformizer} of~$\F_\so$ and~\eqref{eqNormIntegers}, If~$\F/\F_\so$ is ramified then, if~$\w_\F$ is a uniformizer of~$\F$ satisfying~$\ov{\w_\F}=-\w_\F$, then both sides of~\eqref{eqNormFFo} are generated by~$-\w_\F^2$ and~\eqref{eqNormIntegers}. This completes the proof.
}
\end{proof}

Let~$\E=\F[\b]$ be a \rob{field} extension of~$\F$ with a distinguished generator~$\b$. If the generator for~$\Gal(\F/\F_\so)$ extends to an involution on~$\E$ which maps~$\b$~\rob{to}~$-\b$ then we say the pair~$(\E,\b)$ is a \emph{self-dual extension of~$\F/\F_\so$}. We again denote by~$x\mapsto\ov{x}$ this involution on~$\E$, and by~$\E_\so$ the subfield of fixed points. Note that, provided~$\b\ne 0$, the extension~$\E/\E_\so$ is always quadratic, since~$\b\not\in\E_\so$. 

\begin{corollary} \label{cor:sdexts1}
If~$(\E,\b)$ is a self-dual extension of~$\F/\F_\so$ with~$\b\neq 0$, then $-1\in(\E^\times)^2$ if and only if~$\b^2\in\N_{\E/\E_\so}(\E^\times)$.
\end{corollary}

\begin{proof}
By Lemma~\ref{lem:sdexts1}\ref{lem:sdexts1i}, if~$-1\in (\E^\times)^2$ then~$\b^2\in\N_{\E/\E_\so}(\E^\times)$.  Conversely, if~$-1\not\in (\E^\times)^2$ then~$\E/\E_\so$ is ramified and~$\val_\E(\beta)$ is odd since~$\ov{\b}=-\b$.  Hence~$\val_{\E_\so}(\beta^2)$ is odd, and it follows from~\eqref{eqNormFFo} that~$\b^2\not\in\N_{\E/\E_\so}(\E^\times)$. 
\end{proof}

We will also need the following lemma on norms through self-dual extensions.

\begin{lemma}\label{lem:norms}
Suppose~$\F/\F_\so$ is quadratic. Let~$\E_\so/\F_\so$ be a finite extension in an algebraic closure of~$\F$ which does not contain~$\F$ and set~$\E=\F\E_\so$. Then
\[
\N_{\E/\E_\so}(\E^\times)=\{\a\in\E_\so^\times \mid \N_{\E_\so/\F_\so}(\alpha)\in \N_{\F/\F_\so}(\F^\times)\}.
\]
\end{lemma}

We note also that, in the situation of the lemma, for~$\a\in\E_\so$, we have~$\N_{\E_\so/\F_\so}(\alpha)=\N_{\E/\F}(\alpha)$, since any~$\F_\so$-basis for~$\E_\so$ is also an~$\F$-basis for~$\E=\F\E_\so$.

\begin{proof}
We denote the right hand side of the asserted equation by~$\R_{\E/\F}$. If~$\L/\F$ is a subextension of~$\E$ and~$\L_\so=\L\cap\E_\so$ then we have
\[
\R_{\E/\F} = \{\a\in\E_\so^\times \mid \N_{\E_\so/\L_\so}(\alpha)\in \R_{\L/\F}\}
\]
so the lemma follows from the special cases where~$\E/\F$ is separable or purely inseparable.

Suppose first that~$\E/\F$ is purely inseparable, so has odd degree. Then any element of~$\F^\times_\so$ which is not in the image of~$\N_{\F/\F_\so}$ lies in~$\E_\so$ but not in~$\R_{\E/\F}$; in particular~$\R_{\E/\F}\neq\E^\times_\so$. Since certainly~$\N_{\E/\E_\so}(\E^\times)\subseteq\R_{\E/\F}\subseteq\E^\times_\so$, while~$\N_{\E/\E_\so}(\E^\times)$ has index two in~$\E^\times_\so$, it follows that~$\E^\times_\so\neq\R_{\E/\F}=\N_{\E/\E_\so}(\E^\times)$.

Now suppose~$\E/\F$ is separable, so the same is true of~$\E_\so/\F_\so$. By local class field theory, for any finite abelian extension of local fields~$\L/\K$ (contained in a given separable closure) we have the Artin reciprocity isomorphism
\[
\Art_{\L/\K}:\K^\times/\N_{\L/\K}(\L^\times)\simeq\Gal(\L/\K).
\]
Applying this to the extensions~$\E/\E_\so$ and~$\F/\F_\so$, the base change property of class field theory implies that on~$\E_\so^\times/\N_{\E/\E_\so}(\E^\times)$ we have
\[
\Res^{\E}_{\F}\circ\Art_{\E/\E_\so} = \Art_{\F/\F_\so}\circ \N_{\E_\so/\F_\so}.
\]
The restriction map induces an isomorphism~$\Gal(\E/\E_\so)\to\Gal(\F/\F_\so)$ and the Artin reciprocity maps are isomorphisms so we see that~$\Art_{\E/\E_\so}$ is trivial on the class of~$\a\in\E_\so^\times$ if and only if~$\Art_{\F/\F_\so}$ is trivial on the class of~$\N_{\E_\so/\F_\so}(\alpha)$, and the claim follows.
\end{proof}

\shaun{
Finally, we have the following result on ramification indices.
}

\shaun{
\begin{lemma}\label{lem:sdeodd}
Suppose~$(\E,\b)$ is a self-dual extension of~$\F/\F_\so$ with~$\b\ne 0$ and ramification index~$e(\E/\E_\so)=2$. Then~$\val_\E(\b)$ is odd and either
\begin{enumerate}\setlength\itemsep{5pt}
\item $\F=\F_\so$; or
\item $\F/\F_\so$ is quadratic ramified and the ramification index~$e(\E/\F)$ is odd.
\end{enumerate}
\end{lemma}
}

\begin{proof}
%%%%% \red{To write, or reference.}. %%%%%%%
\shaun{
Since~$\b=-\ov\b$ and~$e(\E/\E_\so)=2$, the first assertion is clear. If~$\F/\F_\so$ is \bob{quadratic} unramified then there is a unit~$\z\in\o_\F^\times$ such that~$\ov\z=-\z$; since~$\z\in\o_\E^\times$, this contradicts the assumption that~$\E/\E_\so$ is ramified. For the final assertion, suppose~$e(\E/\F)=e(\E_\so/\F_\so)=2r$ is even, let~$\w_\so$ be a uniformizer of~$\E_\so$ and let~$\w_\F$ be a uniformizer of~$\F$ such that~$\ov{\w_\F}=-\w_\F$; then~$\z=\w_\F\w_\so^{-r}$ is a unit of~$\E^\times$ satisfying~$\ov\z=-\z$, again contradicting the assumption that~$\E/\E_\so$ is ramified.
}
\end{proof}

%%%%%%%%%%%%%%%%%%%%%%%%%%%%%%%%%%%%
\subsection{Hermitian spaces}\label{subsec:spaces}
%%%%%%%%%%%%%%%%%%%%%%%%%%%%%%%%%%%%
Let~$\e=\pm 1$. By an~\emph{$\e$-hermitian space over~$\F/\F_\so$}, we mean a finite-dimensional~$\F$-vector space~$\V$ equipped with a non-degenerate~$\e$-hermitian form~$h:\V\times \V\to \F$, that is, a non-degenerate sesquilinear form (linear in the \emph{second} variable) such that
\[
h(\bw,\bv)=\e \ov{h(\bv,\bw)},\qquad\text{for all }\bv,\bw\in\V.
\]
Given two such spaces~$(\V_i,h_i)$, for~$i=1,2$, for the same~$\e$, we can form their \emph{orthogonal direct sum}, which is the space~$\V=\V_1\oplus\V_2$ equipped with the form~$h=h_1\oplus h_2$ defined by
\[
h(\bv_1+\bv_2,\bw_1+\bw_2)=h_1(\bv_1,\bw_1)+h_2(\bv_2,\bw_2),\qquad\text{for }\bv_i,\bw_i\in\V_i.
\]

If~$(\V,h)$ and~$(\V',h')$ are~$\e$-hermitian spaces over~$\F/\F_\so$, then an \emph{isometry} from~$(\V,h)$ to~$(\V',h')$ is an~$\F$-linear isomorphism~$f:\V\to\V'$ such that
\[
h'(f(\bv),f(\bw))=h(\bv,\bw),\qquad\text{for all }\bv,\bw\in\V.
\]
When there is such an isometry, we say that~$(\V,h)$ and~$(\V',h')$ are \emph{isometric}, and write~$(\V,h)\cong (\V',h')$, or just~$h\cong h'$ for short. Note that orthogonal direct sums behave well with respect to isometry: that is, if~$h_1\cong h'_1$ and~$h_2\cong h'_2$ then~$h_1\oplus h'_1\cong h_2\oplus h'_2$.
%if~$\V_1\cong \V'_1$ and~$\V_2\cong\V'_2$ then~$\V_1\oplus\V'_1\cong\V_2\oplus\V'_2$.

We write~$\Hh_\e(\F/\F_\so)$ for the set of isometry classes of~$\e$-hermitian spaces over~$\F/\F_\so$. It is a monoid with the operation induced by the orthogonal direct sum and identity element the (class of the) zero space.

The \emph{Gram matrix} of an~$\e$-hermitian space~$(\V,h)$ with respect to a basis~$\bv_1,\ldots,\bv_n$ is the~$n\times n$ matrix~$J$ whose~$(i,j)$-entry is~$h(\bv_i,\bv_j)$. This is an~$\e$-hermitian matrix: that is~$J^T=\e\ov J$, where~$J^T$ denotes the transpose of~$J$ and~$\ov J$ denotes the matrix obtained by applying the \rob{Galois} involution~$x\mapsto\ov x$ to each entry. The Gram matrix of~$(\V,h)$ with respect to any other basis takes the form~$\ov B^TJB$, where~$B$ is the change of basis matrix to~$\bv_1,\ldots,\bv_n$. The determinant~$\det(J)$ of the Gram matrix satisfies~$\ov{\det(J)}=\e^{\dim_\F V}\det(J)$. 

\shauns{The \emph{determinant~$\det(\V)$} (or~$\det(h)$) of an~$\e$-hermitian space~$(\V,h)$ is defined to be the class in~$\F^\times/\N_\F$ of the determinant of any Gram matrix for~$(\V,h)$, where
\[
\N_\F=\begin{cases}
\N_{\F/\F_\so}(\F^\times), &\text{ if~$\F/\F_\so$ is quadratic,}\\[3pt]
(\F^\times)^2, &\text{ otherwise.} 
\end{cases}
\]}
This is well-defined and moreover depends only on the isometry class of~$(\V,h)$. Thus we get a morphism of monoids 
\[
\det:\Hh_\e(\F/\F_\so)\to \F^\times/\N_\F.
\]

\medskip

An~$\e$-hermitian space~$(\V,h)$ is called \emph{isotropic} if there is a non-zero~$\bv\in\V$ such that~$h(\bv,\bv)=0$, and \emph{anisotropic} otherwise. (Note that the zero space is anisotropic.) In particular, we have the smallest isotropic~$\e$-hermitian space, the \emph{hyperbolic \daniel{plane}}~$(\H,h_\H)$: it is two-dimensional with basis~$\be_{-1},\be_1$ such that
\[
h_\H(\be_{-1},\be_1)=1\qquad\text{and}\qquad h_\H(\be_i,\be_i)=0,\quad\text{for }i=\pm 1.
\]
Thus the Gram matrix of~$\H$ with respect to the basis~$\be_{-1},\be_1$ is
\[
\begin{pmatrix}
0 & 1 \\ \e & 0
\end{pmatrix},
\]
so that~$\det(\H)=(-\e)\N_\F$. \rob{Up to} isometry,~$\H$ is the unique two-dimensional isotropic~$\e$-hermitian space. For~$n\ge 0$ an integer, we write~$n(\H,h_\H)$ for the orthogonal sum of~$n$ copies of~$(\H,h_\H)$. \daniel{An~$\e$-hermitian space~$(\V,h)$ isometric to~$n(\H,h_\H)$ for some~$n$ is called a~\emph{hyperbolic space}; these spaces 
possess a \emph{complete polarization}, i.e. a direct sum decomposition
\[\V=\V^{1}\oplus\V^{-1}\]
with totally isotropic spaces $\V^{1}$ and~$\V^{-1}$.} 

\begin{remark}\label{remark:hyperbolicplane}
The notation~$(\H,h_\H)$ for hyperbolic \daniel{plane} does not specify either the extension~$\F/\F_\so$ or~$\e$, which will be left implicit. We trust this will cause no confusion, even where it is used for different fields. 
\end{remark}

At the opposite extreme, we have the smallest non-trivial anisotropic spaces, which are one-dimensional when they exist. (There are no non-trivial anisotropic spaces when~$\F=\F_\so$ and~$\e=-1$, the symplectic case.)  They are given by a single element~$\a\in\F^\times$ such that~$\ov\a=\e\a$, and we denote the corresponding space (or its isometry class) by~$\la\a\ra$: it has a basis with Gram matrix~$(\a)$. 

\begin{remark}
Again, the notation~$\la \a\ra$, while standard, does not specify~$\F/\F_\so$, and for example we consider~$\la 1\ra$ as a~$(+1)$-hermitian space over different fields.
\end{remark}

The isometry class of~$\la\a\ra$ is determined precisely by the coset of~$\a$ in~$\F^\times/\N_\F$. Thus we have the following isometry classes of one-dimensional spaces in~$\Hh_\e(\F/\F_\so)$: 
\begin{itemize}\setlength\itemsep{5pt}
\item if~$\F/\F_\so$ is quadratic,~$\b\in\F^\times$ satisfies~$\ov\b=-\b$ and~$\a\in\F_\so^\times\setminus \N_{\F/\F_\so}(\F^\times)$, then 
\begin{align*}
&\la 1\ra\text{ and }\la \a\ra,\qquad \text{if }\e=1, \\
&\la \b\ra\text{ and }\la \b\a\ra,\quad\ \text{if }\e=-1; 
\end{align*}
\item if~$\F=\F_\so$ has uniformizer~$\varpi$, and~$\a$ is a non-square unit of~$\F^\times$, then 
\[
\la 1\ra,\ \la\a\ra,\ \la\varpi\ra,\text{ and }\la\varpi\a\ra.
\]
\end{itemize}
Any anisotropic space is an orthogonal direct sum of one-dimensional anisotropic subspaces so, with respect to a suitable basis, has a diagonal Gram matrix.

\begin{remark}\label{remark:MaxElement}
%\red{Reference} %%%%%%%%%%%%%%%%%%%%%%%%%%%%%%%%%%%%
Up to isomorphism, there is a unique maximal anisotropic~$\e$-hermitian space over~$\F/\F_\so$. More precisely, and with the notation above, it is
\[\begin{cases}
\la 1 \ra \oplus \la-\a\ra, &\text{if~$\F/\F_\so$ is quadratic and }\e=1, \\
\la \b \ra \oplus \la-\b\a\ra, &\text{if~$\F/\F_\so$ is quadratic and }\e=-1, \\
\la 1 \ra \oplus \la-\a\ra\oplus\la \varpi \ra \oplus \la-\varpi\a\ra,\quad &\text{if~$\F=\F_\so$ and }\e=1, \\
\bs 0 &\text{if~$\F=\F_\so$ and }\e=-1.
\end{cases}\]
\end{remark}

By Witt's Theorem, for any~$\e$-hermitian space~$(\V,h)$, we have an isometry
\[
(\V,h)\cong n(\H,h_\H)\oplus(\V_{\an},h_{\an}),
\]
with~$(\V_{\an},h_{\an})$ an anisotropic space; moreover, the \emph{Witt index}~$n$ and the isometry class of~$(\V_{\an},h_{\an})$ are uniquely determined by~$(\V,h)$. We write~$[h]$ for the isometry class of~$(\V_{\an},h_{\an})$ and call it the \emph{anisotropic class} of~$(\V,h)$. We also write~$\diman(\V)=\dim_{\F}(\V_{\an})$ and call it the \emph{anisotropic dimension} of~$(\V,h)$.

\begin{remark}
If~$\F/\F_\so$ is quadratic, the isometry class of an~$\e$-hermitian space~$(\V,h)$ is uniquely determined by the pair~$(\dim(\V),\det(\V))$. If~$\F=\F_\so$ and~$\e=-1$ (the \emph{symplectic case}) then the isometry class of an~$\e$-hermitian space~$(\V,h)$ is uniquely determined by~$\dim(\V)$, which is necessarily even.
\end{remark}

%%%%%%%%%%%%%%%%%%%%%%%%%%%%%%%%%%%%
\subsection{Unitary groups}\label{subsec:groups}
%%%%%%%%%%%%%%%%%%%%%%%%%%%%%%%%%%%%

Let~$(\V,h)$ be an~$\e$-hermitian space over~$\F/\F_\so$. The ring~$\End_\F(\V)$ is equipped with the adjoint anti-involution~$a\mapsto\ov a$ induced by~$h$, defined by
\[
h(a\bv,\bw)=h(\bv,\ov a\bw), \quad\text{ for all }\bv,\bw\in\V.
\]
We set
\[
\U(\V,h)=\{a\in\Aut_\F(\V): a\overline{a}=1\}
\]
which is the group of all isometries from~$\V$ to itself. This is the group of~$\F_\so$-points of a reductive group defined over~$\F_\so$: more precisely, it is a unitary group if~$\F/\F_\so$ is quadratic, a symplectic group if~$\F=\F_\so$ and~$\e=-1$, and a \emph{full} orthogonal group if~$\F=\F_\so$ and~$\e=1$.  

\begin{remark} 
If~$(\V,h)$ and~$(\V',h')$ are isometric~$\e$-hermitian space over~$\F/\F_\so$ then the isometry induces an isomorphism~$\U(\V,h)\simeq\U(\V',h')$. The converse, however, is false: for example, if~$\F\ne\F_\so$ and~$n$ is odd then there are two isometry classes of~$n$-dimensional hermitian spaces over~$\F/\F_\so$ but their isometry groups are isomorphic.
\end{remark}

We introduce the following useful technique which will sometimes allow us to reduce to cases which are easier to treat (in particular, the non-symplectic case). Given an element~$a\in \End_\F(\V) $ such that~$\ov{a}=\eta a$, with~$\eta=\pm$, we define the \emph{twisted form}~$a^*(h)$ on~$\V$ by 
\[
a^*(h)(\bv,\bw)= h(\bv,a\bw),\qquad\text{for }\bv,\bw\in\V.
\]
If~$a$ is invertible then~$(\V,a^*(h))$ is an~$\eta\e$-hermitian space over~$\F/\F_\so$. Moreover, the adjoint anti-involution on~$\End_\F(\V)$ induced by the form~$a^*(h)$ is given by~$b\mapsto a^{-1}\ov{b}a$, for~$b\in\End_\F(\V)$.

A particular case of this twisting occurs when~$a=\g\in\F^\times$ satisfies~$\ov\g=\eta\g$. Given such a~$\gamma$, the twisted form~$\g^*(h)$ makes sense for \emph{any}~$\e$-hermitian space~$(\V,h)$ over~$\F/\F_\so$.

%%%%%%%%%%%%%%%%%%%%%%%%%%%%%%%%%%%%
\subsection{Witt groups}\label{subsec:witt}
%%%%%%%%%%%%%%%%%%%%%%%%%%%%%%%%%%%%
The \emph{Witt group}~$\Ww_\e(\F/\F_\so)$ is defined to be the set of isometry classes of anisotropic~$\e$-hermitian spaces over~$\F/\F_\so$, equipped with the operation induced by taking the orthogonal sum, that is, the unique (well-defined) operation such that the following commutes:
\[
\xymatrix{
\Hh_\e(\F/\F_\so)\times \Hh_\e(\F/\F_\so) \ar[r]^{\ \qquad\oplus}\ar[d] & \Hh_\e(\F/\F_\so)\ar[d] \\
\Ww_\e(\F/\F_\so)\times \Ww_\e(\F/\F_\so) \ar[r]^{\ \qquad\oplus} & \Ww_\e(\F/\F_\so)\\
}
\]
where the map~$\Hh_\e(\F/\F_\so)\to\Ww_\e(\F/\F_\so)$ sends the isometry class of~$(\V,h)$ to its anisotropic class~$[h]$. We will sometimes refer to elements of the Witt group as \emph{Witt towers}: that is, we will identify an element of~$\Ww_\e(\F/\F_\so)$ with its fibre under the map~$\Hh_\e(\F/\F_\so)\to\Ww_\e(\F/\F_\so)$. Note that~$\Ww_\e(\F/\F_\so)$ is an abelian group, and the inverse of the isometry class of an anisotropic space~$(\V,h)$ is given by the class of~$(\V,-h)$. We write~$\bs 0$ for the identity in~$\Ww_\e(\F/\F_\so)$, which is the Witt tower of sums of hyperbolic \daniel{planes}~$n\H$.

The structure of the Witt group is given by the following proposition, where~$\C_n$ denotes the cyclic group of order~$n$.

\begin{proposition}
%\red{Reference} %%%%%%%%%%%%%%%%%%%%%%%%%%%%%%%%%%%%
\begin{enumerate}\setlength\itemsep{5pt}
\item Unitary case: if~$\F/\F_\so$ is quadratic then~$\Ww_\e(\F/\F_\so)$ is of order~$4$ and
\[
\Ww_\e(\F/\F_\so) \simeq \begin{cases} \C_2\times \C_2 \quad&\text{ if }-1\in \N_\F,\\
\C_4 &\text{ otherwise.}
\end{cases}
\]
\item Symplectic case: if~$\F=\F_\so$ and~$\e=-1$ then~$\Ww_\e(\F/\F_\so)$ is trivial.
\item Orthogonal case: if~$\F=\F_\so$ and~$\e=1$ then~$\Ww_\e(\F/\F_\so)$ is of order~$16$ and
\[
\Ww_\e(\F/\F_\so) \simeq \begin{cases} \C_2\times \C_2\times \C_2\times \C_2 \quad&\text{ if }-1\in \N_\F,\\
\C_4\times \C_4 &\text{ otherwise.}
\end{cases}
\]
\end{enumerate}
\end{proposition}

The Witt group is generated by (the classes of) one-dimensional anisotropic spaces~$\la\a\ra$ and we use the same notation to represent the class in~$\Ww_\e(\F/\F_\so)$. For example, we see that~$\la 1\ra\oplus\la 1\ra=\bs 0$ in~$\Ww_1(\F/\F_\so)$ if and only if~$-1\in\N_\F$. We also call the (class of) the unique maximal anisotropic~$\e$-hermitian space over~$\F/\F_\so$  the \emph{maximal} element of~$\Ww_\e(\F/\F_\so)$, or the \emph{Witt tower of maximal anisotropic dimension}. \shauny{ (See Remark~\ref{remark:MaxElement} above for an explicit description of this maximal element.)}  

If~$\g\in\F^\times$ satisfies~$\ov\g=\eta\g$, with~$\eta=\pm$, and~$(\V,h)$ is an~$\e$-hermitian space over~$\F/\F_\so$, then we defined the twisted form~$\g^*(h)$ on~$\V$ in the previous subsection 
%: 
%\[
%\g^*(h)(\bv,\bw)= h(\bv,\bw)\g, \quad\text{ for all }\bv,\bw\in\V,
%\]
so that~$(\V,\g^*(h))$ is an~$\eta\e$-hermitian space. Twisting by~$\g$ preserves orthogonal direct sums, isometries and hyperbolic spaces and thus induces a homomorphism
\[
\g*:\Ww_\e(\F/\F_\so) \to \Ww_{\e\eta}(\F/\F_\so),
\]
which is an isomorphism since it has inverse~$(\g^{-1})^*$.

%%%%%%%%%%%%%%%%%%%%%%%%%%%%%%%%%%%%
\subsection{Transfer}\label{subsec:transfer}
%%%%%%%%%%%%%%%%%%%%%%%%%%%%%%%%%%%%
Let~$(\E,\b)$ be a self-dual extension of~$\F/\F_\so$ and set~$n=[\E:\F]$. Let~$\l:\E\to\F$ be any non-zero~$\F$-linear form on~$\E$ which is \rob{Galois-}equivariant, that is
\[
\l(\ov{x})=\ov{\l(x)},\quad\text{ for all~$x\in\E$.}
\]
Such forms always exist and we have the particular~$\F$-linear form~$\l_\b$ given by setting
\[
\l_\b(\b^i)=\begin{cases} 1 &\text{ for }i=0,\\
0 &\text{ for }1\le i\le n-1.
\end{cases}
\]
Moreover, every such form can be written uniquely as~$\l(x)=\l_\b(\g x)$, for some~$\g\in\E_\so^\times$: indeed, every non-trivial~$\F$-linear form can be written in this way for some~$\g\in\E^\times$, and the \rob{Galois}-equivariance implies that~$\g\in\E_\so$.

\shaun{Now suppose~$\E/\F$ is any finite extension of degree~$n$ to which the \rob{Galois} involution on~$\F$ extends, with fixed field~$\E_\so$.} 
%% I think we want this because we will be using intermediate extensions where there is not necessarily an identified element \beta, so we can't just say ``a self-dual extension.''
Let~$(\V,h)$ be an~$\e$-hermitian space over~$\E/\E_\so$ and let~$\l:\E\to\F$ be a non-zero \rob{Galois-}equivariant~$\F$-linear form on~$\E$. Then it is easy to check that~$(\V,\l\circ h)$ is an~$\e$-hermitian space over~$\F/\F_\so$, called the \emph{transfer}~$(\V,\l^* h)$ of~$(\V,h)$. Transfer preserves orthogonal direct sums and isometries, so induces a morphism of monoids
\[
\l^*:\Hh_\e(\E/\E_\so) \to \Hh_\e(\F/\F_\so)
\]
Moreover, we have~$\l^*(\H)=n\H$, so it also induces a group homomorphism of Witt groups
\[
\l^*:\Ww_\e(\E/\E_\so) \to \Ww_\e(\F/\F_\so)
\]
This map depends on the choice of~$\l$ but nonetheless all these maps share some properties.

\begin{proposition}\label{prop:lambda*}
\begin{enumerate}\setlength\itemsep{5pt}
\item\label{prop:lambda*.i} The image~$\lambda^*(\Ww_\e(\E/\E_\so))$ is independent of the choice of~$\l$.
\item\label{prop:lambda*.ii} The map~$\l^*$ sends the maximal element of~$\Ww_\e(\E/\E_\so)$ to the maximal element of~$\Ww_\e(\F/\F_\so)$.
\end{enumerate}
\end{proposition}

\begin{proof} For arbitrary choices~$\l$ and~$\l'$, we know that~$\l(x)=\l'(\g x)$, for some~$\g\in\E_\so^\times$, so that~$\l^*=\l'^*\circ \g^*$. %, where we recall that~$\g^*$ is induced by twisting by~$\g$. 
Now~\ref{prop:lambda*.i} follows since~$\g^*$ is an isomorphism. On the other hand,~\ref{prop:lambda*.ii} is given by~\cite[Theorem~4.4]{SkSt} for a particular linear form, and follows in general by the proof of~\ref{prop:lambda*.i} since~$\g^*$ maps the maximal element to itself.
\end{proof}

The transfer map~$\l^*$ is in general neither injective nor surjective, as can be seen by taking~$\E/\F$ of even degree. However, we have the following rather surprising result.  We write~$\Ww^\even_\e(\E/\E_\so)$ for the subgroup of~$\Ww_\e(\E/\E_\so)$ consisting of Witt towers of even anisotropic dimension; if~$\Ww_\e(\E/\E_\so)$ is non-trivial then it is a subgroup of index two, and we write~$\Ww^\odd_\e(\E/\E_\so)$ for its non-identity coset, consisting of Witt towers of odd anisotropic dimension. 

\begin{proposition}\label{prop:transferinj}
Let~$(\E,\b)$ be a self-dual extension of~$\F/\F_\so$ and suppose we are not in the symplectic case:~$\F/\F_\so$ is quadratic or~$\e=1$. Then the restrictions of~$\l^*$ to~$\Ww^\even_\e(\E/\E_\so)$ and to~$\Ww^\odd_\e(\E/\E_\so)$ are both injective. 
\end{proposition}

\begin{proof}
Note that by the choice of~$\E$ we have~$\beta=0$ or~$\E\neq \E_\so$. The proof of Proposition~\ref{prop:lambda*} also shows that~$\l^*(\Ww^\even_\e(\E/\E_\so))$ and~$\l^*(\Ww^\odd_\e(\E/\E_\so))$ do not depend on the choice of~$\l$ so it is sufficient to prove the result for a single choice of~$\l$. 

If~$\b=0$ then~$\l_\b^*$ is the identity, and the result is immediate. If~$\E\neq \E_\so$ then~$\Ww^\even_\e(\E/\E_\so)$ and~$\Ww^\odd_\e(\E/\E_\so)$ each contain two elements, whose difference is always the maximal element of~$\Ww_\e(\E/\E_\so)$; injectivity follows, since the image of this maximal element is non-zero, by Proposition~\ref{prop:lambda*}\ref{prop:lambda*.ii}.
%The result is transitive in the extension~$\E/\F$ so may assume~$\E/\F$ is either separable or purely inseparable. In the unitary case, by twisting by an element~$\g\in\F$ such that~$\ov\g=-\g$ if necessary, we may assume~$\e=+1$; moreover,  composing then~$\l^*$ with the transfer~$\T_{\F/\F_\so}^*:\Ww_1(\F/\F_\so)\to\Ww_1(\F_\so/\F_\so)$ induced by the trace map~$\T_{\F/\F_\so}$, we see that it is sufficient to prove the result in the orthogonal case. 
%
%Thus we assume~$\F=\F_\so$ and~$\e=1$. For separable extensions~$\E/\F$, the statement for~$\l=\T_{\E/\F}$ is proved in~\cite{Milnor}, where it is also explained that the result holds for inseparable extensions (cf.~\cite[Remark~1.4]{Milnor}). A different proof for inseparable extensions can be found in the proof of~\cite[Theorem~4.4]{SkSt}. 
\end{proof}

We will also need more precise information on the transfer map in particular instances.

\begin{proposition}\label{prop:transfer}
Let~$(\E,\b)$ be a self-dual extension of~$\F/\F_\so$ and set~$n=[\E:\F]$.
\begin{enumerate}\setlength\itemsep{5pt}
\item 
For~$(\V,h)$ an~$\e$-hermitian space over~$\E/\E_\so$, we have
\[
\det(\l^*(\V))=\det(\l^*\la 1\ra)^{\dim_\E(\V)} \N_{\E/\F}(\det(\V)).
\]
\item\label{prop:transfer.ii} 
In~$\Ww_1(\F/\F_\so)$, we have:
\[
\l_\b^*(\la 1\ra) = \begin{cases} 
\la 1\ra &\text{ if~$n$ is odd,} \\
\la 1\ra \oplus \la (-1)^{\frac n2+1}\N_{\E/\F}(\b)\ra \quad &\text{ otherwise,}
\end{cases}
\]
and in~$\Ww_{-1}(\F/\F_\so)$: 
\[
\l_\b^*(\la \b\ra) = \begin{cases} 
\la (-1)^{\frac{n-1}{2}}\N_{\E/\F}(\b)\ra &\text{ if~$n$ is odd,} \\
\bs 0 \quad &\text{ otherwise.}
\end{cases}
\]
\end{enumerate}
\end{proposition}

\begin{proof}
The analogue of these statements for the transfer of quadratic forms are proved by Scharlau in~\cite[Lemma~5.8,~Theorem~5.12]{Scharlau}. The hermitian case follows \emph{mutatis mutandis}, taking care of the extra signs which appear; for this reason, we sketch the proof of~\ref{prop:transfer.ii}. Suppose~$\b$ has minimal polynomial~$X^n+b_{n-1}X^{n-1}+\cdots+b_1X+b_0$. Then we can easily calculate the Gram matrix of the~$\e$-hermitian \emph{space}~$\l_\b^*(\la 1\ra)$ with respect to the basis~$1,\b,\ldots,\b^{n-1}$, which looks like
\[
\begin{pmatrix}
\phantom{\vdots}1& 0 &\cdots&\cdots&\cdots& 0\\
\phantom{\vdots}0& 0& \cdots& \cdots&0& +b_0\\
\vdots&\vdots&&\bdots& -b_0&\star\\
\vdots&\vdots&\bdots&\bdots&\star&\star\\
\vdots&0&(-1)^{n-1}b_0&\star&\star&\star\\
\phantom{\vdots}0&(-1)^n b_0&\star&\star&\star&\star
\end{pmatrix}
\]
If~$n=[\E:\F]$ is odd then the space~$\l_\b^*(\la 1\ra)$ is the orthogonal direct sum of the subspace~$\la 1\ra$ spanned by~$1$ and the subspace~$X$ spanned by~$\b,\b^2,\ldots,\b^{n-1}$; but~$X$ has a totally isotropic subspace of half its dimension, generated by~$\b,\b^2,\ldots,\b^{\frac{n-1}2}$, hence is hyperbolic. Thus~$\l_\b^*(\la 1\ra)\cong\frac{n-1}2\H\oplus\la 1\ra$.

Similarly, if~$n=[\E:\F]$ is even then we find that~$\l_\b^*(\la 1\ra)\cong\frac{n-2}2\H\oplus\la 1\ra\oplus\la\g\ra$, where~$\g=(-1)^{\frac n2+1}\N_{\E/\F}(\b)$. The proof of the second assertion in~\ref{prop:transfer.ii} is similar.
\end{proof}

\begin{remark}
In the case that~$\F/\F_\so$ is quadratic, since the isometry class of an~$\e$-hermitian space is determined by its dimension and its determinant modulo the norm group~$\N_\F$, Proposition~\ref{prop:transfer} %, together with Lemma~\ref{lem:norms},
 completely characterizes the standard transfer map~$\l_\b^*$, for~$(\E,\b)$ a self-dual extension of~$\F/\F_\so$.
\end{remark}

%%%%%%%%%%%%%%%%%%%%%%%%%%%%%%%%%%%%
\subsection{Embeddings}\label{subsec:embeddings}
%%%%%%%%%%%%%%%%%%%%%%%%%%%%%%%%%%%%
Eventually, we will need to make comparisons of Witt towers for different self-dual extensions. In a first instance, we begin by considering the case of the same extension but embedded in different ways. Thus let~$(\E,\b)$ be a self-dual extension of~$\F/\F_\so$, and fix a non-zero \rob{Galois-}equivariant~$\F$-linear form~$\lambda$ as in the previous subsection. 

Let~$(\V,h)$ be an~$\e$-hermitian space over~$\F/\F_\so$ and let~$\A=\End_\F(\V)$. We say that an embedding~$\varphi:\E\into\A$ is \emph{self-dual} if
\[
\varphi(\ov{x})=\ov{\varphi(x)},\quad\text{ for all }x\in\E,
\]
where we recall that~$x\mapsto \ov x$ denotes the \rob{Galois} involution on~$\E$, while, on the right hand side,~$a\mapsto\ov a$ is the adjoint anti-involution on~$\A$. Such an embedding gives~$\V$ the structure of an~$\E$-vector space, and we write~$\V_\varphi$ when we want to emphasize that we are considering~$\V$ as an~$\E$-vector space via~$\varphi$ in this way. The~$\F$-linear map
\begin{align*}
\Hom_\E(\V_\varphi,\E)&\to \Hom_\F(\V,\F) \\
\psi&\mapsto \l\circ\psi
\end{align*}
is an isomorphism of~$\F$-vector spaces. For each~$\bv\in\V$, there is a unique~$\E$-linear map~$\psi_{\bv}\in\Hom_\E(\V_\varphi,\E)$ such that~$h(\bv,-)=\l\circ\psi_{\bv}$ and we define~$h_\varphi:\V_\varphi\times\V_\varphi\to\E$ by
\[
h_\varphi(\bv,\bw)=\psi_{\bv}(\bw),\qquad\text{for }\bv,\bw\in\V_\varphi.
\]

\begin{lemma}[{\cite[Lemma~5.3]{BS09}}]\label{lem:hphi}
The map~$h_\varphi:\V_\varphi\times\V_\varphi\to\E$ is a nondegenerate~$\e$-hermitian form. Moreover, it is the unique~$\e$-hermitian form on~$\V_\varphi$ such that~$h(\bv,\bw)=\l(h_\varphi(\bv,\bw))$, for all~$\bv,\bw\in\V$.
\end{lemma}

Suppose now that we have a second self-dual embedding~$\varphi':\E\into\A$ and let~$(\V_{\varphi'},h_{\varphi'})$ be the corresponding~$\e$-hermitian space over~$\E/\E_\so$. 
%An observation of the second author in~\cite{SkodField}, is the following useful corollary of Lemma~\ref{lem:hphi}.
\daniel{We have the following useful corollary of Lemma~\ref{lem:hphi}.}

\begin{corollary}[{\cite[Proposition~1.3]{SkodField}}]\label{cor:isomconj}
The~$\e$-hermitian spaces~$(\V_{\varphi},h_{\varphi})$ and~$(\V_{\varphi'},h_{\varphi'})$ over~$\E/\E_\so$ are isometric if and only if the embeddings~$\varphi,\varphi'$ are conjugate in~$\U(\V,h)$.
\end{corollary}

\begin{proof}
\daniel{
Any isometry from~$(\V_{\varphi},h_{\varphi})$ to~$(\V_{\varphi'},h_{\varphi'})$ is an element of~$\U(\V,h)$ which conjugates~$\varphi$ to~$\varphi'$. Conversely, an element~$g$ of~$\U(\V,h)$ conjugating~$\varphi$ to~$\varphi'$ is an isometry from~$(\V_{\varphi},h_{\varphi})$ to~$(\V_{\varphi'},h_{\varphi'})$, because~$h_{\varphi'}\circ(g\times g)$ and~$h_{\varphi}$ coincide by the uniqueness part of Lemma~\ref{lem:hphi}.
}
\end{proof}

\begin{remark}\label{remark:ConjugacyClassesOfEmbedding}
Suppose we have a self-dual embedding~$\varphi$ of~$\E$ into~$\A$. 
\daniel{
Since~$\dim_\E \V_\varphi$ is independent of the embedding, Corollary~\ref{cor:isomconj} implies that the~$\U(\V,h)$-orbits of self-dual embeddings of~$\E$ are in bijection with the set of classes in the fibre of the transfer map~$\l^*$ above~$[h]$ which have dimension of the same parity as~$\dim_\E \V_\varphi$. In particular, Proposition~\ref{prop:transferinj} then implies the following:
\begin{enumerate}
\item\label{remark:ConjugacyClassesOfEmbedding-i} provided we are not in the symplectic case, there is a \emph{unique}~$\U(\V,h)$-orbit of self-dual embeddings of~$\E$; 
\item\label{remark:ConjugacyClassesOfEmbedding-ii} in the symplectic case there are precisely two orbits of embeddings if~$\b\ne 0$.
\end{enumerate}}
\end{remark}

It will be useful also to observe the relationship between this lifting process of forms and the twisting of forms introduced previously. The following lemma comes immediately from the definitions and the uniqueness in Lemma~\ref{lem:hphi}.

\begin{lemma}\label{lem:hphitwist}
Let~$(\V,h)$ be an~$\e$-hermitian space over~$\F/\F_\so$, let~$\A=\End_\F(\V)$ and let~$\varphi:\E\into\A$ be a self-dual embedding. Then
\[
\b^*(h_\varphi)=(\b^* h)_\varphi.
\]
\end{lemma}

Note that this asserts the equality of two~$(-\e)$-hermitian~$\E/\E_\so$-forms on the space~$V_\varphi$.

%%%%%%%%%%%%%%%%%%%%%%%%%%%%%%%%%%%%
\subsection{Comparison}\label{subsec:comparison}
%%%%%%%%%%%%%%%%%%%%%%%%%%%%%%%%%%%%
We suppose now that we are given two self-dual extensions~$(\E,\b)$ and~$(\E',\b')$ of~$\F/\F_\so$, so that we have Witt groups~$\Ww_\e(\E/\E_\so)$ and~$\Ww_\e(\E'/\E'_\so)$. We assume moreover that~$\b,\b'$ are both non-zero, so that these Witt groups are both of order four. There are then unique bijections
\[
\sw_{\e,\b',\b}:\Ww_\e(\E/\E_\so) \to \Ww_\e(\E'/\E'_\so)
\]
which preserve anisotropic dimension and such that
\[
\sw_{-1,\b',\b}(\la\b\ra)=\la\b'\ra,\quad\text{ and }\quad
\sw_{1,\b',\b}(\la\b^2\ra)=\la\b'^2\ra.
\]
We will use these maps~$\sw_{\e,\b',\b}$ to compare self-dual embeddings of~$\E$ and of~$\E'$ in~$\e$-hermitian~$\F/\F_\so$-spaces. It is useful to notice that the maps are related via twisting:
\begin{equation}\label{eq:twistw}
\b'^*\circ\,\sw_{-1,\b',\b}=\sw_{1,\b',\b}\circ\b^*.
\end{equation}
We will sometimes skip the subscripts~$\b,\b'$ and just write~$\sw_\e$ if~$\b,\b'$ are fixed.

\begin{remark}
Since there are two bijections~$\Ww_\e(\E/\E_\so) \to \Ww_\e(\E'/\E'_\so)$ which preserve anisotropic dimension, the choice for~$\sw_{\e,\b',\b}$ made above may seem arbitrary -- for example, in the case~$\e=1$ one could instead have chosen the bijection sending~$\la 1\ra$ to~$\la 1\ra$. However, we will see that this choice is better suited to compatibility with the distinguished transfer maps~$\l_\b^*$ and~$\l_{\b'}^*$. 
\end{remark}

The relationship between~$\sw_{\e,\b',\b}$ and the bijection sending~$\la 1\ra$ to~$\la 1\ra$ will prove to be an important consideration. As an immediate consequence of Corollary~\ref{cor:sdexts1}, we have:

\shaun{
\begin{lemma}\label{lemma:w1oneifsquares}
In the situation above, the following are equivalent:
\begin{enumerate}\setlength\itemsep{5pt}
\item $\sw_{1,\b',\b}(\la 1\ra)=\la 1\ra$;
\daniel{\item either~$-1$ belongs to both~$(\E^\times)^2$ and~$(\E'^\times)^2$ or it belongs to neither of them;}
%\item either~$-1\in (\E^\times)^2\cap (\E'^\times)^2$ or~$-1\not\in(\E^\times)^2\cup (\E'^\times)^2$;
\item $-1\in (\F^\times)^2$ or the residue class degrees~$f(\E/\F),f(\E'/\F)$ have the same parity.
\end{enumerate}
\end{lemma}
}

In the case~$\b=\b'=0$ which we have so far excluded, we define~$\sw_{\e,0,0}$ to be the identity on~$\Ww_\e(\F/\F_\so)$. In \rob{both} cases, we now have the following maps of Witt groups:
\[
\xymatrix{
\Ww_\e(\E/\E_\so) \ar[r]^{\sw_{\e,\b',\b}}\ar[d]_{\l_\b^*} & \Ww_\e(\E'/\E'_\so)\ar[d]^{\l_{\b'}^*} \\
\Ww_\e(\F/\F_\so) & \Ww_\e(\F/\F_\so) \\
}
\]

\begin{remark}\label{rem:isoconcordance}
Suppose~$f:\E'\to\E$ is an~\bob{$\F$-linear} isomorphism of fields such that~$f(\b')=\b$. From the definitions, we see that~$\l_{\b'}=\l_\b\circ f$ and that the map~$\sw_{\e,\b',\b}$ is likewise induced by composition with~$f$: that is, it is induced by the map
\begin{align*}
\Hh_\e(\E/\E_\so)&\to \Hh_\e(\E'/\E'_\so) \\
(\V_\E,h_\E) &\mapsto (\V_\E,f^{-1}\circ h_\E).
\end{align*}
Moreover, if~$(\V,h)$ is an~$\e$-hermitian~$\F/\F_\so$-space and~$\varphi:\E\into\End_\F(\V)$ is a self-dual embedding, then checking the definitions from the previous subsection shows that
\[
h_\varphi=f\circ h_{\varphi\circ f},
\]
so that~$\sw_{\e,\b',\b}([h_\varphi])=[h_{\varphi\circ f}]$.
\end{remark}

Suppose now we are given~$\e$-hermitian~$\F/\F_\so$-spaces~$(\V,h)$ and~$(\V',h')$ which are isometric. Set~$\A=\End_\F(\V)$ and~$\A'=\End_\F(\V')$, suppose we have self-dual embeddings~$\varphi:\E\into\A$ and~$\varphi':\E'\into\A'$. Then we get elements~$[h_\varphi]$ and~$[h_{\varphi'}]$ of the respective Witt groups such that~$\l_\b^*([h_\varphi])=[h]=[h']=\l_{\b'}^*([h'_{\varphi'}])$, as in the previous subsection.

\begin{definition}\label{def:concordance}
The pairs $(\b,\varphi)$ and $(\b',\varphi')$ are \emph{$(h,h')$-concordant} (or just \emph{concordant} if~$h=h'$ and the form is clear from context), if~$\b$ and~$\b'$ are either both zero or both non-zero, and~$\sw_{\e,\b',\b}([h_\varphi])=[h'_{\varphi'}]$.
\end{definition}

\shaun{
\begin{remarks}\label{rem:concordance}
\begin{enumerate}\setlength\itemsep{5pt}
\item It is immediate from the definition that concordance is an equivalence relation. 
%% This is true but perhaps slightly misleading: we really mean~$(h,h')$-concordance is an ``equivalence relation'' but that has to be made to make sense!
\item\label{rem:concordance.ii} In the special case that~$\b=\b'$, so that~$\sw_{\e,\b,\b}$ is the identity map, and~$h=h'$, it follows from Corollary~\ref{cor:isomconj} that~$(\b,\varphi)$ and~$(\b,\varphi')$ are concordant if and only if~$\varphi(\b)$ and~$\varphi'(\b)$ are conjugate by an element of~$\U(\V,h)$.
%% \item\label{concordance.iii} Again if~$\b=\b'$ but now we have~$g$ is an isometry from~$h$ to~$h'$ then~$(\b,\varphi)$ and~$(\b,{}^g\varphi)$ are~$(h,h')$-concordant.
\item If we have an isomorphism~$f:\E'\to\E$ such that~$f(\b')=\b$, Remark~\ref{rem:isoconcordance} shows that~$(\b,\varphi)$ and~$(\b',\varphi\circ f)$ are concordant. Putting this together with%~\ref{rem:concordance.iii} and
~\ref{rem:concordance.ii}, we see that the pairs~$(\b,\varphi)$ and~$(\b',\varphi')$ are~$(h,h')$-concordant if and only if there is an isometry from~$(\V,h)$ to~$(\V',h')$ which conjugates~$\varphi(\b)$ to~$\varphi'(\b')$.
\end{enumerate}
\end{remarks}
}

%% I had put this in, but I don't think it is needed.
%% \begin{remark}\label{rem:independentofisometry}
%% Suppose~$(\V,h)$ is an~$\e$-hermitian~$\F/\F_\so$-space,~$\A=\End_\F(\V)$ and~$\varphi:\E\into\A$ is a self-dual embedding. If~$g\in\U(\V,h)$ is an isometry, then we also get a map~$g_*:\A\to\A$ given by~$g_*(a)=g^{-1}ag$ and hence an embedding~${}^g\varphi=g_*\circ\varphi:\E\into\A$, which is self-dual since~$g$ is an isometry. Then it is straightforward to check that~$g$ is also an isometry between~$(\V_\varphi,h_\varphi)$ and~$(\V_{{}^g\varphi},h_{{}^g\varphi})$. In particular, we get~$[h_\varphi]=[h_{{}^g\varphi}]$ and we see that the pairs~$(\beta,\varphi)$ and~$(\beta,{}^g\varphi)$ are concordant. We deduce also that the definition of~$(h,h')$-concordance is independent of the choice of isometry.
%% \end{remark}

\shaun{
Using the previous remarks, together with Lemma~\ref{lem:hphitwist}, we can use twisting to relate concordance in a symplectic space to concordance in orthogonal spaces obtained by twisting.}

\shaun{
\begin{lemma}\label{lemma:twistSympToOrthConc}
Suppose that~$(\V,h)$ is a skew-hermitian~$\F/\F_\so$-space, that~$\b$ and~$\b'$ are both non-zero, and that the spaces~$(\V,\b^*h)$ and~$(\V,\b'^*h)$ are isometric. Then the following are equivalent:
\begin{enumerate}\setlength\itemsep{5pt}
  \item $(\b,\varphi)$ and~$(\b',\varphi')$ are concordant.
  \item $(\b,\varphi)$ and~$(\b',\varphi')$ are~$(\b^*h,\b'^*h)$-concordant.
 \end{enumerate}
\end{lemma}
}

\shaun{
We will be able to use this whenever we have additional information on~$\b,\b'$ which enables us to see that the spaces~$(\V,\b^*h)$ and~$(\V,\b'^*h)$ are isometric.
}

\shaun{
\begin{proof}
If~$(\b,\varphi)$ and~$(\b',\varphi')$ are concordant then, from Lemma~\ref{lem:hphitwist} and~\eqref{eq:twistw} we get
\begin{align*}
\sw_{1,\b',\b}\([(\b^*h)_\varphi]\) = \sw_{1,\b',\b}\([\b^*(h_\varphi)]\) &= \sw_{1,\b',\b}\circ \b^* \([h_\varphi]\) \\
&= \b'^*\circ \sw_{-1,\b',\b}\([h_\varphi]\) = \b'^*\([h_{\varphi'}]\) = [(\b'^*h)_{\varphi'}].
\end{align*}
Thus~$(\b,\varphi)$ and~$(\b',\varphi')$ are~$(\b^*h,\b'^*h)$-concordant. 
%Since the maps~$\sw_{\e,\b',\b}$ and~$\b^*,\b'^*$ are bijections, 
\bob{Since the map~$\b'^*$ is injective,}
the converse follows immediately.
\end{proof}
}

%%%%%%%%%%%%%%%%%%%%%%%%%%%%%%%%%%%%
\subsection{Concordance in the non-symplectic case}\label{subsec:non-symp}
%%%%%%%%%%%%%%%%%%%%%%%%%%%%%%%%%%%%
\shaun{
We now look more closely at the non-symplectic case: indeed, %since 
Lemma~\ref{lemma:twistSympToOrthConc} allows us to relate the skew-hermitian case to the hermitian case. %, we assume~$\e=1$. 
%Suppose we are given a single hermitian~$\F/\F_\so$-space~$(\V,h)$ and set~$\A=\End_\F(\V)$. 
% so we assume that either~$\e=1$ or~$\F\ne\F_\so$. Suppose we are given a single~$\e$-hermitian~$\F/\F_\so$-space~$(\V,h)$ and set~$\A=\End_\F(\V)$. 
Suppose we are given~$\e$-hermitian~$\F/\F_\so$-spaces~$(\V,h)$ and~$(\V',h')$ which are isometric, and set~$\A=\End_\F(\V)$ and~$\A'=\End_\F(\V')$.
%% Let~$(\E,\b)$,~$(\E',\b')$ be self-dual extensions with~$\b,\b'$ non-zero, and suppose we have self-dual embeddings~$\varphi:\E\into\A$ and~$\varphi':\E'\into\A$. 
Let~$(\E,\b)$,~$(\E',\b')$ be self-dual extensions with~$\b,\b'$ non-zero, and suppose we have self-dual embeddings~$\varphi:\E\into\A$ and~$\varphi':\E'\into\A'$. 
We prove the following first result on concordance in the case~$\e=1$.
}

\shaun{
\begin{proposition}\label{prop:firstmatch}
%% In the situation above, suppose that~$\dim_\E\V_\varphi$ and~$\dim_{\E'}\V_{\varphi'}$ have the same parity and moreover that either
In the situation above, with~$\e=1$, suppose that~$\dim_\E\V_\varphi$ and~$\dim_{\E'}\V'_{\varphi'}$ have the same parity and moreover that either
\begin{enumerate}\setlength\itemsep{5pt}
\item\label{cor:firstmatch-1} this common parity is even; or
\item\label{cor:firstmatch-2} this common parity is odd,%~$\e=1$,
~$\sw_1(\la 1\ra)=\la 1\ra$ and one of the following conditions is satisfied:
\begin{enumerate}\setlength\itemsep{5pt}
\item\label{cor:firstmatch-2a} $\dim_\F\V$ is odd;
\item\label{cor:firstmatch-2b} %% $\F=\F_\so$ and there exists an extension~$\K/\F$ contained in~$\varphi(\E)\cap\varphi'(\E')$ which is invariant under the adjoint anti-involution on~$\A$ but not fixed pointwise. 
$\F=\F_\so$ and there exist an extension~$\K/\F$ contained in~$\varphi(\E)$ which is invariant under the adjoint anti-involution on~$\A$ but not fixed pointwise, and an isometry~$g:\V\to\V'$ such that~$g\K g^{-1}\subseteq\varphi'(\E')$.
\end{enumerate}
\end{enumerate}
%% Then the pairs~$(\b,\varphi)$ and~$(\b',\varphi')$ are concordant.
Then the pairs~$(\b,\varphi)$ and~$(\b',\varphi')$ are~$(h,h')$-concordant.
\end{proposition}
}

\shaun{
In order to prove this we notice that, whenever we are in the non-symplectic case,
 concordance is related to the diagrams
\begin{equation}\label{eqDiagConcOdd}
%\xymatrix{
%\Ww^{\odd}_1(\E/\E_\so) \ar[r]^{\sw_{1,\b',\b}}\ar[d]_{\l_\b^*} & \Ww^{\odd}_1(\E'/\E'_\so)\ar[d]^{\l_{\b'}^*} \\
%\Ww_1(\F/\F_\so) \ar[r]^{\id} & \Ww_1(\F/\F_\so) \\
%}
\xymatrix{
\Ww^{\odd}_\e(\E/\E_\so) \ar[r]^{\sw_{\e,\b',\b}}\ar[d]_{\l_\b^*} & \Ww^{\odd}_\e(\E'/\E'_\so)\ar[d]^{\l_{\b'}^*} \\
\Ww_\e(\F/\F_\so) \ar[r]^{\id} & \Ww_\e(\F/\F_\so) \\
}
\end{equation}
and
\begin{equation}\label{eqDiagConcEven}
%\xymatrix{
%\Ww^{\even}_1(\E/\E_\so) \ar[r]^{\sw_{1,\b',\b}}\ar[d]_{\l_\b^*} & \Ww^{\even}_1(\E'/\E'_\so)\ar[d]^{\l_{\b'}^*} \\
%\Ww_1(\F/\F_\so) \ar[r]^{\id} & \Ww_1(\F/\F_\so) \\
%}
\xymatrix{
\Ww^{\even}_\e(\E/\E_\so) \ar[r]^{\sw_{\e,\b',\b}}\ar[d]_{\l_\b^*} & \Ww^{\even}_\e(\E'/\E'_\so)\ar[d]^{\l_{\b'}^*} \\
\Ww_\e(\F/\F_\so) \ar[r]^{\id} & \Ww_\e(\F/\F_\so) \\
}
\end{equation}
If~$\dim_\E\V_\varphi$ and~$\dim_{\E'}\V'_{\varphi'}$ have the same parity then~$(\b,\varphi)$ and~$(\b',\varphi')$ are concordant if and only if the diagram of the corresponding parity commutes: this follows because, in both diagrams, the maps~$\lambda_\b^*$ and~$\lambda_{\b'}^*$ are injective by Proposition~\ref{prop:transferinj}. Therefore we analyze cases when these diagrams are commutative.}

\shaun{
\begin{lemma}\label{lemma:commutConcDiag}
 %Suppose we are in the non-symplectic case, and~$\b$ and~$\b'$ are non-zero. 
 Suppose that~$\e=1$ and that~$\b$ and~$\b'$ are non-zero.
 \begin{enumerate}\setlength\itemsep{5pt}
 \item\label{lemma:commutConcDiageven} The diagram~\eqref{eqDiagConcEven} is always commutative. 
 \item\label{lemma:commutConcDiagodd} Suppose that~$\sw_1(\la 1\ra)=\la 1\ra$% if~$\e=1$
, that~$[\E:\F]$ and~$[\E':\F]$ have the same parity, and that one of the following conditions is satisfied:
 \begin{enumerate}\setlength\itemsep{5pt}
  \item\label{lemma:commutConcDiag-ii} $\F\neq\F_\so$ %, and~$(-1)^{\frac{\e+1}{2}+[\E:\F]}=1$. 
  and~$[\E:\F]$ is odd.
  \item\label{lemma:commutConcDiag-i} $\F=\F_\so$ 
  and there exist extensions~$\K/\F$ and~$\K'/\F$ contained in~$\E$ and~$\E'$ respectively, which are invariant under the \rob{Galois} involution but not fixed pointwise, such that~$[\E:\K]$ and~$[\E':\K']$ have the same parity and there is a \rob{Galois}-equivariant~$\F$-\daniel{linear field} isomorphism from~$\K$ to~$\K'$. 
  %and there are~$\ov{(\ )}$-invariant but not~$\ov{(\ )}$-fixed field extensions~$\K|\F$ and~$\K'|\F$ in~$\E$ and~$\E'$, respectively, such that~$[\E:\K]$ and~$[\E':\K']$ have the same parity and there is a~$\ov{(\ )}$-equivariant isomorphism from~$\K|\F$ to~$\K'|\F$.
 \end{enumerate}
Then the diagram~\eqref{eqDiagConcOdd} is commutative. 
 \end{enumerate}
\end{lemma}
}

\begin{proof} 
\shaun{
Diagram~\eqref{eqDiagConcEven} is commutative because the maximal anisotropic class is mapped to the maximal anisotropic class, by Proposition~\ref{prop:lambda*}\ref{prop:lambda*.ii}. 
}

\shaun{
We now suppose that~$\sw_1(\la 1\ra)=\la 1\ra$ and that~$[\E:\F]$ and~$[\E':\F]$ have the same parity, and consider diagram~\eqref{eqDiagConcOdd}, recalling that all maps in it are injective. Write
\[
\Ww^{\odd}_\e(\E/\E_\so)=\{\la 1\ra, \la a\ra\} \quad\text{ and }\quad \Ww^{\odd}_\e(\E'/\E'_\so)=\{\la 1\ra, \la a'\ra\}.
\]
Then, since~$\la a\ra - \la 1\ra$ is the maximal element of~$\Ww_\e(\E/\E_\so)$, it follows from Proposition~\ref{prop:lambda*}\ref{prop:lambda*.ii} that~$\l_\b^*(\la a\ra) - \l_\b^*(\la 1\ra)$ is the maximal element of~$\Ww_\e(\F/\F_\so)$. The same applies to~$\l_{\b'}^*(\la a'\ra) - \l_{\b'}^*(\la 1\ra)$ so that 
\[
\l_\b^*(\la a\ra) - \l_\b^*(\la 1\ra) = \l_{\b'}^*(\la a'\ra) - \l_{\b'}^*(\la 1\ra).
\]
Thus it is enough to check
% to
%In each case, we will show that the images of~$\l_\b^*$ and~$\l_{\b'}^*$ coincide. When this image has only two elements, we need only check 
that~$\l_\b^*(\la 1\ra)=\l_{\b'}^*(\la 1\ra)$ to prove commutativity of~\eqref{eqDiagConcOdd}.}

\shaun{
In the situation of~\ref{lemma:commutConcDiag-ii}, %the images of~$\l_\b^*$ and~$\l_{\b'}^*$ are both contained in~$\Ww^{\odd}_1(\F/\F_\so)$, which has two elements, and 
we have~$\l_\b^*(\la 1\ra)=\la 1\ra=\l_{\b'}^*(\la 1\ra)$, by Proposition~\ref{prop:transfer}\ref{prop:transfer.ii}.}
\shaun{In case~\ref{lemma:commutConcDiag-i}, again by Proposition~\ref{prop:transfer}\ref{prop:transfer.ii}, we have
\begin{equation}\label{prooffirstmatch:equation}
\l_\b^*(\la 1\ra) =\la 1\ra \oplus \la (-1)^{m}\N_{\E/\F}(\b)\ra,\text{ and }\l_{\b'}^*(\la 1\ra) =\la 1\ra \oplus \la (-1)^{m'}\N_{\E'/\F}(\b')\ra
\end{equation}
for some integers~$m,m'$ whose values are not needed for the proof. In particular, these both have anisotropic dimension at most two so cannot be the Witt tower of maximal anisotropic dimension.}

\shaun{If~$[\E:\K]$ and~$[\E':\K']$ are even, then the images of~$\l_\b^*$ and~$\l_{\b'}^*$ %in~$\Ww_1(\F/\F_\so)$ coincide, and 
consist of~$\bs 0$ and the maximal element. Since~\eqref{prooffirstmatch:equation} shows that neither is maximal, we have%Witt tower of maximal anisotropic dimension. Thus we have
~$\l_\b^*(\la 1\ra)=\bs 0=\l_{\b'}^*(\la 1\ra)$.}

\shaun{Suppose now~$[\E:\K]$ and~$[\E':\K']$ are odd and denote by~$\K_\so$ the fixed field of~$\K$ under the \rob{Galois} involution on~$\E$, so that~$\K/\K_\so$ is quadratic. Given~$\l:\E\to\K$ a non-zero \rob{Galois-}equivariant~$\K$-linear form, the induced transfer map~$\l^*:\Ww_1(\E/\E_\so)\rightarrow\Ww_1(\K/\K_\so)$ is then bijective. It follows from Proposition~\ref{prop:lambda*}\ref{prop:lambda*.i} that the image of~$\l_\b^*$ coincides with that of~$\l_\K^*$, for any non-zero \rob{Galois-}equivariant~$\F$-linear form~$\l_\K:\K\to\F$. Since there is a \rob{Galois-}equivariant~$\F$-\daniel{linear field} isomorphism from~$\K$ to~$\K'$, this also coincides with the image of~$\l_{\K'}^*$, for any non-zero \rob{Galois-}equivariant~$\F$-linear form~$\l_{\K'}:\K'\to\F$. In particular, the images of~$\l_\b^*$ and~$\l_{\b'}^*$ in~$\Ww_1(\F/\F_\so)$ coincide.}% and we denote this common image by~$\M$.}

\shaun{If this image has order two then it consists of~$\bs 0$ and the maximal element, and we again have~$\l_\b^*(\la 1\ra)=\bs 0=\l_{\b'}^*(\la 1\ra)$. Otherwise, it has order~$4$ and, since~$\Ww_1(\E/\E_\so)$ and~$\Ww_1(\E'/\E'_\so)$ also have order~$4$, the transfer maps~$\l_\b^*$ and~$\l_{\b'}^*$ are injective. %Then the image of~$\Ww_1^{\odd}(\E/\E_\so)$ under~$\l_\b^*$ must consist of two spaces of anisotropic dimension~$2$, and likewise for~$\l_{\b'}^*$. Thus, again, we need only check that~$\l_\b^*(\la 1\ra)=\l_{\b'}^*(\la 1\ra)$ to prove commutativity of~\eqref{eqDiagConcOdd}.}
%\shaun{
We set~$\bs a=\l_\b^*(\la 1\ra)$ and~$\bs b=\l_{\b'}^*(\la 1\ra)$, which are neither~$\bs 0$ nor the maximal element, by injectivity.  Assume for contradiction that~$\bs a \neq \bs b$ so that~\eqref{prooffirstmatch:equation} implies that~$\bs a-\bs b$ is also neither~$\bs 0$ nor maximal. Thus~$\bs a -\bs b$ is either~$\bs a$ or~$\bs b$, %as these are the only two classes in~$\M$ with anisotropic dimension not maximal or zero.  
therefore~$\bs a=2\bs b$, as~$\bs b$ is non-zero. By symmetry, we also have~$\bs b=2\bs a$ and it follows that~$\bs a=\bs b=\bs 0$, which is absurd.}
\end{proof}

\begin{proof}[Proof of Proposition~\ref{prop:firstmatch}]
\shaun{The three parts follow immediately from the corresponding parts of Lemma~\ref{lemma:commutConcDiag}, once we notice, in~\ref{cor:firstmatch-2b}, that~$[\E:\K]$ and~$[\E':\K']$ have the same parity as~$\dim_\K\V=\dim_{\K'}\V'$. 
}
\end{proof}

In order to go further than this, we need to add some conditions on~$\b,\b'$; in particular, we will require them to be related in some way.

%%%%%%%%%%%%%%%%%%%%%%%%%%%%%%%%%%%%
\subsection{Similar extensions}\label{subsec:similar}
%%%%%%%%%%%%%%%%%%%%%%%%%%%%%%%%%%%%
\shaun{
We now introduce a notion of \emph{similarity} on self-dual extensions. We fix a uniformizer~$\w_\F$ of~$\F$; if~$\F\ne\F_\so$ then we assume further that~$\ov{\w_\F}=-\w_\F$. For~$(\E,\b)$ a self-dual extension of~$\F/\F_\so$ with~$\b\ne 0$, we write~$y_\b$ for the image of~$\w_\F^{n/g}\b^{e/g}$ in the residue field~$\mathrm{k}_\E$, where~$e=e(\E/\F)$ is the ramification index,~$n=-\val_\E(\b)$, and~$g=\gcd(n,e)$. We also set~$y_0=0$ in~$k_\F$.
}

\begin{definition}\label{def:similar}
\shaun{
We say that self-dual extensions~$(\E,\b)$ and~$(\E',\b')$ of~$\F/\F_\so$ are \emph{similar} if:
\begin{enumerate}\setlength\itemsep{5pt}
\item %$[\E:\F]=[\E':\F]$, the residue class degree~$f(\E/\F),f(\E'/\F)$ have the same parity, and~$e(\E/\E_\so)=e(\E'/\E'_\so)$;
$f(\E/\F)=f(\E'/\F)$,~$e(\E/\F)=e(\E'/\F)$ and~$e(\E/\E_\so)=e(\E'/\E'_\so)$;
%we have equalities of residue class degrees~$f(\E/\F)=f(\E'/\F)$ and of ramification degrees~$e(\E/\F)=e(\E'/\F)$ and~$e(\E/\E_\so)=e(\E'/\E'_\so)$;
\item $\val_\E(\b)=\val_{\E'}(\b')$; and
\item there is a~$\mathrm{k}_\F$-\daniel{linear field} isomorphism from~$\mathrm{k}_\E$ to~$\mathrm{k}_{\E'}$ which sends~$y_\b$ to~$y_{\b'}$.
%%% Really isomorphism, or just embedding? If isomorphism then we might as well assume residue class degrees and ramification indices are equal.
\end{enumerate}
}
\end{definition}

\shaun{
Note that the notion of similarity is independent of the choice of uniformizer~$\w_\F$. In the end we will mostly be concerned about concordance in cases where we already know that the extensions are similar.
}

\shaun{
%We return now to the situation of the previous subsection, \emph{with the additional assumption that~$\F\ne\F_\so$}. Thus~$\e=1$, 
Suppose, as before, we are given hermitian~$\F/\F_\so$-spaces~$(\V,h)$ and~$(\V',h')$ which are isometric, and set~$\A=\End_\F(\V)$ and~$\A'=\End_\F(\V')$. We also have~$(\E,\b)$ and~$(\E',\b')$, self-dual extensions with~$\b,\b'$ non-zero, and we suppose we have self-dual embeddings~$\varphi:\E\into\A$ and~$\varphi':\E'\into\A'$. We have the following result.
}
%%%%% Do we ever use this when~$\b,\b'$ are not minimal? If not, then we should probably define minimal here and just state it in that case. It would be simpler!
%%%%% Daniel says we do.

%% Do we need F\ne F_\so and/or \e=1 ? 

\shaun{
\begin{lemma}\label{lem:concordwithyb}
Suppose that~$\F\ne\F_\so$ and that the self-dual extensions~$(\E,\b)$ and~$(\E',\b')$ are similar. 
\begin{enumerate}\setlength\itemsep{5pt}
\item\label{lem:concordwithyb.i} $\w_\F^{-1}\b\in \N_{\E/\E_\so}(\E^\times)$ if and only~$\w_\F^{-1}\b'\in \N_{\E'/\E'_\so}(\E'^\times)$.
\item\label{lem:concordwithyb.ia} \orange{The diagrams~\eqref{eqDiagConcOdd} and~\eqref{eqDiagConcEven} are commutative.}
\item\label{lem:concordwithyb.ii} The pairs~$(\b,\varphi)$ and~$(\b',\varphi')$ are~$(h,h')$-concordant.
\end{enumerate}
\end{lemma}
}

\orange{We will see later (see Corollary~\ref{cor:similardiagramscommute}) that~\ref{lem:concordwithyb.ia} is in fact also true without the hypothesis~$\F\ne\F_\so$.}

\begin{proof}
\shaun{
We set~$e=e(\E/\F)=e(\E'/\F)$, ~$e_\so=e(\E/\E_\so)=e(\E'/\E'_\so)$ and~$d=[\E:\F]=[\E':\F]$. We also set~$n=-\val_\E(\b)=-\val_{\E'}\rob{(\b')}$ \daniel{ and observe that~$\w_\F^{-1}\b$ is fixed by the involution.}
}

\shaun{
\ref{lem:concordwithyb.i} If~$e_\so=2$ then both~$e$ and~$n$ are odd, by Lemma~\ref{lem:sdeodd}. Hence \daniel{Hensel's Lemma implies that}~$\w_\F^{-1}\b$ is a square in~$\E$ if and only if~$y_\b$ is a square in~$\mathrm{k}_\E$, and similarly for~$\w_\F^{-1}\b'$. \daniel{On the other hand~$y_\b$ is a square in~$\mathrm{k}_\E$ if and only if~$y_{\b'}$ is a square in~$\mathrm{k}_{\E'}$, because we have a~$\mathrm{k}_\F$-\daniel{linear field} isomorphism from~$\mathrm{k}_\E$ to~$\mathrm{k}_{\E'}$ which sends~$y_\b$ to~$y_{\b'}$.} The result now follows from the description of norms in Lemma~\ref{lem:sdexts1}.
}

\shaun{
If~$e_\so=1$ then~$\w_\F^{-1}\b\in \N_{\E/\E_\so}(\E^\times)$ if and only if~$\w_\F^{-1}\b$ has even valuation, by Lemma~\ref{lem:sdexts1}. Since it has the same valuation as~$\w_\F^{-1}\b'$, the result follows.}

\shaun{
\orange{\ref{lem:concordwithyb.ii} follows immediately from~\ref{lem:concordwithyb.ia}, while the commutativity of~\eqref{eqDiagConcEven} is immediate since the maximal element is mapped to the maximal element. To complete the proof of~\ref{lem:concordwithyb.ia}, we need to prove that the diagram~\eqref{eqDiagConcOdd} is commutative.} Since~$f(\E/\F)=f(\E'/\F)$, Lemma~\ref{lemma:w1oneifsquares} implies that~$\sw_{1}(\la 1\ra)=\la 1\ra$ so that we only need to prove that~$\l_\b^*(\la 1\ra)=\l_{\b'}^*(\la 1\ra)$ (for the case~$\e=1$) and~$\l_\b^*(\la \b\ra)=\l_{\b'}^*(\la \b'\ra)$ (for the case~$\e=-1$). Now Lemma~\ref{lem:norms} and~\ref{lem:concordwithyb.i} imply that~$\N_{\E/\F}(\w_\F^{-1}\b)=\w_\F^{-d}\N_{\E/\F}(\b)$ lies in~$\N_{\F/\F_\so}(\F^\times)$ if and only if~$\w_\F^{-d}\N_{\E'/\F}(\b')$ does also. The result now follows by applying Proposition~\ref{prop:transfer}\ref{prop:transfer.ii}. 
%, we obtain~$\l_\b^*(\la 1\ra)=\l_{\b'}^*(\la 1\ra)$ as required.
}
\end{proof}

%% file: Endo-strata.tex
\section{Classical groups}
%%%%%%%%%%%%%%%%%%%%%%%%%%%%%%%%%%%%
%%%%%%%%%%%%%%%%%%%%%%%%%%%%%%%%%%%%

Let~$(\V,h)$ be an~$\e$-hermitian space over~$\F/\F_\so$, and put~$\A=\End_\F(\V)$ and~$\tG=\Aut_\F(\V)$. The ring~$\A$ is equipped with the adjoint anti-involution~$a\mapsto\ov a$ induced by~$h$. 
\orange{We let our abstract group~$\Sigma=\{1,\s\}$ act both on~$\tG$, with~$\s(g)=(\ov g)^{-1}$ for~$g\in\tG$, and on~$\A$, with~$\s(a)=-\ov a$ for~$a\in\A$.}

%By~$\s$ we will denote both the involution on~$\tG$ defined by~$g\mapsto (\ov g)^{-1}$, and its derivative on~$\A$, which is the map~$a\mapsto -\ov a$. We set~$\Sigma=\{1,\s\}$, where~$1$ acts as the identity on both~$\tG$ and~$\A$.

We set~$\G:=\tG^\Sigma=\U(\V,h)$. We write~$\G^\so$ for the group of~$\F_\so$-points of the connected component of the underlying reductive group, so that~$\G^\so=\G$ except in the orthogonal case when it is the special orthogonal group. We call the group~$\G^\so$ a \emph{classical group}.

For~$\J$ a~$\s$-stable subgroup of~$\tG$, we will write~$\J_-=\J^\Sigma=\J\cap\G$. % and~$\J_-^\so=\J\cap\G^\so$. 
Similarly, if~$\X$ is any~$\s$-stable~$\o_\F$-submodule of~$\A$ then we write
\[
\X_-=\X^\Sigma=\{x\in\X\mid \ov x=-x\}, \qquad
\X_+ = \{x\in\X\mid \ov x=x\},
\]
for the set of skew-symmetric (respectively, symmetric) elements of~$\X$. Note that~$\A_-$ is the Lie algebra of~$\G$ (and~$\G^\so$).

%%%%%%%%%%%%%%%%%%%%%%%%%%%%%%%%%%%%
%%%%%%%%%%%%%%%%%%%%%%%%%%%%%%%%%%%%
\section{Simple strata and concordance}\label{sec:simplestrata}
%%%%%%%%%%%%%%%%%%%%%%%%%%%%%%%%%%%%
%%%%%%%%%%%%%%%%%%%%%%%%%%%%%%%%%%%%

%%%%%%%%%%%%%%%%%%%%%%%%%%%%%%%%%%%%
%The first step in the Bushnell--Kutzko--Stevens explicit construction of cuspidal representations of~$\tG,\G$ is to write down a collection pairs of compact open pro-$p$ subgroups of~$\tG,\G$ and characters of these subgroups such that every cuspidal representation contains one of these characters.  This is rigidified in the theory of strata: 4-tuples of data, which under mild conditions, after one has fixed a character of~$\F$, correspond to explicit characters of an explicit compact open subgroups of~$\tG$.  If a cuspidal representation of~$\tG,$ or $\G$ contains two such characters then they intertwine, and we have a corresponding notion of intertwining of strata.{\color{red}rewrite in introduction need to say semisimple characters necessary etc..}

In this section, we investigate intertwining of self-dual pure strata and introduce concordance of self-dual pure strata (Definition \ref{Def:wittstrata}).  The main result is Proposition~\ref{prop:strataintertwiningovertG}.

%%%%%%%%%%%%%%%%%%%%%%%%%%%%%%%%%%%%
\subsection{Lattice sequences and parahoric subgroups}\label{subsec:latticeseq}
%%%%%%%%%%%%%%%%%%%%%%%%%%%%%%%%%%%%
We recall that an~\emph{$\o_\F$-lattice sequence} in~$\V$ is a map~$\La$ from~$\ZZ$ to the set of~$\o_\F$-lattices in~$\V$ which is decreasing and periodic; that is,
\begin{enumerate}\setlength\itemsep{5pt}
\item $\La(k+1)\subseteq \La(k)$, for all~$k\in\ZZ$;
\item there is a positive integer~$e=e(\La)=e(\La|\o_\F)$ such that~$\p_\F\La(k)=\La(k+e)$, for all~$k\in\ZZ$.
\end{enumerate}
The integer~$e$ is called the~$\o_\F$-period of~$\La$. If~$\dim_{\mathrm{k}_\F}(\La(k)/\La(k+1))$ is independent of~$k\in\ZZ$ we say that~$\La$ is \emph{regular}. We call~$\La$ \emph{strict} if $\La(k+1)\subsetneq \La(k)$, for all~$k\in\ZZ$. For~$a,b\in\ZZ$,~$a>0$, we let~$a\La+b$ denote the~$\o_\F$-lattice sequence in~$\V$ defined by
\[
(a\La+b)(r)=\La(\lfloor (r-b)/a\rfloor),\text{ for all }r\in\rob{\mathbb{Z}}.
\]
\shauns{We call~$a\La+b$ an \emph{affine translation of~$\La$} and say that lattice sequences~$\La,\La'$ are \emph{in the same affine class} if they have a common affine translation.}

\daniel{The direct sum of~$\o_\F$-lattice sequences~$\Lambda$ and~$\Lambda'$ of the same~$\o_\F$-period is defined by
\[
(\La\oplus\La')(r):=\La(r)\oplus\La'(r),\ r\in\mathbb{Z}.
\]
}

An~$\o_\F$-lattice sequence~$\La$ in~$\V$ defines an~$\o_\F$-lattice sequence in~$\A$, by setting
\[
\aa_n(\La)=\{a\in\A \mid a\La(k)\subseteq  \La(k+n),\text{ for all }k\in\ZZ\},\
\]
for~$n\in\ZZ$. The~$\o_\F$-lattice~$\aa_0(\La)$ is a hereditary~$\o_\F$-order in~$\A$ with Jacobson radical~$\aa_1(\La)$. Note that a strict~$\o_\F$-lattice sequence~$\La$ is regular if and only if~$\aa_0(\La)$ is a principal order. We also get a valuation map~$\val_\La$ on~$\A$ by setting
\[
\val_\La(x) = \sup\{n\in\ZZ \mid x\in\aa_n(\La)\}, \qquad\text{for }x\in\A,
\]
with the understanding that~$\val_\La(0)=\infty$. 

The normalizer in~$\Aut_\F(\V)$ of~$\La$ is a compact mod-centre subgroup
\[
\KK(\La)=\{g\in\Aut_\F(\V)\mid \text{there exists }n\in\ZZ\text{ such that }g(\La(k))=\La(k+n),\text{ for all }k\in\ZZ\}.
\]
\bob{The restriction of the valuation map to~$\KK(\La)$ defines a} group homomorphism~$\val_{\La}:\KK(\La)\rightarrow \ZZ$. % by setting~$\val_{\La}(g)=n$, if~$g(\La(0))=\La(n)$. 
The kernel of \bob{which} is a compact open subgroup~$\P(\La)$ of~$\Aut_\F(\V)$ which coincides with the group of units of the order~$\aa_0(\La)$. This subgroup has a decreasing filtration by compact open pro-$p$ subgroups, given by~$\P^n(\La)=1+\aa_n(\La)$, for~$n\geqslant 1$.

For~$\L$ an~$\o_\F$-lattice in~$\V$, we define the \emph{dual lattice}
\[
\L^\#=\{v\in \V\mid h(v,\L)\subseteq \p_\F\}.
\]  
For~$\La$ an~$\o_\F$-lattice sequence in~$\V$ we define the~\emph{dual lattice sequence}~$\La^\#$ in~$\V$ by
\[
\La^\#(r)=\La(1-r)^\#,
\] 
for all~$r\in\ZZ$, and we call~$\La$ \emph{self-dual} if~$\La^\#=\La+d$, for some~$d\in\ZZ$. If~$\La$ is self-dual then the lattices~$\aa_n(\La)$ are fixed by the adjoint anti-involution on~$\A$, and we put
\[
\aa_{n,-}(\La)=\aa_n(\La)\cap \A_-,\quad \P_-(\La)=\P(\La)\cap \G,\quad \P^m_-(\La)=\P^m(\La)\cap \G, \qquad\text{ for }m,n\in\ZZ,\ m\ge 1.
\]
Note that while~$\P(\La)$ is a \emph{parahoric subgroup} of~$\tG$ in the sense of Bruhat--Tits,~$\P_-(\La)$ is not always a parahoric subgroup: it is the full \shauny{stabilizer} of a \bob{point} in the Bruhat--Tits building of~$\G$. See section~\ref{sectionIIC} below for the definition of parahoric subgroup for~$\G^\so$, when we will need it.%The parahoric subgroup~$\P_-^\so(\La)$ of~$\G^\so$ is the inverse image in~$\P_-(\La)\cap\G^\so$ of the connected component of the identity in~$\P_-(\La)/\P^1_-(\La)$, which is a finite reductive group over~$\mathrm{k}_\F$.

\shauns{Finally in this subsection, suppose that~$\E$ is a subfield of~$\A$ containing~$\F$. Then we can consider~$\V$ as an~$\E$-vector space, so we have the notion of~$\o_\E$-lattice sequence in~$\V$; these are in fact~$\o_\F$-lattice sequences which are normalized by~$\E^\times$. We have the following elementary but useful lemma on the existence of lattice sequences with prescribed properties.}

\begin{lemma}\label{lem:changetoconjlatt}
\shauns{Let~$\E,\E'$ be subfields of~$\A$ containing~$\F$, such that~$e(\E/\F)=e(\E'/\F)$ and~$f(\E/\F)=f(\E'/\F)$, and let~$\La$ be an~$\o_\E$-lattice sequence in~$\V$. Then there exist an~$\o_{\E'}$-lattice sequence~$\La'$ in~$\V$ and~$g\in\tG$ such that~$g\La'=\La$.}
\end{lemma} 

\begin{proof}
\shauns{There is an~$\F$-linear isomorphism from~$\E$ to~$\E'$ which maps~$\p_\E^n$ to~$\p_{\E'}^n$, for each~$n\in\ZZ$. Now we choose an~$\E$-basis of~$\V$ which splits~$\La$ and map it to an~$\E'$-basis of~$\V$, using this~$\F$-linear isomorphism, and the image of~$\La$ has the required property.}
\end{proof}

%%%%%%%%%%%%%%%%%%%%%%%%%%%%%%%%%%%%
\subsection{A self-dual~$\dag$-construction}\label{subsec:dag}
%%%%%%%%%%%%%%%%%%%%%%%%%%%%%%%%%%%%
\shaun{We recall briefly the~\emph{$\dag$-construction} of~\cite[\S4]{RKSS}, which is a useful way of generalizing results originally proved only for strict lattice \rob{sequences} to the general case, and introduce a self-dual version.} 

\shaun{Let~$\La$ be an~$\o_\F$-lattice sequence in~$\V$ of~$\o_F$-period~$e=e(\La)$. Let~$\V^\dag=\V\oplus \cdots\oplus \V$ ($e$ times) and define the~$\o_\F$-lattice sequence~$\La^\dag$ in~$\V^\dag$ by
\[
\La^\dag=\bigoplus_{k=0}^{e-1}(\La-k).
\]
Then~$\La^\dag$ is a strict regular~$\o_F$-lattice sequence in~$\V^\dag$ of period~$e$. We denote by~$\Mm^\dag$ the Levi subalgebra of~$\A^\dag=\End_\F(\V^\dag)$ which is the stabilizer of the decomposition~$\V^\dag=\V\oplus \cdots\oplus \V$. Any~$\b\in\A$ then induces an element~$\b^\dag=\b\oplus \cdots\oplus\b$ in~$\Mm^\dag$. We write~$\tM^\dag$ for the group of units of~$\Mm^\dag$, which is a Levi subgroup of~$\tG^\dag=\Aut_\F(\V^\dag)$. Then~$\P(\La^\dag)\cap\M^\dag\simeq \P(\La)\times\cdots\times\P(\La)$, and similarly for~$\P^n(\La^\dag)$, while the~$\tG^\dag$-conjugacy class of~$\P(\La^\dag)$ is independent of~$\La$, depending only on the period~$e$\daniel{, see~\cite[1.5.2(ii)]{bushnellFroehlich:85}}.}

%%%%% Not sure we need all this notation %%%%%

\shaun{Now we introduce a self-dual variant. We will again use the notation~$\dag$; when we use it, we will make it clear when we are applying this self-dual version. Let~$\La$ be a self-dual lattice sequence in~$\V$ with~$e=e(\La)$. Let~$\V^\dag=\V\oplus \cdots\oplus \V$ ($2e$ times, indexed by~$j\in\{\pm 1,\ldots,\pm e\}$) and write vectors~$\bv\in\V^\dag$ as tuples:~$\bv=\(v_j\)_{j=-e}^e$, where we understand that~$0$ is omitted and~$v_j$ is in the~$j$-th copy of~$\V$. We define the form~$h^\dag$ on~$\V^\dag$ by
\[
h^\dag\( \(v_j\)_{j=-e}^e, \(w_j\)_{j=-e}^e \) = \rob{\sum_{j=-e}^e} h(v_j,w_{-j}),
\]
so that each copy of~$\V$ is isotropic\shauns{: indeed, the space~$(\V^\dag,h^\dag)$ is hyperbolic}. Now we define the~$\o_\F$-lattice sequence~$\La^\dag$ in~$\V^\dag$ by
\[
\La^\dag = \bigoplus_{j=1}^{e} (\La-j)\oplus  \bigoplus_{j=1}^{e} (\La-j)^\#,
\]
where we understand that~$\La-j$ is in the~$j$-th copy of~$\V$, and~$(\La-j)^\#$ in the~$(-j)$-th copy. Then~$\La^\dag$ is a regular strict lattice sequence in~$\V^\dag$ and is self-dual with respect to~$h^\dag$; indeed~$(\La^\dag)^\#=\La^\dag$. We again set~$\A^\dag=\End_\F(\V^\dag)$ and~$\tG^\dag=\Aut_\F(\V^\dag)$, and denote by~$\Mm^\dag$ the Levi subalgebra of~$\A^\dag$ which is the stabilizer of the decomposition~$\V^\dag=\V\oplus \cdots\oplus \V$. As above, any~$\b\in\A$ induces an element~$\b^\dag$ in~$\Mm^\dag$, which is skew whenever~$\b$ is skew. We also set~$\G^\dag=\U(\V^\dag,h^\dag)$ \shauns{and note that the map~$g\mapsto g^\dag$ defines an embedding of~$\G$ in~$\G^\dag$.}}

%For~$\La$ a self-dual lattice sequence in~$\V$ with~$e=e(\La)$, we define
%\begin{align*}
%\La^\dag&=(\La-e+1)\oplus\cdots\oplus (\La-0)\oplus (\La-0)^\#\oplus\cdots\oplus (\La-e+1)^\#,\\
%h^\dag&=\left(\begin{smallmatrix}&&h\\&\bdots\\h\end{smallmatrix}\right),~\V^\dag=\V\oplus\cdots\oplus\V\quad\text{(2e)-times},\\
%\A^\dag&=\End_\F(\V^\dag),\quad~\tG^\dag=(\A^\dag)^\times,\quad~\G^\dag=\U(\V^\dag,h^\dag).
%\end{align*}
%Then~$\La^\dag$ is a regular strict lattice sequence in~$\V^\dag$ and is self-dual with respect to~$h^\dag$. 
\shaun{This construction becomes particularly useful when applied to two self-dual lattice sequences~$\La,\La'$ with~$e(\La)=e(\La')$ (which we can always ensure by an affine translation) but with~$\La,\La'$ possibly not~$\G$-conjugate. The lattice sequences~$\La^\dag,\La'^\dag$, as they are self-dual, regular and strict of the same~$\o_\F$-period, are conjugate in $\G^\dag$ by~\cite[Proposition 5.2]{SkodField}.}

%%%%%%%%%%%%%%%%%%%%%%%%%%%%%%%%%%%%
\subsection{Strata}\label{subsec:strata}
%%%%%%%%%%%%%%%%%%%%%%%%%%%%%%%%%%%%
A \emph{stratum} in~$\A$ is a~$4$-tuple~$[\La,n,r,\b]$ where
\begin{enumerate}\setlength\itemsep{5pt}
\item $\La$ is an~$\o_\F$-lattice sequence in~$\V$;
\item $n\geqslant r\geqslant 0$ are integers;
\item $\b\in \aa_{-n}(\La)$. 
\end{enumerate}
The fraction~$\max\{-\val_\La(\b),r\}/{e(\La)}$ is called the \emph{depth} of the stratum. We call the stratum~$[\La,r,r,0]$ a \emph{null stratum}. Two strata~$[\La,n,r,\b_i]$, for~$i=1,2$, are called \emph{equivalent} if
\[
\b_1-\b_2\in\aa_{-r}(\La).
\]

An element~$g\in\tG$ \emph{intertwines} strata~$[\La,n,r,\b]$ and~$[\La',n',r',\b']$ if
\[
g(\b+\aa_{-r}(\La))g^{-1}\cap (\b'+\aa_{-r'}(\La'))\neq \emptyset.
\] 
For a subgroup~$\J$ of~$\tG$, we say~$[\La,n,r,\b]$ and~$[\La',n',r',\b']$ \emph{intertwine in~$\J$} if there exists an element of~$\J$ which intertwines the strata.  We say that~$[\La,n,r,\b]$ and~$[\La',n',r',\b']$ are \emph{conjugate in~$\J$} if~$n=n'$,~$r=r'$ and there exists~$g\in\J$ such that
\[
g\La=\La'\text{ and }g\b g^{-1}=\b'. 
\]

An~\emph{affine translation} of a stratum~$[\La,n,r,\b]$ is a stratum~$[\La',n',r',\b]$ such that there exist~$a,b\in\ZZ$,~$a>0$, with~$\La'=a\La+b$,~$n'=an$ and\rob{~$\lfloor r'/a\rfloor= r$}. \shauns{We say that two strata are \emph{in the same affine class} if they have affine translations which are equal. As we shall see, many objects we later associate to a stratum are in fact shared by all strata in the same affine class.}

\shaun{We can also make a~$\dag$-construction for strata. If~$[\La,n,r,\b]$ is a stratum in~$\A$ then we have the lattice sequence~$\La^\dag$ in~$\V^\dag$ and the element~$\b^\dag$ of~$\A^\dag$, giving us a new stratum~$[\La^\dag,n,r,\b^\dag]$.} \shauns{This process behaves well with respect to intertwining: if~$g\in\tG$ intertwines two strata~$[\La,n,r,\b]$ and~$[\La',n',r',\b']$ then the element~$g^\dag\in\tG^\dag$ intertwines~$[\La^\dag,n,r,\b^\dag]$ and~$[\La'^\dag,n',r',\b'^\dag]$.} 

We fix a uniformizer~$\w_\F$ of~$\F$. Given a stratum~$[\La,n,r,\b]$ with~$r<n$, we write~$y_\b$ for the image of~$\w_F^{n/g}\b^{e/g}$ in~$\aa_0(\La)/\aa_1(\La)$, where~$e=e(\La)$ and~$g=\gcd(n,e)$. The characteristic polynomial of~$y_\b$ (in~$\mathrm{k}_\F[X]$) is called the \emph{characteristic polynomial of the stratum~$[\La,n,r,\b]$}, while its minimal polynomial is called the \emph{minimal polynomial of the stratum~$[\La,n,r,\b]$}. These depend only on the equivalence class of the stratum~$[\La,n,n-1,\b]$ (and the choice of uniformizer). A stratum~$[\La,n,n-1,\b]$ is called \emph{fundamental} if its characteristic polynomial is not a power of~$X$; this property is independent of the choice of uniformizer.

A stratum~$[\La,n,r,\b]$ is called \emph{pure} if either it is null or the following three conditions are satisfied:
\begin{enumerate}\setlength\itemsep{5pt}
\item $\E=\F[\b]$ is a field;
\item $\La$ is an~$\o_\E$-lattice sequence in~$\V$;
\item $\val_{\La}(\b)=-n$.
\end{enumerate}
We call~$[\E:\F]$ the~\emph{degree} of such a stratum, and write~$\B$ for the centralizer in~$\A$ of~$\b$ and~$\bb_n(\La)=\aa_n(\La)\cap\B$, for~$n\in\ZZ$. We set
\[
\nn_k(\b,\La)=\{x\in\aa_0(\La)\mid \b x-x\b\in \aa_k(\La)\}
\]
and define the \emph{critical exponent}~$k_0(\b,\La)$ by
\[
k_0(\b,\La)=
\begin{cases}
-\infty, &\text{ if }\b=0,\\
\max\left\{\val_{\La}(\b),\sup\{k\in\ZZ\mid\nn_{k}(\b,\La)\not\subseteq \bb_0(\La)+\aa_1(\La)\}\right\}, &\text{ otherwise.}
\end{cases}
\]
For~$\E=\F[\b]$ a finite extension of~$\F$, we set
\[
k_\F(\b)=\frac{k_0(\b,\p_\E^\ZZ)}{e(\E/\F)},
\]
where~$\p_\E^\ZZ$ denotes the~$\mathfrak{o}_F$-lattice sequence in~$\E$ (considered as an~$\F$-vector space) given by~$i\mapsto \mathfrak{p}_E^i$,~$i\in\mathbb{Z}$.  Note that, our definition of~$k_\F(\b)$ differs from that in~\cite[(1.4)]{BH96} by the normalization~$1/e(\E/\F)$.  By \cite[Lemma 5.6]{St01}, we have
\[
k_0(\b,\La)= e(\La|\o_\F) k_\F(\b).
\]

A pure stratum~$[\La,n,r,\b]$ is called \emph{simple} if~$k_0(\b,\La)<-r$; in particular, a pure stratum~$[\La,n,n,\b]$ is simple if and only if it is the null stratum~$[\La,n,n,0]$. A particularly nice case occurs when~$r=n-1$ (see~\cite[1.4.15]{BK93}): a pure stratum~$[\La,n,n-1,\b]$ is simple if and only if~$\b$ is \emph{minimal over~$\F$}, that is:
\begin{enumerate}\setlength\itemsep{5pt}
\item $\val_\E(\beta)$ is prime to~$e(\E/\F)$;
\item $\b^{e(\E/\F)}\varpi_\F^{-\val_\E(\beta)}+\p_\E$ generates the residue field~$\mathrm{k}_{\E}$ over~$\mathrm{k}_\F$.
\end{enumerate}

\shaun{We observe that the notions of pure and simple behave well under the~$\dag$-construction: if~$[\La,n,r,\b]$ is pure then~$[\La^\dag,n,r,\b^\dag]$ is also pure, since~$\F[\b]\simeq\F[\b^\dag]$; and if~$[\La,n,r,\b]$ is simple then the same applies to~$[\La^\dag,n,r,\b^\dag]$, since~$k_\F(\b)=k_\F(\b^\dag)$.}

When pure strata intertwine, they share several invariants, and we also get a certain isomorphism between subextensions of the residue fields which is important for concordance in the unitary case (see Lemma~\ref{lem:concordwithyb}).

\begin{lemma}\label{lemma:intertwiningminimalstrataramindices}\label{lemma:isoResfields}
Let~$[\La,n,r,\b]$ and~$[\La',n',r',\b']$ be non-null pure strata, with~$r<n$ and~$r'<n'$, which intertwine in~$\tG$, and put~$\E=\F[\b]$ and~$\E'=\F[\b']$. Then
\begin{enumerate}\setlength\itemsep{5pt}
\item\label{lemma:intertwiningminimalstrataramindices.i} the strata have the same depth, $n/e(\La)=n'/e(\La')$;
\item\label{lemma:intertwiningminimalstrataramindices.ii} if both strata are simple \shauns{and~$r/e(\La)=r'/e(\La')$} then~$e(\E/\F)=e(\E'/\F)$ and~$f(\E/\F)=f(\E'/\F)$;
\item\label{lemma:intertwiningminimalstrataramindices.iii} the strata have the same characteristic and minimal polynomials;
\item\label{lemma:isoResfields.iv} there is a~$\mathrm{k}_\F$-\daniel{linear field} isomorphism from~$\mathrm{k}_{\F}[y_\b]$ to~$\mathrm{k}_{\F}[y_{\b'}]$ which sends~$y_\b$ to~$y_{\b'}$.
\end{enumerate}
\end{lemma}

\begin{proof}
\shaun{\ref{lemma:intertwiningminimalstrataramindices.i} is given by~\cite[Proposition~6.9]{SkSt}. For~\ref{lemma:intertwiningminimalstrataramindices.ii}, by a~$\dag$-construction, we can assume that~$\La$ and~$\La'$ are regular strict and of the same period, hence conjugate; then~\cite[Theorem~2.6.1]{BK93} implies that the strata are conjugate up to equivalence, and the result follows from~\cite[Theorem~2.4.1(ii)]{BK93}. Finally, for~\ref{lemma:intertwiningminimalstrataramindices.iii} and~\ref{lemma:isoResfields.iv}, conjugating by an element which intertwines, we may assume the strata are intertwined by the identity, in which case~$y_\b^i=y_{\b'}^i$ in~$(\aa_0(\La)+\aa_0(\La'))/(\aa_1(\La)+\aa_1(\La'))$, for all~$i\ge 0$, and the results follow.}
\end{proof}

We also note that, under modest conditions, null strata and non-null simple strata do not intertwine.

\begin{lemma}[{\cite[Proposition~6.9]{SkSt}}]\label{lemma:fundNull}
Let~$[\La,n,r,\b]$ and~$[\La',n',n',0]$ be simple strata, with~$\b\neq 0$, and suppose that~$n'/e(\La')<n/e(\La)$. Then the strata do not intertwine in~$\tG$.
\end{lemma}

A stratum~$[\La,n,r,\b]$ is called \emph{self-dual} if~$\La$ is a self-dual~$\o_\F$-lattice sequence and~$\b\in\A_{-}$ \rob{Note that, }\daniel{this is a slight change of terminology: In~\cite[Definition 2.1]{St08} }\rob{these }\daniel{ strata }\rob{ are }\daniel{ called \emph{skew} strata; }\rob{we reserve \emph{skew} for certain self-dual strata which satisfy an additional condition, see Definition \ref{skewsesistratum}.}

A self-dual stratum~$[\La,n,r,\b]$ is called \emph{standard} if~$\La$ has even~$\o_\F$-period and~$\La=\La^\#$. Any self-dual stratum has an affine translation which is standard self-dual. \shaun{Note also that self-dual strata~$[\La,n,r,\b]$ behave well with respect to the self-dual~$\dag$-construction: that is, if~$[\La,n,r,\b]$ is self-dual then the stratum~$[\La^\dag,n,r,\b^\dag]$ obtained by the self-dual~$\dag$-construction is standard self-dual with respect to~$h^\dag$.} 
\shauns{Again, the self-dual~$\dag$-construction behaves well with respect to intertwining: if~$g\in\G$ intertwines two self-dual strata~$[\La,n,r,\b]$ and~$[\La',n',r',\b']$ then the element~$g^\dag\in\G^\dag$ intertwines~$[\La^\dag,n,r,\b^\dag]$ and~$[\La'^\dag,n',r',\b'^\dag]$.}

\shauns{An important application of this shows that equivalent self-dual simple strata have an additional invariant:}

\begin{lemma}\label{lemma:equivstrataramindicesequal}
\shauns{Let~$[\La,n,r,\b]$ and~$[\La',n,r,\b']$ be self-dual simple strata which intertwine in~$\G$, and suppose~$e(\La)=e(\La')$. Put~$\E=\F[\b]$,~$\E'=\F[\b']$. Then we have an equality of ramification indices~$e(\E/\E_\so)=e(\E'/\E'_\so)$.}
%Let~$[\La,n,r,\b],[\La,n,r,\b']$ be equivalent self-dual simple strata and put~$\E=\F[\b]$,~$\E'=\F[\b']$. Then we have an equality of ramification indices~$e(\E/\E_\so)=e(\E'/\E'_\so)$.
\end{lemma}

\begin{proof}
\shauns{By a self-dual~$\dag$-construction, without loss of generality we can assume that~$\La,\La'$ are \bob{regular} standard self-dual, so conjugate in~$\G$. Then~\cite[Theorem~8.5]{SkSt} implies that the strata are conjugate up to equivalence in~$\G$ so, by conjugating, we may assume they are equivalent.} 
Then, by \cite[5.2(i)]{BK94} the residue fields of~$\E$ and~$\E'$ coincide in~$\aa_0(\La)/\aa_1(\La)$ and thus the induced action of the adjoint anti-involution on the residue fields coincides, which finishes the proof.  
\end{proof}

\shaun{Many results concerning simple strata are proved ``by induction along~$r$'': that is, they are proved for minimal strata first, when~$r=n-1$ and then in general using the following fundamental approximation result.}

\begin{proposition}[{\cite[Theorem~2.4.1]{BK93}, \cite[Proposition~1.10]{St00}}]\label{prop:approx}
\shaun{Let~$[\La,n,r,\b]$ be a pure stratum. Then there is a simple stratum~$[\La,n,r,\g]$ equivalent to it and, for any such stratum,
\[
\text{$f(\F[\g]/\F)$ divides~$f(\F[\b]/\F)$\qquad and\qquad $e(\F[\g]/\F)$ divides~$e(\F[\b]/\F)$.}
\]
Moreover, if~$[\La,n,r,\b]$ is self-dual then~$[\La,n,r,\g]$ may be taken to be self-dual also.}
\end{proposition}
%%%%%%%%%%%%%%%%%%%%%%%%%%%%%%%%%%%%%%%%%%%%%%%%%%%%%%%%%%%%%%%%%%%%%%%%%%%%%
%% Perhaps we need (self-dual) pure is equivalent to (self-dual) simple here?
%%%%%%%%%%%%%%%%%%%%%%%%%%%%%%%%%%%%%%%%%%%%%%%%%%%%%%%%%%%%%%%%%%%%%%%%%%%%%

Finally in this subsection, we introduce concordance of pure strata. We introduce the following notation:

\begin{notation} 
For~$\b\in\A$ with~$\E=\F[\b]$ a field, we denote by~$\varphican$ the canonical embedding of~$\E$ in~$\A$.
\end{notation}

%Note that, if
If~$[\La,n,r,\b]$ is a self-dual pure stratum with~$\E=\F[\b]$, then~$(\E,\b)$ is a self-dual extension of~$\F/\F_\so$ and the canonical embedding~$\varphican$ %of~$\E$ in~$\A$ 
is a self-dual embedding.

\begin{definition}\label{Def:wittstrata}
Let~$[\La,n,r,\b]$ and~$[\La',n',r',\b']$ be self-dual pure strata in~$\A$. We say that they are \emph{concordant} if the pairs~$(\b,\varphican)$ and~$(\b',\varphicanprime)$ are concordant.
\end{definition}

If two self-dual pure strata are conjugate in~$\G$ then they are concordant (see Remark~\ref{rem:concordance}\ref{rem:concordance.ii}). The purpose of this section is to investigate the relationship between \emph{intertwining} of self-dual strata and concordance. 

\shauns{One particular case where we get concordance for free is by using the self-dual~$\dag$-construction. If~$[\La,n,r,\b]$ is a self-dual pure stratum with~$\E=\F[\b]$ and we write~$\varphi_\b$ also for the canonical embedding of~$\E$ in~$\A^\dag$ then~$[h^\dag_{\varphi_\b}]$ is always the trivial class. \shauny{Thus we get the following result:}}

\begin{lemma}\label{lem:dagconcord}
\shauny{If~$[\La,n,r,\b]$ and~$[\La',n',r',\b']$ are} \bob{non-null} \shauny{self-dual pure strata with~$e(\La)=e(\La')$, then~$[\La^\dag,n,r,\b^\dag]$ and~$[\La'^\dag,n',r',\b'^\dag]$ are concordant.}
\end{lemma}

%%%%%%%%%%%%%%%%%%%%%%%%%%%%%%%%%%%%
\subsection{Minimal elements and tamely ramified extensions}
%%%%%%%%%%%%%%%%%%%%%%%%%%%%%%%%%%%%
We will need the following \rob{lemmas} on minimal elements and tamely ramified extensions:
%Something motivational ``atoms'' according to BK. 

\begin{lemma}\label{lemma:tameminimal}
Suppose~$\E_1=\F[\b_1]$ is a tamely ramified finite extension, with~$\b_1$ a minimal element, and~$\E_2/\F$ is another finite extension, with~$\b_2\in\E_2^\times$. For~$i=1,2$, write~$e_i=e(\E_i/\F)$ for the ramification index and~$n_i=\val_{\E_i}(\b_i)$, and suppose that 
\begin{enumerate}\setlength\itemsep{5pt}
\item $n_1/e_1=n_2/e_2$, and
\item $\b_1^{e_1}\w_\F^{-n_1}+\p_{\E_1}$ and~$\b_2^{e_1}\w_\F^{-n_{1}}+\p_{\E_2}$ have the same minimal polynomial over~$\mathrm{k}_\F$.
\end{enumerate}
Then, there is a unique~$\F$-embedding~$\phi:\E_1\to\E_2$ such that 
\begin{equation}\label{eq:tameminimal}
\phi(\b_1)\b_2^{-1} \in \U^1_{\E_2}. 
\end{equation}
Furthermore, if~$(\E_1,\b_1)$ and~$(\E_2,\b_2)$ are self-dual field extensions of~$\F$, then~$\phi$ is~$(\ov{\phantom{a}})$-equivariant.
\end{lemma}

\begin{proof}
Let~$\mathsf{p}(X)\in \mathrm{k}_\F[X]$ denote the common minimal polynomial of~$\b_1^{e_1}\w_\F^{-n_1}\bmod{\p_{\E_1}}$ and~$\b_2^{e_1}\w_\F^{-n_1}\bmod{\p_{\E_2}}$ over~$\mathrm{k}_\F$. We take a monic polynomial~$P(X)\in\o_\F[X]$ whose reduction modulo~$\p_\F$ is~$\mathsf{p}(X)$. By Hensel's Lemma, for~$i=1,2$, we have roots~$\gamma_i\in\E_i$ of~$P(X)$ satisfying~$\gamma_i\equiv \b_i^{e_1}\w_\F^{-n_1}\bmod{\p_{\E_i}}$.
There is an~$\F$-monomorphism from~$\E_1$ into a separable closure of~$\E_2$ which maps~$\gamma_1$ to~$\gamma_2$.  Thus we can assume that~$\E_1/\F$ is totally ramified and~$\gamma_1=\gamma_2$.  

We may suppose that the uniformizer~$\w_\F$ is an~$e_1$-th power in~$\E_1$, so that~$\b_1^{e_1}\w_\F^{-n_1}$ is also an~$e_1$-th power. The latter is equal to~$\b_2^{e_1}\w_\F^{-n_1}\bmod{\p_{\E_2}}$ and Hensel's Lemma provides~$e_1$-th roots~$\xi_i\in E_i^\times$ of~$\b_i^{e_1}\w_\F^{-n_1}$  such that~$\xi_1\bmod{\p_{\E_1}}$ is equal to~$\xi_2\bmod{\p_{\E_2}}$ as elements of~$\mathrm{k}_\F$. Then~$\b_i\xi_i^{-1}$, for~$i=1,2$, are roots of the polynomial~$X^{e_1}-\w_\F^{n_1}$, which is irreducible over~$\F$ because~$\b_1\xi_1^{-1}$ generates~$\E_1$. The~$\F$-monomorphism~$\phi$ which maps~$\b_1\xi_1^{-1}$ to~$\b_2\xi_2^{-1}$ then satisfies~\eqref{eq:tameminimal}.

We now prove the uniqueness of~$\phi$. First observe that~\eqref{eq:tameminimal} implies that the map on~$\mathrm{k}_{\E_1}$ induced by~$\phi$ sends~$\b_1^{e_1}\w_\F^{-n_1}+\p_{\E_1}$ 
to~$\b_2^{e_1}\w_\F^{-n_{\blue{1}}}+\p_{\E_2}$. Since~$\b_1^{e_1}\w_\F^{-n_1}+\p_{\E_1}$ generates the residue field of~$\E_1$, we see that~$\phi$ is uniquely determined on the maximal unramified subextension of~$\E_1$ and, as above, we can assume without loss of generality that~$\E_1/\F$ is totally ramified. Again, we may suppose that the uniformizer~$\w_\F$ is an~$e_1$-th power in~$\E_1$.

By B\'ezout's Lemma, there are integers~$r,s$ such that~$\w_1=\b_1^r\w_\F^s$ is a uniformizer of~$\E_1$, and we set~$\w_2=\b_2^r\w_\F^s$ (which is not necessarily a uniformizer of~$\E_2$). Then~\eqref{eq:tameminimal} implies that~$\phi(\w_1)\w_2^{-1}$ is an element of~$\U^1_{\E_2}$. Therefore, if~$\phi'$ also satisfies~\eqref{eq:tameminimal}, then~$\phi(x)\phi'(x)^{-1}$ is an element of~$\U^1_{\E_2}$, for all~$x\in\E_1^\times$. In particular, if~$x$ is an~$e_1$-th root of~$\w_\F$, then~$\phi(x)\phi'(x)^{-1}$ is an~$e_1$-th root of unity in~$\U^1_{\E_2}$ and thus equal to~$1$, since $p$ does not divide~$e_1$. This completes the proof of uniqueness.

The equivariance assertion follows from the uniqueness since~$\phi$ and~$(\ov{\phantom{a}})\circ\phi\circ(\ov{\phantom{a}})$ both satisfy~\eqref{eq:tameminimal}.
\end{proof}

\begin{lemma}\label{lem:tameminimal}
Suppose~$\b$ is a minimal element of an algebraic closure of~$\F$ and set~$\E=\F[\b]$. Let~$\Ft/\F$ be the maximal tamely ramified field extension of~$\E/\F$ and set~$e_p:=[\E:\F]/[\Ft:\F]$, the wild ramification index of~$\E/\F$.
\begin{enumerate}\setlength\itemsep{5pt}
\item\label{lem:tameminimal.i} There is a non-zero element~$\bt$ of~$\Ft$ such that~$\b^{e_p}(\bt)^{-1}\in\U_\E^1$.
\item\label{lem:tameminimal.ii} Any element~$\bt$ as in~\ref{lem:tameminimal.i} is minimal over~$\F$ and generates~$\Ft$ over~$\F$.
\end{enumerate}
\end{lemma}

\begin{proof}
\shauny{We set~$n=\val_\E(\b)$.} 
\begin{enumerate}\setlength\itemsep{5pt}
\item Take a uniformizer~$\wt$ of~$\Ft$. 
\daniel{Then~$\val_{\E}(\wt)=e_p$ and there is a unit~$x$ of~$\Ft$ such that
$\b^{e_p}(\wt)^{-n}x^{-1}$ belong to~$\U_\E^1$.
The element~$\bt:=(\wt)^{n}x$ satisfies the assertion. 
}
% Then~$\val_{\E}(\wt)=e_p$, and therefore there is an integer~$s$ such that
% \[
% \val_{\E}((\wt)^s)=\val_{\E}(\b^{e_p}).
% \]
% Thus there is a unit~$x$ of~$\Ft$ such that
% \[
% \b^{e_p}(\wt)^{-s}x^{-1}\in \U_\E^1.
% \]
% The element~$\bt:=(\wt)^{s}x$ satisfies the assertion. 
\item We take an element~$\bt$ as in~\ref{lem:tameminimal.i}. It generates a sub-extension~$\L/\F$ of~$\Ft/\F$. Set~$e=e(\E/\F)$, $e_\tame=e(\Ft/\F)$ and~$n_\tame=\val_{\Ft}(\bt)$. 
\daniel{Then~$n=n_\tame$ and} there is an element~$u\in \U^1_\E$ such that
\[
\b^{e}\w_\F^{-n}u=(\bt)^{e_\tame}\w_\F^{-n_\tame}.
\]
Thus, by the minimality of~$\b$, we obtain that the residue class of 
\begin{equation}\label{eqResclassMin}
(\bt)^{e_\tame}\w_\F^{-n_\tame}
\end{equation}                                                                                              
generates the residue field extension~$\mathrm{k}_{\E}/\mathrm{k}_\F$. Thus~$\mathrm{k}_\L=\mathrm{k}_{\E}$, and~$\E$ contains the maximal unramified sub-extension of~$\E/\F$. Further~$n_\tame$ is prime to the ramification index~$e_\tame$, since~$e_\tame$ divides~$e$ and~$\b$ is minimal. Thus
\[
\ZZ\cup\{\infty\}=\val_{\Ft}(\Ft)=\val_{\Ft}(\L),
\]
and therefore~$\Ft=\L$. So~$\bt$ generates~$\Ft$ and the minimality follows. 
\end{enumerate}
\end{proof}

%%%%%%%%%%%%%%%%%%%%%%%%%%%%%%%%%%%%
\subsection{Tame subextensions}
%%%%%%%%%%%%%%%%%%%%%%%%%%%%%%%%%%%%
%%%%%% Some words of introduction?

\begin{lemma}\label{lem:tamesubconj}
\shaun{Let~$[\La,n,n-1,\b]$ be a self-dual pure stratum, and~$[\La',n,n-1,\g]$ be a self-dual simple stratum, in~$\A$, which intertwine in~$\G$. Let~$\Ft$ denote the maximal tamely ramified subextension of~$\F[\g]/\F$ and set~$\E=\F[\b]$.
\begin{enumerate}\setlength\itemsep{5pt}
\item\label{lem:tamesubconj.i} There exists~$g\in \G$ such that~$g\Ft g^{-1}\subseteq \E$.
\item\label{lem:tamesubconj.ii} If~$\La'=\La$, then there exists~$g\in \P_-(\La)$ such that~$g\Ft g^{-1}\subseteq \E$.
\end{enumerate}
}
\end{lemma}

\begin{proof}
\shaun{Set~$e_p=[\F[\g]:\Ft]$. By Lemma~\ref{lem:tameminimal}, there is an element~$\gt$ in~$\Ft$ such that~$\gt \g^{-e_p}\in \U^1_{\F[\g]}$; since~$\g$ is skew and~$e_p$ is odd, we can choose~$\gt$ to be skew, and it is also minimal.}

\shaun{The strata~$[\La',ne_p,ne_p-1,\g^{e_p}]$ and~$[\La',ne_p,ne_p-1,\gt]$ are then equivalent. Thus the pure strata~$[\La,ne_p,ne_p-1,\b^{e_p}]$ and~$[\La',ne_p,ne_p-1,\gt]$ intertwine, so have a common characteristic polynomial and a common minimal polynomial, by Lemma~\ref{lemma:intertwiningminimalstrataramindices}. Thus we can apply Lemma~\ref{lemma:tameminimal} with~$\b_1=\gt$ and~$\b_2=\b^{e_p}$ to deduce that there is an equivariant monomorphism~$\phi:\Ft\rightarrow \E$ such that~$\phi(\gt)\b^{-e_p} \in \U^1_\E$. Then the strata~$[\La,ne_p,ne_p-1,\phi(\gt)]$ and~$[\La,ne_p,ne_p-1,\b^{e_p}]$ are equivalent.}

\shaun{It follows that the simple strata~$[\La',ne_p,ne_p-1,\gt]$ and~$[\La,ne_p,ne_p-1,\phi(\gt)]$ intertwine. Then~\cite[Theorem~5.2]{SkSt} implies the existence of an element of~$\G$ which conjugates~$\gt$ to~$\phi(\gt)$. Thus we have found~$g\in\G$ such that~$g\Ft g^{-1}=\phi(\Ft)\subseteq \E$.}

\shaun{Finally, if~$\La'=\La$ then~\cite[Theorem 1.2]{SkodField} implies that the element conjugating~$\gt$ to~$\phi(\gt)$ can be chosen to be in~$\P_-(\La)$.}
\end{proof}

Now let~$[\La,n,n-1,\b]$ and~$[\La',n,n-1,\b']$ be self-dual pure strata in~$\A$ which intertwine in~$\G$, and put~$\E=\F[\b],\E'=\F[\b']$.  Let~$[\La,n,n-1,\gamma]$ and~$[\La',n,n-1,\gamma']$ be self-dual simple strata in~$\A$, respectively equivalent to~$[\La,n,n-1,\b]$ and~$[\La',n,n-1,\b']$.

\begin{corollary}\label{cor:strataK}
With the notation above, there exists~$g\in \G$ such that~$\E\cap g\E'g^{-1}$ contains a tamely ramified extension~$\K/\F$ which is stable under the adjoint anti-involution but is not fixed pointwise. Moreover,~$\K$ can be chosen to be~$\P_-(\La)$-conjugate to the maximal tamely ramified subextension of~$\F[\gamma]/\F$. Furthermore, if~$\La=\La'$ then we can take~$g\in\P_-(\La)$.
\end{corollary}

\begin{proof}
\shaun{Let~$\Ft$ and~$\Fprt$ denote the maximal tamely ramified subextensions of~$\F[\g]/\F$ and~$\F[\g']/\F$ respectively. We apply Lemma~\ref{lem:tamesubconj} several times to find elements~$x\in\P_-(\La)$,~$x'\in\P_-(\La')$ and~$y\in\G$ such that
\[
x\Ft x^{-1}\subseteq\E,\qquad x'^{-1}\Fprt x'\subseteq\E',\qquad y^{-1}\Ft y\subseteq \Fprt.
\]
Then~$\K=x\Ft x^{-1}$ and~$g=xyx'$ are as required. Moreover, if~$\La=\La'$ then~$y$ can be chosen in~$\P_-(\La)$, by Lemma~\ref{lem:tamesubconj}\ref{lem:tamesubconj.ii}, and then~$g\in\P_-(\La)$ as required.}
\end{proof}

%%%%%%%%%%%%%%%%%%%%%%%%%%%%%%%%%%%%
\subsection{Concordance in the symplectic case}
%%%%%%%%%%%%%%%%%%%%%%%%%%%%%%%%%%%%
Using the technique of twisting, we now get an analogue of Proposition~\ref{prop:firstmatch} for the symplectic case, when we have an additional hypothesis on intertwining of strata.

\begin{lemma}\label{lemma:symplecticstrataconditionsimplyconcordance1}
Suppose that~$\e=-1$ and~$\F=\F_\so$. Let~$[\La,n,n-1,\b]$ and~$[\La',n,n-1,\b']$ be self-dual pure strata in~$\A$ which intertwine in~$\G$, and put~$\E=\F[\b],\E'=\F[\b']$. Suppose further that~$\dim_\E\V,\dim_{\E'}\V$ have the same parity, and either
\begin{enumerate}\setlength\itemsep{5pt}
\item this common parity is even; or
\item this common parity is odd and~$\sw_1(\la 1\ra)=\la 1\ra$.
\end{enumerate}
Then the strata are concordant.
\end{lemma}

\begin{proof}
\shaun{Conjugating by an element of~$\G$ (which does not affect concordance, by Remark~\ref{rem:concordance}), we can assume that the strata are intertwined by the identity. Then~\cite[Lemma~5.3]{SkSt} applied twice (as in the proof of~\cite[Theorem~5.2]{SkSt}) implies that the spaces~$(\V,\b^*h)$ and~$(\V,\b'^*h)$ are isometric (by an element of~$\P^1(\La')\P^1(\La)$).}

\shaun{Now Corollary~\ref{cor:strataK} implies that the hypotheses of Proposition~\ref{prop:firstmatch} are satisfied so that the pairs~$(\b,\varphican)$ and~$(\b',\varphicanprime)$ are~$(\b^*h,\b'^*h)$-concordant, and it follows from Lemma~\ref{lemma:twistSympToOrthConc} that our original strata are concordant.}
\end{proof}

\shaun{The following Lemma will be useful when we need to understand whether concordance is preserved when we pass from pure strata to equivalent simple strata.}

\begin{lemma}\label{lemma:symplecticstrataconditionsimplyformshyperbolic}
Suppose that~$\e=-1$ and~$\F=\F_\so$. Let~$[\La,n,n-1,\b],~[\La,n,n-1,\b']$ be self-dual pure strata in~$\A$ which intertwine in~$\G$ and put~$\E=\F[\b]$ and~$\E'=\F[\b']$. % and let~$\varphican$ and~$\varphicanprime$ denote their respective canonical embeddings into~$\A$.  
Suppose further that~$\dim_\E\V$ is odd,~$\dim_{\E'}\V$ is even and~$[\La,n,n-1,\b']$ is simple.  
\shaun{Then~$h_{\varphicanprime}$ is hyperbolic if and only if~$[\b^*(h_{\varphican})]=\la 1\ra$.}
\end{lemma}

\shaun{Note that the condition~$[\b^*(h_{\varphican})]=\la 1\ra$ can be translated into a condition on~$h_\varphican$, using Corollary~\ref{cor:sdexts1}: either~$[h_{\varphican}]=\la \b \ra$ and~$-1$ is a square in~$\E$, or~$[h_{\varphican}]\neq\la \b \ra$ and~$-1$ is a non-square in~$\E$.}

\begin{proof}
\shaun{The stratum~$[\La,n,n-1,\b]$ is equivalent to a self-dual simple stratum~$[\La,n,n-1,\g]$ in~$\A$, which also intertwines~$[\La,n,n-1,\b']$. By~\cite[8.5]{SkSt}, this implies that~$[\La,n,n-1,\g]$ is, up to equivalence, conjugate to~$[\La,n,n-1,\b']$ in~$\G$ and, replacing~$[\La,n,n-1,\b']$ by its conjugate, we may assume that~$[\La,n,n-1,\b]$ and~$[\La,n,n-1,\b']$ are equivalent. Then~$(\V,\b^*(h))$ is isometric to~$(\V,\b'^*(h))$ by an element~$u\in\P^1(\La)$ by~\cite[Lemma~5.3]{SkSt}.%\cite[Proposition~3.1]{SkSt} (applied to the identity map).
}

\shaun{We show first that~$\l_\b^*(\la 1\ra)=\bs 0$. By Corollary~\ref{cor:strataK} we can choose~$u$ such that~$u^{-1}\E' u\cap \E$ contains a \rob{Galois-}invariant subfield~$\K$ with~$\K\neq \K_\so$. The assumptions imply that~$[\E:\K]$ is even, since~$\dim_\E\V$ is odd while~$\dim_\K\V$ is even because~$\K$ is contained in~$u^{-1}\E' u$.}

\shaun{Choose non-zero \rob{Galois-}equivariant linear forms
\[
\lambda_{\K/\F}:\K\rightarrow \F,\ \lambda_{\E/\K}:\E\rightarrow \K.
\]
Then~$\l_\b^*$ and~$\lambda_{\K/\F}^*\circ\lambda_{\E/\K}^*$ on~$\Ww_1(\E/\E_\so)$ have the same image, by Proposition~\ref{prop:lambda*}\ref{prop:lambda*.i}. Now the image of~$\lambda_{\E/\K}^*$ is contained in~$\Ww_1^\even(\K/\K_\so)$, so consists of~$\bs 0$ and the maximal element. Then, by Proposition~\ref{prop:lambda*}, the image of~$\l_\b^*$ also consists of~$\bs 0$ and the maximal element (which has anisotropic dimension four). However,~$\lambda_\b^*(\la 1\ra)$ has anisotropic dimension at most~$2$, by Proposition~\ref{prop:transfer}\ref{prop:transfer.ii}, so~$\lambda_\b^*(\la 1\ra)=\bs 0$.}

\shaun{Now, since~$\dim_\E\V$ is odd, the class~$[\b^*(h_\varphican)]$ is that of a~$1$-dimensional anisotropic space. Since~$[\b^*h] = \l_\b^*([\b^*(h_\varphican)])$ and~$\l_\b^*$ is injective on the~$1$-dimensional anisotropic spaces, it follows that~$\b^*h$ is hyperbolic if and only if~$[\b^*(h_\varphican)]=\la 1\ra$.}

\shaun{On the other hand, since~$\dim_{\E'}\V$ is even, the class~$[\b'^*h]$ is either~$\bs 0$ or maximal. By injectivity of~$\l_{\b'}^*$ on~$\Ww_1^\even(\E'/\E'_\so)$ and Proposition~\ref{prop:lambda*}\ref{prop:lambda*.ii}, we deduce that~$\b'^*h$ is hyperbolic if and only if~$\b'^*(h_{\varphicanprime})$ is hyperbolic, which occurs if and only if~$h_{\varphicanprime}$ is hyperbolic since~$\b'^*$ is an isomorphism.}

\shaun{Putting together the results of the last two paragraphs, since~$(\V,\b^*h)$ and~$(\V,\b'^*h)$ are isometric we see that~$h_{\varphicanprime}$ is hyperbolic if and only if~$[\b^*(h_\varphican)]=\la 1\ra$, as required.}
\end{proof}

%%%%%%%%%%%%%%%%%%%%%%%%%%%%%%%%%%%%
\subsection{Concordance of intertwining simple strata}
%%%%%%%%%%%%%%%%%%%%%%%%%%%%%%%%%%%%
\shaun{We are now in a position to prove that self-dual simple strata which intertwine in~$\G$ are concordant.} 
\shauns{In fact, we deduce it from the following numerical criterion for concordance of pure strata.}

\begin{lemma}\label{lemma:purestrataintertwineandramimpliesconcordant}
\shauns{Let~$[\La,n,n-1,\b]$ and~$[\La',n,n-1,\b']$ be self-dual pure strata which intertwine in~$\G$. Setting~$\E=\F[\b]$ and~$\E'=\F[\b']$, suppose further that
\[
e(\E/\F)=e(\E'/\F),~f(\E/\F)=f(\E'/\F),~\text{and}~e(\E/\E_\so)=e(\E'/\E'_\so),
\]
Then the strata are concordant.}
\end{lemma}

\begin{proof}
\shauns{It follows from the assumptions and Lemma~\ref{lemma:isoResfields} that~$e(\La)=e(\La')$ and~$\val_\E(\b)=\val_{\E'}(\b')$, while there is a~$\mathrm{k}_\F$-\daniel{linear field isomorphism between~$\mathrm{k}_\E$ and~$\mathrm{k}_{\E'}$} which sends~$y_\b$ to~$y_{\b'}$.} 
\shaun{Thus the self-dual \rob{extensions}~$(\E,\b)$ and~$(\E',\b')$ are \emph{similar}, in the sense of Definition~\ref{def:similar}.}
 
\shaun{The result now follows from Lemma~\ref{lemma:symplecticstrataconditionsimplyconcordance1} when~$\G$ is symplectic, from Proposition~\ref{prop:firstmatch} together with Corollary~\ref{cor:strataK} when~$\G$ is orthogonal, and from Lemma~\ref{lem:concordwithyb} when~$\G$ is unitary.}
\end{proof}

\begin{corollary}\label{cor:strataintertwineandramimpliesconcordant}
\shauns{Let~$[\La,n,r,\b]$ and~$[\La',n,r,\b']$} be self-dual simple strata which intertwine in~$\G$. Then they are concordant.
\end{corollary}

\begin{proof}
\shauns{If~$r=n$ then both strata are null so there is nothing to prove. Otherwise, the elements~$\b$ and~$\b'$ are necessarily both non-zero by the definition of simple stratum, because the third parameter~$r$ is smaller than~$n$. Put~$\E=\F[\b]$ and~$\E'=\F[\b']$. Then \rob{Lemmas}~\ref{lemma:intertwiningminimalstrataramindices} and~\ref{lemma:equivstrataramindicesequal} imply that~$e(\La)=e(\La')$ and that
\[
e(\E/\F)=e(\E'/\F),~f(\E/\F)=f(\E'/\F),~\text{and}~e(\E/\E_\so)=e(\E'/\E'_\so),
\] 
and the result follows from Lemma~\ref{lemma:purestrataintertwineandramimpliesconcordant}.}
\end{proof}

%%%%%%%%%%%%%%%%%%%%%%%%%%%%%%%%%%%%
\subsection{Intertwining of concordant pure strata}
%%%%%%%%%%%%%%%%%%%%%%%%%%%%%%%%%%%%
Finally in this section, we consider the~$\G$-intertwining of self-dual pure strata which intertwine in~$\tG$:

\begin{proposition}\label{prop:strataintertwiningovertG}
\shauns{Let~$[\La,n,r,\b]$ and~$[\La',n,r,\b']$} be self-dual pure strata in~$\A$ which intertwine in~$\tG$.
\begin{enumerate}\setlength\itemsep{5pt}
\item\label{prop:strataintertwiningovertG.i} 
\shaun{If~$\e=1$ or~$\F\neq\F_\so$ then they intertwine in~$\G$.}
\item\label{prop:strataintertwiningovertG.ii} 
\shaun{If the strata are simple and concordant then they intertwine in~$\G$.}
\end{enumerate}
\end{proposition}
\shaun{Note that, together with Corollary~\ref{cor:strataintertwineandramimpliesconcordant}, this implies that, in the non-symplectic case, self-dual simple strata which intertwine in~$\tG$ are automatically concordant; this is not true in the symplectic case.}

\begin{proof}
\shauns{If~$r=n$ then all strata are null and there is nothing to prove, so we assume~$r<n$.
As the strata intertwine in~$\tG$, the stratum~$[\La\oplus \La',n,r,\b+\b']$ is equivalent to a simple stratum~$[\La\oplus \La',n,r,\gamma]$ in~$\End_{\F}(\V\oplus \V)$ by~\cite[Proposition~7.1]{SkSt}. By~\cite[Theorem~6.16]{SkSt}, we can moreover take~$[\La\oplus \La',n,r,\gamma]$ to be self-dual with respect to~$h\oplus h$ and~$\gamma=(\gamma_1,\gamma_2)\in\A\oplus\A$. Replacing~$\b,\b'$ with~$\gamma_1,\gamma_2$, we can therefore assume that~$\b,\b'\in\A_-$ have the same minimal polynomial over~$\F$. In the situation of~\ref{prop:strataintertwiningovertG.ii}, Corollary~\ref{cor:strataintertwineandramimpliesconcordant} ensures that the strata are still concordant so this replacement is possible.}

\shaun{Now~\ref{prop:strataintertwiningovertG.i} follows from \cite[Corollary~5.1]{SkSt}, while~\ref{prop:strataintertwiningovertG.ii} follows from Remark~\ref{rem:concordance}\ref{rem:concordance.ii}.}
\end{proof}

\orange{
For later use, we also get the following corollary on similar self-dual extensions.
\begin{corollary}\label{cor:similardiagramscommute}
Suppose~$(\E,\b)$ and~$(\E',\b')$ are similar self-dual extensions. Then the diagrams~\eqref{eqDiagConcOdd} and~\eqref{eqDiagConcEven} are commutative.
\end{corollary}}
\begin{proof}
\orange{There is nothing to prove in the symplectic case, while the result is given by Lemma~\ref{lem:concordwithyb}\ref{lem:concordwithyb.ia} if~$\F\ne\F_\so$. Thus we suppose we are in the orthogonal case:~$\F=\F_\so$ and~$\e=1$. Replacing~$\b,\b'$ by~$\w_\F^{2k}\b,\w_\F^{2k}\b'$ respectively, for suitable~$k<0$, we may assume that~$\val_\E(\b)=\val_{\E'}(\b')<0$.}

\orange{
\bob{
Let~$(\V,h_\E)$ be a~$2$-dimensional hyperbolic space over~$\E/\E_\so$ and~$\La$ a strict self-dual~$\o_\E$-lattice sequence in~$\V$ of~$\o_\E$-period~$2$ such that~$\La(0)^{\#_{h_\E}}=\La(0)$. \shauny{By~\cite[Lemma 5.5]{BS09}, we can choose a non-zero Galois-equivariant~$\F$-linear map~$\lambda:\E\rightarrow\F$ such that, setting~$h:=\lambda\circ h_\E$, we have 
\[
\L^{\#_{h_\E}}=\L^{\#_h}, \text{for all~$\o_\E$-lattices~$\L$ in~$\V$.}
\]
In particular,~$\La(0)^{\#_{h}}=\La(0)$.}
%\[\La(z)^{\#_h}=\La(-z),\ z\in\ZZ.\]
Doing the same with~$\E'/\E'_\so$, we obtain a space~$(\V',h')$ isometric to~$(\V,h)$ and a regular self-dual~$\o_{\E'}$-lattice sequence~$\La'$ with~$e(\La|\o_\F)=~e(\La'|\o_\F)$ and~$\La'(0)^{\#_{h'}}=\La'(0)$.}
% Let~$(\V,h_\E)$ be a~$2$-dimensional hyperbolic space over~$\E/\E_\so$ and~$\La$ a strict self-dual~$\o_\E$-lattice sequence in~$\V$ of~$\o_\E$-period~$2$. We set~$h=\l_\b^*(h_\E)$ so that~$(\V,h)$ is a hyperbolic orthogonal space over~$\F$ and~$\La$ has~$\o_\F$-period~$e(\La|\o_\F)=2e(\E/\F)$. Doing the same with~$\E'/\E'_\so$, we obtain a space~$(\V',h')$ isometric to~$(\V,h)$ and a strict self-dual~$\o_{\E'}$-lattice sequence~$\La'$ with~$e(\La|\o_\F)=~e(\La'|\o_\F)$.
By~\cite[Proposition 5.2]{SkodField}, there is an isometry from~$(\V,h)$ to~$(\V',h')$ which sends~$\La$ to~$\La'$ so we may assume~$(\V,h)=(\V',h')$ and~$\La=\La'$.
}

\orange{
%We identify~$\b,\b'$ with their images under~$\vphi.\vphi'$ respectively, so that
Now~$[\La,n,n-1,\b]$ and~$[\La,n,n-1,\b']$ are self-dual pure strata, where~$n=-\val_{\La}(\b)=-\val_{\La}(\b')$. By~\cite[2.5.8,2.5.11]{BK93} both strata are equivalent to simple strata in~\emph{$\g$-standard form} for the same~$\g$, since~$y_\b,y_{\b'}$ have the same minimal polynomial. Thus the strata intertwine in~$\tG$, and Proposition~\ref{prop:strataintertwiningovertG} implies that they intertwine in~$\G$. Now Corollary~\ref{cor:strataK} implies that there is a} \gre{$g\in\G$ and an}  \orange{extension~$\K/\F$ contained in~$\E\cap g\E'g^{-1}$ which is stable under the adjoint anti-involution but is not fixed pointwise. Moreover, since~$\E/\F$ and~$\E'/\F$ have the same residue degree, Lemma~\ref{lemma:w1oneifsquares} implies that~$\sw_1(\la 1\ra)=\la 1\ra$. Thus the hypotheses of Lemma~\ref{lemma:commutConcDiag} are satisfied and we conclude that the diagrams commute as required.
}
\end{proof}

%% file: Endo-chars.tex
%%%%%%%%%%%%%%%%%%%%%%%%%%%%%%%%%%%%

%%%%%%%%%%%%%%%%%%%%%%%%%%%%%%%%%%%%
\section{Self-dual simple characters: intertwining and concordance}\label{sec:simplechars}
%%%%%%%%%%%%%%%%%%%%%%%%%%%%%%%%%%%%
In this section, we investigate intertwining of self-dual simple characters and concordance of their underlying simple strata.  The main result is Proposition~\ref{prop:TiGandGIntertwiningSameNonSympl}.  %We will say that self-dual simple characters are \emph{Witt-concordant} if they are defined relative to concordant self-dual field extensions~$\F[\b],\F[\b']$, and, as a simple character does not determine the field extension, show that this is equivalent to any self-dual field extensions defining the simple characters being concordant.

In previous works on self-dual simple characters it is often assumed that the characters take values in the complex numbers~$\mathbb{C}$, for example in~\cite{St05}. However, as they are characters of pro-$p$ groups and~$\CC$ contains a complete set of $p$-power roots of unity, all of the results apply equally well over~$\CC$, and often we just refer to the results over~$\mathbb{C}$.

For the rest of the paper, we fix a non-trivial character~$\psi_{\so}:\F_\so\rightarrow \CC^\times$ of conductor~$\p_{\F_\so}$, and define~$\psi:\F\rightarrow\CC^\times$ by~$\psi=\psi_\so\circ\T_{\F/\F_\so}$.

%%%%%%%%%%%%%%%%%%%%%%%%%%%%%%%%%%%%
\subsection{Simple characters}
%%%%%%%%%%%%%%%%%%%%%%%%%%%%%%%%%%%%
Let~$[\La,n,r,\b]$ be a non-zero simple stratum in~$\A$.  Associated to~$[\La,n,r,\b]$ is an~$\o_{\F}$-order~$\HH(\b,\La)$~in~$\A$ defined inductively, see~\cite[\S3.1]{BK93} for the original definition when~$\La$ is strict and~\cite[\S5]{BK99} in general. 
% If~$[\La,n,n-1,\b]$ is simple, i.e.~$\b$ is minimal, then we set  \[\HH(\b,\La)=;\] otherwise, let~$[....]$ be a simple stratum equivalent to.... and we set, 
%\[\HH(\b,\La)=...\]
For~$m\geqslant 1$, we put
\[
\H^{m}(\b,\La)=\HH(\b,\La)\cap \P^{m}(\La),
\]
a compact open subgroup of~$\tG$.  

Also associated to~$[\La,n,r,\b]$, and our fixed character~$\psi$, is a set~$\Cc(\La,r,\b)$ of \emph{simple characters} of~$\H^{r+1}(\b,\La)$, defined in \cite[\S 3.2]{BK93} when~$\La$ is strict and in~\cite[\S5]{BK99} in general. In the case of a zero simple stratum~$[\La,n,n,0]$ we put~$\H^{n+1}(0,\La)=\P^{n+1}(\La)$ and define the associated set of simple characters to be~$\Cc(\La,n,0)=\{1_{\P^{n+1}(\La)}\}$.  

Given two simple strata~$[\La,n,r,\b]$ and~$[\La',n',r',\b]$ in~$\A$, with~$\E=\F[\b]$, such that~$\left\lfloor\frac{r}{e(\La|\o_\E)}\right\rfloor=\left\lfloor\frac{r'}{e(\La'|\o_\E)}\right\rfloor$, there is a canonical bijection~$\tau_{\La',\La,\b}:\Cc(\La,r,\b)\rightarrow \Cc(\La',r',\b)$ called~\emph{transfer}, see \cite[3.6.1]{BK93},~\cite[Section~2.1]{St05} and~\cite[Section~3.1]{secherreI}; if~$\theta\in\Cc(\La,r,\b)$ then~$\tau_{\La',\La,\b}(\theta)$ is the unique simple character~$\theta'\in\Cc(\La',r',\b)$ such that $1\in\tG$ intertwines~$\theta$ with~$\theta'$. \shaun{Note that, although we omit it from our notation, the transfer map depends on the integers~$(r,r')$.}

\daniel{In general intertwining between simple characters does not imply intertwining between underlying strata, but we still have the following important implication: }\rob{Suppose that}\daniel{~$[\La,n,r,\b]$ and~$[\La',n',r',\b']$ are simple strata and~$g$ is an element of~$\tG$ which intertwines a character in~$\Cc(\La,r,\b)$ with a character in~$\Cc(\La',r',\b')$. Then~$g$ intertwines~$[\La,n,\max(n-1,r),\b]$ with~$[\La',n',\max(n'-1,r'),\b']$, noting that every element of~$\Cc(\La,r,\b)$ and~$\Cc(\La',r',\b')$ restricts to the character attached to~$[\La,n,\max(n-1,r),\b]$ and~$[\La',n',\max(n'-1,r'),\b']$, respectively, see~\cite[2.1]{St08} and~\cite[9.5]{SkSt}.} 

%\rob{THIS DOES NOT SEEM SYMMETRIC IN ~$[\La,n,r,\b]$ and~$[\La,n',r',\b']$ as we have $max(n-1,r')$ in the second? ALSO want to note something for the restriction of the characters in~$\Cc(\La',r',\b')$ as well rather than leave this implicit.}

By induction on~$k_\F(\b)$, the groups and sets of simple characters only depend on the affine class of the strata:

\begin{proposition}
Let~$[\La,n,r,\b]$ and~$[\La',n',r',\b]$ be simple strata in the same affine class.  Then
\[
\H^{r+1}(\b,\La)=\H^{r'+1}(\b,\La'),\qquad \Cc(\La,r,\b)=\Cc(\La',r',\b),
\]
and the transfer map~$\tau_{\La',\La,\b}:\Cc(\La,r,\b)\rightarrow \Cc(\La',r',\b)$ is the identity.
\end{proposition}

%by \cite{}{\color{red}Provide reference} this is well-defined independent of the choice of stratum.
The next proposition shows how intertwining of simple characters interacts with certain arithmetic invariants of the underlying simple strata: 

\begin{proposition}[{cf. \cite[3.5.1]{BK93}}]\label{prop:SimpleDegrees}
Let~$[\La,n,r,\b]$ and~$[\La',n',r',\b']$ be simple strata in~$\A$ satisfying~$e(\La)=e(\La')$. Suppose that one of the following two conditions is satisfied:
\begin{enumerate}\setlength\itemsep{5pt}
\item $r=r'$ and there are simple characters~$\theta\in\Cc(\La,r,\b)$ and~$\theta'\in\Cc(\La',r,\b')$ which intertwine in~$\tG$; \label{prop:SimpleDegrees.i}
\item $\Cc(\La,r,\b)$ and~$\Cc(\La',r',\b')$ intersect non-trivially.\label{prop:SimpleDegrees.ii}
\end{enumerate}
Then, we have
\[ 
k_0(\b,\La)=k_0(\b',\La'),\quad e(\F[\b]/\F)=e(\F[\b']/\F),~\text{and}~f(\F[\b]/\F)=f(\F[\b']/\F).
\]
\end{proposition}

Note that in the second part we allow~$r\neq r'$. The condition~$e(\La)=e(\La')$ can always be obtained by changing  the strata in their affine classes.

\begin{proof}
Suppose that condition~\ref{prop:SimpleDegrees.i} holds.  If~$\theta$ is trivial then the first stratum is null, but then the other stratum is also null (since otherwise~$[\La,r,r,0]$ would intertwine with~$[\La',n',n'-1,\b']$, which is impossible by Lemma~\ref{lemma:fundNull}) and the result follows. Otherwise, both characters are non-trivial and the strata are non-null. Furthermore, by a~$\dag$-construction we can assume that the lattice sequences are strict and regular, of the same period. Then there is an element of~$\tG$ which maps~$\La$ to~$\La'$ so, by~\cite[Theorem~3.5.11]{BK93}, the simple characters are conjugate, and the result then follows from~\cite[Proposition~3.5.1]{BK93}. 

Suppose that condition \ref{prop:SimpleDegrees.ii} holds, and let~$\theta\in\Cc(\La,r,\b)\cap\Cc(\La',r',\b')$. Without loss of generality we assume that~$r\leq r'$.  Let~$\widetilde{\theta}\in\Cc(\La',r,\b')$ be an extension of~$\theta$ (cf.~\cite[(3.2.5)]{BK93}). Then~$\tilde{\theta}$ intertwines with~$\theta$ and condition~\ref{prop:SimpleDegrees.i} holds, and we conclude by the last case. 
\end{proof}

The \emph{degree} of a simple character~$\theta$ is the index~$[\F[\b]:\F]$ for a stratum~$[\La,n,r,\b]$ such that~$\theta\in\Cc(\La,r,\b)$; by Proposition~\ref{prop:SimpleDegrees} this is well-defined, i.e. it is independent of the choice of the stratum.

\shaun{In order to prove results by induction along~$r$, we also need to know what happens when we restrict a simple character.}

\begin{proposition}[{\cite[Corollary~3.3.20,~3.2.5]{BK93},~\cite[Remarks~3.14]{St05}}]\label{prop:restrictsimplechar}
\shaun{Let~$[\La,n,r,\b]$ be a simple stratum and let~$[\La,n,r+1,\g]$ be a simple stratum equivalent to~$[\La,n,r+1,\b]$. Then~$\H^{r+1}(\b,\La)=\H^{r+1}(\g,\La)$ and restriction to~$\H^{r+2}(\b,\La)$ induces a surjective map~$\Cc(\La,r,\b)\to\Cc(\La,r+1,\g)$. Moreover, the assignment
\[
\theta\mapsto \theta\psi_{\b-\g}
\]
defines a bijection~$\Cc(\La,r,\g)\to\Cc(\La,r,\b)$.}
\end{proposition}

\shaun{Let~$[\La,n,r,\b]$ be a simple stratum and let~$[\La,n,r+1,\g]$ be a simple stratum equivalent to~$[\La,n,r+1,\b]$. In this situation, we will write~$\B_\g=\End_{\F[\g]}(\V)$ and~$\tG_\g=\B_\g^\times$. Associated to the field~$\F[\g]$, there is a map~$s_\g:\A\to\B_\g$, called a \emph{tame corestriction} (see~\cite[Definition~1.3.3]{BK93} and~\cite[Definition~6.12]{SkSt}). If~$\theta$ \rob{belongs to}~$\Cc(\La,r,\b)$ and~$\theta_0\in\Cc(\La,r,\g)$ is any extension of~$\theta|_{\H^{r+2}(\b,\La)}$, then we can write~$\theta=\theta_0\psi_{\b-\g+c}$, for some~$c\in\aa_{-(r+1)}(\La)$. The stratum
\[
[\La,r+1,r,s_\g(\b-\g+c)]
\]
is called a \emph{derived stratum} in~$\B_\g$; this derived stratum is equivalent to a simple stratum (\cite[Theorem~6.14]{SkSt}).}

%%%%%%%%%%%%%%%%%%%%%%%%%%%%%%%%%%%%
\subsection{Self-dual simple characters}
%%%%%%%%%%%%%%%%%%%%%%%%%%%%%%%%%%%%
Let~$[\La,n,r,\b]$ be a self-dual simple stratum in~$\A$.  Then the subgroup~$\H^{r+1}(\b,\La)$ together with the set of simple characters~$\Cc(\La,r,\b)$ are stable under the involution~$\sigma$, and we define
\begin{align*}
\H^{r+1}_-(\b,\La)&=\H^{r+1}(\b,\La)^\Sigma=\H^{r+1}(\b,\La)\cap\G;\\
\Cc^\Sigma(\La,r,\b)&=\{\theta\in \Cc(\La,r,\b):\theta^\sigma=\theta\}.
\end{align*}
Thus~$\H^{r+1}_-(\b,\La)$ is a compact open subgroup of~$\G$, and we have a set of~\emph{self-dual simple characters} of~$\H^{r+1}_-(\b,\La)$ defined by restriction:
\[
\Cc_-(\La,r,\b)=\left\{\theta\mid_{\H^{r+1}_-(\b,\La)}:\theta\in\Cc^\Sigma(\La,r,\b)\right\}.
\]
This restriction of characters coincides with the Glauberman correspondence (see \cite[\S 2]{St00}), and defines a bijection $\Cc^\Sigma(\La,r,\b)\rightarrow \Cc_-(\La,r,\b)$. Given~$\theta_-\in\Cc_-(\La,r,\b)$ we will call the unique~$\theta\in\Cc^\Sigma(\La,r,\b)$ such that~$\theta\mid_{\H^{r+1}_-(\b,\La)}=\theta_-$ the \emph{lift of~$\theta_-$ \shauny{with respect to~$(\La,r,\b)$}}. \shauny{ (Below, we will simply write \emph{lift of~$\theta_-$}, because the stratum will be given implicitly.)}  We also define the \emph{degree} of a self-dual simple character to be the degree of \daniel{any} of its lift. \daniel{This is well-defined by Proposition~\ref{prop:SimpleDegrees}\ref{prop:SimpleDegrees.i} and the following proposition:}

A consequence of the Glauberman correspondence, see \cite[2.5]{St00}, is the following proposition:
 
\begin{proposition}\label{prop:Glauberman}
Let~$\theta_-\in\Cc_-(\La,r,\b)$ and~$\theta'_-\in\Cc_-(\La',r',\b')$ be self-dual simple characters with lifts~$\theta$ and~$\theta'$ respectively.  Then an element of~$\G$ intertwines~$\theta$ with~$\theta'$ if and only if it intertwines~$\theta_-$ with~$\theta_-'$.
\end{proposition} 
% if~$\theta_-\in\Cc_-(\La,r,\b)$ and~$\theta'_-\in\Cc_-(\La',r',\b')$ are self-dual simple characters with lifts~$\theta$ and~$\theta'$ respectively, then~$g\in\G$ intertwines~$\theta$ with~$\theta'$ if and only if it intertwines~$\theta_-$ with~$\theta_-'$.
 %
%
%{\color{red}Define~$\H^{r+1}_-(\b,\La)\leqslant \P_-^{r+1}(\La)$, \emph{self-dual simple characters}~$\Cc_-(\La,r,\b)$, %Glauberman correspondence,  
%}
%
%
%\begin{definition}
%Let~$\theta\in\Cc_-(\La,r,\b)$,~$\theta'\in\Cc_-(\La',r',\b')$ be simple characters with~$\b,\b'\not\in\F$.  We say~$\theta,\theta'$ are \emph{Witt-concordant} if the canonical embeddings of~$\F[\b],\F[\b']$ in~$\A$ are concordant.
%{\color{red}We could extend the definition of Witt-concordance from strata to simple characters: however~$\F[\b]$ is not determined by~$\theta$, so maybe this doesn't make sense so don't want to do this, and it should be to do with the data one has fixed to define~$\theta$, i.e. the stratum.}

\shaun{When we restrict self-dual simple characters, as in Proposition~\ref{prop:restrictsimplechar}, we may do so relative to a self-dual simple approximation~$[\La,n,r+1,\g]$ of our stratum. In particular, when only self-dual simple characters are considered, the derived strata we obtain will also be self-dual. Writing~$\varphi_\g$ for the canonical embedding of~$\F[\g]$ in~$\A$ as usual, we also write~$\G_\g$ for the unitary group of the form~$h_{\varphi_\g}$ (defined relative to the standard linear form~$\l_\g$).}

\shauns{We also have a self-dual~$\dag$-construction for self-dual simple characters. Since self-dual simple characters are in bijection with~$\sigma$-stable characters via Glauberman, we describe it for these. Let~$[\La,n,r,\b]$ be a self-dual simple stratum and~$\theta\in\Cc^\Sigma(\La,r,\b)$, and denote by~$\M^\dag$ the Levi subgroup of~$\G^\dag$ which stabilizes the decomposition~$\V^\dag=\V\oplus\cdots\oplus\V$. There is a unique simple character~$\theta^\dag\in\Cc(\La^\dag,r,\b^\dag)$ whose restriction to~$\H^{r+1}(\b^\dag,\La^\dag)$ has the form~$\theta\otimes\cdots\otimes\theta$ (see~\cite[Lemma~2.7]{BSS}). Moreover, by uniqueness~$\theta^\dag\in\Cc^\Sigma(\La^\dag,r,\b^\dag)$. If we have two self-dual simple characters~$\theta,\theta'$ and~$g\in\G$ intertwines~$\theta$ with~$\theta'$ then~$g^\dag\in\G^\dag$ intertwines~$\theta^\dag$ with~$\theta'^\dag$, by~\cite[Lemma~4.4]{RKSS}.}

%%%%%%%%%%%%%%%%%%%%%%%%%%%%%%%%%%%%
\subsection{\rob{Lemmas}}
%%%%%%%%%%%%%%%%%%%%%%%%%%%%%%%%%%%%
In this subsection, we prove some partial results towards the main result in the next subsection. Let~$[\La,n,r,\b],[\La',n,r,\b]$ be self-dual simple strata in~$\A$, put~$\E=\F[\b]$ and~$\E'=\F[\b']$. %{\color{red}Is it clear to the reader what~$\E_\so$ etc is here... check we have said this in the Witt sections before.}  {\color{red}Make sure earlier sections say that~$e(\La)=e(\La|\o_\F)$.}
\shaun{The first lemma we prove gives a strong ``intertwining implies conjugacy'' result in the non-symplectic case: one does not even need to assume that the intertwining takes place in~$\G$.}
 
% \begin{lemma}\label{lemma:intoverGmeansramsequal}
% Let~$\theta\in\Cc^\Sigma(\La,r,\b)$ and~$\theta'\in\Cc^\Sigma(\La',r,\b')$ be simple characters which intertwine in~$\G$ and suppose that~$e(\La)=e(\La')$.   Then~$e(\E/\E_\so)=e(\E'/\E'_\so)$.
% \end{lemma}
% \begin{proof}
% \blue{We use a~$\dagger$-construction to reduce to the case of principal lattice chains and by~\cite[Proposition 5.2]{skodFields} there is an element of~$G^\dag$ 
% which maps~$\La^\dag$ to~$\La'^\dag$. Then by~\cite[Theorem 10.2]{SkSt} there is an element of~$G^\dag$ which conjugates~$\theta^\dag$ to~$\theta'^\dag$. 
% And by~\cite[Theorem (3.5.8)]{BK93} we can assume that both strata define the same set of semisimple characters. 
% Now we conclude as in the proof of Lemma~\cite{lemma:equivstrataramindicesequal}.}
% %{\color{red}self-dual~$\dagger$-construction Lemma 3.2, should be the same as the one in the new strata section.
% }
% \end{proof}

\begin{lemma}\label{lemma:InTildeGConGSimple}
Suppose that we are in the non-symplectic case, and let~$\theta_{-}\in \Cc_-(\La,r,\b)$ and $\theta'_{-}\in \Cc_-(\La,r,\b')$ be self-dual simple characters with lifts~$\theta$ and~$\theta'$, respectively. Suppose that~$\theta$ and~$\theta'$ intertwine in~$\tG$, then~$\theta_-$ and~$\theta'_-$ are conjugate in~$\P_-(\La)$.
%Then, their lifts are conjugate over~$\G$ if and only if their lifts intertwine in~$\tG$.
\end{lemma}

\begin{proof}
% Let~$\theta\in \Cc^\Sigma(\La,r,\b)$ be the lift of~$\theta_-$ and~$\theta'\in \Cc^\Sigma(\La,r,\b')$ be the lift of~$\theta'_-$.  
\shaun{There is nothing to show if both strata are null. The proof is by induction along~$r$. For the base case~$r=n-1$, Proposition~\ref{prop:strataintertwiningovertG}\ref{prop:strataintertwiningovertG.i} implies that the strata intertwine in~$\G$, and then~\cite[Theorem 8.5]{SkSt} implies that they are conjugate in~$\P_-(\La)$. Assume now that~$r<n-1$. Since~$\theta|_{\H^{r+2}(\b,\La)}$ and~$\theta'|_{\H^{r+2}(\b',\La)}$ are simple characters by Proposition~\ref{prop:restrictsimplechar} \orange{(for approximations~$[\La,n,r+1,\g]$ and~$[\La,n,r+1,\g']$ of the strata with~$\b$ and~$\b'$ respectively)}, by the inductive hypothesis they are conjugate by an element of~$\P_-(\La)$; conjugating by this element, we can assume that they are equal, so that
\[
\Cc(\La,r+1,\g)=\Cc(\La,r+1,\g')
\]
by~\cite[Proposition~9.23]{SkSt}. \shauns{%Let~$[\La,n,r+1,\g]$ be a self-dual simple stratum equivalent to~$[\La,n,r+1,\b]$. 
Then, by~\cite[Theorem~9.26]{SkSt}, there is a self-dual simple stratum~$[\La,n,r,\b'']$ such that
\[
\Cc(\La,r,\b')=\Cc(\La,r,\b'') \quad\text{and}\quad\text{$[\La,n,r+1,\b'']$ is equivalent to~$[\La,n,r+1,\g]$.}
\]
Replacing~$\b'$ by~$\b''$, we may therefore assume that~$\g'=\g$; that is,~$[\La,n,r+1,\g]$ is equivalent to both $[\La,n,r+1,\b]$ and~$[\La,n,r+1,\b']$.}
%such that~$\theta$ and~$\theta'$ coincide on~$\H^{r+1}(\g,\La)$.
Thus there are a skew-symmetric element~$c$ of~$\aa_{-(r+1)}(\La)$ and a simple character~$\theta_0\in \Cc^\Sigma(\La,r,\g)$ 
%{\color{red}What is~$\a_{-(r+1)}$} 
such that~$\theta=\theta_0\psi_{\b-\g+c}$ and~$\theta'=\theta_0\psi_{\b'-\g}$. Now~\cite[Proposition~9.17(i)]{SkSt} implies that the derived strata~$[\La,r+1,r,s_\g(\b-\g+c)]$ and~$[\La,r+1,r,s_\g(\b'-\g)]$ intertwine in the centralizer~$\tG_\g$, whence also in~$\G_\g$ by the base case. But then~\cite[Proposition~9.27(ii)]{SkSt} implies that~$\theta$ and~$\theta'$ intertwine in~$\G$ and the result then follows from~\cite[Theorem 10.3]{SkSt}.}
\end{proof}

% This lemma leads to a useful corollary, improving Lemma \ref{lemma:intoverGmeansramsequal} in the non-symplectic setting:
% 
% \begin{corollary}\label{corEqualEE0}
% Suppose that we are in the non-symplectic case, let~$\theta\in \Cc^\Sigma(\La,r,\b)$ and~$\theta'\in \Cc^\Sigma(\La',r,\b')$ be simple characters which intertwine in~$\tG$, and suppose that~$e(\La)=e(\La')$. Then~$e(\E/\E_\so)=e(\E'/\E'_\so)$.
% \end{corollary}
% 
% \begin{proof}
% As in the proof of Lemma~\ref{lemma:intoverGmeansramsequal} 
% %~\ref{lemRamOfEE0determinedbyGIntertwining}
% %{\color{red}Check this step}, 
% we can reduce to the case where~$\La$ and~$\La'$ coincide.  Proposition~\ref{propInTildeGConGSimple} implies that the characters are conjugate over~$\G$, and hence intertwine \blue{in}~$\G$.  Hence we can conclude by Lemma~\ref{lemma:intoverGmeansramsequal}.
% \end{proof}

%\gre{We should maybe somewhere say that if we write~$\Cc^\Sigma(\La',r,\b')$ that we mean implicitly that the stratum is self-dual.}
%\gre{We also should define Witt concordance for the case when~$\b$ and~$\b'$ are zero. Ans we should also emphasize thateven if we say that two embeddings are concordant that
%we mean that there are choices of~$\b$ and~$\b'$ behind it. For example, if we change~$\b$ to~$-\b$ then the embeddings may fail to be concordant. 
%See, def 4.1 Def:wittstrata and the def of concordant embeddings in the the Witt file.}
\begin{lemma}\label{lemma:MatchWittforIntertwCharacters}
%Let~$\theta\in \Cc^\Sigma(\La,r,\b)$ and~$\theta'\in \Cc^\Sigma(\La',r,\b')$ be simple characters which intertwine in~$\tG$. Suppose that~$e(\La)=e(\La')$, and if~$\G$ is symplectic suppose further that~$\theta$ and~$\theta'$ intertwine in~$\G$. Then 
\shaun{Let~$\theta\in \Cc^\Sigma(\La,r,\b)$ and~$\theta'\in \Cc^\Sigma(\La',r,\b')$ be simple characters which intertwine in~$\tG$, and suppose that~$e(\La)=e(\La')$. If~$\G$ is symplectic suppose further that~$\theta$ and~$\theta'$ intertwine in~$\G$. Then} 
\begin{enumerate}\setlength\itemsep{5pt}
\item\label{lemma:MatchWittforIntertwCharacters.i} $e(\E/\E_\so)=e(\E'/\E'_\so)$;
\item\label{lemma:MatchWittforIntertwCharacters.ii} \orange{the pairs~$(\b,\varphican)$ and~$(\b',\varphicanprime)$ are concordant.}
%the canonical embeddings of~$\b$ and~$\b'$ in~$\A$ are concordant.
\end{enumerate}
\end{lemma}

\begin{proof}
\shaun{We first prove the equality in~\ref{lemma:MatchWittforIntertwCharacters.i}. By a self-dual~$\dagger$-construction we can reduce to the case of standard strict regular lattice sequences of the same period. By~\cite[Proposition 5.2]{SkodField} there is then an element of~$\G$ which maps~$\La$ to~$\La'$. Thus~$\theta$ and~$\theta'$ intertwine by an element of~$\G$, in the symplectic case by assumption and in the non-symplectic case by Lemma~\ref{lemma:InTildeGConGSimple}. Then by~\cite[Theorem 10.3]{SkSt} there is an element of~$\G$ which conjugates~$\theta$ to~$\theta'$ and, by~\cite[Theorem~3.5.8]{BK93}, we can conjugate to assume that both strata define the same set of simple characters. We now conclude as in the proof of Lemma~\ref{lemma:equivstrataramindicesequal}, by looking at the image of the residue fields~$\mathrm{k}_\E$,~$\mathrm{k}_{\E'}$ in~$\aa_0(\La)/\aa_1(\La)$.}

\shaun{We now turn to~\ref{lemma:MatchWittforIntertwCharacters.ii}. If either character is trivial then, since they intertwine, both are by Lemma~\ref{lemma:fundNull}; then~$\b$ and~$\b'$ vanish, and the result follows. Otherwise, both characters are non-trivial and both strata are non-null. From the intertwining of the two characters we find
\begin{itemize}\setlength\itemsep{5pt}
\item $[\La,n,n-1,\b]$ and~$[\La',n,n-1,\b']$ intertwine in~$\G$, by restriction; 
\item $e(\E/\F)=e(\E'/\F)$ and~$f(\E/\F)=f(\E'/\F)$, by Proposition~\ref{prop:SimpleDegrees}\ref{prop:SimpleDegrees.i};
\item $e(\E/\E_0)=e(\E'/\E'_0)$, by~\ref{lemma:MatchWittforIntertwCharacters.i}. 
\end{itemize}
The result now follows from Lemma~\ref{lemma:purestrataintertwineandramimpliesconcordant}.}
\end{proof}

%Let~$[\La,n,r,\b]$ and~$[\La',n,r,\b']$ be self-dual simple strata in~$\A$.  Fix defining sequences $([\La,n,r+i,\g_i])_{i=0}^{n-r}$ for~$[\La,n,r,\b]$ and ~$([\La',n,r+i,\g'_i])_{i=0}^{n-r}$ for~$[\La',n,r,\b']$ of self-dual simple strata. We now show that, in the presence of intertwining simple characters, concordance is inherited along the defining sequences:

\begin{lemma}\label{lemma:WitttowersDefSecuence}
%Under the notation above, let~$\theta\in \Cc^\Sigma(\La,r,\b)$ and~$\theta'\in \Cc^\Sigma(\La',r,\b')$ be simple characters which intertwine in~$\tG$ and suppose that~$e(\La)=e(\La')$. If~$\G$ is symplectic that the canonical embeddings of~$\b$ and~$\b'$ in~$\A$ are concordant.  Then, for all~$i$, the canonical embeddings of~$\g_i$ and~$\g_i'$ in~$\A$ are concordant.
\shaun{Let~$\theta\in \Cc^\Sigma(\La,r,\b)$ and~$\theta'\in \Cc^\Sigma(\La',r,\b')$ be simple characters which intertwine in~$\tG$ and suppose that~$e(\La)=e(\La')$. If~$\G$ is symplectic suppose further that \orange{the pairs~$(\b,\varphican)$ and~$(\b',\varphicanprime)$ are concordant.} %the canonical embeddings of~$\b$ and~$\b'$ in~$\A$ are concordant. 
Let~$[\La,n,r+1,\g]$ and~$[\La',n,r+1,\g']$ be self-dual simple strata equivalent to~$[\La,n,r+1,\b]$ and~$[\La',n,r+1,\b']$ respectively. \orange{Then the pairs~$(\g,\varphigcan)$ and~$(\g',\varphigcanprime)$ are concordant.}
%Then the canonical embeddings of~$\g$ and~$\g'$ are concordant.
}
\end{lemma}

\begin{proof}
\shaun{Since the restrictions~$\theta|_{\H^{r+2}(\b,\La)}$ and~$\theta'|_{\H^{r+2}(\b',\La')}$ are simple characters for~$\g,\g'$ respectively which intertwine in~$\tG$, in the non-symplectic case the result follows from Lemma~\ref{lemma:MatchWittforIntertwCharacters}. Suppose now that we are in the symplectic case. %We denote the canonical embedding of~$\F[\b]$ by~$\varphi_\b$, and use similar notation for the other fields. 
If~$r=n-1$ then~$\g,\g'$ are both zero and there is nothing to prove, so we suppose~$r<n-1$.}

\shaun{If~$\dim_{\F[\b]}(\V)$ and~$\dim_{\F[\g]}(\V)$ have the same parity then~$(\b,\varphican)$ and~$(\g,\varphigcan)$ are concordant, by Lemma~\ref{lemma:symplecticstrataconditionsimplyconcordance1}. 
(Note that, since~$f(\F[\g]/\F)$ divides~$f(\F[\b]/\F)$ and~$e(\F[\g]/\F)$ divides~$e(\F[\b]/\F)$, if the common parity is odd then~$f(\F[\g]/\F)$ and~$f(\F[\b]/\F)$ have the same~$2$-power divisor so~$\sw_{1,\g,\b}(\la 1\ra)=\la 1\ra$, by Lemma~\ref{lemma:w1oneifsquares}.) Since the invariants for~$\b',\g'$ are the same as for~$\b,\g$ respectively (by Proposition~\ref{prop:SimpleDegrees}), we also have that~$(\b',\varphicanprime)$ and~$(\g',\varphigcanprime)$ are concordant, and the result follows by transitivity of concordance.}

\shaun{Otherwise,~$\dim_{\F[\b]}(\V)$ is odd and~$\dim_{\F[\g]}(\V)$ is even. Let~$[\La,n,n-1,\g_0]$ be a simple stratum equivalent to~$[\La,n,n-1,\g]$, so that~$\dim_{\F[\g_0]}(\V)$ is also even. As in the previous case, the pairs~$(\g,\varphigcan)$ and~$(\g_0,\varphi_{\g_0})$ are concordant. Using analogous notation for~$\b',\g'$, we also have concordant pairs~$(\g',\varphigcanprime)$ and~$(\g'_0,\varphi_{\g'_0})$. Finally, since~$-1$ is a square in~$\F[\b]$ if and only if it is a square in~$\F[\b']$, it follows from Lemma~\ref{lemma:symplecticstrataconditionsimplyformshyperbolic} (see also the remark following that lemma) and the fact that~$(\b,\varphican)$ and~$(\b',\varphicanprime)$ are concordant that~$(\g_0,\varphi_{\g_0})$ and~$(\g'_0,\varphi_{\g'_0})$ are also concordant. The result again follows by transitivity.}
\end{proof}

%%%%%%%%%%%%%%%%%%%%%%%%%%%%%%%%%%%%
\subsection{Intertwining self-dual simple characters}
%%%%%%%%%%%%%%%%%%%%%%%%%%%%%%%%%%%%
%After Sections~\ref{MatchWitt} and~\ref{Symplecticcasesection}, 
In the main result of the section, Proposition~\ref{prop:TiGandGIntertwiningSameNonSympl} below, we investigate the relation between~$\G$-intertwining of self-dual simple characters and~$\tG$-intertwining of their lifts.  This improves Lemma~\ref{lemma:InTildeGConGSimple} in the non-symplectic case to allow for non-conjugate lattice sequences, and proves the analogue in the symplectic case using concordance.  We start with the case of the same lattice sequence:

\begin{proposition}\label{prop:InTildeGConGSimpleSymp}
Let~$\theta_-\in \Cc_-(\La,r,\b)$  and~$\theta'_-\in \Cc_-(\La,r,\b')$ \gre{be self-dual simple characters} with lifts~$\theta$ and~$\theta'$ respectively.
Then the following assertions are equivalent:
\begin{enumerate}\setlength\itemsep{5pt}
\item\label{prop:InTildeGConGSimpleSymp.i} $\theta_-$ and~$\theta'_-$ are conjugate in~$\P_-(\La)$.
\item\label{prop:InTildeGConGSimpleSymp.ii} $\theta_-$ and~$\theta'_-$ intertwine in~$\G$.
\item\label{prop:InTildeGConGSimpleSymp.iii} $\theta$ and~$\theta'$ are conjugate in~$\P(\La)$ and \orange{the pairs~$(\b,\varphican)$ and~$(\b',\varphicanprime)$ are concordant.} %the canonical embeddings of~$\b$ and~$\b'$ in~$\A$ are concordant.
\item\label{prop:InTildeGConGSimpleSymp.iv} $\theta$ and~$\theta'$ intertwine in~$\tG$ and \orange{the pairs~$(\b,\varphican)$ and~$(\b',\varphicanprime)$ are concordant.} %the canonical embeddings of~$\b$ and~$\b'$ in~$\A$ are concordant.
\end{enumerate}

%Then~$\theta_-$ and~$\theta'_-$ intertwine in~$\G$ if and only if their lifts are conjugate over~$\tG$ and the canonical embeddings of~$\b$ and~$\b'$ in~$\A$ are concordant.
\end{proposition}

\begin{proof}

%The proof is similar to the proof of Proposition~\ref{propInTildeGConGSimple}.  
The first equivalence \ref{prop:InTildeGConGSimpleSymp.i}$\Leftrightarrow$\ref{prop:InTildeGConGSimpleSymp.ii} and the last equivalence \ref{prop:InTildeGConGSimpleSymp.iii}$\Leftrightarrow$\ref{prop:InTildeGConGSimpleSymp.iv} follow from~\cite[Theorem 10.2 and 10.3]{SkSt}.  So we only have to prove the second equivalence  \ref{prop:InTildeGConGSimpleSymp.ii}$\Leftrightarrow$\ref{prop:InTildeGConGSimpleSymp.iii}.

If~$\theta_-$ and~$\theta'_-$ intertwine in~$\G$, then their lifts intertwine in~$\G$ and hence are conjugate by an element of~$\P(\La)$, by~\cite[Theorem 10.2]{SkSt} (see also~\cite[3.5.11]{BK93} when~$\La$ is \bob{strict}). Moreover, by Lemma~\ref{lemma:MatchWittforIntertwCharacters}, \orange{the pairs~$(\b,\varphican)$ and~$(\b',\varphicanprime)$ are concordant.} %the canonical embeddings of~$\b$ and~$\b'$ are concordant. 

\shauns{We prove the converse by induction along~$r$. If~$r=n-1$ then the simple characters~$\theta,\theta'$ (that is, the strata~$[\La,n,n-1,\b]$ and~$[\La,n,n-1,\b']$) intertwine in~$\G$ by Proposition~\ref{prop:strataintertwiningovertG}. The proof of the inductive step is now identical to that in Lemma~\ref{lemma:InTildeGConGSimple}, with two small additional arguments: first we use Lemma~\ref{lemma:WitttowersDefSecuence} and the induction hypothesis to conjugate~$\t|_{\H^{r+2}(\b,\La)}$ to~$\t'|_{\H^{r+2}(\b',\La)}$; secondly, when we obtain derived strata~$[\La,r+1,r,s_\g(\b-\g+c)]$ and~$[\La,r+1,r,s_\g(\b'-\g)]$ which intertwine in the centralizer~$\tG_\g$, %since this centralizer is a unitary group, 
Proposition~\ref{prop:strataintertwiningovertG}\ref{prop:strataintertwiningovertG.i} implies that, since~$\G_\g$ is a unitary group, %and Corollary~\ref{cor:strataintertwineandramimpliesconcordant} imply 
these strata intertwine in~$\G_\g$.}%are also concordant, which means that the base case can be applied.}
\end{proof}

\shaun{From Proposition~\ref{prop:InTildeGConGSimpleSymp} we now get a strengthening (in the symplectic case) of Lemma~\ref{lemma:MatchWittforIntertwCharacters}\ref{lemma:MatchWittforIntertwCharacters.i}.}
%From Proposition \ref{prop:InTildeGConGSimpleSymp} we get an analogue of Lemma~\ref{lemma:MatchWittforIntertwCharacters}\ref{lemma:MatchWittforIntertwCharacters.i} with the difference that we replace~$\G$-intertwining hypothesis by concordance in the symplectic case. 
%\red{and~$\tG$-intertwining of the lifts?}. 
% \gre{The Lemma and the Corollary are quite similar. One could state the Corollary only for the symplectic case.}\red{Sure, but~$\G$-intertwining cannot be replaced solely by concordant embeddings? one needs~$\tG$ intertwining as well?}
%to Corollary~\ref{corEqualEE0}, also valid in the symplectic case, with a similar proof: 

\begin{corollary}\label{cor:EqualEE0Sympl}
\shaun{Let~$\theta\in \Cc^\Sigma(\La,r,\b)$ and~$\theta'\in \Cc^\Sigma(\La',r,\b')$ be simple characters which intertwine in~$\tG$, and suppose that~$e(\La)=e(\La')$. %If~$\G$ is symplectic suppose further that the canonical embeddings of~$\b$ and~$\b'$ in~$\A$ are concordant. 
%Then~$e(\E/\E_\so)=e(\E'/\E'_\so)$.}
\orange{Then~$(\E,\b)$ and~$(\E',\b')$ are similar self-dual extensions (see Definition~\ref{def:similar}).}}
\end{corollary}

\begin{proof} 
\orange{We already know, by Proposition~\ref{prop:SimpleDegrees}, that~$\E/\F$ and~$\E'/\F'$ have the same ramification index and residue class degree, and, by Lemma~\ref{lemma:intertwiningminimalstrataramindices} applied to the pure strata~$[\La,n,n-1,\b]$ and\rob{~$[\La',n,n-1,\b']$}, that the elements~$y_\b$ and~$y_{\b'}$ have the same (irreducible) minimal polynomial over~$\mathrm{k}_\F$. Since~$\val_\E(\b)=-n e(\E/\F)/e(\La)=\val_{\E'}(\b')$, it only remains to show that~$e(\E/\E_\so)=e(\E'/\E'_\so)$.}

\shaun{For this, by a self-dual~$\dag$-construction, we may assume further that \orange{the pairs~$(\b,\varphican)$ and~$(\b',\varphicanprime)$ are concordant\shauny{, by Lemma~\ref{lem:dagconcord}}.} %the canonical embeddings of~$\b$ and~$\b'$ in~$\A$ are concordant. 
The proof is now the same as that of Lemma~\ref{lemma:MatchWittforIntertwCharacters}\ref{lemma:MatchWittforIntertwCharacters.i}, except that we use Proposition~\ref{prop:InTildeGConGSimpleSymp} to obtain that the characters~$\t^\dag$ and~$\t'^\dag$ intertwine in~$\G^\dag$.}
%
%The non-symplectic case is given by Lemma~\ref{lemma:MatchWittforIntertwCharacters}\ref{lemma:MatchWittforIntertwCharacters.i}, and the proof of the symplectic case is the same as that of Lemma~\ref{lemma:MatchWittforIntertwCharacters}\ref{lemma:MatchWittforIntertwCharacters.i}, except that we use Proposition~\ref{prop:InTildeGConGSimpleSymp} to obtain that the characters~$\t^\dag$ and~$\t'^\dag$ intertwine in~$\G^\dag$. 
%
%\red{There is a problem of sorts here: we have not proved that if~$\b,\b'$ are concordant then so are~$\b^\dag,\b'^\dag$. But in fact, if~$\varphi_\b,\varphi_{\b'}$ are the canonical embeddings (of~$\b^\dag,\b'^\dag$), then aren't~$h^\dag_{\varphi_\b}$ and~$h^\dag_{\varphi_{\b'}}$ both hyperbolic anyway? In which case we always have concordance and the concordance hypothesis in the corollary is unnecessary?}
\end{proof}

\begin{proposition}\label{prop:TiGandGIntertwiningSameNonSympl}
Let~$\theta_-\in\Cc_{-}(\La,r,\b)$ and~$\theta'_-\in\Cc_{-}(\La',r,\b')$ be self-dual simple characters of~$\G$, and suppose that~$e(\La)=e(\La')$.  
%\begin{enumerate}
%\item\label{prop:TiGandGIntertwiningSameNonSympl-i} In the non-symplectic case,~$\theta_-$ and~$\theta'_-$ intertwine in~$\G$ if and only if their lifts intertwine in~$\tG$.  
%\item In the symplectic case,~$\theta_-$ and~$\theta'_-$ intertwine in~$\G$ if and only if their lifts intertwine in~$\tG$ and the canonical embeddings of~$\b$ and~$\b'$ are concordant.
%\end{enumerate}
\shaun{Then~$\theta_-$ and~$\theta'_-$ intertwine in~$\G$ if and only if their lifts intertwine in~$\tG$ and \orange{the pairs~$(\b,\varphican)$ and~$(\b',\varphicanprime)$ are concordant.}}%the canonical embeddings of~$\b$ and~$\b'$ are concordant.}
\end{proposition}

\shaun{We remark that, in the non-symplectic case the hypothesis on concordance %of \orange{the pairs~$(\b,\varphican)$ and~$(\b',\varphicanprime)$ %the canonical embeddings of~$\b,\b'$ 
is in fact not necessary: if the lifts of~$\theta_-$ and~$\theta'_-$ intertwine in~$\tG$ then, by Lemma~\ref{lemma:MatchWittforIntertwCharacters}, \orange{the pairs~$(\b,\varphican)$ and~$(\b',\varphicanprime)$ are automatically concordant.}} %the canonical embeddings of~$\b$ and~$\b'$ are automatically concordant.}

%\blue{Apparently this proposition needs a more technical proof. The circumstance of different lattice sequences produces a stronger complication. 
%This is the reason for the next lemma.}

To prove Proposition~\ref{prop:TiGandGIntertwiningSameNonSympl}, we will need the following lemma:

\begin{lemma}\label{lemma:diagonal}
\shaun{Let~$[\La,n,r,\b]$ and~$[\La',n,r,\b']$ be self-dual simple strata and let~$\theta\in\Cc^\Sigma(\La,r,\b)$ and~$\theta'\in\Cc^\Sigma(\La',r,\b')$. Suppose that~$\theta,\theta'$ intertwine in~$\tG$ and that \orange{the pairs~$(\b,\varphican)$ and~$(\b',\varphicanprime)$ are concordant.} %the canonical embeddings of~$\b$ and~$\b'$ are concordant. 
Then, there are self-dual simple strata~$[\La,n,r,\b_1]$ and~$[\La',n,r,\b'_1]$ such that~$\b_1$ and~$\b'_1$ have the same characteristic polynomial,~$\theta\in\Cc^\Sigma(\La,r,\b_1)$ and~$\theta'\in\Cc^\Sigma(\La',r,\b'_1)$.  Moreover, for any such~$\b_1,\b'_1$, \orange{the pairs~$(\b_1,\vphi_{\b_1})$ and~$(\b'_1,\vphi_{\b'_1})$ are concordant.} }%the canonical embeddings of~$\b_1$ and~$\b'_1$ are concordant.}
\end{lemma}

Granting this, we complete the proof of Proposition~\ref{prop:TiGandGIntertwiningSameNonSympl}:

\begin{proof}
Let~$\theta\in \Cc^\Sigma(\La,r,\b)$ and~$\theta'\in \Cc^\Sigma(\La',r,\b')$ denote the lifts of~$\theta_-$ and~$\theta'_-$, respectively. %We write~$\varphi$ an~$\varphi'$ for the canonical embeddings of~$\b$ and~$\b'$. 
Suppose first that~$\theta$ and~$\theta'$ intertwine by an element of~$\G$. Then they intertwine by an element of~$\tG$ and, by Lemma~\ref{lemma:MatchWittforIntertwCharacters}\ref{lemma:MatchWittforIntertwCharacters.ii}, \orange{the pairs~$(\b,\varphican)$ and~$(\b',\varphicanprime)$ are concordant.} %the embeddings~$\varphi$ and~$\varphi'$ are concordant. 

Suppose now that~$\theta$ and~$\theta'$ intertwine by an element of~\rob{$\tG$} and \orange{the pairs~$(\b,\varphican)$ and~$(\b',\varphicanprime)$ are concordant.} %the embeddings~$\varphi,\varphi'$ are concordant. 
By Lemma~\ref{lemma:diagonal} we can assume that~$\b$ and~$\b'$ have the same characteristic polynomial and \orange{the pairs~$(\b,\varphican)$ and~$(\b',\varphicanprime)$ are still concordant;} %are concordant; 
thus, by Remark~\ref{rem:concordance}\ref{rem:concordance.ii}, they are conjugate by an element~$g\in\G$. The characters~$\theta$ and~$\theta'':=\tau_{g\La,\La',\b'}(\theta')$ intertwine by an element of~$\tG$ by~\cite[Theorem~1.11]{BSS}. Thus~$\theta$ and~$\theta''$ are conjugate by an element of~$\G$, by Proposition~\ref{prop:InTildeGConGSimpleSymp}. We deduce that~$\theta$ and~$\theta'$ intertwine in~$\G$, since~$\theta'$ and~$\theta''$ are intertwined by~$1$. 
\end{proof}

It remains only to prove Lemma~\ref{lemma:diagonal}:

\begin{proof}
The proof is by induction along~$r$. There is nothing to show if both strata are null so we assume that they are both non-null; in particular both characters are non-trivial.

The base case is~$r=n-1$. Applying \cite[Proposition~7.1]{SkSt}, as in the proof of Proposition~\ref{prop:strataintertwiningovertG}, we can replace~$\b$ and~$\b'$ by elements~$\b_1$ and~$\b'_1$ without changing the equivalence classes of the simple self-dual strata and such that~$\b_1$ and~$\b'_1$ have the same minimal polynomial. Proposition~\ref{prop:strataintertwiningovertG}\ref{prop:strataintertwiningovertG.ii} also implies that the strata intertwine in~$\G$. Thus \orange{the pairs~$(\b_1,\vphi_{\b_1})$ and~$(\b'_1,\vphi_{\b'_1})$ are concordant} %the canonical embeddings of~$\b_1$ and~$\b'_1$ are concordant 
by Corollary~\ref{cor:strataintertwineandramimpliesconcordant}.

\shaun{Suppose now that~$r<n-1$ and let~$[\La,n,r+1,\g]$ and~$[\La',n,r+1,\g']$ be self-dual simple strata equivalent to~$[\La,n,r+1,\b]$ and~$[\La',n,r+1,\b']$ respectively. Note that \orange{the pairs~$(\g,\varphigcan)$ and~$(\g',\varphigcanprime)$ are concordant,} %the canonical embeddings of~$\g$ and~$\g'$ are concordant, 
by Lemma~\ref{lemma:WitttowersDefSecuence}. By the induction hypothesis there are concordant self-dual simple strata $[\La,n,r+1,\g_1]$ and~$[\La',n,r+1,\g'_1]$, such that~$\Cc(\La,r+1,\g)=\Cc(\La,r+1,\g_1)$ and~$\Cc(\La',r+1,\g')=\Cc(\La',r+1,\g'_1)$ and such that the minimal polynomials of~$\g_1$ and~$\g'_1$ coincide. By Remark~\ref{rem:concordance}\ref{rem:concordance.ii}, concordance provides an element~$g$ of~$\G$ such that~$g\g_1g^{-1}=\g'_1$. Now,~$\theta|_{\H^{r+2}(\g_1,\La)}$ and~$\tau_{g\La,\La',\g'_1}(\theta'|_{\H^{r+2}(\g'_1,\La')})$ intertwine by an element of~$\tG$ by~\cite[Theorem~1.11]{BSS}, so Proposition~\ref{prop:InTildeGConGSimpleSymp} implies they are conjugate by an element of~$\G$ which maps~$\La$ to~$g\La$. Thus, conjugating by this element, we can assume that~$\theta|_{\H^{r+2}(\g_1,\La)}=\tau_{\La,\La',\g'_1}(\theta'|_{\H^{r+2}(\g'_1,\La')})$ and that~$\g_1=\g'_1$.
}

\shaun{By~\cite[Theorem~9.26]{SkSt}, there is a self-dual simple stratum~$[\La,n,r,\b_0]$ such that
\[
\Cc(\La,r,\b)=\Cc(\La,r,\b_0) \quad\text{and}\quad\text{$[\La,n,r+1,\b_0]$ is equivalent to~$[\La,n,r+1,\g_1]$;}
\]
moreover, \orange{the pairs~$(\b,\varphican)$ and~$(\b_0,\vphi_{\b_0})$ are concordant,} %the canonical embeddings of~$\b,\b_0$ are then concordant, 
by Lemma~\ref{lemma:MatchWittforIntertwCharacters}. Thus we may replace~$\b$ by~$\b_0$ and assume that~$\g=\g_1$. By the same argument for~$\b'$, we see that we may assume~$\g=\g'$. Then there are a character~$\theta_0\in\Cc^\Sigma(\La,r,\g)$ with transfer~$\theta'_0\in\Cc^\Sigma(\La',r,\g)$ and an element~$c\in\aa^-_{-(1+r)}(\La')$
such that
\[
\theta=\theta_0\psi_{\b-\g},\ \theta'=\theta'_0\psi_{\b'-\g+c}.
\]
}
	
%Step 1: we show that we can reduce to the case~$\g=\g'$ and~$\tau_{\La,\La',\g}(\theta'|_{\H^{r+2}(\g,\La')})=\theta|_{\H^{r+2}(\g,\La)}$. By the induction hypothesis there are concordant self-dual simple strata $[\La,n,r+1,\g_1]$ and~$[\La',n,r+1,\g'_1]$, such that~$\Cc(\La,r+1,\g)=\Cc(\La,r+1,\g_1)$ and~$\Cc(\La',r+1,\g')=\Cc(\La',r+1,\g'_1)$ and such that the minimal polynomials of~$\g_1$ and~$\g'_1$ coincide. \blue{Concordance provides an element~$g$ of~$\G$ such that~$g\g_1g^{-1}$ is equal to~$\g'_1$. Now,~$\theta|_{\H^{r+2}(\g_1,\La)}$ and~$\tau_{g\La,\La',\g'_1}(\theta'|_{\H^{r+2}(\g'_1,\La')})$ intertwine by an element of~$\tG$ by~\cite[Theorem 1.11]{BSS}, and thus by Proposition~\ref{prop:InTildeGConGSimpleSymp} they are conjugate by an element of~$\G$ which maps~$\La$ to~$g\La$. Thus without loss of generality we can assume~$\theta|_{\H^{r+2}(\g_1,\La)}=\tau_{\La,\La',\g'_1}(\theta'|_{\H^{r+2}(\g'_1,\La')})$. By the translation principle~\cite[Theorem 9.26]{SkSt} we can assume~$\g=\g'_1=\g'$. Note the latter Theorem changes~$\b$ and~$\b'$ but without changing the sets of simple characters. Therefore Lemma~\ref{lemma:MatchWittforIntertwCharacters} still ensures that in the symplectic case~$\b$ and~$\b'$ are still concordant after the change. Now, there are a character~$\theta_0\in\Cc^\Sigma(\La,r,\g)$ with transfer~$\theta'_0\in\Cc^\Sigma(\La',r,\g)$ and an element~$c\in\aa^-_{-(1+r)}(\La')$ such that
%\[\theta=\theta_0\psi_{\b-\g}\ \theta'=\theta'_0\psi_{\b'-\g+c}.\]
%}

\shaun{By a self-dual~$\dag$-construction we obtain characters
\[
\theta^\dag=\theta_0^\dag\psi_{\b^\dag-\g^\dag},\ \theta'^\dag=\theta'^\dag_0\psi_{\b'^\dag -\g^\dag+c^\dag}
\]
which intertwine in~$\tG^\dag$. Moreover, there exists an element~$g\in\tG^\dag_{\g^\dag}$ such that~$g\La^\dag=\La'^\dag$; then~$^g\theta_0^\dag=\theta'^\dag_0$, because $\tau_{\La'^\dag,\La^\dag,\g^\dag}(\theta_0^\dag)=\theta'^\dag_0$ (as, by its definition, the~$\dag$-construction commutes with transfer). Writing~$s$ for a tame corestriction with respect to~$\g^\dag$, the strata~$[g\La^\dag,r+1,r,s(g\b^\dag g^{-1}-\g^\dag)]$ and~$[\La'^\dag,r+1,r,s(\b'^\dag-\g^\dag+c^\dag)]$ intertwine by an element of~$\tG^\dag_{\g^\dag}$, by~\cite[Proposition~9.17(i)]{SkSt}. Conjugating by this element, we may apply~\cite[Proposition~7.6]{SkSt}, which then implies that the strata~$[\La^\dag,n,r,\b^\dag]$ and~$[\La'^\dag,n,r,\b'^\dag+c^\dag]$ intertwine.}

\shaun{\orange{Since~$\psi_{c^\dag}$ is intertwined by~$\tG_\g$, it follows from~\cite[2.4.11]{BK93} that~$s(c^\dag)$ is congruent to an element of~$\F[\g]$ modulo~$\aa_{-r}(\La'^\dag)$.} 
Thus~$[\La'^\dag,n,r,(\b'+c)^\dag]$ is equivalent to a simple stratum by~\cite[Proposition 6.14]{SkSt}. Since it intertwines with~$[\La^\dag,n,r,\b^\dag]$, the stratum~$[\La^\dag\oplus\La'^\dag,n,r,\b^\dag+(\b'+c)^\dag]$ is also equivalent to a simple stratum, by~\cite[Proposition~7.1]{SkSt}. Then~\cite[Theorem~6.16]{SkSt} implies that it is moreover equivalent to a simple stratum which is split by the decomposition~$\V^\dag\oplus\V^\dag=\V\oplus\cdots\oplus\V$. In particular, restricting to two copies of~$\V$, we see that~$[\La\oplus\La',n,r,\b+(\b'+c)]$ is equivalent to a simple stratum, and again, by~\cite[Theorem~6.16]{SkSt}, to a self-dual simple stratum~$[\La\oplus\La',n,r,\b_1+\b'_1]$ split by}~\bob{$\V\oplus \V$}.

\shaun{Thus we have found self-dual simple strata~$[\La,n,r,\b_1]$ and~$[\La',n,r,\b'_1]$ equivalent to~$[\La,n,r,\b]$ and~$[\La',n,r,\b'+c]$ respectively, such that~$\b_1$ and~$\b'_1$ have the same minimal polynomial. Then~$\Cc(\La,r,\b)$ and~$\Cc(\La,r,\b_1)$ coincide, while 
\[
\theta'\in\Cc(\La',r,\g)\psi_{\b'-\g+c}=\Cc(\La',r,\g)\psi_{\b'_1-\g}=\Cc(\La',r,\b'_1),
\]
so we are done.}

%\blue{Step 4: We show that the canonical embeddings for~$\b_1$ and~$\b'_1$ are concordant. We write~$\varphi_\b$ etc. for the canonical embeddings. }

\shaun{Finally, we prove that \orange{the pairs~$(\b_1,\vphi_{\b_1})$ and~$(\b'_1,\vphi_{\b'_1})$ are concordant.} %the canonical embeddings of~$\b_1$ and~$\b'_1$ are concordant. We write~$\varphi_\b$ etc. for the canonical embeddings. 
Since~$\theta$ lies in~$\Cc(\La,r,\b)$ and in~$\Cc(\La,r,\b_1)$, Proposition~\ref{lemma:MatchWittforIntertwCharacters}\ref{lemma:MatchWittforIntertwCharacters.ii} implies that~$(\b,\varphican)$ and~$(\b_1,\vphi_{\b_1})$ are concordant; similarly~$(\b',\varphicanprime)$ and~$(\b'_1,\vphi_{\b'_1})$ are concordant. Since~$(\b,\varphican)$ and~$(\b',\varphicanprime)$ are concordant by assumption, the result follows by transitivity of concordance.} 
\end{proof}

%% file: Endo-ps.tex
%%%%%%%%%%%%%%%%%%%%%%%%%%%%%%%%%%%%

%%%%%%%%%%%%%%%%%%%%%%%%%%%%%%%%%%%%%
\section{Self-dual ps-characters and simple endo-classes}\label{secPS}
%%%%%%%%%%%%%%%%%%%%%%%%%%%%%%%%%%%%%
In this section we consider the collection of all self-dual simple characters while varying our~$\e$-hermitian space, for fixed~$\e$ and~$\F/\F_\so$. We first recall results of Bushnell and Henniart~\cite{BH96}, and their extensions to non-strict lattice sequences which are special cases of results in~\cite{BSS}, on the foundational theory of ps-characters and simple endo-classes. %In fact, we mildly relax the conditions given in the definition of endo-equivalence of~\cite{BSS} to allow for intertwining simple characters attached to different lattice sequences not related by translation. 
Then we develop the theory in the presence of an~$\e$-hermitian form over~$\F$.

%
%\blue{We need to extend their results to relate endo-equivalence of ps-characters to intertwining simple characters attached to different lattice sequences.} \red{Can you write a brief sentence here saying why we need to extend these results? did they assume the lattice sequences were the same?}\gre{The point is that we have to allow~$\La\neq\La'$ even after translation of the domain. } 

For the remainder, while our~$\F$-vector space~$\V$ and~$\e$-hermitian space~$(\V,h)$ over~$\F$ may be varying, we still use the notation~$\tG=\Aut_\F(\V)$ and~$\G=\U(\V,h)$.  

%%%%%%%%%%%%%%%%%%%%%%%%%%%%%%%%%%%%%
\subsection{Ps-characters}
%%%%%%%%%%%%%%%%%%%%%%%%%%%%%%%%%%%%%
A \emph{simple pair over~$\F$} is a pair~$(k,\b)$ 
\shauns{consisting of an element~$\b$ of some finite field extension of~$\F$ and \rob{an} integer~$k$ satisfying~$0\leqslant k<-k_\F(\b)e(\F[\b]/\F)$.}
% %such that~$\E=\F[\b]$ is a finite field extension and~$k$ is an integer satisfying~$0\leqslant k<-k_\F(\b)e(\E/\F)$. 
For~$(k,\b)$ a simple pair, we write~$\E=\F[\b]$ and denote by~$\Qq(k,\b)$ the class of all quadruples~$(\V,\vphi,\La,r)$ consisting of
\begin{enumerate}\setlength\itemsep{5pt}
\item a finite dimensional~$\F$-vector space~$\V$; 
\item an embedding~$\vphi:\E\hookrightarrow \A$, where~$\A=\End_\F(\V)$; 
\item an~$\o_{\vphi(\E)}$-lattice sequence~$\La$ in~$\V$; 
\item and an integer~$r$ such that~$\left\lfloor r/e(\La|\o_{\vphi(\E)})\right\rfloor=k$.
\end{enumerate}
In this situation, we will abuse notation and write~$e(\La|\o_\E)$ for~$e(\La|\o_{\vphi(\E)})$, the period of~$\La$ as an~$\o_{\vphi(\E)}$-lattice sequence. We will also abbreviate~$e(\La)=e(\La|\o_\F)$ for the period of~$\La$ as an~$\o_\F$-lattice sequence.

Given~$(\V,\vphi,\La,r)\in\Qq(k,\b)$, we set
\[
n=\begin{cases} 
r &\text{ if }\b=0; \\
-\nu_\E(\b)e(\La|\o_\E)&\text{ otherwise;}
\end{cases}
\]
we then obtain a simple stratum~$[\La,n,r,\vphi(\b)]$ in~$\A$ which we call a \emph{realization} of the simple pair~$(k,\b)$. It is simple because~$k_0(\b,\La)=e(\La|\o_\E)e(\E/\F)k_\F(\b)$ by \cite[1.4.13]{BK93} (see also~\cite[5.1]{BK99}).  

We let~$\fCc(k,\b)$ denote the collection of all simple characters defined by a realization of a simple pair~$(k,\b)$: 
\[%\begin{equation}\label{eqnCkbetaNonsd}
\fCc(k,\b)=\bigcup_{\substack{(\V,\vphi,\La,r)\\ \in\Qq(k,\b)}}\Cc(\La,r,\vphi(\b)).
\]%\end{equation}
%Note that, as introduced in \cite{BH96}, the standard notation for~$\Cc(k,\b)$ is~$\mathfrak{R}(k,\b)$.
Given two realizations~$[\La,n,r,\vphi(\b)]$ and~$[\La',n',r',\vphi'(\b)]$ of a simple pair~$(k,\b)$ there is a canonical bijection
\[
\shaun{\tau_{\La',\vphi',\La,\vphi,\b}}:\Cc(\La,r,\vphi(\b))\rightarrow \Cc(\La',r',\vphi'(\b)),
\]
defined in~\cite[3.6.14]{BK93},~\cite[Section 2.1]{St05} and~\cite[Section~3.1(53)]{secherreI}, called \emph{transfer}, and generalizing the transfer recalled in the previous section.
\shaun{Although the transfer depends on~$r,r'$, we do not include them in our notation; indeed, we will usually omit~$\vphi',\vphi$ also and just write~$\tau_{\La',\La,\b}$, as is usual in the literature.}

\shauny{We recall briefly some of the main properties of transfer. Given realizations~$[\La,n,r,\vphi(\b)]$, $[\La',n',r',\vphi'(\b)]$, and~$[\La'',n'',r'',\vphi''(\b)]$ of a simple pair~$(k,\b)$, the associated transfer maps satisfy the following:
\begin{itemize}
\item (symmetry) $\tau_{\La,\La',\b}=\tau_{\La',\La,\b}^{-1}$;
\item (transitivity) $\tau_{\La'',\La,\b}=\tau_{\La'',\La',\b}\circ \tau_{\La',\La,\b}$;
\item (intertwining) suppose the embeddings~$\vphi,\vphi'$ have image in the endomorphisms of the same space~$\vphi,\vphi':\E\hookrightarrow \End_\F(\V)$, and let~$\t\in\Cc(\La,r,\vphi(\b))$; then~$\tau_{\La',\La,\b}(\t)$ is the unique simple character~$\t'\in\Cc(\La',r',\vphi'(\b))$ such that~$\t$ is intertwined with~$\t'$ by an element of~$\tG$ which conjugates~$\vphi$ to~$\vphi'$.
\end{itemize}
In the final property, in fact every element of~$\tG$ which conjugates~$\vphi$ to~$\vphi'$ also intertwines~$\t$ with its transfer~$\tau_{\La',\La,\b}(\t)$.}

\shauny{It is also possible to describe the transfer map explicitly in terms of restrictions. Suppose we are given realizations~$[\La,n,r,\vphi(\b)]$ and~$[\La',n',r',\vphi'(\b)]$ of a simple pair~$(k,\b)$ on spaces~$\V,\V'$ respectively, and~$\La,\La'$ have the same period (so that also~$n'=n$). We set~$\V''=\V\oplus\V'$ so that we have an embedding~$\vphi''=\vphi+\vphi'$ of~$\E$ in~$\End_\F(\V'')$; we also set~$\La''=\La\oplus\La'$ and~$r''=\min\{r,r'\}$. Then we get a further realization~$[\La'',n,r'',\vphi''(\b)]$ on~$\V''$. Now, given a simple character~$\t\in\Cc(\La,r,\vphi(\b))$ there is a unique simple character~$\t''\in\Cc(\La'',r'',\vphi''(\b))$ such that~$\t$ is the restriction of~$\t''$ to~$\H^{r+1}(\vphi''(\b),\La'')\cap\Aut_\F(\V)$. Then the transfer~$\t'=\tau_{\La',\La,\b}(\t)$ is the restriction of~$\t''$ to~$\H^{r'+1}(\vphi''(\b),\La'')\cap\Aut_\F(\V')$.}

\medskip

A \emph{potential simple character}, or \emph{ps-character},\emph{ supported on the simple pair}~$(k,\b)$ is a function~$\Th:\Qq(k,\b)\rightarrow \fCc(k,\b)$ such that
\begin{enumerate}\setlength\itemsep{5pt}
\item $\Th(\V,\vphi,\La,r)\in \Cc(\La,r,\vphi(\b))$, for~$(\V,\vphi,\La,r)\in\Qq(k,\b)$;
\item\label{property2} $\Th(\V',\vphi',\La',r')=\tau_{\La',\La,\b}(\Th(\V,\vphi,\La,r)),$
for~$(\V,\vphi,\La,r),(\V',\vphi',\La',r')\in\Qq(k,\b)$. 
\end{enumerate}
For~$(\V,\vphi,\La,r) \in\Qq(k,\b)$, we call~$\Th(\V,\vphi,\La,r)$ a \emph{realization} of~$\Th$. Thus, by property~\ref{property2}, a ps-character is determined by any one of its realizations. We define the \emph{degree} of~$\Th$ to be~$\deg(\Th)=[\F[\b]:\F]$.

%Let~$\Th$ be a ps-character supported on the simple pair~$(k,\b)$ and~$\Th'$ be a ps-character supported on the simple pair~$(k',\b')$.  
Let~$\Th,\Th'$ be ps-characters supported on the simple pairs~$(k,\b),(k',\b')$ respectively. 

\begin{definition}\label{def:EndoEquivalent}
We say that~$\Th$ and~$\Th'$ are \emph{endo-equivalent}, denoted~$\Th\approx\Th'$, if 
\begin{enumerate}\setlength\itemsep{5pt}
\item\label{Prop1endodef} $\deg(\Th)=\deg(\Th')$;
\item $k=k'$; 
\item\label{Prop3endodef} there exist realizations on a common~$\F$-vector space which intertwine, i.e. there exist a finite dimensional~$\F$-vector space~$\V$ and quadruples $(\V,\vphi,\La,r)\in \Qq(k,\b)$ and~$(\V,\vphi',\La',r')\in\Qq(k',\b')$, such that~$\Th(\V,\vphi,\La,r)$ and~$\Th'(\V,\vphi',\La',r')$ intertwine in~$\tG=\Aut_\F(\V)$.
\end{enumerate}
\end{definition}

Note that the formulation of endo-equivalence in~\cite[8.6]{BH96} and~\cite[1.10]{BSS} differs mildly from the above. In particular they do not consider ps-characters with a trivial character in the image. Therefore we need the following \rob{remarks}.  

\begin{remarks}\label{rem:defEndoEquivalent}
\begin{enumerate}\setlength\itemsep{5pt}
\item\label{rem:defEndoEquivalent-1} In Definition~\ref{def:EndoEquivalent}\ref{Prop3endodef}, we could impose that~$\La=\La'$ and that~$\La$ is strict without changing the relation. Indeed, suppose that~$\La\neq\La'$ or~$\La$ is not strict. By changing the lattice sequences in their affine classes, we can assume that~$e(\La)=e(\La')$. Then, performing a~$\dag$-construction, there exists~$g\in\tG^\dag$ such that~$g\La^\dag=\La'^\dag$, and the characters~$^g\t^\dag=\tau_{g\La^\dag, ^g\vphi^\dag,\La,\vphi,\b}(\t)$ and $\t'^\dag=\tau_{\La'^\dag,\vphi'^\dag,\La',\vphi',\b}(\t')$ intertwine in~$\tG^\dag$, because~$\t$ and~$\t'$ intertwine in~$\tG$.
\item\label{rem:defEndoEquivalent-2} For every non-negative integer~$k$ we have exactly one ps-character supported on~$(k,0)$, which we call the~\emph{zero ps-character}~${\bf 0}_k$. It is not endo-equivalent to any other ps-character, which can be seen as follows. Suppose~$\Th$ is a ps-character supported on~$(k,\b)$, with~$\b\ne 0$, which is endo-equivalent to~${\bf 0}_k$. Then~$\F[\b]=\F$ and there are realizations~$\t\in\Cc(\La,r,\b)$ of~$\Th$ and~$\t'\in\Cc(\La',r',0)$ of~${\bf 0}_k$ on the same vector space such that~$e(\La)=e(\La')$ and~$\t,\t'$ intertwine.  We then get
\[
\left\lfloor r'/e(\La)\right\rfloor=k<\frac{-k_0(\b,\La)}{e(\La)}=\frac{-\val_{\La}(\b)}{e(\La)}\in\ZZ,
\]
but this contradicts Lemma~\ref{lemma:fundNull}.
\item \label{rem:defEndoEquivalentiii} \shaun{It follows from the previous remarks and~\cite[Corollary 8.10]{BH96} that endo-equivalence is indeed an equivalence relation.} 
\end{enumerate}
\end{remarks}

We can now state some initial results on endo-equivalence of ps-characters from~\cite{BH96}.  

\begin{proposition}[{cf.~\cite[8.4,~8.10]{BH96}}]\label{prop:firstPropEndoSimple}
Let~$\Th,\Th'$ be ps-characters supported on the simple pairs~$(k,\b),(k,\b')$ respectively, and put~$\E=\F[\b]$ and~$\E'=\F[\b']$.  Suppose that~$\Th\approx\Th'$. Then:
%Let~$\Th$ be a ps-character supported on the simple pair~$(k,\b)$ and~$\Th'$ be a ps-character supported on the simple pair~$(k',\b')$, and put~$\E=\F[\b]$ and~$\E'=\F[\b']$.  Suppose that~$\Th\approx\Th'$. Then:
\begin{enumerate}\setlength\itemsep{5pt}
\item\label{prop:firstPropEndoSimple.i} We have~$e(\E/\F)=e(\E'/\F)$,~$f(\E/\F)=f(\E'/\F)$ and~$k_\F(\b)=k_\F(\b')$. 
\item\label{prop:firstPropEndoSimple.ii} If~$(\V,\vphi,\La,r)\in\Qq(k,\b)$,~$(\V,\vphi',\La',r')\in\Qq(k,\b')$ and~$e(\La)=e(\La')$ then we have~$(\V,\vphi,\La,r')\in\Qq(k,\b)$, i.e.~$\left\lfloor\frac{r'}{e(\La|\o_\E)}\right\rfloor=k$. 
\end{enumerate}
%Moreover, endo-equivalence is an equivalence relation.
\end{proposition}

\begin{proof}
If~$\Th$ is zero then~$\Th'=\Th$ by Remark~\ref{rem:defEndoEquivalent}\ref{rem:defEndoEquivalent-2}, and the result follows, so we suppose both~$\Th,\Th'$ are non-zero. Then~\ref{prop:firstPropEndoSimple.i} follows from Remark~\ref{rem:defEndoEquivalent}\ref{rem:defEndoEquivalent-1} and~\cite[Proposition 8.4]{BH96}. By~\ref{prop:firstPropEndoSimple.i} we have~$e(\La|\o_\E)=\frac{e(\La)}{e(\E/\F)}=e(\La'|\o_{\E'})$ and thus
\[
\left\lfloor r'/e(\La|\o_\E)\right\rfloor=\left\lfloor r'/e(\La'|\o_{\E'})\right\rfloor=k,
\]
which proves~\ref{prop:firstPropEndoSimple.ii}. 
\end{proof}

\begin{definition}
We call the equivalence classes of ps-characters under endo-equivalence \emph{simple endo-classes}. We define the~\emph{degree} of a simple endo-class to be the degree of any ps-character in the equivalence class.  
\end{definition}

A miracle of the theory is that, while endo-equivalence is defined via the existence of realizations on a common vector space which intertwine (property~\ref{Prop3endodef} of the definition), all realizations of endo-equivalent ps-characters on common vector spaces intertwine:

\begin{theorem}[{cf.~\cite[Theorem 1.11]{BSS} and~\cite[Corollary 8.7]{BH96}}]\label{thm:EndoEquivMeanspairwiseIntforAllrealizations}
 %\red{Suppose that we are given two }endo-equivalent ps-characters~$\Th$ and~$\Th'$ supported on~$(k,\b)$ and~$(k,\b')$, respectively, and 
Let~$\Th,\Th'$ be \rob{endo-equivalent} ps-characters supported on the simple pairs~$(k,\b),(k,\b')$ respectively. 
%Let~$\Th$ be a ps-character supported on the simple pair~$(k,\b)$ and~$\Th'$ be a ps-character supported on the simple pair~$(k',\b')$.  
Let~$\t,\t'$ be realizations of~$\Th,\Th'$ respectively, on the same vector space~$\V$. Then,~$\t$ and~$\t'$ intertwine in~$\tG$. 
\end{theorem}

\begin{proof}
%We consider the realizations~$\t\in\Cc(\La,r,\vphi(\b))$ and~$\t'\in\Cc(\La',r',\vphi'(\b))$, of~$\Th$ and~$\Th'$ respectively, such that~$e(\La)=e(\La')$. \red{
Without loss of generality, we can assume~$\t\in\Cc(\La,r,\vphi(\b))$ and\gre{~$\t'\in\Cc(\La',r',\vphi'(\b'))$}, for simple strata~$[\La,n,r,\vphi(\b)]$ and\gre{~$[\La',n',r',\vphi'(\b')]$} with~$e(\La)=e(\La')$ by adjusting the strata in their affine classes.  If one of the ps-characters is zero then~$\t$ and~$\t'$ are trivial and therefore intertwine. Thus we assume that both ps-characters are non-zero.

\shaun{We first consider the case~$r=r'$. By Proposition~\ref{prop:firstPropEndoSimple}\ref{prop:firstPropEndoSimple.i} we can apply Lemma~\ref{lem:changetoconjlatt} to find an}\gre{~$\o_{\vphi'(\E')}$-lattice} \shaun{sequence~$\La''$ in~$\V$ and an element~$g\in\tG$ such that~$g\La''=\La$. Then~$\t$ and~$\Th'(\V, \presuper{g}\vphi',\La,r')$ are conjugate in~$\tG$ by~\cite[1.13]{BSS}. Thus~$\t$  and~$\Th'(\V,\vphi',\La'',r')$ are conjugate and hence, as the latter is intertwined with~$\t'$ by~$1$, we see that~$\t$ and~$\t'$ intertwine.}
 
We now assume, without loss of generality, that~$r\leq r'$. The quadruple~$(\V,\vphi',\La',r)$ is an element of~$\Qq(k,\b)$ by Proposition~\ref{prop:firstPropEndoSimple}\ref{prop:firstPropEndoSimple.ii}, and by the~$r=r'$ case the characters~$\t$ and~$\Th'(\V,\vphi',\La',r)$ intertwine. Thus~$\t$ and~$\t'$ intertwine because~$\t'$ is the restriction of~$\Th'(\V,\vphi',\La',r)$. 
\end{proof}

% We then have:
% \begin{theorem}[{\cite[Corollary 8.3]{BSS}}]\label{BHendoequivalence} 
% Endo-equivalence defines an equivalence relation on the class of ps-characters. 
% \end{theorem}

That endo-equivalence is a transitive relation leads to the following transitivity of intertwining statement for simple characters:
%of Bushnell--Henniart (\emph{c.f.} \cite[Corollary 8.3]{BSS} for non-strict lattice sequences):

\begin{theorem}\label{thm:transintertwiningoverGL}
Let~$\t_i\in\Cc(\La_i,r_i,\b_i)$, for~$i=1,2,3$. Suppose that~$\t_1$ and~$\t_2$ intertwine in~$\tG$,~$\t_2$ and~$\t_3$ intertwine in~$\tG$, and that either
\begin{enumerate}\setlength\itemsep{5pt}
\item\label{thm:transintertwiningoverGL.ii} $\left\lfloor\frac{r_1}{e(\La_1|\o_{\E_1})}\right\rfloor=\left\lfloor\frac{r_2}{e(\La_2|\o_{\E_2})}\right\rfloor=\left\lfloor\frac{r_3}{e(\La_3|\o_{\E_3})}\right\rfloor$ and~$\t_1,\t_2$ and~$\t_3$ have the same degree; or
\item\label{thm:transintertwiningoverGL.i} $e(\La_1)=e(\La_2)=e(\La_3)$ and~$r_1=r_2=r_3$.
\end{enumerate}
Then~$\t_1$ and~$\t_3$ intertwine in~$\tG$.
%
%~$\I_{\tG}(\t_1,\t_2)$ and~$\I_{\tG}(\t_2,\t_3)$ are non-empty then~$\t_1$ and~$\t_3$ intertwine in~$\tG$. 
\end{theorem}

\begin{proof}
\shaun{In case~\ref{thm:transintertwiningoverGL.ii}, let~$\Th_i$ be the ps-character with realization~$\t_i$ in~$\Cc(\La_i,r_i,\b_i)$, for~$i=1,2,3$. We have~$\Th_1\approx\Th_2$ and~$\Th_2\approx\Th_3$ by assumption, and thus~$\Th_1\approx\Th_3$ by transitivity, and therefore~$\t_1$ and~$\t_3$ intertwine by Theorem~\ref{thm:EndoEquivMeanspairwiseIntforAllrealizations}.
% In the case all lattice sequences are strict this is~\cite[Corollary 2]{BHIntertwiningSimple}.  {\color{red}In general,....}  The ps-characters~$\Th$ and~$\Th''$ defined by~$\t$ and~$\t''$ are endo-equivalent to each other by Theorem~\ref{thmIntertwiningOfSimpleCharIsEQuivalenceRelation}{\color{red}Wrong reference} and thus~$\t$ and~$\t''$ intertwine by some element of~$\tG$ by Theorem~\ref{thm:EndoEquivMeanspairwiseIntforAllrealizations}.
Case~\ref{thm:transintertwiningoverGL.i} follows from case~\ref{thm:transintertwiningoverGL.ii} by Proposition~\ref{prop:SimpleDegrees}\ref{prop:SimpleDegrees.i}.}
\end{proof}

% In the definition of endo-class, the equality of the degrees of the ps-characters (Property \ref{Prop1endodef}) is a result of their intertwining if one assumes that both characters are defined by lattice sequences of the same period and~$r=r'$:
% \begin{theorem}\label{thmEqualDegrees}
% Let~$\t\in\Cc(\La,r,\b)$ and~$\t'\in\Cc(\La',r,\b')$ be~$\tG$-intertwining simple characters with~$e(\La)=e(\La')$. Then
%  $e(\E/\F)=e(\E'/\F)$,~$f(\E/\F)=f(\E'/\F)$ and~$k_0(\b,\La)=k_0(\b',\La')$. 
% \end{theorem}
% 
% \begin{proof}
%  Using the~$\dag$-construction, we can assume that~$\La$ and~$\La'$ are lattice chains of the same period. Then by~\cite[3.5.11]{BK93} these characters are conjugate and thus~\cite[3.5.1]{BK93} provides the desired equalities.
% \end{proof}
%  

%%%%%%%%%%%%%%%%%%%%%%%%%%%%%%%%%%%%%
\subsection{Self-dual ps-characters}
%%%%%%%%%%%%%%%%%%%%%%%%%%%%%%%%%%%%%
We fix the extension~$\F/\F_\so$, as usual let~$\ov{\phantom{a}}$ denote the generator of~$\Gal(\F/\F_\so)$, and fix a sign~$\e=\pm1$. In this section we introduce the theory of endo-class for self-dual simple characters under an~$\e$-hermitian form over~$\F$.
 
A simple pair~$(k,\b)$ over~$\F$ is called \emph{self-dual} if~$(\E,\b)$ is a self-dual extension of~$\F/\F_\so$. For~$(k,\b)$ a self-dual simple pair, we denote by~$\Qq_-(k,\b)$ the class of all quadruples~$((\V,h),\vphi,\La,r)$ consisting of
\begin{enumerate}\setlength\itemsep{5pt}
\item a finite-dimensional~$\varepsilon$-hermitian space~$(\V,h)$ over~$\F/\F_\so$; 
\item a self-dual embedding~$\vphi:\E\rightarrow \A$, where~$\A=\End_\F(\V)$; 
\item a self-dual~$\o_{\vphi(\E)}$-lattice sequence~$\La$ in~$\V$; 
\item and an integer~$r$ such that~$\left\lfloor r/e(\La|\o_{\vphi(\E)})\right\rfloor=k$.
\end{enumerate}
In particular, we then have $(\V,\vphi,\La,r)\in\Qq(k,\b)$. Given~$((\V,h),\vphi,\La,r)\in\Qq_-(k,\b)$ we obtain a self-dual simple stratum~$[\La,n,r,\vphi(\b)]$ which we call a \emph{self-dual realization} of the simple pair~$(k,\b)$. 
 
Let~$(k,\b)$ be a self-dual simple pair. By \cite[Proposition 2.12]{St05}, if we have self-dual realizations~$[\La,n,r,\vphi(\b)]$ and~$[\La',n',r',\vphi'(\b)]$ of~$(k,\b)$, then the transfer map~$\tau_{\La',\La,\b}$ commutes with the involutions defined on~$\Cc(\La,r,\vphi(\b))$ and~$\Cc(\La',r',\vphi'(\b))$ and restricts to give a bijection 
\[
\tau_{\La',\La,\b}:\Cc_{-}(\La,r,\vphi(\b))\rightarrow \Cc_{-}(\La',r',\vphi'(\b)).
\]  
\bob{We let~$\fCc_-(k,\b)$ denote the collection of all self-dual simple characters defined by a realization of the self-dual simple pair~$(k,\b)$:
\[
\fCc_-(k,\b)=\bigcup_{((\V,h),\vphi,\La,r)\in\Qq_-(k,\b)}\Cc_-(\La,r,\vphi(\b)).
\]}

\begin{definition}
Let~$(k,\b)$ be a self-dual simple pair.  
\begin{enumerate}\setlength\itemsep{5pt}
\item A ps-character~$\Th$ supported on~$(k,\b)$ is called \emph{$\sigma$-invariant} if, for all quadruples $((\V,h),\vphi,\La,r)\in\Qq_-(k,\b)$,~$\Th(\V,\vphi,\La,r)$ is~$\sigma$-invariant with respect to~$(\V,h)$.
\item A \emph{self-dual ps-character supported on~$(k,\b)$} is a function,~$\Th_-:\Qq_{-}(k,\b)\rightarrow\bob{ \fCc_{-}(k,\b)}$ such that, for all~$((\V,h),\vphi,\La,r),((\V',h'),\vphi',\La',r')\in\Qq_{-}(k,\b)$,
\begin{align*}
&\Th_-((\V,h),\vphi,\La,r)\in \Cc_{-}(\La,r,\vphi(\b));\\
&\Th_-((\V',h'),\vphi',\La',r')=\tau_{\La',\La,\b}(\Th_-((\V,h),\vphi,\La,r)).
\end{align*}
\end{enumerate}
\end{definition}

We call a value of a self-dual ps-character, a \emph{self-dual realization} of the self-dual ps-character. Thus, again, a self-dual ps-character is determined by any one of its self-dual realizations. By the Glauberman correspondence, every self-dual ps-character arises uniquely by restriction from of a~$\sigma$-invariant ps-character. 

\shaun{More precisely, for a self-dual ps-character~$\Th_-$ supported on~$(k,\b)$, there is a unique~$\sigma$-invariant ps-character~$\Th$, supported on~$(k,\b)$, such that the following diagram commutes:}
%\[
%\xymatrix{\Qq_-(k,\b)\ar[dr]^{\mathrm{Gl}\ \circ \ \Th_-}\ar[d]\\%&\fCc_-(k,\b)\ar[d] \\
%\Qq(k,\b)\ar[r]^\Th&\fCc(k,\b)
%}
%\]
\[
\begin{tikzcd}[column sep=3pc]
 \bob{ \Qq_-(k,\b)}\arrow[d] \arrow[dr,"\mathrm{Gl}\ \circ\ \Th_-"]& {} \\
  %B \arrow{rru}{} \arrow{r}{} &
  \bob{\Qq(k,\b)} \arrow[r,"\Th"] &\bob{\fCc(k,\b)}
\end{tikzcd}
\]
where the vertical arrow is the forgetful map~$((\V,h),\vphi,\La,r)\mapsto(\V,\vphi,\La,r)$ and\bob{~$\mathrm{Gl}\ \circ\ \Th_-((\V,h),\vphi,\La,r)$ is the Glauberman lift in~$\Cc(\La,r,\varphi(\b))$ of~$\Th_-((\V,h),\vphi,\La,r)\in\Cc_-(\La,r,\varphi(\b))$.}  
%
%
%right vertical arrow is the Glauberman lifting of self-dual simple characters. 
We call~$\Th$ the \emph{lift} of~$\Th_-$. We also define the \emph{degree} of~$\Th_-$ to be the degree of its lift, so~$\deg(\Th_-)=[\F[\b]:\F]$.

%Let~$\Th_-$ be a self-dual ps-character supported on the self-dual simple pair~$(k,\b)$ and~$\Th'_-$ be a self-dual ps-character supported on the self-dual simple pair~$(k',\b')$.  
Let~$\Th_-,\Th'_-$ be self-dual ps-characters supported on the self-dual simple pairs~$(k,\b),(k',\b')$ respectively. 
\begin{definition}
We say that~$\Th_-$ and~$\Th'_-$ are \emph{endo-equivalent}, denoted~$\Th_-\approx\Th'_-$, if 
\begin{enumerate}\setlength\itemsep{5pt}
\item\label{Prop1sdendodef} $\deg(\Th_-)=\deg(\Th'_-)$;
\item $k=k'$; 
\item\label{Prop3sdendodef} there exist realizations on a common~$\e$-hermitian space~$(\V,h)$ over~$\F$ which intertwine in~$\G=\U(\V,h)$, i.e. there exist a finite-dimensional~$\e$-hermitian space~$(\V,h)$ over~$\F/\F_\so$ and quadruples $((\V,h),\vphi,\La,r)\in \Qq_-(k,\b)$ and~$((\V,h),\vphi',\La',r')\in\Qq_-(k',\b')$, such that~$\Th_-(\bob{(\V,h)},\vphi,\La,r)$ and~$\Th'_-(\bob{(\V,h)},\vphi',\La',r')$ intertwine in~$\G$.
\end{enumerate}
\end{definition}

We can now prove our main result on endo-equivalence of self-dual ps-characters.  

\begin{theorem}\label{thm:Endo}
%Let~$\Th_-$ be a self-dual ps-character supported on the self-dual simple pair~$(k,\b)$ and~$\Th'_-$ be a self-dual ps-character supported on the self-dual simple pair~$(k',\b')$.  Let~$\Th$ denote the lift of~$\Th_-$, and~$\Th'$ denote the lift of~$\Th'_-$.
%
Let~$\Th_-,\Th'_-$ be self-dual ps-characters supported on the self-dual simple pairs~$(k,\b),(k,\b')$ respectively. Denote by~$\Th,\Th'$ the lifts of~$\Th_-,\Th'_-$ respectively. 
%Let~$(k,\b)$ and~$(k',\b')$ be self-dual simple pairs which satisfy~$k=k'$ and~$[\E:\F]=[\E':\F]$.  Let~$\Th_-$ and~$\Th'_-$ be two self-dual ps-characters supported on~$(k,\b)$ and~$(k',\b')$, respectively, and~$\Th$ and~$\Th'$ their lifts.
Suppose that~$\deg(\Th_-)=\deg(\Th'_-)$. Then the following assertions are equivalent:
\begin{enumerate}\setlength\itemsep{5pt}
\item\label{point1} $\Th$ and~$\Th'$ are endo-equivalent;
\item\label{point2} $\Th_-$ and~$\Th'_-$ are endo-equivalent;
\item\label{point3} for all~$((\V,h),\vphi,\La,r)\in\Qq_-(k,\b)$ and~$((\V,h),\vphi',\La',r')\in\Qq_-(k,\b')$ with~$(\b,\vphi)$ and~$(\b',\vphi')$ concordant the realizations~$\Th_{-}((\V,h),\vphi,\La,r)$ and~$\Th'_{-}((\V,h),\vphi',\La',r')$ intertwine in~$\G=\U(\V,h)$; 
\item\label{point4} there are~$((\V,h),\vphi,\La,r)\in\Qq_-(k,\b)$ and~$((\V,h),\vphi',\La',r')\in\Qq_-(k,\b')$ with~$(\b,\vphi)$ and~$(\b',\vphi')$ concordant such that~$\Th_{-}((\V,h),\vphi,\La,r)$ and $\Th'_{-}((\V,h),\vphi',\La',r')$ intertwine in~$\G=\U(\V,h)$. 
%  intertwine is a comparison pair such that the realizations of~$\Th_-$ and~$\Th'_-$ intertwine in~$\G$; 
%  \item\label{point4} For all Witt comparison pairs, the realizations of~$\Th_-$ and~$\Th'_-$ intertwine in~$\G$.  {\color{red}remove Witt comparison pair and just write what it is}
\end{enumerate}
Suppose further that~$\F\ne\F_\so$ or~$\e=1$. Then these four assertions are equivalent to:
\begin{enumerate}\addtocounter{enumi}{4}\setlength\itemsep{5pt}
\item \label{point5} for all~$((\V,h),\vphi,\La,r)\in\Qq_-(k,\b)$ and~$((\V,h),\vphi',\La',r')\in\Qq_-(k,\b')$ the realizations~$\Th_{-}((\V,h),\vphi,\La,r)$ and~$\Th'_{-}((\V,h),\vphi',\La',r')$ intertwine in~$\G=\U(\V,h)$. 
\end{enumerate}
\end{theorem}

Indeed, we will see in the proof that, in the non-symplectic case (i.e.~$\F\ne\F_\so$ or~$\e=1$), if~$\Th$ and~$\Th'$ are endo-equivalent then, for any two self-dual embeddings~$\vphi,\vphi'$ of~$\b,\b'$ respectively into any~$\e$-hermitian space~$(\V,h)$ over~$\F/\F_\so$, the pairs~$(\b,\vphi)$ and~$(\b',\vphi')$ are concordant.

\begin{proof}
Certainly~\ref{point2} follows from~\ref{point4}, and~\ref{point1} follows from~\ref{point2}, by definition and Proposition~\ref{prop:Glauberman}. Clearly~\ref{point3} follows from~\ref{point5}.

\shaun{Suppose now that~$\Th,\Th'$ are endo-equivalent and we have~$((\V,h),\vphi,\La,r)\in\Qq_-(k,\b)$ and~$((\V,h),\vphi',\La',r')\in\Qq_-(k,\b')$. Replacing~$\La,\La'$ by an affine translation, which does not affect the realization, we can assume~$e(\La)=e(\La')$; moreover, by Proposition~\ref{prop:firstPropEndoSimple}\ref{prop:firstPropEndoSimple.ii} we may replace~$r,r'$ by~$\min\{r,r'\}$ and hence assume they are equal. Note that these changes do not affect the concordance of~$(\b,\vphi)$ and~$(\b',\vphi')$. Then~$\Th(\V,\vphi,\La,r)$ and~$\Th'(\V,\vphi',\La',r)$ intertwine by Theorem~\ref{thm:EndoEquivMeanspairwiseIntforAllrealizations} and, further, when~$\F\ne\F_\so$ or~$\e=1$, the pairs~$(\b,\vphi)$ and~$(\b',\vphi')$ are automatically concordant by Lemma~\ref{lemma:MatchWittforIntertwCharacters}\ref{lemma:MatchWittforIntertwCharacters.ii}.
%. If either~$(\vphi,\b)$ and~$(\vphi',\b')$ are concordant, or~$\F\ne\F_\so$ or~$\e=1$, then 
Then, provided we have concordant pairs~$(\b,\vphi)$ and~$(\b',\vphi')$ in the symplectic case,} \gre{Proposition~\ref{prop:TiGandGIntertwiningSameNonSympl}} \shaun{implies that~$\Th_{-}((\V,h),\vphi,\La,r)$ and~$\Th'_{-}((\V,h),\vphi',\La',r')$ intertwine in~$\G=\U(\V,h)$. Thus~\ref{point1} implies~\ref{point3}, and~\ref{point1} implies~\ref{point5} when~$\F\ne\F_\so$ or~$\e=1$.}

\shaun{Now suppose the two ps-characters satisfy~\ref{point3}. In order to show~\ref{point4}, we have to find elements of~$\Qq_-(k,\b)$ and~$\Qq_-(k,\b')$ defined on the same~$\e$-hermitian space \emph{with concordant pairs}. Let~$h_\E$ and~$h_{\E'}$ be hyperbolic~$\e$-hermitian spaces over~$\E$ and~$\E'$, respectively, of the same dimension, and recall that~$[\E:\F]=[\E':\F]$, since~$\Th_-,\Th'_-$ have the same degree. Then (in the notation of subsection~\ref{subsec:transfer})~$\l_\b^*(h_\E)$ and~$\l_{\b'}^*(h_{\E'})$ are hyperbolic spaces over~$\F$ of the same dimension; hence they are isometric and we can assume without loss of generality that they are the same space~$(\V,h)$. The pairs~$(\b,\vphi_\b)$ and~$(\b',\vphi_{\b'})$ %of~$\b$ and~$\b'$ into~$\A$ 
are then concordant, because~$h_\E$ and~$h_{\E'}$ are hyperbolic. Now, for any self-dual~$\o_\E$-lattice sequence~$\La$ in~$\V$ and self-dual~$\o_{\E'}$-lattice sequence~$\La'$ in~$\V$, we have~$((\V,h),\vphi,\La,ke(\La|\o_\E))\in\Qq_-(k,\b)$ and~$((\V,h),\vphi',\La',ke(\La'|\o_{\E'}))\in\Qq_-(k,\b')$.} 
\end{proof}

\shaun{In fact, concordance exactly determines whether realizations of endo-equivalent self-dual ps-characters intertwine:}

\begin{proposition}\label{prop:endoconcordinter}
\shaun{Let~$\Th_-,\Th'_-$ be endo-equivalent self-dual ps-characters supported on the self-dual simple pairs~$(k,\b),(k,\b')$ respectively. %Suppose~$\deg(\Th_-)=\deg(\Th'_-)$ and 
Let~$((\V,h),\vphi,\La,r)\in\Qq_-(k,\b)$ and~$((\V,h),\vphi',\La',r')\in\Qq_-(k,\b')$. Then the realizations~$\Th_{-}((\V,h),\vphi,\La,r)$ and~$\Th'_{-}((\V,h),\vphi',\La',r')$ intertwine in~$\G=\U(\V,h)$ if and only if~$(\b,\vphi)$ and~$(\b',\vphi')$ are concordant.}
\end{proposition}

\begin{proof}
\shaun{Suppose~$\theta=\Th_{-}((\V,h),\vphi,\La,r)$ and~$\theta'=\Th'_{-}((\V,h),\vphi',\La',r')$ intertwine in~$\G=\U(\V,h)$. Replacing~$\La,\La'$ by an affine translation, which does not affect the realization, we can assume~$e(\La)=e(\La')$; moreover, by Proposition~\ref{prop:firstPropEndoSimple}\ref{prop:firstPropEndoSimple.ii} we may replace~$r,r'$ by~$\max\{r,r'\}$ (which is equivalent to restricting~$\theta,\theta'$ to subgroups) and hence assume they are equal. Then Proposition~\ref{prop:TiGandGIntertwiningSameNonSympl} implies that~$(\b,\vphi)$ and~$(\b',\vphi')$ are concordant. The converse is given by Theorem~\ref{thm:Endo}.}
\end{proof}

%The symplectic case is completely analogous using matching Witt towers.
%
%\begin{proposition}\label{propTiGandMatchWittTowerandGIntertwiningSameSympl}
%In the symplectic case: Two self-dual simple characters for same~$r$ and lattice sequences of same period intertwine in~$G$ if and only if 
%their lifts intertwine in~$\tG$ and the Witt towers match.
%\end{proposition}
%
%\begin{proof}
%The ``only-if-part''  follows from Proposition~\ref{propMatchWittforIntertwCharacters}. The ``if-part'' is completely analogous to the proof of Proposition 
%\ref{propTiGandGIntertwiningSameNonSympl} using Corollary~\ref{corEqualEE0Sympl} and Theorem~\ref{thm:EndoStrongSympl}.
%\end{proof}
%
%
%We now get to the Theorem for the symplectic case. 
%
%\begin{theorem}\label{thm:EndoSympl}
%In the symplectic case: Let~$\Th_-$ and~$\Th'_-$ be two self-dual ps-characters and~$\Th$ and~$\Th'$ their lifts.
%Then, the following assertions are equivalent
%\begin{enumerate}
% \item$\Th_-$ and~$\Th'_-$ are equivalent,
% \item There is a comparison pair such that the realizations of~$\Th_-$ and~$\Th'_-$ intertwine in~$G$.  
%  \item For all comparison pairs with matching Witt towers the realizations of~$\Th_-$ and~$\Th'_-$ intertwine in~$G$. 
%\end{enumerate}
%\end{theorem}
%
%\begin{proof}
% Completely analogous to the proof of Theorem~\ref{thm:Endo} using Proposition~\ref{propTiGandMatchWittTowerandGIntertwiningSameSympl}.
%\end{proof}
As endo-equivalence of ps-characters for general linear groups is an equivalence relation, from Theorem~\ref{thm:Endo} we deduce the corresponding result for self-dual ps-characters:

\begin{corollary}\label{remEndoIsEqRelationSimple}
Endo-equivalence defines an equivalence relation on the class of self-dual ps-characters.
% with respect to~$(\s,\e)$.{\color{red}State as corollary}
%\emph{simple $(\sigma,\epsilon)$-endo-classes}.  or \emph{simple classical endo-classes} with respect to~$(\sigma,\epsilon)$.
\end{corollary}

\orange{We also deduce that endo-equivalent self-dual ps-characters must be supported on similar self-dual extensions (see Definition~\ref{def:similar}).
\begin{corollary}\label{cor:endosimilar}
Let~$\Th_-,\Th'_-$ be endo-equivalent self-dual ps-characters supported on the self-dual simple pairs~$(k,\b),(k,\b')$ respectively and put~$\E=\F[\b]$ and~$\E'=\F[\b']$. Then the self-dual extensions~$(\E,\b)$ and~$(\E',\b')$ are similar.
\end{corollary}
\begin{proof}
Denote by~$\Th,\Th'$ the lifts of~$\Th_-,\Th'_-$ respectively; they are endo-equivalent by Theorem~\ref{thm:Endo}. We choose self-dual realizations of~$\Th_-,\Th'_-$ on a common space~$\V$ for which the lattice sequences have the same period. Then their lifts are realizations of the endo-equivalent~$\Th,\Th'$ so intertwine in~$\tG=\Aut_\F(\V)$ by Theorem~\ref{thm:EndoEquivMeanspairwiseIntforAllrealizations}. Then Corollary~\ref{cor:EqualEE0Sympl} says that the extensions~$(\E,\b)$ and~$(\E',\b')$ are similar.
\end{proof}
}

We call the equivalence classes of self-dual ps-characters under endo-equivalence \emph{self-dual simple endo-classes}. We also obtain the self-dual version of Theorem~\ref{thm:transintertwiningoverGL}, the transitivity of intertwining of self-dual simple characters: 
%that~$\G$-intertwining of self-dual simple characters of~$\G$ is a transitive relation, hence an equivalence relation:

\begin{corollary}\label{corIntEquivalenceRelation}
Let~$\t_{i,-}\in\Cc_-(\La_i,r_i,\b_i)$ be self-dual simple characters, for~$i=1,2,3$. Suppose that~$\t_{1,-}$ and~$\t_{2,-}$ intertwine in~$\G$,~$\t_{2,-}$ and~$\t_{3,-}$ intertwine in~$\G$, and that either%that their respective lifts~$\t_i$,~$i=1,2,3$, satisfy~\ref{thm:transintertwiningoverGL.ii} in Theorem~\ref{thm:transintertwiningoverGL}. 
\begin{enumerate}\setlength\itemsep{5pt}
\item\label{corIntEquivalenceRelation.ii} $\left\lfloor\frac{r_1}{e(\La_1|\o_{\E_1})}\right\rfloor=\left\lfloor\frac{r_2}{e(\La_2|\o_{\E_2})}\right\rfloor=\left\lfloor\frac{r_3}{e(\La_3|\o_{\E_3})}\right\rfloor$ and~$\t_{1,-},\t_{2,-}$ and~$\t_{3,-}$ have the same degree; or 
\item\label{corIntEquivalenceRelation.i} $e(\La_1)=e(\La_2)=e(\La_3)$ and~$r_1=r_2=r_3$.
\end{enumerate}
%\gre{I thought about putting it in but somehow I do not want to put it here.  You can decide. If so, then extend the proof. Be careful (i) does not imply (ii) but (i) and intertwining impies (ii).} % and that~$\I_\G(\t_{1,-},\t_{2,-})$ and~$\I_\G(\t_{2,-},\t_{3,-})$ are non-empty.
Then~$\t_{1,-}$ and~$\t_{3,-}$ intertwine in~$\G$.
% ,\t'_-$ and~$\t''_-$ be self-dual simple characters with the same group level and the same degree.
% If~$\t_-$ intertwines with~$\t'_-$ and~$\t'_-$ intertwines with~$\t''_-$ over~$\G$, then~$\t_-$ intertwines with~$\t''_-$ over~$\G$.
\end{corollary}

\begin{proof}
\shaun{In case~\ref{corIntEquivalenceRelation.ii}, let~$\Th_{i-}$ be the ps-character with realization~$\t_{i,-}$ in~$\Cc_-(\La_i,r_i,\b_i)$, for~$i=1,2,3$. We have~$\Th_{1,-}\approx\Th_{2,-}$ and~$\Th_{2,-}\approx\Th_{3,-}$ by assumption, and thus~$\Th_{1,-}\approx\Th_{3,-}$ by Corollary~\ref{remEndoIsEqRelationSimple}. We abbreviate~$\vphi_i$ for the canonical embedding of~$\b_i$ in~$\A$. Then~$(\b_1,\vphi_1)$ and~$(\b_2,\vphi_2)$ are concordant, by Proposition~\ref{prop:endoconcordinter}, and likewise~$(\b_2,\vphi_2)$ and~$(\b_3,\vphi_3)$ are concordant. Thus~$(\b_1,\vphi_1)$ and~$(\b_3,\vphi_3)$ are concordant, by transitivity of concordance, and~$\t_{1,-}$ and~$\t_{3,-}$ intertwine in~$\G$ by Proposition~\ref{prop:endoconcordinter} again. Case~\ref{corIntEquivalenceRelation.i} follows from case~\ref{corIntEquivalenceRelation.ii} by Proposition~\ref{prop:SimpleDegrees}\ref{prop:SimpleDegrees.i}.}
\end{proof}

%% file: Endo-semisimple.tex
%%%%%%%%%%%%%%%%%%%%%%%%%%%%%%%%%%%%

%%%%%%%%%%%%%%%%%%%%%%%%%%%%%%%%%%%%
\section{Self-dual semisimple characters: intertwining and concordance}
\label{secSemisimpleChars}
%%%%%%%%%%%%%%%%%%%%%%%%%%%%%%%%%%%%

In this section we recall, from~\cite{St05,MiSt}, the basic properties of semisimple strata and characters, and of their self-dual versions. We also recall from~\cite{SkSt} how the intertwining of semisimple characters induces a matching between their splittings, and use this to deduce both results on concordance and Skolem--Noether type results.

%%%%%%%%%%%%%%%%%%%%%%%%%%%%%%%%%%%%
\subsection{Semisimple strata}
%%%%%%%%%%%%%%%%%%%%%%%%%%%%%%%%%%%%
Suppose that~$\V=\bigoplus_{i\in \I} \V^i$ is a decomposition into~$\F$-subspaces. For~$\J$ any subset of~$\I$, we write~$\V^\J=\bigoplus_{i\in\J}\V^i$ and~$\mathbf{e}^\J:\V\to \V^\J$ for the projection with kernel~$\bigoplus_{j\in \I\setminus\J}\V^j$. We also set~$\A^\J=\End_\F(\V^\J)$ and~$\tG_\J=\Aut_\F(\V^\J)$. When~$\J=\{i\}$ is a singleton, then we will write~$\A^i$ rather than~$\A^{\{i\}}$, etc.

Now let~$[\La,n,r,\b]$ be a stratum in~$\A$. For~$\J$ a subset of~$\I$, we set~$\La^\J=\La\cap \V^\J$ and~$\b_\J=\mathbf{e}^\J\b\mathbf{e}^\J$, and~$n_\J=\max\{-\val_{\La^\J}(\b_\J),r\}$, so that~$[\La^\J,n_\J,r,\b_\J]$ is a stratum in~$\V^\J$. The decomposition~$\V=\bigoplus_{i\in \I} \V^i$ of~$\V$ is called a \emph{splitting} of~$[\La,n,r,\b]$ if~$\b=\sum_{i\in \I}\b_i$ and~$\La(k)=\bigoplus_{i\in \I}\La^i(k)$, for all~$k\in\mathbb{Z}$.  

\begin{definition}\label{defSemisimpleStratum}
A stratum~$[\La,n,r,\b]$ in~$\A$ is called \emph{semisimple} if it is a null stratum or if~$\val_{\La}(\b)=-n$ and there exists a splitting~$\bigoplus_{i\in \I} \V^i$ for~$[\La,n,r,\b]$ such that
\begin{enumerate}\setlength\itemsep{5pt}
\item for~$i\in \I$, the stratum~$[\La^i,n_i,r,\b_i]$ in~$\End_\F(\V^i)$
is simple; 
\item for~$i,j\in \I$ with~$i\neq j$, the stratum~$[\La^{\{i,j\}},n_{\{i,j\}},r,\b_{\{i,j\}}]$ is not equivalent to a simple stratum in~$\End_\F(\V^{\{i,j\}})$.
%%% Could be stated as: if~$\J$ is a subset of~$\I$ with~$|\J|\ge 2$ then the stratum~$[\La^{\J},n_\J,r,\b_\J]$ is not equivalent to a simple stratum in~$\End_\F(\V^\J)$.
\end{enumerate}
\end{definition}

Let~$[\La,n,r,\b]$ be a semisimple stratum in~$\A$. We write~$\E=\F[\b]$ and~$\E_i=\F[\b_i]$, so that~$\E=\bigoplus_{i\in \I}\E_i$ is a sum of fields, and set~$\B_\b=\C_\A(\b)$ and~$\tG_\b=\B_\b^\times$. By abuse of notation, we call an~$\o_\F$-lattice sequence which is a sum of~$\o_{\E_i}$-lattice sequences in~$\V^i$ an~\emph{$\o_\E$-lattice sequence}; thus~$\La$ is an~$\o_\E$-lattice sequence. We also call~$[\E:\F]=\dim_\F\E$ the~\emph{degree} of the semisimple stratum. 

\shauns{%
For~$[\La,n,0,\b]$ a non-null semisimple stratum in~$\A$, we let
\[
k_0(\b,\La)=-\min\{r\in\ZZ:r\geq 0,\ [\La,n,r,\b]\text{ is not semisimple}\}
\] 
denote the \emph{critical exponent} of~$[\La,n,0,\b]$ and set~$k_\F(\b)=\frac{1}{e(\La)}k_0(\b,\La)$; by~\cite[\S3.1]{St05}, this is independent of~$\La$. For null strata we put~$k_0(0,\La)=k_\F(0)=-\infty$.
}

\shauns{%
It is possible to generalize the critical exponent to all pairs~$(\b,\La)$ where~$\b$ generates a product of fields, and~$\La$ is an~$\o_{\E}$-lattice sequence as follows. We set~$n=-\val_\La(\b)$ and~$e=e(\La)$.
}

\begin{lemma}\label{lemmasufflargeinterstratasesi}
\shauns{%
With the notation above, if~$\b$ is non-zero then, for any sufficiently large integer~$l$, the stratum~$[\La,n+le,0,\w_\F^{-l}\b]$ is semisimple.
}
\end{lemma}

\shauns{%
We can then define
\[
k_0(\b,\La)=k_0(\w_\F^{-l}\b,\La)+le,
\]
for any integer~$l$ such that~$[\La,n+le,0,\w_\F^{-l}\b]$ is semisimple; this is independent of the choice of~$l$. We also set~$k_\F(\b)=\frac{1}{e}k_0(\b,\La)$, which is again independent of~$\La$.
}

\shauns{%
\begin{proof} Replacing~$\b$ by~$\w_\F^{-l}\b$ for sufficiently large~$l$, we can assume that~$k_\F(\b_i)<0$ for all~$i\in\I$, \shaun{in which case~$n$ is positive}. We need only show that there is an integer$~l$ such that~$[\La,n+le,0,\w_\F^{-l}\b]$ is semisimple so we suppose for contradiction that there is no such integer. From the definition of semisimple stratum, it is sufficient to consider the case that~$\I$ has cardinality two. For each~$l\ge 0$ there is then \daniel{by~\cite[Theorem 6.16]{SkSt}} a simple stratum~$[\La,n+le,0,\g^{(l)}]$ equivalent to~$[\La,n+le,0,\w_\F^{-l}\b]$ \shaun{with~$\g^{(l)}\in\bigoplus_{i\in\I}\A^i$. %\orange{We can further assume that~$[\La,n+el,0\g_l]$ is split be the splitting of~$\V$ associated to~$\b$. Then~$[\La,n,0,\w_\F^l\g_l]$ is a simple stratum, because~$k_\F(\g_l)=k_\F(\b_i)<0$ for any~$i$, and~$\w_\F^l\g_l$ converges to~$\b$ as~$l\to\infty$.} 
Restricting these to the~$i$th block (where both strata are simple), we see that%~$[\La,n,0,\w_\F^l\g^{(l)}]$ is a pure stratum equivalent to~$[\La,n,0,\b]$ and, restricting to a single block we have~
~$k_\F(\g^{(l)})= k_\F(\w_\F^{-l}\b_i)$; in particular,~$k_\F(\w_\F^l\g^{(l)})=k_\F(\b_i)<0$ so that~$[\La,n,0,\w_\F^l\g^{(l)}]$ is a simple stratum. But then~$\w_\F^l\g^{(l)}$ converges to~$\b$ as~$l\to\infty$} so~\cite[Proposition~1.9]{St00} implies that~$[\La,n,0,\b]$ is simple. In particular,~$\F[\b]$ is a field, which contradicts the fact that~$\I$ has cardinality two.
\end{proof}
}

Now we turn to the self-dual case. If~$[\La,n,r,\b]$ is self-dual and semisimple with associated splitting~$\V=\bigoplus_{i\in \I}\V^i$ then, for each~$i\in \I$, there exists a unique~$\s(i)=j\in \I$ such that~$\ov{\b_i}=-\b_j$. We set~$\I_0=\{i\in \I:\s(i)=i\}$ and choose a set of representatives~$\I^+$ for the orbits of~$\s$ in~$\I\setminus \I_0$.  Then we let~$\I_{-}=\s(\I_{+})$ so that we have a disjoint union~$\I=\I_+\cup \I_0\cup \I_{-}$. If~$\J$ is a~$\s$-stable subset of~$\I$, then we write~$h_\J$ for the restriction of the form~$h$ to~$\V^\J$, so that~$(\V^\J,h_\J)$ is an~$\e$-hermitian space over~$\F/\F_\so$; this applies in particular when~$\J$ is a singleton subset of~$\I_0$. 

\begin{definition}\label{skewsesistratum}
A semisimple stratum~$[\La,n,r,\b]$ in~$\A$ is called \emph{skew} if it is self-dual and the associated splitting~$\bigoplus_{i\in \I} \V^i$ is orthogonal with respect to the~$\e$-hermitian form~$h$, i.e.~$\I=\I_0$ in the notation above. 
\end{definition}

In particular, a self-dual simple stratum is automatically skew. \rob{At the start of the appendix, there is a brief discussion on the roles of skew and non-skew self-dual semisimple objects.}

\shaun{%
As in the simple case, many results concerning semisimple strata are proved ``by induction along~$r$'' using the following fundamental approximation result.
}

\begin{proposition}[{\cite[Proposition~3.4]{St05}, \cite[Lemma~3.1]{MiSt}}]\label{prop:semiapprox}
\shaun{%
Let~$[\La,n,0,\b]$ be a non-null semisimple stratum with associated splitting~$\bigoplus_{i\in \I} \V^i$ and let~$0<r\le n$. Then there is a semisimple stratum~$[\La,n,r,\g]$ equivalent to~$[\La,n,r,\b]$ with~$\g\in\bigoplus_{i\in\I}\A^i$. Moreover, if~$[\La,n,0,\b]$ is self-dual then~$[\La,n,r,\g]$ may be taken to be self-dual also.
}
\end{proposition}

%It will also be useful to have an analogue in the semisimple case of Lemma~\ref{lem:changetoconjlatt}, on the existence of lattice sequences with prescribed properties. %%%% But I haven't worked out the statement yet!

%\begin{lemma}\label{lem:changetoconjlattsemi}
%\shauns{Let~$[\La,n,0,\b]$ be a semisimple stratum in~$\A$ with associate splitting~$\bigoplus_{i\in \I} \V^i$. Let~$\E'=\F[\b']$ be an~$\F$-subalgebra of~$\A$ which is isomorphic to a sum containing~$\F$, such that~$e(\E/\F)=e(\E'/\F)$ and~$f(\E/\F)=f(\E'/\F)$, and let~$\La$ be an~$\o_\E$-lattice sequence in~$\V$. Then there exist an~$\o_{\E'}$-lattice sequence~$\La'$ in~$\V$ and~$g\in\tG$ such that~$g\La'=\La$.}
%\end{lemma} 

%\begin{proof}
%\shauns{There is an~$\F$-linear isomorphism from~$\E$ to~$\E'$ which maps~$\p_\E^n$ to~$\p_{\E'}^n$, for each~$n\in\ZZ$. Now we choose an~$\E$-basis of~$\V$ which splits~$\La$ and map it to an~$\E'$-basis of~$\V$, using this~$\F$-linear isomorphism, and the image of~$\La$ has the required property.}
%\end{proof}

%%%%%%%%%%%%%%%%%%%%%%%%%%%%%%%%%%%%
\subsection{Semisimple characters}
%%%%%%%%%%%%%%%%%%%%%%%%%%%%%%%%%%%%
Let~$[\La,n,r,\b]$ be a semisimple stratum in~$\A$. Associated to it \rob{are} an~$\o_{\F}$-order~$\HH(\b,\La)$ in~$\A$ defined inductively (see~\cite[Section 3.2]{St05}) and, for\gre{~$m\geqslant 1$,} the compact open subgroups~$\H^{m}(\b,\La)=\HH(\b,\La)\cap \P^{m}(\La)$ of~$\tG$. We also have a set~$\Cc(\La,r,\b)$ of characters of~$\H^{r+1}(\b,\La)$ called \emph{semisimple characters} (which depend on our fixed choice of additive character~$\psi$). For each subset~$\J$ of~$\I$, there is a natural embedding~$\H^{r+1}(\b_\J,\La^\J)\hookrightarrow \H^{r+1}(\b,\La)$ and hence a map~$\Cc(\La,r,\b)\to\Cc(\La^\J,r,\b_\J)$ which we write~$\t\mapsto\t_\J$. In the case when~$\J=\{i\}$ is a singleton, we call~$\t_i$ a \emph{simple block restriction} of~$\t$.

As in the simple case we have a notion of transfer for semisimple characters. We define~$e(\La_\E)$ to be the greatest common divisor of the integers~$e(\La^i|\o_{\E_i})$, for~$i\in \I$. Given two semisimple strata~$[\La,n,r,\b]$ and~$[\La',n',r',\b]$ in~$\A$, which satisfy~$\left\lfloor\frac{r}{e(\La_\E)}\right\rfloor=\left\lfloor\frac{r'}{e(\La'_{\E'})}\right\rfloor$, there is a canonical bijection~$\tau_{\La',\La,\b}:\ \Cc(\La,r,\b)\rightarrow\Cc(\La',r',\b)$ called~\emph{transfer} (see~\cite[Proposition 3.26]{St05} and~\cite[Remark 3.3]{St08}): if~$\t\in\Cc(\La,r,\b)$ then~$\tau_{\La',\La,\b}(\t)$ is the unique semisimple character~$\t'\in\Cc(\La',r',\b)$ such that~$1\in\tG$ intertwines~$\t$ with~$\t'$. Again, despite the dependence of the bijection on~$(r,r')$ we omit it from our notation. 

Now suppose that~$[\La,n,r,\b]$ is also self-dual. Then the subgroup~$\H^{r+1}(\b,\La)$ and the set~$\Cc(\La,r,\b)$ of semisimple characters are stable under~$\s$ (see~\cite[\S3.6]{St05} and~\cite[\S3.6]{MiSt}), and we set
\begin{align*}
\H_-^{r+1}(\b,\La)&=\H^{r+1}(\b,\La)^\Sigma= \H^{r+1}(\b,\La)\cap\G;\\
\Cc^\Sigma(\La,r,\b)&=\{\t\in\Cc(\La,r,\b):\t^\s=\t\}.
\end{align*} 
As in the simple setting,~$\H_-^{r+1}(\b,\La)$ is a compact open subgroup of~$\G$, and we define the set of~\emph{self-dual semisimple characters} of~$\H_-^{r+1}(\b,\La)$ by restriction:
\[
\Cc_-(\La,r,\b)=\{\t\mid_{\H_-^{r+1}(\b,\La)}:\t\in\Cc^\Sigma(\La,r,\b)\}.\]
This restriction coincides with the Glauberman correspondence. By the Glauberman correspondence, if~$\t_-\in \Cc_-(\La,r,\b)$ then there is a unique~$\t\in \Cc^{\Sigma}(\La,r,\b)$ whose restriction to~$\H^{r+1}_{-}(\b,\La)$ is~$\t_-$; we call~$\t$ the \emph{lift of~$\t_-$ \shauny{with respect to~$(\La,r,\b)$}}.  \shauny{(As before, we will simply write \emph{lift of~$\theta_-$}, since the stratum will be given implicitly.)}  

\daniel{A self-dual semisimple character~$\t_-$ is called~\emph{skew semisimple} if there is a skew semisimple stratum~$[\La,n,r,\b]$ such that~$\t_-\in\Cc_-(\La,r,\b)$.}

If we have two self-dual semisimple strata~$[\La,n,r,\b]$ and~$[\La',n',r',\b]$ in~$\A$, which satisfy~$\left\lfloor\frac{r}{e(\La_\E)}\right\rfloor=\left\lfloor\frac{r'}{e(\La'_{\E'})}\right\rfloor$, then the transfer map~$\tau_{\La',\La,\b}$ commutes with the involution~$\s$ (see~\cite[Proposition 3.32]{St05} and~\cite[\S3.7]{MiSt}). In particular it restricts to a bijection~$\tau_{\La',\La,\b}:\Cc_-(\La,r,\b)\rightarrow\rob{\Cc_-(\La',r',\b)}$.

%%%%%%%%%%%%%%%%%%%%%%%%%%%%%%%%%%%%
\subsection{Matching splittings, intertwining and conjugacy}
%%%%%%%%%%%%%%%%%%%%%%%%%%%%%%%%%%%%
%In this section, we l
Let~$[\La,n,r,\b]$ and $[\La',n',r',\b']$ be semisimple strata in~$\A$ %with~$e(\La)=e(\La')$, and 
with associated splittings~$\V=\bigoplus_{i\in \I}\V^i$ and~$\V=\bigoplus_{i\in \I'}\V'^i$, respectively. The starting point for this section is the \emph{Matching Theorem} of the second and third authors.

\begin{definition}\label{def:matching}
\shauns{Let~$\t\in \Cc(\La,r,\b)$ and~$\t'\in \Cc(\La',r',\b')$ be semisimple characters and
suppose there are a bijection~$\z:\I\to\I'$ and~$g\in\Aut_\F(\V)$ such that, for each~$i\in\I$, we have:
\begin{enumerate}\setlength\itemsep{5pt}
\item $g \V^i=\V'^{\z(i)}$;
\item $\presuper{g}\t_i$ and~$\t_{\z(i)}'$ intertwine in $\Aut_\F(\V'^{\z(i)})$. 
\end{enumerate}
Then we say that~$\z$ is a \emph{matching} from~$(\t,\b)$ to~$(\t',\b')$, and that~$\t$ intertwines~$\t'$ in~$\Aut_\F(\V)$ with matching~$\z$.
}
\end{definition}

\shauns{The use of the terminology ``intertwines with matching~$\z$'' is justified by the following result.}

\begin{proposition}\label{prop:intdecomp}
\shauns{Let~$\t\in \Cc(\La,r,\b)$ and~$\t'\in \Cc(\La',r',\b')$ be semisimple characters in~$\tG$, and suppose~$\z:\I\rightarrow \I'$ is a bijection between their index sets. For~$i\in\I$, write~$\I(\t_i,\t'_{\z(i)})$ for the set of isomorphisms~$g\in\Hom_\F(\V^i,\V'^{\z(i)})$ such that~$\presuper{g}\t_i$ is intertwined with~$\t'_{\z(i)}$ by the identity. Then
\begin{equation}\label{eqn:intdecomp}
\left(\I_{\tG}(\t,\t')\cap \prod_{i\in\I}\Hom_\F(\V^i,\V'^{\z(i)})\right) = \prod_{i\in\I}\I(\t_i,\t'_{\z(i)}).
\end{equation}
\ignore{
Suppose further that~$\t_-\in \Cc_-(\La,r,\b)$ and~$\t'_-\in \Cc(\La',r,\b')$ are self-dual semisimple characters in~$\G$ whose lifts are~$\t,\t'$ respectively, and that~$\z$ commutes with the involution~$\s$. For~$i\in\I_0$, we write~$\I_\G(\t_i,\t'_{\z(i)})$ for the set of isometries~$g\in\Hom_\F(\V^i,\V'^{\z(i)})$ such that~$\presuper{g}\t_i$ is intertwined with~$\t'_{\z(i)}$ by the identity. Then
\[
\left(\I_{\G}(\t_-,\t'_-)\cap \prod_{i\in\I}\Hom_\F(\V^i,\V'^{\z(i)})\right) = \prod_{i\in\I_+}\I(\t_i,\t'_{\z(i)})\times \prod_{i\in\I_0}\I_{\G}(\t_i,\t'_{\z(i)}),
\]
where, on the right hand side, for~$i\in\I_+$, we have identified~$\I(\t_i,\t'_{\z(i)})$ with the set of~$(g,\s(g))$ in~$\I(\t_i,\t'_{\z(i)})\times\I(\t_{\s(i)},\t'_{\s\z(i)})$.
\todo{Does this make sense?}
}
}
\end{proposition}

\begin{proof}
\shauns{It is clear that the left hand side of~\eqref{eqn:intdecomp} is contained in the right hand side. Conversely, if~$g_i\in \I(\t_i,\t'_{\z(i)})$, for~$i\in\I$, then~$g=\sum_{i\in\ I}g_i$ intertwines~$\t$ with~$\t'$, by the Iwahori decomposition of semisimple characters.} %The second assertion follows immediately by Glauberman.}
\end{proof}

\shauns{For the matching theorem we restrict to the case~$r=r'$ and~$e(\La)=e(\La')$, though we will see later that these hypotheses could be relaxed somewhat.}

\begin{theorem}[{\cite[Theorem~10.1]{SkSt}}]\label{thm:MatchingForChar}
\shauns{Suppose that~$e(\La)=e(\La')$ and} let~$\t\in \Cc(\La,r,\b)$ and~$\t'\in \Cc(\La',r,\b')$ be semisimple characters %in~$\A$ 
which intertwine in~$\tG$. \shauns{Then there is a unique matching~$\z:\I\to\I'$ from~$(\t,\b)$ to~$(\t',\b')$. Moreover, if~$g\in\tG$ satisfies~$g \V^i=\V'^{\z(i)}$ for~$i\in\I$, then~$\presuper{g}\t_i$ and~$\t_{\z(i)}'$ intertwine in $\Aut_\F(\V'^{\z(i)})$.}
%bijection~$\z:\I\rightarrow \I'$ such there is an element~$g\in\tG$ with,  for all~$i\in\I$,
%\begin{enumerate}\setlength\itemsep{5pt}
%\item\label{match1} we have~$g \V^i=\V'^{\z(i)}$;
%\item\label{match2} the characters~$\presuper{g}\t_i$ and~$\t_{\z(i)}'$ intertwine in $\tG_{\z(i)}$. 
%\end{enumerate}
%Moreover, any element of~$\tG$ which satisfies the first property also satisfies the second property. 
\end{theorem}
In particular, under the notation of the theorem we have~$e(\E_i/\F)=e(\E'_{\z(i)}/\F)$,~$f(\E_i/\F)=f(\E'_{\z(i)}/\F)$, and~$k_0(\b_i,\La^i)=k_0(\b'_{{\z(i)}},\La'^{{\z(i)}})$ by Proposition~\ref{prop:SimpleDegrees}. 

%\begin{definition}\label{def:matching}
%A map~$\z:\I\rightarrow\I'$ between the indexing sets of splittings which satisfies properties~\ref{match1},~\ref{match2} of Theorem \ref{thm:MatchingForChar} is called a \emph{matching} from~$(\t,\b)$ to~$(\t',\b')$. We also say that~$\t$ and~$\t'$ \emph{intertwine by an element of~$\tG$ with matching~$\z$}.
%\end{definition}

\begin{remark}\label{rem:matching}
\shauns{Suppose~$[\La,n,r,\b]$ and $[\La',n,r,\b']$ are self-dual semisimple strata and~$\t_-\in \Cc_-(\La,r,\b)$ and~$\t'_-\in \Cc_-(\La',r,\b')$ are self-dual semisimple characters, with lifts~$\t,\t'$ respectively. If~$\t_-,\t'_-$ intertwine in~$\G$ then the lifts intertwine so we have a matching~$\z:\I\to\I'$ from~$(\t,\b)$ to~$(\t',\b')$. Moreover, by uniqueness of matchings,~$\z$ is~$\s$-equivariant. We will also say that~$\z$ is a matching from~$(\t_-,\b)$ to~$(\t'_-,\b)$.}
\end{remark}

\begin{remark}\label{rem:skewsemi}
\shauny{It follows from the~$\s$-equivariance of the matching %and a  proof similar to the one of~\ref{prop:SimpleDegrees}\ref{prop:SimpleDegrees.ii} it follows
that if~$\t_-$ is a skew semisimple character and~$[\La,n,r,\b]$ is any self-dual semisimple stratum such that~$\t_-\in\Cc(\La,r,\b)$ then the stratum~$[\La,n,r,\b]$ is in fact skew semisimple.}
\end{remark}

\begin{corollary}\label{corCharPolBlockwisev4}
Under the assumptions of Theorem~\ref{thm:MatchingForChar}, suppose that there exists~$g\in\tG$ such that~$g\b g^{-1}=\b'$ and~$\t'=\tau_{\La',g\La,\b'}(\presuper{g}{\t})$. Then~$\b_i$ and~$\b'_{\z(i)}$ have the same characteristic polynomial, for all~$i\in \I$. 
\end{corollary}

\begin{proof}
\shauns{By conjugating by~$g$ we reduce to the case that~$\b=\b'$. As~$\t'=\tau_{\La',\La,\b}(\t)$ we have~$1\in \I_\tG(\t,\t')$. The identity map~$\I\rightarrow\I$ is then a matching from~$(\t,\b)$ to~$(\t',\b')$ %then satisfies conditions~\ref{match1} and~\ref{match2} of Theorem~\ref{thm:MatchingForChar}, and the uniqueness statement implies that
so the uniqueness in Theorem~\ref{thm:MatchingForChar} implies that~$\z$ is the trivial permutation of the index set, which finishes the proof.} 
\end{proof}

\begin{remark}\label{remark:CyclicPermutationSemisimpleChar}
If the semisimple characters are not related by transfer, then the conclusion of Corollary~\ref{corCharPolBlockwisev4} need not hold, as the following example shows. Suppose the characteristic of~$\F$ is~$p$ and set~$\I=\{0,\ldots,p-1\}$. 
\orange{Take an element~$\b_0\in\F$ of negative even valuation, and  a \bob{regular lattice sequence}~$\La^0$. Set~$n=-\val_\La(\b_0)$ and~$r=\frac{n}{2}$, both of which are non-zero multiples of~$e(\La^0)$.}
%Take an element~$\b_0$ which is minimal over~$\F$ and a principal lattice chain~$\La^0$ normalized by~$\b_0$. Set~$n=-\val_\La(\b_0)$ and~$r=\left\lfloor \frac{n}{2}\right\rfloor$ and suppose that~$n>0$ and that~$r$ is a non-zero multiple of~$e(\La^0)$. 
Take an element~$\l\in (\F\cap\aa_{-r}(\La^0))\setminus\aa_{-r+1}(\La^0)$ and set~$\b_i=\b_0+i\l$, for~$i\in\I$. Then, putting~$\b=\sum_{i\in\I}\b_i$ and~$\La=\bigoplus_{i\in\I} \La^0$, the stratum~$[\La,n,0,\b]$ is semisimple.

By the results of~\cite[\S3.5]{BK93}, the sets~$\Cc(\La^0,r-1,\b_i)$ coincide, so multiplication by~$\psi_\l$ induces a permutation of $\Cc(\La^0,r-1,\b_0)$. Choosing any~$\t_0\in\Cc(\La^0,r-1,\b_0)$, there are unique semisimple characters~$\t,\t'\in\Cc(\La,r-1,\sum_{i\in\I}\b_i)$ whose~$i$th simple block restrictions are~$\t_0\psi_{i\l}$ and~$\t_0\psi_{(i+1)\l}$ respectively. Then the matching from~$(\t,\b)$ to~$(\t',\b)$ is a cyclic permutation, not the identity, but~$\b_i,\b_{i+1}$ do not have the same characteristic polynomial.  
%%%%%% The existence of \t,\t' is not totally obvious but I don't think I want to include an argument for it, since this is an unused remark.
% and extend~$\bigotimes_{l=0}^{p-1} \t_1\psi_{l\l}$ and~$\bigotimes_{l=1}^{p} \t_1\psi_{l\l}$  to elements~$\t$ and~$\t'$ of~$\Cc(\bigoplus_i\La,r-1,\bigoplus_l(\b_1+l\l))$ by extending them trivially on the upper and lower unipotent matrices \red{one needs to use the Iwahori decomposition with respect to the parabolic?... Also one should explain why this is well-defined and defines a semisimple character - see mine and shauns paper section 3}. Then the matching from~$\t$ to~$\t'$ is a cyclic permutation, but not the identity.  
\end{remark}

We need a description of the intertwining of transfers, which generalizes~\cite[Theorem~3.22]{St05} and which we prove in the appendix (as Proposition~\ref{prop:IntertwiningOfTransfers}). The statement involves a particular subgroup~$\mathrm{S}_r(\b,\La)$ of~$\P^1(\La)$ associated to the semisimple stratum~$[\La,n,r,\b]$, which normalizes every character in~$\Cc(\La,r,\b)$ and is defined in~\cite[Section 3.2]{St05}, where it is denoted~$\Gamma_r(\b,\La)$ (see also~\cite[Proposition~9.8]{SkSt}, and~\cite[(3.5.1)]{BK93} for the simple case).

\begin{proposition}\label{propIntertwiningOfTransfers} Suppose~$e(\La)=e(\La')$ and~$\b'=\b$.
\begin{enumerate}\setlength\itemsep{5pt}
\item Let~$\t\in\Cc(\La,r,\b)$ and%~$\t'\in\Cc(\La',r,\b)$ be such that
~$\t'=\tau_{\La',\La,\b}(\t)$. Then %we have
\[
\I_{\tG}(\t,\t')=\mathrm{S}_{r}(\b,\La')\tG_{\b} \mathrm{S}_r(\b,\La).
\]
\label{propIntertwiningOfTransfersPart1}
\item Suppose~$[\La,n,r,\b]$ and~$[\La',n,r,\b]$ are self-dual and let~$\t\in\Cc_-(\La,r,\b)$ and%~$\t'\in\Cc_-(\La',r,\b)$ be such that
~$\t'_-=\tau_{\La',\La,\b}(\t_-)$. Then %we have
\[
\I_\G(\t_-,\t'_-)=(\SS_{r}(\b,\La')\cap \G) \G_{\b} (\SS_r(\b,\La)\cap \G).
\]
\end{enumerate}
\end{proposition}

We also have the following intertwining implies conjugacy theorem. In the case of semisimple characters and skew semisimple characters, this is~\cite[Theorems~10.2,10.3]{SkSt} respectively; we prove the case of self-dual semisimple characters in the appendix (as Theorem~\ref{thm:IntImplConjSelfDual}). The condition on a matching~$\z$ which allows one to deduce conjugacy is
\begin{equation}\label{eqn:matchdims}
\Lambda^i(j)/\Lambda^i(j+1)\cong \Lambda'^{\zeta(i)}(j)/\Lambda'^{\zeta(i)}(j+1), \quad\text{for all~$i\in \I$ and all integers~$j$.}
\end{equation}

\begin{theorem}\label{thmIntImplConjSelfDual}
\shauns{Suppose that~$e(\La)=e(\La')$.   
\begin{enumerate}\setlength\itemsep{5pt}
\item\label{thmIntImplConjSelfDual.i} Let~$\theta\in\Cc(\Lambda,r,\beta)$ and~$\theta'\in\Cc(\Lambda,r,\beta')$ be semisimple characters which intertwine, such that the matching~$\zeta$ from~$(\t,\b)$ to~$(\t',\b')$ satisfies~\eqref{eqn:matchdims}. Then there is an element of $\P(\Lambda)\cap\prod_{i\in\I}\Hom_\F(\V^i,\V'^{\z(i)})$ which conjugates~$\theta$ to~$\theta'$.
\item\label{thmIntImplConjSelfDual.ii} Suppose~$[\La,n,r,\b]$ and~$[\La',n,r,\b]$ are self-dual and let~$\theta_-\in\Cc_-(\Lambda,r,\beta)$ and~$\theta'_-\in\Cc_-(\Lambda,r,\beta')$ be self-dual semisimple characters which intertwine in~$\G$, such that the matching~$\zeta$ from~$(\t_-,\b)$ to~$(\t'_-,\b')$ satisfies~\eqref{eqn:matchdims}. Then there is an element of $\P_-(\Lambda)\cap\prod_{i\in\I}\Hom_\F(\V^i,\V'^{\z(i)})$ which conjugates~$\theta_-$ to~$\theta'_-$.
\end{enumerate}}
\end{theorem}

From these results we get the following important corollaries.

\begin{corollary}\label{cor:Intertwining}
%\begin{enumerate}\setlength\itemsep{5pt}
%\item 
Suppose that~$e(\La)=e(\La')$ and let~$\t\in \Cc(\La,r,\b)$ and~$\t'\in \Cc(\La',r,\b')$ be semisimple characters which intertwine in~$\tG$, and let~$\z:\I\rightarrow \I'$ be the matching from~$(\t,\b)$ to~$(\t',\b')$. Then
\[
\I_\tG(\t,\t')=\mathrm{S}_{r}(\b',\La')\left(\I_{\tG}(\t,\t')\cap \prod_{i\in\I}\Hom_\F(\V^i,\V'^{\z(i)})\right)
\mathrm{S}_r(\b,\La).
\]
%\item 
Suppose further that~$[\La,n,r,\b]$ and~$[\La',n,r,\b']$ are self-dual and that we have a partition~$\mathcal{P}$ of~$\I$ into~$\s$-stable subsets such that for each~$\J\in\mathcal{P}$ the hermitian spaces~$(\V^\J,h_\J)$ and~$(\V'^{\z(\J)},h_{\z(\J)})$ are isometric. Let~$\t_-\in \Cc_-(\La,r,\b)$ and~$\t'_-\in \Cc(\La',r,\b')$ be self-dual semisimple characters which intertwine in~$\G$%, and denote their lifts by~$\t$ and~$\t'$ respectively
. Then
\[
\I_\G(\t_-,\t'_-)=(\mathrm{S}_{r}(\b',\La')\cap \G)\left(\I_{\G}(\t_-,\t'_-)\cap\prod_{\J\in\mathcal{P}}\Hom_\F(\V^{\J},\V'^{\z(\J)})\right)(\mathrm{S}_r(\b,\La)\cap \G).
\] 
%\end{enumerate}
\end{corollary}

\begin{proof}%[Proof of Corollary~\ref{cor:Intertwining}]
%\ref{cor:Intertwining.i} 
Let~$g\in\tG$ be an element inducing the matching, that is, satisfying the conditions in Theorem~\ref{thm:MatchingForChar}. For each index~$i$ the field extensions~$\E_i/\F$ and~$\E'_{\z(i)}/\F$ have equal ramification indices and inertia degrees, by Proposition~\ref{prop:SimpleDegrees}; thus, by Lemma~\ref{lem:changetoconjlatt}, there is an~$\o_{\E_i}$-lattice sequence~$\La''^i$ in~$\V^i$ which is conjugate in~$\tG_i$ to~$g^{-1}\La'^{\z(i)}$. In particular,~$\La''=\bigoplus_{i\in\I}\La''^i$ is then an~$\o_\E$-lattice sequence in~$\V$ which is~$\tG$-conjugate to~$\La'$ by an element which maps~$\La''^i$ to~$\La'^{\z(i)}$.

Let~$\t''=\tau_{\La'',\La,\b}(\t)$ be the transfer of~$\t$ to~$\Cc(\La'',r,\b)$. Applying Theorem~\ref{thm:transintertwiningoverGL} to the simple block restrictions of~$\t'',\t,\t'$ and Proposition~\ref{propIntertwiningOfTransfers}, we see that~$\t''$ intertwines with~$\t'$. Then Theorem~\ref{thmIntImplConjSelfDual}\ref{thmIntImplConjSelfDual.i}
%~\cite[Theorem 10.2]{SkSt} 
implies that~$\t''$ is conjugate to~$\t'$ by an element~$g\in\tG$ which maps~$\La''^i$ to~$\La'^{\z(i)}$. 

Conjugating by this element, we can assume we are in the case~$\t''=\t'$,~$\La''=\La'$ and~$\V^i=\V'^{\z(i)}$ for all~$i\in\I$. We can then identify the index sets so that~$\z$ is the identity. We have~$\t\in\Cc(\La,r,\b)$ and $\t'\in\Cc(\La',r,\b)\cap\Cc(\La',r,\b')$ so that
\[
\SS_r(\b,\La')(\KK(\La')\cap\tG_\b)=\I_{\tG}(\t')\cap\KK(\La')=\SS_r(\b',\La')(\KK(\La')\cap\tG_{\b'}).
\]
% 
% matching~$\z$ is the identity. Note that we are only interested in the splittings, and hence, as they and the sets of characters agree after conjugation, we can assume~$\b=\b'$. 
Then Proposition~\ref{propIntertwiningOfTransfers} implies
\begin{align*}
\I_\tG(\t,\t')&=\SS_r(\b,\La')\tG_{\b} \SS_r(\b,\La)=\SS_r(\b,\La')(\KK(\La')\cap\tG_\b)\tG_{\b} \SS_r(\b,\La)\\
&=\SS_{r}(\b',\La')(\KK(\La')\cap\tG_{\b'})\tG_{\b} \SS_r(\b,\La)\subseteq \SS_{r}(\b',\La')(\I_\tG(\t,\t')\cap\prod_i \tG_{i}) \SS_r(\b,\La),
\end{align*}
since~$\tG_\b$ and~$\tG_{\b'}$ are both contained in~$\prod_i \tG_{i}$. %The completes the proof of~\ref{cor:Intertwining.i}, and~\ref{cor:Intertwining.ii} 
The final assertion now follows from a standard cohomology argument as in~\cite[4.14]{St05} (cf. also~\cite[2.4]{RKSS}). 
\end{proof}

\ignore{
Finally in this section, we state an intertwining implies conjugacy theorem for self-dual semisimple characters, which generalizes a result of the second and third author~\cite[10.2,10.3]{SkSt} for skew semisimple characters and which we prove in the appendix (as Theorem~\ref{thm:IntImplConjSelfDual}). 
\begin{theorem}\label{thmIntImplConjSelfDual}
\shauns{Let~$\theta\in\Cc_-(\Lambda,r,\beta)$ and~$\theta'\in\Cc_-(\Lambda,r,\beta')$ be self-dual semisimple characters which intertwine in~$\G$ such that the matching~$\zeta$ from~$(\t_-,\b)$ to~$(\t'_-,\b')$ satisfies 
\[
\Lambda^i(j)/\Lambda^i(j+1)\cong \Lambda'^{\zeta(i)}(j)/\Lambda'^{\zeta(i)}(j+1), 
\]
for all~$i\in \I$ and all integers~$j$. Then there is an element of $\P_-(\Lambda)\cap\Hom_\F(\V^i,\V'^{\z(i)})$ which conjugates~$\theta$ to~$\theta'$.}
\end{theorem}
}

%%%%%%%%%%%%%%%%%%%%%%%%%%%%%%%%%%%%
\subsection{Concordance and Skolem--Noether}
%%%%%%%%%%%%%%%%%%%%%%%%%%%%%%%%%%%%
We now prove a conjecture of the second and third authors,~\cite[Conjecture~10.4]{SkSt}. We let~$[\La,n,r,\b]$ and~$[\La',n,r,\b']$ be self-dual semisimple strata in~$\A$, with~$e(\La)=e(\La')$ and with associated splittings~$\V=\bigoplus_{i\in \I}\V^i$ and~$\V=\bigoplus_{i\in \I'}\V'^i$, respectively. %In the following, when~$[\La,n,r,\beta]$ is a semisimple stratum with associated splitting~$\V=\bigoplus_{i\in\I}\V^i$, w
We will write~$\vphi_i$ for the canonical embedding of~$\E_i=\F[\b_i]$ in~$\A^i$, and similarly~$\vphi'_i$. 
\orange{For~$\t_-\in \Cc_-(\La,r,\b)$, we write~$\t$ for its lift and~$\t_i$ for the simple block restrictions of~$\t$; if~$i\in\I_0$ then~$\t_i$ is self-dual and we write~$\t_{-,i}$ for its restriction in~$\Cc_-(\La^i,r,\b_i)$. We use similar notation for~$\t'_-\in \Cc_-(\La',r,\b')$. When we write about a \emph{matching from~$(\t_-,\b)$ to~$(\t'_-,\b')$}, we mean a matching from~$(\t,\b)$ to~$(\t',\b')$.}

\begin{theorem}\label{thm:MatchingForCharForG}
Let~$\t_-\in \Cc_-(\La,r,\b)$ and~$\t'_-\in \Cc_-(\La',r,\b')$ be self-dual semisimple characters which intertwine in~$\G$ and let~$\z:\I\rightarrow\I'$ be the matching from~$(\t_-,\b)$ to~$(\t'_-,\b')$. Then, for~$i\in\I_0$, the spaces~$(\V^i,h_i)$ and~$(\V'^{\z(i)},h'_{\z(i)})$ are isometric and the characters~$\t_{-,i}$ and~$\t'_{-,\z(i)}$ intertwine by an isometry from~$(\V^i,h_i)$ to~$(\V'^{\z(i)},h'_{\z(i)})$. Moreover, the pairs~$(\b_i,\vphi_i)$ and~$(\b'_{\z(i)},\vphi'_{\z(i)})$ are~$(h_i,h'_{\z(i)})$-concordant.
%for~$i\in\I_0$, let~$g_i$ be an isometry from~$h_i$ to~$h_{\z(i)}$. Then the canonical embeddings of~$g_i\b_i g_i^{-1}$ and~$\b'_{\z(i)}$ are Witt concordant.
%\red{My preference is to scrap the Witt-concordant notation, and just use concordant for everything}
\end{theorem}

%In the sequel we are going to say the canonical embeddings of~$\b_i$ and~$\b'_{\z(i)}$ are Witt-concordant instead of the last sentence of the Theorem. Note that this means 
%we consider~$h_i$ for~$\b_i$ and~$h_{\z(i)}$ for~$\b'_{\z(i)}$.  
In the proof of the theorem (and the subsequent corollaries), we abbreviate~$\V_0$ for~$\V^{\I_0}$, so that~$\V_0^\perp=\V^{\I_+\cup\I_-}$, and similarly~$\V'_0,\V_0'^\perp$. We also write~$\t,\t'$ for the lifts of~$\t_-,\t'_-$ respectively.
%We also abbreviate~$\V_\so$ for~$\V^{\I_0}$ and~$\V_{+-}$ for~$\V^{\I_+\cup\I_-}$. \red{I think it should be~$\V_0$ rather than~$\V_\so$ same with~$\I_0$.... for me~$\so$ isn't a~$0$ it is just a symbol...}

\begin{proof}
\shaun{%
Since the spaces~$\V_0^\perp$ and~$\V_0'^\perp$ have the same~$\F$-dimension and are hyperbolic, they are isometric~$\e$-hermitian spaces over~$\F/\F_\so$; thus~$\V_0$ and~$\V'_0$ are also isometric and Corollary~\ref{cor:Intertwining} then reduces us to the case~$\I=\I_0$. Moreover, the final assertion follows from the first and Proposition~\ref{prop:TiGandGIntertwiningSameNonSympl} so, by Corollary~\ref{cor:Intertwining} again, it is enough to show that~$\V^i$ is isometric to~$\V'^{\z(i)}$ for all~$i\in\I$.
}

\shaun{%
We proceed by induction along~$r$. When~$r=n$, both characters are trivial so there is nothing to show, while the case~$r=n-1$ is given by~\cite[Proposition~7.10]{SkSt}. Suppose now that~$r<n-1$ and let~$[\La,n,r+1,\g]$ and~$[\La',n,r+1,\g']$ be self-dual semisimple strata equivalent \rob{to}~$[\La,n,r+1,\b]$ and~$[\La',n,r+1,\b']$, respectively, such that~$\g\in \prod_{i\in\I}\A^{i}$ and~$\g'\in \prod_{j\in\I'}\A'^{j}$. By the induction hypothesis and Corollary~\ref{cor:Intertwining} it is sufficient to assume that~$\g$ and~$\g'$ generate field extensions of~$\F$. Then Lemma~\ref{lemma:MatchWittforIntertwCharacters}\ref{lemma:MatchWittforIntertwCharacters.ii} and Lemma~\ref{lemma:diagonal} imply that there exist simple self-dual strata~$[\La,n,r+1,\tilde{\g}]$ and~$[\La',n,r+1,\tilde{\g}']$ such that~$\tilde{\g}$ and~$\tilde{\g}'$ have the same minimal polynomial, their canonical embeddings are concordant, and
\[
\Cc(\La,r+1,\g)=\Cc(\La,r+1,\tilde{\g}),\qquad \Cc(\La',r+1,\g')=\Cc(\La',r+1,\tilde{\g}').
\]
In particular, there is an element~$g$ in~$\G$ such that~$g\tilde{\g}g^{-1}=\tilde{\g}'$. The characters~$\t|_{\H^{r+2}(\tilde{\g},\La)}$ and~$\tau_{g\La,\La',\tilde{\g}'}(\t'|_{\H^{r+2}(\tilde{\g}',\La')})$ intertwine by an element of~$\G$, by Corollary~\ref{corIntEquivalenceRelation}, and therefore are conjugate by an element~$g'\in\G$ which maps~$\La$ to~$g\La$, by Proposition~\ref{prop:InTildeGConGSimpleSymp}. Thus
\[\
\Cc(\La,r+1,\g)=\Cc(\La,r+1,\tilde{\g}'^{g'})
\]
and we can replace~$\tilde{\g}$ by~$\tilde{\g}'^{g'}$; that is, we can assume without loss of  generality that~$\tilde{\g}=\tilde{\g}'^{g'}$. We only need to prove the result for~$\presuper{g'}\t$ and~$\t'$ so we can assume further that~$\tilde{\g}=\tilde{\g}'$ and~$g'=1$, that is,~$\t|_{\H^{r+2}(\La,\tilde{\g})}$ is the transfer of~$\t'|_{\H^{r+2}(\La',\tilde{\g}')}$ from~$\La'$ to~$\La$.
}

\shauns{
By the translation principle~\cite[Theorem~9.26]{SkSt} there are a skew semisimple stratum~$[\La,n,r,\tilde{\b}]$\daniel{, such that~$[\La,n,r+1,\tilde{\b}]$ is equivalent to~$[\La,n,r+1,\tilde{\g}]$,} and~$u\in\P^1_-(\La)$\daniel{,} which normalizes~$[\La,n,r+1,\tilde{\g}]$ up to equivalence, such that
\[
\Cc(\La,r,\b)=\Cc(\La,r,\tilde{\b}), \qquad\text{and}\qquad u\tilde{\g}u^{-1}\in\prod_{i\in\tilde{\I}}\A^i,
\]
where~$\V=\bigoplus_{i\in\tilde{\I}}\V^i$ is the splitting of~$[\La,n,r,\tilde{\b}]$. If we denote by~$\tau$ the matching between~$(\t,\b)$ and~$(\t,\tilde\b)$ then Proposition~\ref{prop:ConjIdempToEachOther}\ref{prop:ConjIdempToEachOther-ii} implies that~$\V^i$ is isometric to~$\V^{\tau(i)}$ for all~$i\in\I$ and there is an element of the normalizer of~$\t$ in~$\G$ which realizes this matching. Analogously, we have a skew semisimple stratum~$[\La',n,r,\tilde{\b}']$\daniel{, such that~$[\La',n,r+1,\tilde{\b}']$ is equivalent to $[\La',n,r+1,\tilde{\g}]$,} and~$u'\in\P^1_-(\La')$, and a matching~$\tau'$ between~$(\t',\b')$ and~$(\t',\tilde{\b}')$. Thus we need only show that~$\V^{\tau(i)}$ and~$\V^{\tau'\z(i)}$ are isometric. Since matchings are unique,~$\tau'\z\tau^{-1}$ is the matching between~$(\t,\tilde\b)$ and~$(\t',\tilde\b')$, and we are thus reduced to the case~$\b=\tilde\b$ and~$\b'=\tilde{\b}'$.
}

\shauns{Now~$[\La,n,r+1,u^{-1}\b u]$ has splitting~$\bigoplus_{i\in\I} u^{-1}\V^i$ and is equivalent to~$[\La,n,r+1,\tilde{\g}]$, with~$\tilde{\g}\in\prod_{i\in\I} u^{-1}\A^i u$, and~$\t^u|_{\H^{r+2}(\tilde{\g},\La)}=\t|_{\H^{r+2}(\tilde{\g},\La)}$. Since~$u^{-1}\V^i$ is isometric to~$\V^i$, if we replace~$(\t,\b)$ by~$(\t^u,u^{-1}\b u)$ we reduce further to the case~$\tilde\g=\g$. Similarly, replacing~$(\t',\b')$ by~$(\t'^{u'},u'^{-1}\b' u')$ we see that we may assume~$\g=\tilde{\g}=\tilde{\g}'=\g'$.
}

\ignore{
By the translation principle~\cite[Theorem~9.26]{SkSt} there are skew-semisimple strata~$[\La,n,r,\tilde{\b}]$ 
and~$[\La,n,r,\tilde{\b}']$ such that
\begin{itemize}
 \item~$\Cc(\La,r,\b)=\Cc(\La,r,\tilde{\b})$
 \item~$\Cc(\La',r,\b')=\Cc(\La',r,\tilde{\b}')$
 \item~$[\La,n,r+1,\tilde{\b}]\approx [\La,n,r+1,\tilde{\g}]$
 \item~$[\La',n,r+1,\tilde{\b}']\approx [\La',n,r+1,\tilde{\g}]$
 \item~$u\tilde{\g}u^{-1}\in\prod_{i\in\tilde{\I}}\A^i$ for some~$u\in (1+\mathfrak{m}_{r+1}(\tilde{\gamma},\La))$ 
 \item~$v\tilde{\g}v^{-1}\in\prod_{i\in\tilde{\I}'}\A^i$ for some~$v\in (1+\mathfrak{m}_{r+1}(\tilde{\gamma},\La'))$.
\end{itemize}
By Proposition~\ref{prop:ConjIdempToEachOther}\ref{prop:ConjIdempToEachOther-ii} we only need to prove the result for~$\tilde{\b}$ and~$\tilde{\b}'$, and by conjugation with~$u$ and~$v$ we only need to prove the assertion for 
the pair~$ ^u\tilde{\b}$ and~$ ^v\tilde{\b}'$. Thus we can assume~$\g=\g'$.} 
% \red{%
% I can't follow this: By the translation principle~\cite[Theorem~9.26]{SkSt} and Proposition~\ref{prop:ConjIdempToEachOther}\ref{prop:ConjIdempToEachOther-ii} we can assume
% \[
% \g=\tilde{\g}=\tilde{\g}'=\g'.
% \] 
% }

\shaun{%
Now we take~$\t'_0\in\Cc(\La',r,\g)$ such that~$\t'=\t'_0\psi_{\b'-\g}$. If we set~$\t_0=\tau_{\La,\La',\g}(\t'_0)$ then the restrictions of~$\t$ and~$\t_0$ coincide on~$\H^{r+2}(\g,\La)$) so there exists~$c\in (\prod_i \A^{i})_-\cap\aa_{-r-1}$ such that~$\t=\t_0\psi_{\b-\g+c}$. Moreover, since then~$\t_0$ and~$\t_0\psi_c$ are both simple characters in~$\Cc(\La,r,\g)$, \orange{the character~$\psi_c$ is intertwined by all of~$\tG_\g$ and, writing~$s_\g$ for an equivariant tame corestriction with respect to~$\g$, it follows from~\cite[Lemma~3.10]{secherreStevensVI:10} that~$s_\g(c)$ is congruent to an element of~$\F[\g]$ modulo~$\aa_{-r}$.}
}

\ignore{%
Now we take~$\t_0\in\Cc(\La,r,\g)$ and~$\t'_0\in\Cc(\La',r,\g)$ such that~$\t_0=\tau_{\La,\La',\g}(\t'_0)$ and we can take~$\t'_0$ such that the restrictions of~$\t'$ and~$\t'_0$ on~$\H^{r+2}(\g,\La')$ coincide (from which it follows that the restrictions of~$\t$ and~$\t_0$ coincide on~$\H^{r+2}(\g,\La)$) and such that~$\t'=\t'_0\psi_{\b'-\g}$. Then there exists~$c\in (\prod_i \A^{i})_-\cap\aa_{-r-1}$ such that~$\t=\t_0\psi_{\b-\g+c}$; moreover, since then~$\t_0$ and~$\t_0\psi_c$ are both simple characters in~$\Cc(\La,r,\g)$, it follows from \red{Reference}, writing~$s_\g$ for an equivariant tame corestriction with respect to~$\g$, that~$s_\g(c)$ is congruent to an element of~$\F[\g]$ modulo~$\aa_{-r}$.
}

\shaun{%
Now Lemma~\ref{lem:DerivedCharactersForG} implies that the self-dual strata~$[\La,r+1,r,s_\g(\b-\g+c)]$ and~$[\La',r+1,r,s_\g(\b'-\g)]$ intertwine in~$\G$; moreover, these strata are equivalent to self-dual semisimple strata with splittings~$\V=\oplus_i\V^i$ and~$\V=\oplus_{i\in\I'}\V'^{i}$, respectively, by~\cite[Theorem~6.15]{SkSt}. }%
\shauns{The matching between these latter strata is still~$\z$. Indeed, by~\cite[Proposition 7.1]{SkSt} there are a bijection~$\xi:\I\rightarrow\I'$ and an element~$\tilde{g}\in\tG_\g$ which intertwines~$[\La,r+1,r,s_\g(\b-\g+c)]$ with~$[\La',r+1,r,s_\g(\b'-\g)]$ and satisfies~$\tilde{g}\V^i=\V'^{\xi(i)}$. Then~$\id_{\V^i}$ intertwines~$[\La^i,r+1,r,s_\g(\b-\g+c)|_{\V^i}]$ with~$[\tilde{g}^{-1}\La'^{\xi(i)},r+1,r,s_\g(\tilde{g}^{-1}\b'_{\xi(i)}\tilde{g}-\g)]$, for each~$i\in\I$. Thus~$\t_i$ and~$\t_{\xi(i)}'^{\tilde{g}}$ intertwine, by Lemma~\ref{lem:DerivedCharacters}\ref{lem:DerivedCharacters.ii}, and the uniqueness of the matching implies that~$\xi=\z$. Finally, the base case implies that there is an element~$g\in\G_\g$ such that~$g\V^i=\V'^{\z(i)}$, for all~$i\in\I$, which gives the required isometry.}
\end{proof}

Finally, we deduce from Theorem~\ref{thm:MatchingForCharForG} a semisimple \emph{Skolem--Noether} result:

\begin{corollary}\label{cor:SkolemNoetherSemisimplev4}
Let~$\t_-\in \Cc_-(\La,r,\b)$ and~$\t'_-\in \Cc_-(\La',r,\b')$ be self-dual semisimple characters which intertwine in~$\G$, let~$\z:\I\rightarrow\I'$ be the matching from~$(\t_-,\b)$ to~$(\t'_-,\b')$, and suppose that~$\b_i$ and~$\b'_{\z(i)}$ have the same characteristic polynomial for all indices~$i$. Then~$\b$ and~$\b'$ are conjugate in~$\G$.
%, there is an element of~$G^+$ which conjugates~$\b$ to~$\b'$.
\end{corollary}

\begin{proof}
By Theorem~\ref{thm:MatchingForCharForG}, we can assume that the matching~$\z$ is the identity of~$\I$. Now, the characters~$\t_i$ and~$\t'_i$ intertwine in~$\G_i$ for all~$i\in\I_0$. Hence, by \cite[Theorem 5.2]{SkSt},~$\b_i$ and~$\b'_i$ are conjugate by an element of~$\U(\V^i,h_i)$,~$i\in\I_0$. Thus~$\b$ and~$\b'$ are conjugate in~$\G$.
\end{proof}

For Theorem~\ref{thmIIC} below, the following %consequence of Theorem~\ref{thm:MatchingForCharForG} and 
generalization of Lemma~\ref{lemma:diagonal} is crucial: it allows us, when we have self-dual semisimple characters which intertwine, to find strata giving rise to these characters whose defining elements are \emph{conjugate}.

\begin{corollary}\label{cor:diagonalSemisimple}
Let~$\t_-\in\Cc_-(\La,r,\b)$ and~$\t'_-\in\Cc_-(\La',r,\b')$ be self-dual semisimple characters which intertwine by an element of~$\G$, and let~$\z:\I\rightarrow\I'$ be \rob{the} matching from~$(\t_-,\b)$ to~$(\t'_-,\b')$. Then there exist self-dual semisimple strata~$[\La,n,r,\tilde{\b}]$ and~$[\La',n,r,\tilde{\b}']$ with splittings~$\V=\oplus_{i\in \I}\V^i$ and~$\V=\oplus_{i\in \I'}\V'^i$ respectively, such that~$\Cc(\La,r,\b)=\Cc(\La',r,\tilde{\b})$ and $\Cc(\La,r,\b')=\Cc(\La',r,\tilde{\b}')$, and such that, for all~$i\in\I$, the characteristic polynomials of~$\tilde{\b}_i$ and~$\tilde{\b}'_{\z(i)}$ coincide. Moreover,~$\tilde{\b}$ and~$\tilde{\b}'$ are conjugate by an element of~$\G$.
\end{corollary}

\shaun{%
Note that the final claim of the theorem is immediate from Corollary~\ref{cor:SkolemNoetherSemisimplev4}.
}

\begin{proof}
\shaun{%
The proof is by induction along~$r$. In the case of trivial characters we can just take~$\tilde{\b}$ and~$\tilde{\b}'$ to be zero. The case~$r=n-1$ is the case of self-dual semisimple strata. By~\cite[Proposition 7.1]{SkSt} we have, for all~$i\in\I_0\cup\I_+$, that~$[\La^i\oplus\La'^{\z(i)},n,r,\b_i\oplus\b'_{\z(i)}]$ is equivalent to a simple stratum~$[\La^i\oplus\La'^{\z(i)},n,r,\tilde{\b}_i\oplus\tilde{\b}'_{\z(i)}]$ split by~$\V^i\oplus\V'^{\z(i)}$; moreover, for~$i\in\I_0$, this can be chosen self-dual by~\cite[Theorem 6.16]{SkSt}. We set~$\tilde{\b}_i:=-\ov{\tilde{\b}_{\s(i)}}$ for~$i\in\I_-$ and~$\tilde{\b}:=\sum_i\tilde{\b}_i$ and we define~$\tilde{\b}'$ analogously. They satisfy the requirements. 
}

\shaun{%
We now come to the induction step. We consider self-dual semisimple strata~$[\La,n,r+1,\g]$ equivalent to~$[\La,n,r+1,\b]$, and~$[\La,n,r+1,\g']$ equivalent to~$[\La,n,r+1,\b']$, such that~$\g\in\prod_{i\in\I}\A^{i}$ and~$\g'\in\prod_{i\in\I'}\A^{i}$. We first show that we can assume that~$\g=\g'$.} 
%\orange{We split the remaining part of the proof into two steps:}

%\orange{Step 1: We reduce to~$\g=\g'$:}
\shaun{
We denote the splittings of~$\g$ and~$\g'$ by~$\V=\oplus_{j\in \J}\V^j$ and~$\V=\oplus_{j\in \J'}\V'^j$, respectively. Now the semisimple characters~$\t|_{\H^{r+2}(\g,\La)}$ and~$\t'|_{\H^{r+2}(\g',\La)}$ intertwine. Then, denoting by~$\xi: \J\rightarrow\J'$ their matching, the induction hypothesis implies that there are self-dual semisimple strata~$[\La,n,r+1,\tilde{\g}]$ and~$[\La',n,r+1,\tilde{\g}']$, with splittings~$\V=\oplus_{j\in\J}\V^j$ and~$\V=\oplus_{j\in\J'}\V'^j$ respectively, such that~$\Cc(\La,r+1,\g)=\Cc(\La,r+1,\tilde{\g})$ and~$\Cc(\La',r+1,\g')=\Cc(\La',r+1,\tilde{\g}')$, and such that~$\g_j$ and~$\g'_{\xi(j)}$ have the same characteristic polynomial. By Corollary~\ref{cor:SkolemNoetherSemisimplev4} there is then an element~$g\in \G$ such that~$ ^g\tilde{\g}=\tilde{\g}'$. 
}
% \red{%
% This is now too quick for me to understand: 
% }
\shaun{We write~$\t_\g,\t'_{\g'}$ and~$\t''_{\tilde{\g}'}$ for the characters $\t|_{\H^{r+2}(\g,\La)},\t'|_{\H^{r+2}(\g',\La')}$ and the transfer $\tau_{g\La,\La',\tilde{\g}'}(\t'|_{\H^{r+2}(\g',\La')})$, respectively. By Theorem~\ref{thm:MatchingForChar}, for~$j\in\J_+$ there exists an~$\F$-\daniel{linear} isomorphism~$\V^j\rightarrow\V'^{\xi(j)}$ which intertwines~$\t_{\g,j}$ and~$\t'_{\g',\xi(j)}$; since~$\t'_{\g'}$ and~$\t''_{\tilde{\g}'}$ are intertwined by the identity (so also~$\t'_{\g',\xi(j)}$ and~$\t''_{\tilde{\g}',\xi(j)}$ are intertwined by the identity), Theorem~\ref{thm:transintertwiningoverGL} implies that there is an isomorphism~$g_j:\V^j\rightarrow\V'^{\xi(j)}$ which intertwines~$\t_{\g,j}$ and~$\t''_{\tilde{\g}',\xi(j)}$. By the same argument, using Theorem~\ref{thm:MatchingForCharForG} and Corollary~\ref{corIntEquivalenceRelation}, for each~$j\in\J_\so$ there is an isometry~$g_j:\V^j\rightarrow\V'^{\xi(j)}$ which intertwines~$\t_{\g,j}$ and~$\t''_{\tilde{\g}',\xi(j)}$. Finally, we put~$g_j=\ov{g_{-j}}^{\,-1}$, for~$j\in\J_-$, and~$g=\sum_{j\in\J} g_j$. Then Proposition~\ref{prop:intdecomp} implies that~$g$ is an element of~$\G$ which intertwines~$\t_\g$ and~$\t''_{\tilde{\g}'}$ with matching~$\xi$.
}

\shauns{
Now Theorem~\ref{thmIntImplConjSelfDual} implies that there is an element of~$\G$ which conjugates~$\t|_{\H^{r+2}(\g,\La)}$ to~$\tau_{g\La,\La',\tilde{\g}'}(\t'|_{\H^{r+2}(\g',\La')})$, and conjugates their splittings. Conjugating by this element, we are reduced to the situation that~$\t|_{\H^{r+2}(\g,\La)}$ and~$\t'|_{\H^{r+2}(\g',\La')}$ intertwine by the identity, there exists an  element~$\tilde{\g}'$ such that the strata~$[\La,n,r+1,\tilde{\g}']$ and~$[\La',n,r+1,\tilde{\g}']$ are self-dual semisimple with~$\Cc(\La,r,\g)=\Cc(\La,r,\tilde{\g}')$ and~$\Cc(\La',r,\g')=\Cc(\La',r,\tilde{\g}')$, and the splittings for~$\g$,~$\g'$ and~$\tilde{\g}'$ coincide. Finally, applying the translation principle Theorem~\ref{thmTranslationPrincipleG} and Proposition~\ref{prop:ConjIdempToEachOther}\ref{prop:ConjIdempToEachOther-ii}, as in the proof of Theorem~\ref{thm:MatchingForCharForG}, we reduce to the case that~$\g=\g'=\tilde{\g}'$ and~$\xi$ is the identity map.
% to obtain self-dual semisimple strata~$[\La,n,r,\b_1]$ and~$[\la',n,r,\b'_1]$
% such that
% \begin{itemize}
%  \item $\Cc(\La,r,\b_1)=\Cc(\La,r,\b)$ and~$\Cc(\La',r,\b'_1)=\Cc(\La',r,\b')$,
%  \item $[\La,n,r+1,\tilde{\g}]\sim [\La,n,r+1,\b_1]$ and~$[\La',n,r+1,\tilde{\g}]\sim [\La',n,r+1,\b_1]$,
%  \item there exists an element~$u\in\P_-^1(\La)$ which normalises the equivalence class of~$[\La,n,r+1,\tilde{\g}]$ such that 
%  $u\tilde{\g}u^{-1}\in\prod_{i\in\I_1}\A^{ii}$,
%  \item there exists an element~$u'\in\P_-^1(\La')$ which normalises the equivalence class of~$[\La',n,r+1,\tilde{\g}']$ such that 
%  $u'\tilde{\g}'u'^{-1}\in\prod_{i\in\I'_1}\A^{ii}$.
% \end{itemize}
% By Proposition~\ref{prop:ConjIdempToEachOtherG} we can replace~$\b$ by~$u^{-1}\b_1u$. Thus we can assume without loss of 
% generality,~$\g=\tilde{\g}$ and~$\tilde{\g}'=\g'$. The element~$\tilde{\g}$ is conjugate to~$\tilde{\g}'$ by an element of~$\G$, and we therefore can assume without loss of generality~$\g=\tilde{\g}=\tilde{\g}'=\g'$,~$\J=\J'$ and~$\xi=\id_{\J}$.
}

\shauns{
Now we take~$\t'_0\in\Cc(\La',r,\g)$ such that~$\t'=\t'_0\psi_{\b'-\g}$. If we set~$\t_0=\tau_{\La,\La',\g}(\t'_0)$ then the restrictions of~$\t$ and~$\t_0$ coincide on~$\H^{r+2}(\g,\La)$) so there exists~$c\in (\prod_i \A^{i})_-\cap\aa_{-r-1}$ such that~$\t=\t_0\psi_{\b-\g+c}$. 
}%
\shauns{By Corollary~\ref{cor:Intertwining} the characters~$\t_j$ and~$\t'_j$ intertwine (and by an isometry if~$j\in\J_0$). Writing~$s_{\g_j}$ for an equivariant tame corestriction with respect to~$\g_j$, \rob{Lemmas}~\ref{lem:DerivedCharacters} and~\ref{lem:DerivedCharactersForG} imply that the strata~$[\La^j,r+1,r,s_{\g_j}(\b_j-\g_j+c_j)]$ and~$[\La'^j,r+1,r,s_{\g_j}(\b'_j-\g_j)]$ intertwine (by an element of~$\G_{\g_j}$ if~$j\in\J_0$). Thus~\cite[Proposition~7.6]{SkSt} implies that~$[\La^j,n,r,\b_j+c_j]$ and~$[\La'^j,n,r,\b'_j]$ intertwine (by an element of%~$\G_{\g_j}$
~\bob{$\G\cap \Aut_\F(\V^j)$} if~$j\in\J_0$) and moreover that~$[\La,n,r,\b+c]$ and~$[\La',n,r,\b']$ intertwine by an element of~$\G$. Furthermore~$[\La,n,r,\b+c]$ is equivalent to a semisimple stratum whose splitting is a coarsening of the splitting of~$\b$ by~\cite[Theorem~6.15]{SkSt}. In fact the splitting cannot be a proper coarsening by~\cite[Proposition 7.1]{SkSt}, because~$[\La,n,r,\b+c]$ intertwines with~$[\La,n,r,\b']$ which has the same number of blocks as~$\b$. Now we proceed as in the base case to obtain skew elements~$\tilde{\b}$ and~$\tilde{\b}'$ such that~$[\La,n,r,\b+c]$ is equivalent to~$[\La,n,r,\tilde{\b}]$ and~$[\La',n,r,\tilde{\b}']$ is equivalent to~$[\La',n,r,\b']$. We further have
\[
\t\in\Cc(\La,r,\b) \quad\text{and}\quad \t\in\Cc(\La,r,\g)\psi_{\b-\g+c}=\Cc(\La,r,\b)\psi_c=\Cc(\La,r,\tilde{\b})
\]
so that~$\Cc(\La,r,\tilde{\b})=\Cc(\La,r,\b)$, and also~$\Cc(\La',r,\tilde{\b}')=\Cc(\La',r,\b')$, as required.
}
\end{proof}

%% file: Endo-pss.tex
%%%%%%%%%%%%%%%%%%%%%%%%%%%%%%%%%%%%

%%%%%%%%%%%%%%%%%%%%%%%%%%%%%%%%%%%%
\section{Self-dual semisimple endo-classes}\label{secPSS}
%%%%%%%%%%%%%%%%%%%%%%%%%%%%%%%%%%%%
In this section, we introduce one of the central concepts of the article, \emph{self-dual semisimple endo-equivalence}. We generalize the previous notions of (self-dual) simple endo-equivalence to the semisimple setting and, via the Matching Theorem, reduce the fundamental properties of (self-dual) semisimple endo-equivalence to the (self-dual) simple setting we treated in Section~\ref{secPS}.

%%%%%%%%%%%%%%%%%%%%%%%%%%%%%%%%%%%%
\subsection{Self-dual semisimple pairs}
%%%%%%%%%%%%%%%%%%%%%%%%%%%%%%%%%%%%
% Let~$[\La,n,0,\b]$ be a semisimple stratum with positive~$n$.  We let
% \[
% k_0(\b,\La)=-\min\{r\in\ZZ:[\La,n,r,\b]\text{ is not semisimple}\}
% \] 
% denote the \emph{critical exponent} of~$[\La,n,0,\b]$ 
% and~$k_F(\b):=\frac{1}{e(\La)}k_0(\b,\La)$; by~\cite[\S3.1]{St05}, this
% is independent of~$\La$. For the zero-stratum we keep~$k(0,\La)=k_F(0)=-\infty$.

\begin{definition}\label{def:sesipair}
A \emph{semisimple pair} is a pair~$(k,\b)$ 
\shaun{consisting of an element~$\b$ of a finite-dimensional semisimple commutative~$\F$-algebra and an integer~$k$ such that, writing~$\E=\F[\b]=\bigoplus_{i\in \I} \E_i$ as a sum of fields, we have} 
%where~$\E=\F[\b]$ is a %semisimple~$\F$-algebra, i.e.~
%sum of %pairwise non-isomorphic 
%fields~$\E=\bigoplus_{i\in \I} \E_i$, and~$k$ is an integer satisfying  %{\color{red} want non-isomorphic fields?}
\[
0\leqslant k/e_\E <-k_\F(\b),
\]
where~$e_\E=\lcm_{i\in \I}e_i$ is the lowest common multiple of the ramification indices~$e_i=e(\E_i/\F)$. 
\shaun{(See after \rob{Lemma~\ref{lemmasufflargeinterstratasesi}} for the definition of~$k_\F(\b)$ in this generality.)} We say that~$\I$ is the \emph{index set} of~$(k,\b)$. Writing~$\b=\sum_{i\in\I}\b_i$ for the decomposition of~$\b$ in~$\E=\bigoplus_{i\in \I} \E_i$ and setting~$k_i=\left\lfloor\frac{ke_i}{\shauny{e_\E}}\right\rfloor$, each~$(k_i,\b_i)$ is a simple pair\daniel{, because~$k_\F(\b_i)\leq k_\F(\b)$,} and we call these the \emph{component simple pairs of~$(k,\b)$}.  \shauny{More generally, if~$\J$ is a non-empty subset of~$\I$ and we set~$\b_\J=\sum_{j\in\J}\b_j$ and~$k_\J=\left\lfloor \frac{ke_\J}{e_\E}\right\rfloor$, where $e_\J=\lcm_{j\in \J} e_j$, then~$(k_\J,\b_\J)$ is a semisimple pair.}

A semisimple pair~$(k,\b)$ is called \emph{self-dual} if there exists an extension of \shaun{the \rob{Galois} involution~$x\mapsto\ov x$ on~$\F$ to an involution on~$\E=\F[\b]$ such that~$\ov\b=-\b$}. %\gre{We have to be cautious, because~$\s$ has another meaning here.}\red{good point, although rather than~$\bar{(\ )}^\E$, maybe we can abuse notation like in the Witt section and still call the extension~$\bar{\ }$ without the~$\E$?  Looks much neater notationally.}
\end{definition}

Let~$(k,\b)$ be a self-dual semisimple pair, and write the minimal polynomial of~$\b$ as~$\Psi(X)=\prod_{i\in \I}\Psi_i(X)$ with~$\Psi_i(X)$ irreducible, so that~$\E_i\simeq \F[X]/(\Psi_i(X))$. The action of~$x\mapsto\ov x$ on the primitive idempotents of~$\E$ defines an action of~$\s$ on~$\I$. We let~$\I_0=\{i\in \I:\s(i)=i\}$, and choose a set of representatives~$\I^+$ for the orbits of~$\s$ in~$\I\backslash \I_0$.  Then we let~$\I_{-}=\s(\I_+)$ so that we have a disjoint union~$\I=\I_+\cup \I_0\cup \I_-$.  

\begin{definition}
A self-dual semisimple pair~$(k,\b)$ is called \emph{skew} if~$\I=\I_0$ in the notation above.
\end{definition}

Let~$(k,\b)$ be a semisimple pair, and let~$\Qq(k,\b)$ denote the class of quadruples $(\V,\vphi,\La,r)$ consisting of: 
\begin{enumerate}\setlength\itemsep{5pt}
\item a finite dimensional~$\F$-vector space~$\V$; 
\item an embedding~$\vphi:\E\hookrightarrow \A$, where~$\A=\End_\F(\V)$; 
\item an~$\o_{\vphi(\E)}$-lattice sequence~$\La$ in~$\V$; 
\item and an integer~$r$ such that~$\left\lfloor r/e(\La_\E)\right\rfloor=k$, where we recall that~$e(\La_\E)=e(\La|\o_\F)/e_\E$ is the greatest common divisor of the~$e(\La^i|\o_{\vphi(\E_i)})$. 
\end{enumerate}
Given~$(\V,\vphi,\La,r)\in\Qq(k,\b)$, setting~$n=\max\{r,-\val_{\La}(\vphi(\b))\}$ we obtain a semisimple stratum~$[\La,n,r,\vphi(\b)]$ with splitting~$\V=\bigoplus_{i\in \I} \V^i$, where~$\V^i=\ker(\Psi_i(\vphi(\b)))$, which we call a \emph{realization} of the semisimple pair~$(k,\b)$. \orange{Note also that, since~$\vphi$ is an embedding, the spaces~$\V^i$ are all non-zero.}

We let~$\fCc(k,\b)$ denote the class of all semisimple characters defined by a realization of the semisimple pair~$(k,\b)$:
\[
\fCc(k,\b)=\bigcup_{(\V,\vphi,\La,r)\in\Qq(k,\b)}\Cc(\La,r,\vphi(\b)).
\]

Now let~$(k,\b)$ be a self-dual semisimple pair. Let~$(\V,h)$ be a finite-dimensional~$\e$-hermitian space over~$\F/\F_\so$ and~$\A=\End_\F(\V)$. We say that an embedding~$\vphi:\E\hookrightarrow \A$ is \emph{self-dual} if~$\varphi(\ov{x})=\ov{\varphi(x)}$, for all~$x\in\E$. Let~$\Qq_-(k,\b)$ denote the class of quadruples~$((\V,h),\vphi,\La,r)$ where
\begin{enumerate}\setlength\itemsep{5pt}
\item $(\V,h)$ is a finite-dimensional~$\e$-hermitian space over~$\F/\F_\so$; 
\item $(\V,\vphi,\La,r)\in\Qq(k,\b)$;
\item and~$\vphi$ and~$\La$ are self-dual.
\end{enumerate}
Given such a quadruple~$((\V,h),\vphi,\La,r)\in\Qq_-(k,\b)$, the realization~$[\La,n,r,\vphi(\b)]$ is a self-dual semisimple stratum, which we call a \emph{self-dual realization} of~$(k,\b)$.
%setting~$n=\max\{r,-\val_{\La}(\vphi(\b))\}$, we obtain a self-dual semisimple stratum~$[\La,n,r,\vphi(\b)]$, with splitting~$\V=\bigoplus_{i\in \I} \V^i$ where~$\V^i=\ker(\Psi_i(\b))$, which we call a \emph{self-dual realization} of~$(k,\b)$. 
%%%
%\red{why is the definition of~$n$ here given differently to the non-self-dual case?... setting~$n=$....}
%\gre{It is not different.}
%%%

We let~$\fCc_-(k,\b)$ denote the class of all self-dual semisimple characters defined by a realization of the self-dual semisimple pair~$(k,\b)$:
\[
\fCc_-(k,\b)=\bigcup_{((\V,h),\vphi,\La,r)\in\Qq_-(k,\b)}\Cc_-(\La,r,\vphi(\b)).
\]

%%%%%%%%%%%%%%%%%%%%%%%%%%%%%%%%%%%%
\subsection{Transfer of semisimple characters}\label{sect:Transferofsesichars}
%%%%%%%%%%%%%%%%%%%%%%%%%%%%%%%%%%%%
%%%
%\red{Shaun suggested a BK compositio semisimple types reference I didn't have.}
%%%
Let~$(k,\b)$ be a semisimple pair with index set~$\I$, and let~$(\V,\vphi,\La,r),(\V',\vphi',\La',r')\in\Qq(k,\b)$. % with~$e(\La_\E)=e(\La'_\E)$.
%~$[\La,n,r,\vphi(\b)]$ and~$[\La',n',r',\vphi'(\b)]$ be realizations of~$(k,\b)$ with~$e(\La|\o_\E)=e(\La'|\o_\E)$.  
Let~$\t\in\Cc(\La,r,\vphi(\b))$ and let~$\tM$ denote the Levi subgroup of~$\tG=\Aut_\F(\V)$ associated to the decomposition~$\V=\bigoplus_{i\in\I}\V^i$. Then
\[
\t\mid_{\H^{r+1}(\vphi(\b),\La)\cap\tM}=\bigotimes_{i\in\I}\t_i,
\]
with~$\t_i\in\Cc(\La_i,r,\vphi(\b_i))$ simple characters. Put~$\t'_i=\tau_{\La'^i,\La^i,\b}(\t_i)$ and also write~$\tM'$ for the Levi subgroup of~$\tG'=\Aut_\F(\V')$ which is the stabilizer of~$\V'=\bigoplus_{i\in\I}\V'^i$.

\begin{lemma}\label{lemTransfer}
There is a unique semisimple character~$\t'\in\Cc(\La',r',\vphi'(\b))$ satisfying
\[
\t'\mid_{\H^{r'+1}(\vphi'(\b),\La')\cap\tM'}=\bigotimes_{i\in\I}\t'_i.
\]
\end{lemma}

Writing~$\tau_{\La',\La,\b}(\t)$ for the character given by the lemma, this then defines a bijection
\[
\tau_{\La',\La,\b}:\Cc(\La,r,\vphi(\b))\to \Cc(\La',r',\vphi'(\b))
\] 
which we call \emph{transfer}.

\begin{proof}
\shaun{%
If~$\tP'$ is any parabolic subgroup of~$\tG'$ with Levi factor~$\tM'$ and unipotent radical~$\tU'$ then~$\H^{r'+1}(\vphi'(\b),\La')$ has an Iwahori decomposition with respect to~$(\tM',\tP')$ and the restriction to~$\H^{r'+1}(\vphi'(\b),\La')\cap\tU'$ of any semisimple character in~$\Cc(\La',r',\vphi'(\b))$ is trivial by~\cite[Lemma 3.15]{St05}. Uniqueness is then immediate so it only remains to prove existence.
}

%%%
%\red{Should follow similar to~\cite[Lemma~3.15]{St05}.}
%\red{Needs a proper proof, with detail, not just a vague sketch.}
%%%

\shaun{
Passing to an affine translation of~$[\La,n,r,\varphi(\b)]$, we can assume \rob{that}~$e(\La|\o_\F)>\dim_\F\V'$. Then, by the~$\dag$-construction, we obtain~$\t^\dag\in\Cc(\La^\dag,r,\varphi(\b)^\dag)$ (see~\cite[Lemma 3.3]{RKSS}), where~$\V^\dag=\bigoplus_{i\in\I}(\V^i)^\dag$ and~$\t^\dag\mid_{\H^{r+1}(\b_i^\dag,(\La^i)^\dag)}=\t_i^\dag$. Moreover,
\[
\dim_\F(\V^i)^\dag\geq e(\La|\o_\F)\geq\dim_\F\V'\geq\dim_\F\V'^i.
\]
Hence, as~$\t_i^\dag$ is the transfer of~$\t_i$ to~$(\La^i)^\dag$ and the transfer for simple characters is transitive, by replacing~$\t$ with~$\t^\dag$ we can assume~$\dim_\F\V^i\geq\dim_\F\V'^i$ for each~$i\in\I$. Then there are an~$\E_i$-linear monomorphism~$g_i:\V'^i\rightarrow\V^i$ and an~$\E_i$-subspace~$\W^i$ of~$\V^i$ such that~$\W^i\oplus g_i\V'^i=\V^i$ splits~$\Lambda^i$. Write~$g$ for the direct sum of the~$g_i$ so that we have
\[
\varphi(x)\mid_{g\V'}=g\varphi'(x)g^{-1},\qquad\text{ for }x\in\E.
\]
Let~$\tilde{\t}$ be the transfer of~$\t\mid_{\H^{r+1}(\varphi(\b)|_{g\V'},\La\cap g\V')}$ from~$\La\cap g\V'$ to~$g\La'$,  see~\cite[Proposition 3.26]{St05}. 
Then~$ ^{g^{-1}}\tilde{\t}$ satisfies the desired properties. 
}
% \blue{By a~$\dag$ construction as in~\cite[Lemma 3.3]{RKSS} we can assume without loss of generality that there are~$\E_i$-vector spaces~$\W^i$,\ $i\in\I$, such that~$\W^i\oplus\V'^i=\V^i$ as an~$\E_i$-splitting for~$\La^i$, in particular the~$\E_i$-structure of~$\V'^i$ is the restriction of the~$\E_i$-structure on~$\V^i$, i.e.~$\vphi(\b)|_{\End_\F(\V')}=\vphi'(\b)$. Then we define~$\t'$ as the transfer of~$\t|_{\H^{r+1}(\vphi(\b),\La)\cap\Aut_\F(\V')}$ from~$\La\cap \V'$ to~$\La'$ with respect to~$\vphi'(\b)$, see~\cite[Proposition 3.26]{St05}.
% }
%We now show uniqueness. Choose any parabolic subgroup~$\P$ of~$\G$ with Levi factor~$\M$, unipotent radical~$\U$ and opposite unipotent~$\U^{-}$, then~$\H^{r'+1}(\vphi'(\b),\La')$ has an Iwahori decomposition
%\[
%\H^{r'+1}(\vphi'(\b),\La')=(\H^{r'+1}(\vphi'(\b),\La')\cap\U)(\H^{r'+1}(\vphi'(\b),\La')\cap\M)(\H^{r'+1}(\vphi'(\b),\La')\cap\U^{-})
%\]
%with respect to~$\P$ by~\cite[Proposition 5.2]{St08}, and when restricted to the unipotent factors of this decomposition~$\t'$ is trivial by~\cite[Lemma 3.15]{St05}.
\end{proof}

Now suppose~$(k,\b)$ is a self-dual semisimple pair and~$[\La,n,r,\vphi(\b)]$ and~$[\La',n',r',\vphi'(\b)]$ are self-dual realizations. Then, as in \cite[Proposition 3.32]{St05}, the bijection~$\tau_{\La',\La,\b}$ commutes with the restrictions of the actions of the adjoint anti-involutions of~$h$ and~$h'$ on~$\Cc(\La,r,\vphi(\b))$ and~$\Cc(\La',r',\vphi'(\b))$, respectively. Thus it restricts to give a bijection 
\[
\tau_{\La',\La,\b}:\Cc_{-}(\La,r,\vphi(\b))\to \Cc_{-}(\La',r',\vphi'(\b)).
\]
Thanks to this result and with the definition of semisimple pairs, we can now define (self-dual) potential semisimple characters. 

%%%%%%%%%%%%%%%%%%%%%%%%%%%%%%%%%%%%
\subsection{Self-dual pss-characters}
%%%%%%%%%%%%%%%%%%%%%%%%%%%%%%%%%%%%
Let~$(k,\b)$ be a semisimple pair.  
\begin{definition}\label{defPseudosemisimpleCharacter}
A \emph{potential semisimple character}, or \emph{pss-character}, \emph{supported on~$(k,\b)$} is a function~$\Th:\Qq(k,\b)\to \fCc(k,\b)$ such that
\begin{enumerate}\setlength\itemsep{5pt}
\item for all~$(\V,\vphi,\La,r)\in\Qq(k,\b)$, we have $\Th(\V,\vphi,\La,r)\in\Cc(\La,r,\vphi(\b))$; 
\item and, for all pairs~$(\V,\vphi,\La,r),(\V',\vphi',\La',r')\in\Qq(k,\b)$, we have
\[
\Th(\V',\vphi',\La',r')=\tau_{\La',\La,\b}(\Th(\V,\vphi,\La,r)).
\] 
\end{enumerate}
\end{definition}
We call the values of a pss-character its \emph{realizations}; by definition a pss-character is determined by any one of its realizations. 
\shaun{
We also define the \emph{degree} of a pss-character~$\Th$ supported on~$(k,\b)$ to be~$\deg(\Th)=[\F[\b]:\F]$.
}
%%% I wonder if this is really strong enough. Should instead $\deg(\Th)$ be the multiset~$\{\[\F[\b_i]:\F]\}\}$?

\begin{definition}
Let~$(k,\b)$ be a self-dual semisimple pair. 
\begin{enumerate}\setlength\itemsep{5pt}
\item \orange{A pss-character~$\Th$ supported on~$(k,\b)$ is called \emph{$\s$-invariant} if, for any (or equivalently, some)~$((\V,h),\vphi,\La,r)\in\Qq_{-}(k,\b)$ the realization~$\Th(\V,\vphi,\La,r)$ is~$\s$-invariant.}
%for any (or equivalently one) self-dual realization of~$(k,\b)$ the value is~$\s$-invariant.
\item A \emph{self-dual pss-character supported on~$(k,\b)$} is a function~$\Th_-:\Qq_-(k,\b)\to\bob{\fCc_-(k,\b)}$ such that, for all~$((\V,h),\vphi,\La,r),((\V',h'),\vphi',\La',r')\in\Qq_{-}(k,\b)$,
\begin{align*}
&\Th_-((\V,h),\vphi,\La,r)\in \Cc_{-}(\La,r,\vphi(\b));\\
&\Th_-((\V',h'),\vphi',\La',r')=\tau_{\La',\La,\b}(\Th_-((\V,h),\vphi,\La,r)).
\end{align*}
%\begin{enumerate}\setlength\itemsep{5pt}
%\item for all~$((\V,h),\vphi,\La,r)\in\Qq_-(k,\b)$, we have $\Th_-((\V,h),\vphi,\La,r)\in\Cc_-(\La,r,\vphi(\b))$;
%\item and, for all pairs~$((\V,h),\vphi,\La,r)$,~$((\V',h'),\vphi',\La',r')\in\Qq_-(k,\b)$, we have
%\[
%\Th_-((\V',h'),\vphi',\La',r')=\tau_{\La,\La',\b}(\Th_-((\V,h),\vphi,\La,r)).
%\]
%\end{enumerate}
\end{enumerate}
\end{definition}

We call the values of a self-dual pss-character its \emph{self-dual realizations}; by definition a self-dual pss-character is determined by any one of its realizations. As in the simple setting, by the Glauberman correspondence, every self-dual pss-character comes uniquely from the restriction of a~$\s$-invariant pss-character. 
\shaun{%
More precisely, for a self-dual pss-character~$\Th_-$ supported on~$(k,\b)$, there is a unique~$\sigma$-invariant pss-character~$\Th$, supported on~$(k,\b)$, such that the }\bob{following diagram commutes:}
%\[
%\xymatrix{\Qq_-(k,\b)\ar[r]^{\Th_-}\ar[d]&\fCc_-(k,\b)\ar[d] \\
%\Qq(k,\b)\ar[r]^\Th&\fCc(k,\b)
%}
%\]
%commutes. 
%
\[
\begin{tikzcd}[column sep=3pc]
 \bob{ \Qq_-(k,\b)}\arrow[d] \arrow[dr,"\mathrm{Gl}\ \circ\ \Th_-"]& {} \\
  %B \arrow{rru}{} \arrow{r}{} &
  \bob{\Qq(k,\b)} \arrow[r,"\Th"] &\bob{\fCc(k,\b)}
\end{tikzcd}
\]
where the vertical arrow is the forgetful map~$((\V,h),\vphi,\La,r)\mapsto(\V,\vphi,\La,r)$ and\bob{~$\mathrm{Gl}\ \circ\ \Th_-((\V,h),\vphi,\La,r)$ is the Glauberman lift in~$\Cc(\La,r,\varphi(\b))$ of~$\Th_-((\V,h),\vphi,\La,r)\in\Cc_-(\La,r,\varphi(\b))$.}  
%where the left vertical arrow is the forgetful map~$((\V,h),\vphi,\La,r)\mapsto(\V,\vphi,\La,r)$ and the right vertical arrow is the Glauberman lifting of self-dual simple characters.
We call~$\Th$ the \emph{lift} of~$\Th_-$. We also define the \emph{degree} of~$\Th_-$ to be the degree of its lift, so~$\deg(\Th_-)=[\F[\b]:\F]$.

%%%
%\gre{I have to change the next paragraph, because it was to redundant. Maybe the first part could be in the semisimple section} 
%\red{Yes}
%%%
We now see how a pss-character determines a set of ps-characters. Suppose we are given a semisimple character~$\t\in\Cc_-(\La,r,\b)$, whose index set splits as a disjoint union~$\I=\I_+\cup \I_0\cup \I_-$ as above; then we have 
\[
\V=\bigoplus_{i\in \I_+}(\V^i\oplus \V^{\s(i)})\oplus\bigoplus_{i\in \I_0} \V^i.
\]
Moreover, writing~$\M$ for the subgroup of~$\G$ stabilizing this decomposition, 
\shauny{which is the intersection with~$\G$ of the Levi subgroup~$\tM$ of~$\tG$ stabilizing the decomposition}, we have
%as in \cite[\S3.11]{MiSt} we have 
\[
\H_-^{r+1}(\b,\La)\cap \M\simeq \prod_{i\in \I_+}\H^{r+1}(\b_i,\La^i)\times\prod_{i\in \I_0} \H_-^{r+1}(\b_i,\La^i),
\]
\shauny{by applying~\cite[Proposition~5.4]{St08} to~$\H^{r+1}(\b,\La)\cap \tM$ and intersecting with~$\G$.} 
Then, after identifying $\H^{r+1}_-(\b,\La)\cap \M$ with the decomposition above we have
\[
\t|_{\H^{r+1}_-(\b,\La)\cap M}=\bigotimes_{i\in\I_+}\t_i^2\otimes\bigotimes_{i\in\I_0}\t_{i,-}
\]
with~$\t_i\in\Cc(\La^i,r,\b_i),\ i\in\I_+$ and~$\t_{i,-}\in\Cc_-(\La^i,r,\b_i),\ i\in\I_0$\shauny{, by applying~\cite[Proposition~5.5]{St08}}. Moreover, as in~\cite[4.3~Lemma~1]{Blondel}, for~$i\in \I_+$ we have
\[
\H^{r+1}(\b_i,\La^i)=\H^{r+1}(2\b_i,\La^i)\quad\text{and}\quad \t_i^2\in \Cc(\La^i,r,2\b_i).
\]
The above decomposition generalizes to pss-characters as follows:

\begin{lemma}%
%%%%
%[{\red{Reference or proof}}]%%%%
%%%%
\label{lemma:pssdecompositions}
Let~$(k,\b)$ be a semisimple pair with index set~$\I$ and component simple pairs~$(k_i,\b_i)$. %~$\b=\bigoplus_{i\in \I}\b_i$.
\begin{enumerate}\setlength\itemsep{5pt}
\item\label{lemma:pssdecompositions-i} Let~$\Th$ be a pss-character supported on~$(k,\b)$ and let~$(\V,\vphi,\La,r)\in\Qq(k,\b)$. % and write~$k_i=\lfloor\frac{r}{e(\La^i|\o_{\E_i})}\rfloor,\ i\in\I$. Then there is for every
%For each~$i\in\I$ there is a unique ps-character~$\Th_i$ supported on~$(k_i,\b_i)$ such that 
%\begin{equation}\label{eqThetai}
%\Th(\V,\vphi,\La,r)\mid_{\H^{r+1}(\vphi(\b_i),\La^i)}=\Th_i(\V^i,\vphi|_{\E_i},\La^i,r).
%\end{equation}
%Moreover, the ps-character~$\Th_i$ does not depend on the choice of~$(\V,\vphi,\La,r)\in\Qq(k,\b)$.
\orange{For each~$\J\subseteq\I$ there is a unique pss-character~$\Th_\J$ supported on~$(k_\J,\b_\J)$ such that 
\begin{equation}\label{eqThetai}
\Th(\V,\vphi,\La,r)\mid_{\H^{r+1}(\vphi(\b_\J),\La^\J)}=\Th_\J(\V^\J,\vphi|_{\E_\J},\La^\J,r).
\end{equation}
Moreover, the pss-character~$\Th_\J$ does not depend on the choice of~$(\V,\vphi,\La,r)\in\Qq(k,\b)$.}
\item\label{lemma:pssdecompositions-iii} \shaun{If~$\Th$ is a pss-character supported on~$(k,\b)$ and, for~$i\in\I$, we write~$\Th_i$ for the ps-character given by~\ref{lemma:pssdecompositions-i}, then the~$\Th_i$ are pairwise endo-inequivalent.}
\item\label{lemma:pssdecompositions-ii} 
\shaun{Suppose that~$(k,\b)$ is self-dual. Let~$\Th_-$ be a self-dual pss-character supported on~$(k,\b)$ with lift~$\Th$. 
\orange{For~$\J\subseteq\I$, write~$\Th_\J$ for the pss-character given by~\ref{lemma:pssdecompositions-i}.} 
%For~$i\in\I$, write~$\Th_i$ for the ps-character given by~\ref{lemma:pssdecompositions-i}. 
Let~$((\V,h),\vphi,\La,r)\in\Qq_-(k,\b)$. % and write~$k_i=\lfloor\frac{r}{e(\La^i|\o_{\E_i})}\rfloor,\ i\in\I$. 
\begin{enumerate}\setlength\itemsep{5pt}
\item\label{lemma:pssdecompositions-iia}  
%For all~$i\in\I_0$ there is a unique self-dual ps-character~$\Th_{i,-}$ supported on~$(k_i,\b_i)$ such that
%\[
%\Th_-((\V,h),\vphi,\La,r)\mid_{\H^{r+1}_-(\vphi(\b_i),\La^i)}=\Th_{i,-}((\V^i,h|_{\V^i}),\vphi|_{\E_i},\La^i,r).
%\]
%Moreover,~$\Th_i$ is the lift of~$\Th_{i,-}$.
\orange{For all~$\s$-stable~$\J\subseteq\I$, there is a unique self-dual pss-character~$\Th_{\J,-}$ supported on~$(k_\J,\b_\J)$ such that
\[
\Th_-((\V,h),\vphi,\La,r)\mid_{\H^{r+1}_-(\vphi(\b_\J),\La^\J)}=\Th_{\J,-}((\V^\J,h|_{\V^\J}),\vphi|_{\E_\J},\La^\J,r).
\]
Moreover,~$\Th_\J$ is the lift of~$\Th_{\J,-}$.}
%Further~$\Th_{i,-}$, for~$i\in\I_0$, do not depend on the choice of $((\V,h),\vphi,\La,r)\in\Qq_-(k,\b)$. 
\item\label{lemma:pssdecompositions-iib}  For all~$i\in\I_+$, the map~$\Th_i^2$ is the unique ps-character supported on~$(k_i,2\b_i)$ such that %, writing~$i:\H^{r+1}(\vphi(\b_i),\La^i)\hookrightarrow\H^{r+1}_-(\vphi(\b),\La)$ for the embedding, we have
\[
\Th_-((\V,h),\vphi,\La,r)\mid_{\H^{r+1}(\vphi(\b_i),\La^i)}=\Th^2_i(V^i,\vphi|_{\E_i},\La^i,r).
%\Th_-((\V,h),\vphi,\La,r)\circ i=\Th_i(V^i,\vphi|_{\E_i},\La^i,r)^2.
\] 
%Furthermore,~$\Th_i^2$ is a ps-character supported on~$(k_i,2\b_i)$.
\end{enumerate}
}
\end{enumerate}
\end{lemma}

%%%%%%
%\shaun{Note that, for~$\Th$ a ps-character supported on~$(k,\b)$, by~$\Th^2$ we mean the unique ps-character supported on~$(k,2\b)$ such that~$\Th^2(V,\vphi,\La,r)(y)=\Th(V,\vphi,\La,r)(y)^2$ for~$y\in\H^1(\vphi(\b),\La)$, for all~$(\V,\vphi,\La,r)\in\Qq(k,\b)$.}
%%%%%%

\begin{proof}
\shaun{Note first that all uniqueness statements follow from the compatibility with transfer. So we concentrate on the remaining parts of the assertions. \orange{For~\ref{lemma:pssdecompositions-i} we define~$\Th_\J$ by~\eqref{eqThetai} and transfer; it does not depend on the choice of~$(\V,\vphi,\La,r)$ by Lemma~\ref{lemTransfer}. The same strategy works for the construction of~$\Th_{\J.-}$ in~\ref{lemma:pssdecompositions-iia}; but then the values of~$\Th_{\J.-}$ are the restrictions of the corresponding values of~$\Th_\J$, which proves both that the definition is independent of the choice of~$((\V,h),\vphi,\La,r)$ and that~$\Th_\J$ is the lift of~$\Th_{\J,-}$.} Finally, for~$i\in\I_+$, once one has checked that~$\Th_i^2$ is a ps-character supported on~$(k_i,2\b_i)$, the same argument proves~\ref{lemma:pssdecompositions-iib}.
}

\shaun{
It remains to prove~\ref{lemma:pssdecompositions-iii}. Suppose for contradiction that there are distinct~$i,j\in\I$ such that~$\Th_i\approx\Th_j$ and let~$\z:\I\to\I$ be the transposition which exchanges~$i$ and~$j$. Let~$(\V,\vphi,\La,r)\in\Qq(k,\b)$ be such that~$\dim_\F\V^i=\dim_\F\V^j$ and consider~$\t=\Th(\V,\vphi,\La,r)$, which has simple block restrictions~$\t_i,\t_j$. If~$g:\V^i\to\V^j$ is any isomorphism then, since~$\Th_i,\Th_j$ are endo-equivalent, the simple characters~$\presuper{g}\t_i$ and~$\t_j$ intertwine in~$\Aut_\F(\V^j)$; replacing~$g$ if necessary, we can assume that~$\presuper{g}\t_i$ and~$\t_j$ are intertwined by the identity. In particular, we see that~$\I(\t_i,\t_{\z(i)})\ne\emptyset$, in the notation of Proposition~\ref{prop:intdecomp}. By symmetry we also have~$\I(\t_j,\t_{\z(j)})\ne\emptyset$. Since the identity lies in~$\I(\t_l,\t_{\z(l)})$, for any~$l\ne i,j$, Proposition~\ref{prop:intdecomp} implies that~$\t$ is intertwined with itself by an element of~$\Aut_\F(\V)$ with matching~$\z$. However, this contradicts the uniqueness of matchings in Theorem~\ref{thm:MatchingForChar}, since the identity certainly intertwines~$\t$ with itself, with matching the identity.
}
\end{proof}

This lemma shows that we can identify the index set of a semisimple pair~$(k,\b)$ with the index set of any realization of any pss-character supported on~$(k,\b)$. Given a pss-character~$\Th$ supported on~$(k,\b)$, we call the ps-characters~$\Th_i$ supported on~$(k_i,\b_i)$ given by the lemma the \emph{component ps-characters of~$\Th$}.

%%%%%%%%%%%%%%%%%%%%%%%%%%%%%%%%%%%%
\subsection{Semisimple endo-classes}
%%%%%%%%%%%%%%%%%%%%%%%%%%%%%%%%%%%%
Let~$\Th$ be a pss-character supported on the semisimple pair~$(k,\b)$, and let~$\Th'$ be a pss-character supported on the semisimple pair~$(k',\b')$.

\begin{definition}
We say that~$\Th$ and~$\Th'$ are \emph{endo-equivalent}, denoted~$\Th \approx\Th'$, if
\begin{enumerate}\setlength\itemsep{5pt}
\item $\deg(\Th)=\deg(\Th')$; 
\item $k=k'$;
\item there exist realizations on a common~$\F$-vector space which intertwine, i.e.\ there exist a finite-dimensional~$\F$-vector space~$\V$ and quadruples~$(\V,\vphi,\La,r)\in\Qq(k,\b)$ and~$(\V,\vphi',\La',r')\in\Qq(k',\b')$ such that~$\Th(\V,\vphi,\La,r)$ and~$\Th'(\V,\vphi',\La',r')$ intertwine in~$\tG=\Aut_{\F}(\V)$.
\end{enumerate}
\end{definition}

\begin{theorem}\label{thm:Endopss}
Let~$\Th$ and~$\Th'$ be pss-characters supported on semisimple pairs~$(k,\b)$ and~$(k,\b')$ respectively. 
\begin{enumerate}\setlength\itemsep{5pt}
\item\label{thm:Endopss-i} We have~$\Th\approx\Th'$ if and only if there is a bijection~$\z:\I\to\I'$ such that, for all~$i\in\I$, the component ps-characters~$\Th_i$ and~$\Th_{\z(i)}$ are endo-equivalent. Moreover, if~$\Th\approx\Th'$ then the map~$\z$ is uniquely determined.
\item \label{thm:Endopss-ii-new} Suppose that~$\Th\approx\Th'$ and let~$\z:\I\to\I'$ be the bijection of~\ref{thm:Endopss-i}. Let~$(\V,\vphi,\La,r)\in\Qq(k,\b)$ and~$(\V',\vphi',\La',r')\in\Qq(k,\b')$.
\begin{enumerate}\setlength\itemsep{5pt}
\item\label{thm:Endopss-iia1} For all~$i\in\I$, we have 
\begin{equation}\label{eqppsFieldDegrees}
e(\E_i|\F)=e(\E'_{\z(i)}|\F),\ f(\E_i|\F)=f(\E'_{\z(i)}|\F),\ k_\F(\b_i)=k_\F(\b'_{\z(i)}).
\end{equation}
\item\label{thm:Endopss-iia2} If%~$(\V,\vphi,\La,r)\in\Qq(k,\b)$ and~$(\V',\vphi',\La',r')\in\Qq(k,\b')$ are such that
~$e(\La)=e(\La')$ then~$(\V,\vphi,\La,r')\in\Qq(k,\b)$.
\item\label{thm:Endopss-iib} If~$\V=\V'$ and~$\Th(\V,\vphi,\La,r)$ and~$\Th'(\V,\vphi',\La',r')$ intertwine in~$\tG$ with matching~$\xi$, then~$\xi=\z$.
\item\label{thm:Endopss-iic} If~$\V=\V'$ and~$\dim_\F(\V^i)=\dim_\F(\V'^{\z(i)})$, for all~$i\in\I$, then~$\Th(\V,\vphi,\La,r)$ and~$\Th'(\V,\vphi',\La',r')$ intertwine in~$\tG$ with matching~$\z$.
\end{enumerate}
%\item\label{thm:Endopss-iia1} If~$\Th\approx\Th'$ and~$\z:\I\to\I'$ is the bijection given by~\ref{thm:Endopss-i} then, for all~$i\in\I$, we have 
%\begin{equation}\label{eqppsFieldDegrees}
%e(\E_i|\F)=e(\E'_{\z(i)}|\F),\ f(\E_i|\F)=f(\E'_{\z(i)}|\F),\ k_\F(\b_i)=k_\F(\b'_{\z(i)}).
%\end{equation}
\item\label{thm:Endopss-iii-new} Endo-equivalence of pss-characters is an equivalence relation.
\end{enumerate}
\end{theorem}

\begin{proof}
\shauns{
Note first that~\ref{thm:Endopss-iii-new} follows from~\ref{thm:Endopss-i} and \gre{Remark \ref{rem:defEndoEquivalent}\ref{rem:defEndoEquivalentiii}}. Moreover, the uniqueness statement in~\ref{thm:Endopss-i} follows immediately from Lemma~\ref{lemma:pssdecompositions}\ref{lemma:pssdecompositions-iii} and the transitivity of endo-equivalence for ps-characters.
}

\shauns{
Suppose that~$\xi:\I\to\I'$ is a bijection such that~$\Th_i\approx\Th'_{\xi(i)}$, for all~$i\in\I$. From Proposition~\ref{prop:firstPropEndoSimple}\ref{prop:firstPropEndoSimple.i} it follows that~$\Th$ and~$\Th'$ have the same degree. Let~$(\V,\vphi,\La,r)\in\Qq(k,\b)$ and~$(\V,\vphi',\La',r')\in\Qq(k,\b')$ be such that~$\V^i$ and~$\V'^{\xi(i)}$ have the same dimension for all~$i\in\I$, and set~$\t=\Th(\V,\vphi,\La,r)$ and~$\t'=\Th'(\V,\vphi',\La',r')$. Then the simple block restrictions~$\t_i$ and~$\t'_{\xi(i)}$ intertwine by an~$\F$-linear isomorphism from~$\V^i$ to~$\V'^{\xi(i)}$, for all~$i\in\I$, by Theorem~\ref{thm:EndoEquivMeanspairwiseIntforAllrealizations}. Then~$\t$ and~$\t'$ intertwine in~$\Aut_\F(\V)$ with matching~$\xi$, by Proposition~\ref{prop:intdecomp}, so~$\Th\approx\Th'$. This proves one direction of~\ref{thm:Endopss-i}, and also that~\ref{thm:Endopss-iic} follows from~\ref{thm:Endopss-i}.
}

\shauns{
Conversely, suppose that we have~$(\V,\vphi,\La,r)\in\Qq(k,\b)$ and~$(\V,\vphi',\La',r')\in\Qq(k,\b')$ such that~$\t=\Th(\V,\vphi,\La,r)$ and~$\t'=\Th'(\V,\vphi',\La',r')$ intertwine in~$\tG=\Aut_\F(\V)$. % with matching~$\xi:\I\to\I'$. 
Replacing~$\La,\La'$ in their affine classes, we can assume that~$\La$ and~$\La'$ have the same~$\o_\F$-period and, exchanging them if necessary, that~$r\leq r'$. To ease notation, we will identify~$\b,\b'$ with their images under the embeddings~$\vphi,\vphi'$ respectively.
%~$\Th\approx\Th'$. Thus there are realizations~$\t\in\Cc(\La,r,\b)$ and~$\t'\in\Cc(\La',r',\b')$ of~$\Th$ and~$\Th'$ respectively on a common vector space~$\V$ which intertwine in~$\tG=\Aut_\F(\V)$, where we have identified~$\b$ and~$\b'$ with their images under the embeddings. Without loss of generality we can assume that~$\La$ and~$\La'$ have the same~$\o_\F$-period and that~$r\leq r'$.
}

\shauns{
%We are going to show that~$[\La,n,r',\b]$ is semisimple.  %It needs some preparation. 
Suppose for contradiction that~$[\La,n,r',\b]$ is not semisimple and let~$[\La,n,r',\g]$ be a semisimple stratum equivalent to it such that~$\g\in\prod_i\A^{i}$. %This choice is possible by coarsening and~\cite[6.16]{SkSt}, cf.~\cite[3.4]{St05}. 
By Theorem~\ref{thm:MatchingForChar} and Proposition~\ref{prop:SimpleDegrees} the algebras~$\F[\g]$ and~$\E'=\F[\b']$ have the same degree and~$e_{\E'}=e_{\F[\g]}$. As the pss-characters are endo-equivalent~$\F[\g]$ also has the same~$\F$-degree as~$\E=\F[\b]$. If~$\J$ is the index set for~$\g$, then~$\#\J\leq\#\I$ and, for~$i\in\I,j\in\J$ with~$\V^i\subseteq\V^j$, we have~$e(\F[\g_j]/\F) \mid e(\F[\b_i]/\F)$ and~$f(\F[\g_j]/\F) \mid f(\F[\b_i]/\F)$. Thus the equality~$[\F[\g]:\F]=[\F[\b]:\F]$ implies that~$\#\J=\#\I$, so we can identify~$\J$ with~$\I$, and that~$e(\F[\g_i]/\F) = e(\F[\b_i]/\F)$ and~$f(\F[\g_i]/\F) = f(\F[\b_i]/\F)$. In particular, we deduce that~$e_{\E}=e_{\F[\g]}=e_{\E'}$.
}

\shauns{
Since~$[\La,n,r',\b]$ is not semisimple, there must then be a (unique) index~$i_0$ with~$\g_{i_0}=0$ and~$\b_{i_0}\in\F^\times$, by~\cite[6.4,~6.1]{SkSt}. This implies that~$k_0(\b,\La)=k_0(\b_{i_0},\La^{i_0})$ is a multiple of~$e(\La^{i_0}|\o_{\E_{i_0}})=e(\La|\o_F)$, so that
\[
\frac{r'}{e(\La|\o_\F)}\ge  \frac{-k_0(\b,\La)}{e(\La|\o_\F)} > \frac{r}{e(\La|\o_\F)},
\]
with the middle term an integer. In particular, we deduce that
\[
\left\lfloor\frac{r'}{e(\La|\o_\F)}\right\rfloor>\left\lfloor\frac{r}{e(\La|\o_\F)}\right\rfloor.
\]
But
\[
\left\lfloor\frac{r}{e(\La|\o_\F)}\right\rfloor=\left\lfloor\frac{r}{e(\La_\E)e_\E}\right\rfloor = \left\lfloor\frac{k}{e_\E}\right\rfloor = \left\lfloor\frac{k}{e_{\E'}}\right\rfloor=\left\lfloor\frac{r'}{e(\La'_{\E'})e_{\E'}}\right\rfloor = \left\lfloor\frac{r'}{e(\La'|\o_\F)}\right\rfloor,
\]
which contradicts the previous inequality since~$e(\La|\o_\F)=e(\La'|\o_\F)$.
}

\ignore{
% Applying Theorem~\ref{thm:MatchingForChar} to~$\t|_{\H^{r'+1}(\g,\La)}$ and~$\t'$,  Proposition~\ref{prop:SimpleDegrees} and above the equality~$e(\La_\E)=e(\La'_{\E'})$, 
%in particular (by
Since~$e(\La|\o_\F)=e(\La'|\o_\F)$,
\[
\left\lfloor\frac{r}{e(\La|\o_\F)}\right\rfloor=\left\lfloor\frac{k e(\La|\o_{\E})}{e(\La)}\right\rfloor=\left\lfloor\frac{r'}{e(\La')}\right\rfloor.
\]
Now we can conclude the desired semi-simplicity. 
If~$[\La,-,r',\b']$ is not semisimple then there must be an index~$i_\so\in\I$ such that~$\g_{i_\so}=0$ and~$\b_{i_\so}\in\F^\times$ by~\cite[6.4, 6.1]{SkSt}. Thus
\[
\left\lfloor\frac{r'}{e(\La)}\right\rfloor>\left\lfloor\frac{r}{e(\La)}\right\rfloor=\left\lfloor\frac{r'}{e(\La')}\right\rfloor.
\]
A contradiction. 
}

\shauns{
Thus~$[\La,n,r',\b]$ is semisimple. 
\orange{Now~$\t|_{\H^{r'+1}(\b,\La)}$ and~$\t'$ intertwine in~$\tG$ as~$\t$ and~$\t'$ do, Theorem~\ref{thm:MatchingForChar} provides a matching~$\z:\I\to\I'$, and~\eqref{eqppsFieldDegrees} follows from Proposition~\ref{prop:SimpleDegrees}. Note also that if we had started with a matching~$\xi:\I\to\I'$ from~$(\b,\t)$ to~$(\b',\t')$ then we would obtain~$\xi=\z$ by restriction.
}
%Then~$\t|_{\H^{r'+1}(\b,\La)}$ and~$\t'$ intertwine in~$\tG$ with matching~$\xi$, so~$\xi=\z$ is the matching given by Theorem~\ref{thm:MatchingForChar}, and~\eqref{eqppsFieldDegrees} follows from Proposition~\ref{prop:SimpleDegrees}. 
We see that:%, we see that
%\[
%e(\E_i|\F)=e(\E'_{\z(i)}|\F),\ f(\E_i|\F)=f(\E'_{\z(i)}|\F),\ k_\F(\b_i)=k_\F(\b'_{\z(i)}),
%\] 
%for all~$i\in\I$. Then:
%We need to show~$\Th_i\approx\Th'_{\z(i)}$ for all~$i\in\I$. 
%We collect the information about~$\Th_i$ and~$\Th'_{\z(i)}$:
\begin{itemize}
\item the field extensions~$\E_i/\F$ and~$\E_{\z(i)}/\F$ have the same degree;
\item we have~$e(\La^i|\o_{\E_i})=e(\La'^{\z(i)}|\o_{\E'_{\z(i)}})$  and~$e(\La_\E)=e(\La'_{\E'})$ so, setting~$q=\frac{e(\La^i|\o_{\E_i})}{e(\La_\E)}$, we obtain
\[
\left\lfloor\frac{r}{e(\La^i|\o_{\E_i})}\right\rfloor=\left\lfloor\frac{k}{q}\right\rfloor=\left\lfloor\frac{r'}{e(\La'^{\z(i)}|\o_{\E'_{\z(i)}})}\right\rfloor,
\]
which integer we denote by~$k_i$; 
\item we have~$(\V^i,\varphi_i,\Lambda^i,r')\in\Qq(k_i,\b_i)$, while~$\t_i|_{\H^{r'+1}(\b,\La)}$ and~$\t'_{\z(i)}$ intertwine. 
\end{itemize}
Thus the ps-characters~$\Th_i$ and~$\Th'_{\z(i)}$ are endo-equivalent. This completes the proof of the converse direction of~\ref{thm:Endopss-i}, and also that of~\ref{thm:Endopss-iia1} and~\ref{thm:Endopss-iib}.
}
%\orange{
%For the uniqueness statement suppose~$\Th\approx\Th'$ and the existence of a bijection~$\xi:\I\to\I'$ such that~$\Th'_{\z(i)}$ and~$\Th'_{\xi(i)}$ are endo-equivalent for all~$i\in\I$. Take a realization~$\t''$ of~$\Th'$ such that~$\V'^{\xi(i)}$ and~$\V'^{\z(i)}$ have the same dimension over~$\F$. Then~$\t''_{\z(i)}$ and~$\t''_{\xi(i)}$ intertwine by Theorem~\ref{thm:EndoEquivMeanspairwiseIntforAllrealizations}. The uniqueness of the matching from~$\t''$ to itself implies~$\xi(i)=\z(i)$ for all~$i\in\I$.
%}

\shauns{
Finally, assertion~\ref{thm:Endopss-iia2} follows as in the proof of Proposition~\ref{prop:firstPropEndoSimple}\ref{prop:firstPropEndoSimple.ii}. 
}
\end{proof}

We call the equivalence classes of pss-characters under endo-equivalence~\emph{semisimple endo-classes}. 

\begin{definition}
Given endo-equivalent pss-characters~$\Th$ and~$\Th'$ we call the bijection~$\z$ of Theorem~\ref{thm:Endopss-i} the~\emph{matching}~$\z_{\Th',\Th}$ from~$\Th$ to~$\Th'$.
\end{definition}

The uniqueness statement in Theorem~\ref{thm:Endopss}\ref{thm:Endopss-i} immediately gives us:

\begin{corollary}\label{cor:zetaEqualityPss}
Let~$\Th$ and~$\Th'$ and~$\Th''$ be pss-characters such that~$\Th\approx\Th'\approx\Th''$.  Then~$\z_{\Th'',\Th}=\z_{\Th'',\Th'}\circ\z_{\Th',\Th}$. 
\end{corollary}

We obtain, as another consequence, that intertwining is an equivalence relation for semisimple characters with the same degree and the same parameter~$k$ (cf. Theorem~\ref{thm:transintertwiningoverGL}).

\begin{corollary}\label{cor:equivOfIntertwiningSschar}
Suppose~$\t^{(l)}\in\Cc(\La^{(l)},r^{(l)},\b^{(l)})$, for~$l=1,2,3$, are semisimple characters of the same degree such that~$\t^{(1)}$ intertwines with~$\t^{(2)}$, and~$\t^{(2)}$ intertwines with~$\t^{(3)}$. Suppose that $\left\lfloor\frac{r^{(l)}}{e(\La^{(l)}_{{\E^{(l)}}})}\right\rfloor$ is independent of~$l$. Then~$\t^{(1)}$ and~$\t^{(3)}$ intertwine. 
\end{corollary}

%%%%%%%%%%%%%%%%%%%%%%%%%%%%%%%%%%%%
\subsection{Self-dual semisimple endo-classes}
%%%%%%%%%%%%%%%%%%%%%%%%%%%%%%%%%%%%
Let~$\Th_-$ be a self-dual pss-character supported on the self-dual semisimple pair~$(k,\b)$, and let~$\Th'_-$ be a self-dual pss-character supported on the self-dual semisimple pair~$(k',\b')$.

\begin{definition}
We say that~$\Th_-$ and~$\Th'_-$ are \emph{endo-equivalent}, denoted~$\Th_- \approx\Th'_-$, if
\begin{enumerate}\setlength\itemsep{5pt}
\item $\deg(\Th_-)=\deg(\Th'_-)$; 
\item $k=k'$;
%\red{is this defined earlier? or should it be~$\deg(\Th_i)=\deg(\Th'_i)$ for all~$i$}
\item there exist self-dual realizations on a common~$\e$-hermitian space which intertwine, i.e.\ there exist~$((\V,h),\vphi,\La,r)\in\Qq_-(k,\b)$ and~$((\V,h),\vphi',\La',r')\in\Qq(k',\b')$ such that~$\Th_-((\V,h),\vphi,\La,r)$ and~$\Th'_-((\V,h),\vphi',\La',r')$ intertwine in~$\G=\U(\V,h)$.
\end{enumerate}
\end{definition}

Given two endo-equivalent self-dual pss-characters~$\Th_-$ \rob{and}~$\Th'_-$ with lifts~$\Th$ \rob{and}~$\Th'$ respectively, then~$\Th\approx\Th'$ by the Glauberman correspondence, and the matching from~$\Th$ to~$\Th'$ will also be written as the matching~$\z_{\Th'_-,\Th_-}$ from~$\Th_-$ to~$\Th'_-$. This matching is~$\s$-equivariant because it is also the matching between any intertwining realizations of~$\Th_-$ and~$\Th'_-$, by %Glauberman and 
Theorem~\ref{thm:Endopss}\ref{thm:Endopss-iib}.
% This matching is self-dual because it is also the matching between intertwining realizations of~$\Th_-$ to~$\Th'_-$, by Glauberman and Theorem~\ref{thm:Endopss} \ref{thm:Endopss-iii}. 

We need to generalize the notion of concordance to embeddings of semisimple algebras over~$\F$. 

\begin{definition}\label{def:WittConcordanceSemisimple}
\shaun{Let~$(k,\b)$ and~$(k,\b')$ be self-dual semisimple pairs with index sets~$\I$ and~$\I'$ respectively. Suppose that~$(\V,h)$ and~$(\V',h')$ are isometric~$\e$-hermitian spaces and~$\vphi:\E\hookrightarrow \End_\F(\V)$ and~$\vphi':\E'\hookrightarrow \End_{\F}(\V')$ are self-dual~$\F$-algebra embeddings. 
%%
%\red{did we define self-dual in this generality}\orange{After Def. 9.2}
%%
Let~$\z:\I\to\I'$ be a bijection. We say that~$(\vphi,\b)$ and~$(\vphi',\b')$ are \emph{$\z$-concordant} if, for all~$i\in\I_0$, the spaces~$(\V^i,h_i)=(\V'^{\z(i)},h'_{\z(i)})$ are isometric and~$(\vphi|_{\E_i},\b_i)$ and~$(\vphi'|_{\E'_{\z(i)}},\b'_{\z(i)})$ are~$(h_i,h'_{\z(i)})$-concordant.} 
%if there is an isometry~$g$ from~$h$ to~$h'$ such that~$g(\V^i)=\V'^{\z(i)}$ and such that, for all~$i\in\I_0$, the embeddings~$g\vphi|_{\E_i}(*)g^{-1}$ and~$\vphi'|_{\E'_{\z(i)}}$ are concordant with respect to~$\b_i$ and~$\b_{\z(i)}'$.
\end{definition}
%%
%\red{It feels like the above could go in the Witt section, a new subsection at the end...}\gre{No, here is fine, because we do not use it before.}
%%

Now we can gather all the results of the previous sections to get the following:

\begin{theorem}\label{thmEndoSemisimplev4}
Let~$\Th_-$ and~$\Th'_-$ be self-dual pss-characters {supported on~$(k,\b)$ and~$(k,\b')$, respectively,} 
and~$\Th$ and~$\Th'$ their respective lifts.
Then, the following assertions are equivalent:
\begin{enumerate}\setlength\itemsep{5pt}
 \item \label{part1endosemi}The self-dual pss-characters~$\Th_-$ and~$\Th'_-$ are endo-equivalent;
 \item \label{part2endosemi}The lifts~$\Th$ and~$\Th'$ are endo-equivalent.
 \item \label{part3endosemi} \bob{$\deg(\Th_-)=\deg(\Th'_-)$ and t}\shaun{here is a  bijection~$\z:\I\to\I'$ which commutes with~$\s$ \shauny{ with the following property: if~$((\V,h),\vphi,\La,r)\in\Qq_-(k,\b)$ and~$((\V,h),\vphi',\La',r')\in\Qq(k,\b')$ are such that~$(\vphi,\b)$ and~$(\vphi',\b')$ are~$\z$-concordant and~$\dim_\F\V^i=\dim_\F\V'^{\zeta(i)}$, for~$i\in\I$,} then the realizations~$\Th_-((\V,h),\vphi,\La,r)$ and~$\Th'_-((\V,h),\vphi',\La',r')$ intertwine in~$\G=\U(\V,h)$ with matching~$\z$.}
 %for all pairs~$\t,\t'$ of realizations of~$\Th,\Th'$ respectively on a common hermitian space~$(\V,h)$ with $\z$-concordant embeddings, the characters~$\t,\t'$ intertwine in~$\G=\U(\V,h)$ with matching~$\z$}.
\end{enumerate}
\end{theorem}

\begin{proof}
If~$\Th_-$ and~$\Th'_-$ are endo-equivalent then so are~$\Th$ and~$\Th'$ by the Glauberman correspondence, i.e.~\ref{part1endosemi}$\Rightarrow$\ref{part2endosemi}. 

\shaun{
We have~\ref{part3endosemi} $\Rightarrow$\ref{part1endosemi}~because we can find realizations of the semisimple pairs such that, for all indices~$i\in \I_0$, the forms~$h_{i,\vphi(\b_i)}$ and~$h_{\z(i),\vphi'(\b'_{\z(i)})}$ are hyperbolic so that, in particular,~$(\vphi|_{\E_i},\b_i)$ and~$(\vphi'|_{\E'_{\z(i)}},\b'_{\z(i)})$ are~$(h_i,h'_{\z(i)})$-concordant.
}

\shauns{
It remains to show that~\ref{part2endosemi}$\Rightarrow$\ref{part3endosemi}, so suppose that~$\Th$ and~$\Th'$ are endo-equivalent. Let~$\z$ be the matching from~$\Th$ to~$\Th'$. Take realizations~$\t_-=\Th_-((V,h),\vphi,\La,r)$ and~$\t'_-=\Th'_-((V,h),\vphi',\La',r')$, such that for all~$i\in\I$ the~$\F$-vector spaces~$\V^i$ and~$\V'^{\z(i)}$ have the same dimension. Let~$\t$ and~$\t'$ be the lifts of~$\t_-$ and~$\t'_-$. Then~$\t$ and~$\t'$ intertwine with matching~$\z$ by Theorem~\ref{thm:Endopss}\ref{thm:Endopss-iic}. %Then~$\s\z\s$ is a matching from~$(\t,\vphi(\b))$ to~$(\t',\vphi'(\b'))$ too, and therefore~$\z$ commutes with~$\s$ by Theorem~\ref{thm:Endopss}\ref{thm:Endopss-iib}. 
Now suppose that~$\vphi$ and~$\vphi'$ are~$\z$-concordant. Since~$\Th_i$ and~$\Th'_{\z(i)}$ are endo-equivalent,~$\t_i$ and~$\t'_{\z(i)}$ intertwine by an~$\F$-linear isomorphism~$g_i\in\Hom_\F(\V^i,\V'^{\z(i)})$ by Theorem~\ref{thm:EndoEquivMeanspairwiseIntforAllrealizations}. For~$i\in\I_-$ we can replace~$g_i$ by~$\ov{g_{-i}}^{-1}$; moreover, for~$i\in\I_0$ we may assume that~$g_i$ is an isometry from~$(V^i,h_i)$ to~$(V'^{\z(i)},h_{\z(i)})$ which intertwines~$\t_{i,-}$ with~$\t'_{\z(i),-}$, by Proposition~\ref{prop:TiGandGIntertwiningSameNonSympl} (since the embeddings are~$\z$-concordant). In particular, the element~$g=\sum_{i\in\I} g_i$ is then in~$\G$ and intertwines~$\t$ with~$\t'$ with matching~$\z$, by Proposition~\ref{prop:intdecomp}.
}
\end{proof}

One consequence of Theorem~\ref{thmEndoSemisimplev4}, Theorem~\ref{thm:Endopss}\ref{thm:Endopss-iii-new} and Corollary~\ref{cor:zetaEqualityPss} is:

\begin{corollary}\label{cor:EndoequivTransSdPss}
Endo-equivalence of self-dual pss-characters is an equivalence relation and, for self-dual pss-characters~$\Th_-\approx\Th'_-\approx\Th''_-$, we have~$\z_{\Th''_-,\Th'_-}\circ\z_{\Th'_-,\Th_-}=\z_{\Th''_-,\Th_-}$.
\end{corollary}

\begin{definition}
We call the equivalence classes of self-dual pss-characters under endo-equivalence \emph{self-dual semisimple endo-classes}.
\end{definition}

\shaun{
As another corollary of Theorem~\ref{thmEndoSemisimplev4}  and Theorem~\ref{thm:MatchingForCharForG}, we see the remarkable result that, for self-dual semisimple characters of same degree and with the same~$k$, intertwining is an equivalence relation.
}
% the property intertwining for self-dual semisimple characters of same degree and same~$k$ is an equivalence relation.
%}% the remarkable result that intertwining of self-dual semisimple characters with same group level is transitive, and hence equivalence relation:

\begin{corollary}\label{cor:IntertwiningEquivalenceRelG}
Suppose~$\t^{(l)}_-\in\Cc_-(\La^{(l)},r^{(l)},\b^{(l)})$, for~$l=1,2,3$, are self-dual semisimple characters of the same degree such that~$\t^{(1)}_-$ intertwines with~$\t^{(2)}_-$ in~$\G$, and~$\t^{(2)}_-$ intertwines with~$\t^{(3)}_-$ in~$\G$. Suppose that $\left\lfloor\frac{r^{(l)}}{e(\La^{(l)}_{{\E^{(l)}}})}\right\rfloor$ is independent of~$l$. Then~$\t^{(1)}_-$ and~$\t^{(3)}_-$ intertwine in~$\G$. 
\end{corollary}

\begin{proof}
Let~$\Th^{(l)}_-$ be the self-dual pss-character supported on~$(k,\b^{(l)})$ with realization~$\t^{(l)}_-$. Now~$\Th^{(1)}_-\approx\Th^{(2)}_-$ and~$\Th^{(2)}_-\approx\Th^{(3)}_-$ and thus~$\Th^{(1)}_-\approx\Th^{(3)}_-$, by Corollary~\ref{cor:EndoequivTransSdPss}. Let~$\vphi^{(l)}$ be the canonical embedding of~$\E^{(l)}$ into~$\A$. We need to show that~$(\vphi^{(1)},\b^{(1)})$ and~$(\vphi^{(3)},\b^{(3)})$ are~$\z_{\Th^{(3)}_-,\Th^{(1)}_-}$-concordant. Without loss of generality we can assume that~$e(\La^{(l)})$ is independent of~$l$, and by Theorem~\ref{thm:Endopss}\ref{thm:Endopss-iia2} we can assume without loss of generality that~$r^{(l)}$ is independent of~$l$. \bob{By the Glauberman correspondence and Corollary~\ref{cor:equivOfIntertwiningSschar}, the lifts of~$\t^{(1)}_-$ and~$\t^{(3)}_-$ intertwine in~$\tG$ and, by Theorem~\ref{thm:Endopss}\ref{thm:Endopss-iib}, \shauny{they do so} with matching~$\zeta:=\zeta_{\Th^{(1)}_-,\Th^{(3)}_-}$; in particular,~$\dim_\F\V^i=\dim_\F\V^{\zeta(i)}$, for~$i\in\I$.} By Theorem~\ref{thm:MatchingForCharForG} we have that~$(\vphi^{(1)},\b^{(1)})$ and~$(\vphi^{(2)},\b^{(2)})$ are~$\z_{\Th^{(2)}_-,\Th^{(1)}_-}$-concordant, and that~$(\vphi^{(2)},\b^{(2)})$ and~$(\vphi^{(3)},\b^{(3)})$ are~$\z_{\Th^{(3)}_-,\Th^{(2)}_-}$-concordant. Now the transitivity of concordance and Corollary~\ref{cor:EndoequivTransSdPss} finish the proof.
\end{proof}

%%%%%%%%%%%%%%%%%%%%%%%%%%%%%%%%%%%%
\section{Intertwining and conjugacy for special orthogonal groups}
%%%%%%%%%%%%%%%%%%%%%%%%%%%%%%%%%%%%

We now investigate intertwining and conjugacy of semisimple characters of special orthogonal groups, so for this section we suppose that~$(\V,h)$ is a~$1$-hermitian space with~$\F=\F_\so$, so that~$\G=\U(\V,h)$ is an orthogonal group and~$\G^\so$ is its special orthogonal subgroup.   %This, can not happen if~$\b$ has no zero-component. 

%%%%%%%%%%%%%%%%%%%%%%%%%%%%%%%%%%%%
\subsection{Intertwining self-dual semisimple characters}
%%%%%%%%%%%%%%%%%%%%%%%%%%%%%%%%%%%%

%%%
%\red{I think this should move later.}\orange{Where?}
%%%
%Here we consider endo-equivalent pss-characters~$\Th_-$ and~$\Th'_-$ for self-dual semisimple pairs~$[k,\b]$ and~$[k,\b']$ in the orthogonal setting~$\s=1$,~$\e=1$. 
%Let~$\t_-\in \Cc_-(\La,r,\vphi(\b))$ and~$\t'_-\in \Cc_-(\La',r,\vphi'(\b'))$ be realizations of~$\Th_-$ and~$\Th'_-$ respectively on an orthogonal space~$(\V,h)$. In this section we analyse the intertwining classes of an orthogonal endo-class.  Let us assume that~$\La$ and~$\La'$ have the same period, i.e.~$e(\La)=e(\La')$.

%For this section, we suppose that we are in the orthogonal setting~$\s=1$, $\e=1$, so~$\G$ will denote an orthogonal group, and~$\G^\so$ its special orthogonal subgroup.  We l
Let~$[\La,n,r,\b]$ and~$[\La',n,r,\b']$, be self-dual semisimple strata in~$\A$, \rob{with} associated splittings~$\V=\bigoplus_{i\in \I}\V^i$ and~\rob{$\V=\bigoplus_{i\in \I'}\V'^i$}.

%{\color{red}Should we say that it follows that~$e(\La)=e(\La')$, because if~$n>r$ and~$[\La,n,n-1,\b]$ and~$[\La',n,n-1,\b']$ intertwine then both share the depth Reference.}

\begin{lemma}\label{propSOintertwiningInvertibleBeta}
Suppose that~$\b$ and~$\b'$ are non-zero. Let~$\t_-\in \Cc_-(\La,r,\b)$ and~$\t'_-\in \Cc_-(\La',r,\b')$ be self-dual semisimple characters which intertwine in~$\G$.
%\begin{enumerate}\setlength\itemsep{5pt}
% \item Suppose that there is~$i_0\in\I$ such that~$\b_{i_0}=0$. Then,~$\t_-$ and~$\t'_-$ intertwine by an element of~$\G^\so$.
% \item 
\shauns{Suppose also that~$\b$ normalizes~$\La$.}  %Suppose also that~$\val_\La(\b_i)=-n$ for all~$i\in\I$. 
Then~$\t_-$ and~$\t'_-$ intertwine in~$\G^\so$ (respectively in~$\G\setminus\G^\so$) if and only if the symplectic spaces~$(\V,\b^*(h))$ and~$(\V,\b'^*(h))$ are isometric by an automorphism of~$\V$ of determinant congruent to~$1$ modulo~$\p_\F$ (respectively, to~$-1$ modulo~$\p_\F$).
 %Assume that all~$\vphi(\b_i)$ have the same valuation for the lattice sequence~$\La$, i.e.~$\val_\La(\vphi(\b_i))=-q$ for all indexes~$i$.    
 %The realizations~$\t_-$ and~$\t'_-$ intertwine over~$\G^\so$ if and only if the symplectic forms~$h^{\vphi(\b)}$ and~$h^{ \vphi'(\b')}$ are isometric by 
 %an automorphism of~$\V$ of determinant congruent~$1$ modulo~$\mathfrak{p}_\F$.
 %\end{enumerate}
\end{lemma}

\begin{proof}
\shauns{By hypothesis, there is an element~$g\in\G$ which intertwines~$\t_-$ with~$\t'_-$. Then the fundamental strata~$[\La,n,n-1,\b]$ and~$[\La',n,n-1,\b']$ are intertwined by~$g$ so, by~\cite[Proposition~6.9]{SkSt}, have the same level. In particular, we deduce that~$e(\La)=e(\La')$. Moreover, writing~$\z:\I\to\I'$ for the matching from~$(\t_-,\b)$ to~$(\t'_-,\b')$, Theorem \ref{thm:MatchingForChar} and~\cite[Proposition~6.9]{SkSt} together imply that~$\nu_{\La'}(\b'_{\z(i)})=-n$ for all~$i\in\I$ also. In particular, this implies that~$\b'$ normalizes~$\La'$ also.}

\shauns{By the intertwining of the fundamental strata, there are skew elements~$c\in\b+\aa_{1-n}^-$ and~$c'\in\b'+\aa_{1-n}'^-$ such that~$gcg^{-1}=c'$, and~$g$ then gives an isometry from~$c^*(h)$ to~$c'^*(h)$. Then~\cite[Lemma~5.3]{SkSt} implies that there is an~$\F$-linear isometry~$u\in\P^1(\La)$ from~$\b^*(h)$ to~$c^*(h)$; similarly, there is an~$\F$-linear isometry~$u'\in\P^1(\La')$ from~$c'^*(h)$ to~$\b'^*(h)$. Thus\gre{~$u'gu$} is an isometry from~$\b^*(h)$ \rob{to}~$\b'^*(h)$, and\gre{~$\det(u'gu)\equiv \det(g) \pmod{\p_\F}$}. Since any isometry of a symplectic space has determinant~$1$, there cannot be isometries from~$\b^*(h)$ \rob{to}~$\b'^*(h)$ with determinant congruent to both~$\pm 1$ modulo~$\p_\F$ and the result follows.}
%The assumption on the~$\b_i$ is also true for the~$\b'_i$ by Theorem \ref{thm:MatchingForChar} and~\cite[6.9]{SkSt}.  By hypotheses there is an element~$g\in\G$ which intertwines~$\t_-$ with~$\t'_-$. Then, the fundamental strata~$[\La,n,n-1,\b]$ and~$[\La',n,n-1,\b']$ intertwine under~$g$ and the twists of~$h$,~$\b^*(h)$ and~$\b'^*(h)$,  are isometric by an element of the form~$u'gu$ with elements~$u\in 1+\aa_1(\La)$ and~$u'\in 1+\aa_1(\La')$, by~\cite[Proposition 3.1 and Proof of Theorem 5.2]{SkSt}. The element~$u'gu$ has determinant congruent to the determinant of~$g$ modulo~$\mathfrak{p}_\F$. The fact that an isometry of a symplectic space has determinant~$1$ finishes the proof. 
\end{proof}

\begin{theorem}\label{thmSOintertwining}
Let~$\t_-\in \Cc_-(\La,r,\b)$ and~$\t'_-\in \Cc_-(\La',r,\b')$ be self-dual semisimple characters which intertwine in~$\G$.
\begin{enumerate}\setlength\itemsep{5pt}
\item\label{thmSOintertwining.i} Suppose that there is~$i_0\in\I$ such that~$\b_{i_0}=0$. Then,~$\t_-$ and~$\t'_-$ intertwine by an element of~$\G^\so$ and by an element of~$\G\setminus\G^\so$.
\item\label{thmSOintertwining.ii} If~$\b$ has no zero component, then~$\t_-$ and~$\t'_-$ intertwine under an element of~$\G^\so$ if and only if~$(\V,\b^*(h))$ and~$(\V,\b'^*(h))$ are isometric by an automorphism of~$\V$ of determinant congruent to~$1$ modulo~$\mathfrak{p}_\F$. In this case every element of~$\G$ intertwining~$\t_-$ and~$\t'_-$ is in~$\G^\so$.
\end{enumerate}
\end{theorem}
\shaun{
Note that the statement in~\ref{thmSOintertwining.ii} is equivalent to saying that~$\t_-$ and~$\t'_-$ intertwine under an element of~$\G\setminus\G^\so$ if and only if~$(\V,\b^*(h))$ and~$(\V,\b'^*(h))$ are isometric by an automorphism of~$\V$ of determinant congruent to~$-1$ modulo~$\mathfrak{p}_\F$, in which case every element of~$\G$ intertwining~$\t_-$ and~$\t'_-$ is in~$\G\setminus\G^\so$.
}

\begin{proof}
\shaun{Write~$\z:\I\to\I'$ for the matching from~$(\t_-,\b)$ to~$(\t'_-,\b')$. By Theorem~\ref{thm:MatchingForCharForG}, for~$i\in\I_+$ we can choose an isomorphism~$g_i:\V^i\to\V'^{\z(i)}$, and by Theorem~\ref{thm:MatchingForCharForG}, for~$i\in\I_0$ we can find an isometry~$g_i$ from~$(\V^i,h_i)$ to~$(\V'^{\z(i)},h'_{\z(i)})$. For~$i\in\I_-$ we set~$g_i=\ov{g_{-i}}^{-1}$, so that~$g=\sum_{i\in\I}g_i$ is an element of~$\G$ which conjugates the splittings. Conjugating by this element~$g$ (which may have determinant~$-1$), we reduce to the case that the characters have the same splitting and the matching is the identity map.}
%and by  are isometric we can assume that both characters have the same associated splitting, the index sets are the same and the matching is the identity.}
\begin{enumerate}\setlength\itemsep{5pt}
\item \shaun{The characters~$\t_{i_0}$ and~$\t'_{i_0}$ are trivial, and therefore intertwine under any element of the group~$\U(\V^{i_0},h_{i_0})$, in particular by an element of determinant~$1$ and by an element of determinant~$-1$. The result follows immediately by applying Proposition~\ref{prop:intdecomp} to the lifts of~$\t_-,\t'_-$.}
%\orange{By Corollary~\ref{cor:Intertwining}, there is an element~$g$ of~$\I(\t,\t')$ such that~$g\V^i=\V^{i}$ for all~$i\in\I$, we write~$g=(g_i)_i$. Now one replaces~$g_{i_0}$ by any element of~$\U(\V^{i_0},h_{i_0})$. The new element still intertwines~$\t$ with~$\t'$, because both characters have an Iwahori decomposition with respect to their associated splitting. Therefore, there are an element of~$\G^\so$ and an element of~$\G\setminus\G^\so$ which intertwine the semisimple characters.}
%%%
% \red{Why does this mean one can adjust the original intertwiner so that it has determinant 1}
% \gre{See 8..} 
% \red{8..what does that mean?}
%%%
% Thus, by Corollary~\ref{cor:Intertwining}, there \orange{are} an element of~$\G^\so$ \orange{and an element of~$\G\setminus\G^\so$} which intertwine the semisimple characters. Here we have used that the characters~$\t$ and~$\t'$ have Iwahori decomposition with respect to the associated spittings.
\item %\shaun{For~$s\in\ZZ$, we write~$\I_s=\{i\in\I\mid\nu_\La(\b_i)=-s\}$, which gives a (finite) partition~$\I=\bigcup_{s\in\ZZ}\I_s$ into~$\s$-stable subsets. By Corollary~\ref{cor:Intertwining}, there are~$g_{\I_s}\in \U(\V^{\I_s},h_{\I_s})$ such that~$g=\sum_{s\in\ZZ} g_{\I_s}$ is an element of~$\G$ which intertwines~$\t$ with~$\t'$; then~$g_{\I_s}$ intertwines~$\t_{\I_s}$ with~$\t'_{\I_s}$. Applying Lemma~\ref{propSOintertwiningInvertibleBeta} to~$\t_{\I_s,-}$ and~$\t'_{\I_s,-}$, we see that~$\b_{\I_s}^*(h)$ is isometric to~$\b_{\I_s}'^*(h)$ by an element of determinant congruent to~$\det(g_{\I_s})$ modulo~$\p_\F$. Summing the blocks,~$\b^*(h)$ is isometric to~$\b'^*(h)$ by an element of determinant congruent to~$\det(g)$ modulo~$\p_\F$. The result now follows since~$\b^*(h)$ and~$\b'^*(h)$ are symplectic forms and any isometry of a symplectic space has determinant~$1$.}
\shaun{We write~$\I=\bigcup_{\J\in\mathcal{P}} \J$ in the coarsest way such that, for each~$\J$, all \rob{elements}~$\b_j$ with~$j\in\J$ have the same valuation with respect to~$\La$. By Corollary~\ref{cor:Intertwining}, there are~$g_{\J}\in \U(\V^{\J},h_{\J})$ such that~$g=\sum_{\J\in\mathcal{P}} g_{\J}$ is an element of~$\G$ which intertwines~$\t$ with~$\t'$; then~$g_{\J}$ intertwines~$\t_{\J}$ with~$\t'_{\J}$. Applying Lemma~\ref{propSOintertwiningInvertibleBeta} to~$\t_{\J,-}$ and~$\t'_{\J,-}$, we see that~$\b_{\J}^*(h)$ is isometric to~$\b_{\J}'^*(h)$ by an element of determinant congruent to~$\det(g_\J)$ modulo~$\p_\F$. Summing the blocks,~$\b^*(h)$ is isometric to~$\b'^*(h)$ by an element of determinant congruent to~$\det(g)$ modulo~$\p_\F$. The result now follows since~$\b^*(h)$ and~$\b'^*(h)$ are symplectic forms and any isometry of a symplectic space has determinant~$1$.}
%\orange{We write the index set~$\I$ in the coarsest way as the union of subsets~$S$ such that the elements~$\b_s$ all have the same valuation with respect to~$\La$, for all~$s\in S$.  We now apply Lemma~\ref{propSOintertwiningInvertibleBeta} to all blocks~$S$ and we get the result.}
\end{enumerate}
\end{proof}

We deduce an analogue of Corollary~\ref{cor:IntertwiningEquivalenceRelG} (transitivity of intertwining) for special orthogonal groups.

\begin{corollary}\label{cor:IntertwiningEquivalenceRelGo}
Suppose~$\t^{(l)}_-\in\Cc_-(\La^{(l)},r^{(l)},\b^{(l)})$, for~$l=1,2,3$, are self-dual semisimple characters of the same degree such that~$\t^{(1)}_-$ intertwines with~$\t^{(2)}_-$ in~$\G^\so$, and~$\t^{(2)}_-$ intertwines with~$\t^{(3)}_-$ in~$\G^\so$. Suppose that $\left\lfloor\frac{r^{(l)}}{e(\La^{(l)}_{{\E^{(l)}}})}\right\rfloor$ is independent of~$l$. Then~$\t^{(1)}_-$ and~$\t^{(3)}_-$ intertwine in~$\G^\so$. 
\end{corollary}

\begin{proof}
\shaun{%
Changing the lattice sequences in their affine class, we can assume that they all have the same~$\o_\F$-period; then Theorem~\ref{thm:Endopss}\ref{thm:Endopss-iia1} implies that~$e(\La^{(l)}_{{\E^{(l)}}})$ is also independent of~$l$. Set~$r_{\min}=\min\{r_i\mid i=1,2,3\}$ and~$r_{\max}=\max\{r_i\mid i=1,2,3\}$, so that we can restrict the characters to~$\Cc_-(\La^{(l)},r_{\max},\b^{(l)})$ and extend them uniquely to~$\Cc_-(\La^{(l)},r_{\min},\b^{(l)})$ without changing endo-class (i.e.\ we pass to their transfers). Moreover, by Corollary~\ref{cor:IntertwiningEquivalenceRelG}, any two of these intertwine by an element of~$\G$. Then, if any~$\b^{(l)}$ has a zero component, then~$\t_-^{(1)}$ and~$\t_-^{(3)}$ intertwine by an element of~$\G^\so$ by Theorem~\ref{thmSOintertwining}\ref{thmSOintertwining.i}.
}

\shaun{%
Suppose now that none of the~$\b^{(l)}$ has a zero component. Applying Theorem~\ref{thmSOintertwining}\ref{thmSOintertwining.ii} to the restrictions to~$\Cc_-(\La^{(l)},r_{\max},\b^{(l)})$, we have that
\[
\b^{(1)*}(h)\cong  \b^{(2)*}(h) \cong \b^{(3)*}(h)
\]
by isometries of determinant congruent to~$1$ modulo~$\p_F$. Then, by Theorem~\ref{thmSOintertwining}\ref{thmSOintertwining.ii} again, the extensions to~$\Cc_-(\La^{(l)},r_{\min},\b^{(l)})$ all intertwine by an element of~$\G^\so$. In particular, their restrictions~$\t_-^{(1)}$ and~$\t_-^{(3)}$ intertwine by an element of~$\G^\so$.
}
\ignore{
At first we prove the case~$r^{(1)}=r^{(2)}=r^{(3)}$. If~$\b^{(1)}$ or~$\b^{(3)}$ has a zero-component then~$\t_-^{(1)}$ and~$\t_-^{(3)}$ intertwine by an element of~$\G^\so$ by Corollary~\ref{cor:IntertwiningEquivalenceRelG} and Theorem~\ref{thmSOintertwining}. If all three elements~$\b^{(l)}$ have no zero-component, then by~\emph{loc.cit.} we have that 
\[
h^{\b^{(1)}}\simeq  h^{\b^{(2)}} \simeq h^{\b^{(3)}}.
\]
by isometries of determinant congruent to~$1$ modulo~$\p_F$. Thus by~\emph{loc.cit.} and Corollary~\ref{cor:IntertwiningEquivalenceRelG} we obtain the result.
} 
\ignore{%
The general case: \orange{We restrict the characters to} $\Cc_-(\La^{(l)},\max(r^{(i)}|i=1,2,3),\b^{(l)})$ and then~$ h^{\b^{(1)}}\simeq h^{\b^{(3)}}$ by an \orange{element of determinant congruent to~$1$ mod~$\mathfrak{p}_\F$} if~$\b^{(1)}$ has no zero-component, by~\emph{loc.cit.}. Further the canonical embeddings of~$\b^{(1)}$ and~$\b^{(3)}$ are~$\z_{\Th^{(1)}_-,\Th^{(3)}_-}$-concordant as shown in Corollary~\ref{cor:IntertwiningEquivalenceRelG}. Now the extensions of~$\t^{(1)}$ and~$\t^{(3)}$ to~$\Cc_-(\La^{(l)},\min(r^{(i)}|i=1,3),\b^{(l)})$ intertwine by an element of~$\G$ by Theorem~\ref{thmEndoSemisimplev4}\ref{part3endosemi} and thus intertwine in~$\G^\so$ by Theorem~\ref{thmSOintertwining}.
}
\end{proof}

%%%%%%%%%%%%%%%%%%%%%%%%%%%%%%%%%%%%
\subsection{Conjugacy of self-dual semisimple characters}
%%%%%%%%%%%%%%%%%%%%%%%%%%%%%%%%%%%%

There exist self-dual semisimple characters for~$\G^\so$ which intertwine in~$\G^\so$, are conjugate in~$\G$, but are not conjugate in~$\G^\so$. For example, let~$[\La,n,r,\b]$ be a self-dual semisimple stratum in~$\A$ with associated splitting~$\V=\bigoplus_{i\in\I}\V^i$ and let~$\t\in\Cc^{\Sigma}(\La,r,\b)$ be a semisimple character.  Suppose there exists~$i_0\in\I$ with~$\b_{i_0}=0$, such that~$\bob{\P_-(\La^{i_0})}$ has no element of determinant~$-1$ in its normalizer. Take an element~$g_{i_0}$ of determinant~$-1$ in~$\U(\V^{i_0},h_{i_0})$ and~$g_i=\id$ for~$i\ne i_0$, and put~$g=\sum_{i\in\I} g_i$. Then~$\t$ and~$\t'= \presuper{g}\t$ intertwine by an element of~$\G^\so$ by Theorem~\ref{thmSOintertwining}, but they are not conjugate by an element of~$\G^\so$ by an exercise using Theorem~\ref{thmSOintertwining} and Proposition~\ref{prop:ConjIdempToEachOther}\ref{prop:ConjIdempToEachOther-ii}.  %We fix this problem in the following way: 

Nonetheless, we do have the following intertwining implies conjugacy theorem.

\begin{theorem}\label{thmSOIntertwiningConjugacy}
Suppose that~$e(\La)=e(\La')$, let~$[\La,n,r,\b]$ and\rob{~$[\La',n,r,\b']$} be self-dual semisimple strata in~$\A$ with splittings~$\V=\bigoplus_{i\in\I}\V^i$ and~$\V=\bigoplus_{i\in\I'}\V'^i$, and let~$\t\in\Cc_-(\La,r,\b)$ and~$\t'\in\Cc_-(\La',r,\b')$ be self-dual semisimple characters which intertwine in~$\G^\so$ with matching~$\z:\I\to\I'$. Suppose there exists~$g\in\G$ such that~$g\La^i=\La'^{\z(i)}$, for all~$i\in\I$, and 
%Suppose~$\t$ and~$\t'$ intertwine under an element~$g$ of~$\G^\so$ 
%Let~$[\La,n,r,\b]$ and~$[\La',n,r,\b']$ be self-dual semisimple strata in~$\A$ with respective splittings~$\V=\bigoplus_{i\in\I}\V^i$ and~$\V=\bigoplus_{i\in\I'}\V'^i$, and let~$\t\in\Cc_-(\La,r,\b)$ and~$\t'\in\Cc_-(\La',r,\b')$ be self-dual semisimple characters which intertwine in~$\G^\so$ with matching~$\z:\I\to\I'$. Suppose there exists~$g\in\G$ such that~$g\La^i=\La'^{\z(i)}$, for all~$i\in\I$, and 
%Suppose that~$g\in\I_{\G^\so}(\t,\t')$ and there exists~$g'\in\G$ such that~$g'\La^i=\La'^{\z(i)}$. We write \[g=u\diag(g_i|i\in \I_+\cup \I_0)v,\ v\in\SS_r(\b,\La)\cap\G,\ u\in\SS_r(\b',\La')\cap\G\] using Corollary~\ref{cor:Intertwining}. 
one of the following two assertions:
\begin{enumerate}\setlength\itemsep{5pt}
%\item The element~$\b$ has a zero component~$\b_{i_0}$,~$\La^{i_0}$ and~$\t_{i_0}$ have the same normalizer in~$\U(h_{i_0})$,~$g'\in \G^{\so}$ and~$ g_{i_0}^{-1}g'_{i_0}$ is an element of~$\SU(h_{i_0})$. 
%\item The element~$\b$ has a zero component, say~$\b_{i_0}$, and the normalizer of~$\La^{i_0}$ contains an element of~$\U(\V^{i_0},h_{i_0})$ with determinant~$-1$. 
%\item The element~$\b$ has no zero component.
\item\label{thmSOIntertwiningConjugacy.i} there is an~$i_0$ such that~$\b_{i_0}=0$, and~$\P_-(\La^{i_0})$ contains an element of determinant~$-1$; or
\item\label{thmSOIntertwiningConjugacy.ii} $\b_i\ne 0$, for all~$i\in\I$.
\end{enumerate}
Then~$\t$ is conjugate to~$\t'$ by an element of~$\G^\so\cap \P_-(\La)$.
%Then~$\t$ is conjugate to~$\t'$ by an element~$g\in\G^\so$ such that~$g\La=\La'$.
\end{theorem}

%\begin{remark}
%Unfortunately the theorem is a little unsatisfying, because in the first part we have the condition that~$\La^{i_0}$ and~$\t_{i_0}$ have the same normalizer in~$\G$. In general, this is a proper condition; there are many examples where this fails. For example, \blue{consider~$\mathrm{O}(1,1)(\F)=\U(h)$ and a lattice sequence~$\La$ with~$e(\La)=3$, such that~$\La^\#(j)=\La(-j)$ for all~$j\in\mathbb{Z}$ and such that~$\mathfrak{a}_\La$ is an Iwahori order of~$\GL_2(\F)$.} Then,~$\mathfrak{a}_2$ is the radical of a maximal order and thus the trivial character of~$\P_-^2(\La)$ has a normalizer in~$\mathrm{O}(1,1)(F)$ bigger than the normalizer of~$\La$.  But, in spite of the example, the normalizer condition holds if~$r=0$, we call this later~\emph{full}, because~$\mathfrak{a}_1$ is the radical of~$\La$ and both have the same normalizer. 
%\end{remark}

\begin{remark}\label{remIICforCharsSOvertex}
If~$\La$ is a self-dual lattice sequence which corresponds to a vertex in the Bruhat--Tits building of~$\G$, then there is an element of~$\G\setminus \G^\so$ in the normalizer of~$\La$. In particular, in the situation of Theorem~\ref{thmSOIntertwiningConjugacy}, if~$\P^\so_-(\La^{i_0})$ is a maximal parahoric subgroup of~$\U(\V^{i_0},h_{i_0})$, then condition~\ref{thmSOIntertwiningConjugacy.i} is satisfied so~$\t$ is conjugate to~$\t'$ by an element~$g\in\G^\so$ such that~$g\La=\La'$.
\end{remark}

\begin{proof}[Proof of Theorem~\ref{thmSOIntertwiningConjugacy}]
By Theorem~\ref{thmIntImplConjSelfDual}\ref{thmIntImplConjSelfDual.ii}, there is a~$y\in\G$ which conjugates~$\t$ to~$\t'$.
\begin{enumerate}\setlength\itemsep{5pt}
%\item We can assume~$g'$ is the identity. Therefore $\z$ is the identity and~$\La^i$ and~$\La'^i$ equal.  In particular, on the block for the zero component of~$\b$ the characters equal, i.e.~$\t_{i_0}=\t'_{i_0}$. There is nothing to prove if~$\La^{i_0}$ has an element of determinant~$-1$ in its normalizer in $\U(h_{i_0})$. Thus let us assume the opposite, which by assumption implies that the~$\U(h_{i_0})$-normalizer of~$\t_{i_0}$ is also contained in~$\SU(h_{i_0})$. By  Theorem~\ref{thm:IntImplConjSelfDual},~$\t$ and~$\t'$ are thus conjugate by an element of~$\P^-(\La)$ \orange{preserving every~$\V^i$}. 
%Corollary~\ref{cor:Intertwining} and 
%Theorem~\ref{thmSOintertwining} now implies that this conjugating element must have determinant~$1$.
\item If~$g_{i_0}\in\P_-(\La^{i_0})$ has determinant~$-1$, put~$g_i=\id$, for~$i\ne i_0$, and~$g=\sum_{i\in\I} g_i$. Then~$g$ is an element of determinant~$-1$ which normalizes~$\t$. Then~$y$ and~$yg$ both conjugate~$\t$ to~$\t'$ and one of them lies in~$\G^\so$.
%
%This follows directly from Theorem~\ref{thm:IntImplConjSelfDual}. 
\item Since~$y$ intertwines~$\t$ with~$\t'$, it lies in~$\G^\so$ by Theorem~\ref{thmSOintertwining}\ref{thmSOintertwining.ii}.
%The determinant of every conjugating element (\blue{which exists by Theorem~\ref{thm:IntImplConjSelfDual}}) \blue{has to be in~$\G^\so$ by Theorem~\ref{thmSOintertwining}.}
\end{enumerate}
\end{proof}

%% file: Endo-IIC.tex
%%%%%%%%%%%%%%%%%%%%%%%%%%%%%%%%%%%%

\section{Intertwining implies conjugacy for cuspidal types}\label{sectionIIC}

We recall the construction of cuspidal types for~$\G^\so$ of \cite{St08}, or more precisely its extension to representations over~$\CC$ in \cite{RKSS}, and then prove that two cuspidal types for~$\G^\so$ intertwine in~$\G^\so$ if and only if they are conjugate in~$\G^\so$. This completes the classification by types of the irreducible cuspidal representations of~$\G^\so$. 
\shaun{In the whole section we assume that~$\G^\so$ has compact centre, i.e. that~$\G$ is not~$\F$-isomorphic to~$\O(1,1)(\F)$.}

\shaun{Let~$[\La,n,0,\b]$ be a \emph{skew} semisimple stratum with index set~$\I=\I_0$, and~$\E=\F[\b]=\bigoplus_{i\in\I}\E_i$. We write~$\bb_n=\bb_n(\La)$ for the intersection of the~$\o_\F$-lattice~$\aa_n=\aa_n(\La)$ with the centralizer~$\B=\B_\b$ of~$\b$, so that~$\bb_n=\bigoplus_{i\in\I}\bb^i_n$. The quotient~$\P_-(\La_\E)/\P^1_-(\La_\E)$ is the set of rational points of the reductive group (defined over~$\mathrm{k}_\so$)
\begin{equation}\label{eqReductiveGroup}
\prod_{i\in\I}\Res_{\mathrm{k}_{\E_{i,\so}}|\mathrm{k}_\so} (\mathbb{U}(\ov{\phantom{a}}^i)),
\end{equation}
where~$\mathbb{U}(\ov{\phantom{a}}^i)$ is the reductive group defined by the anti-involution which is the restriction of~$\ov{\phantom{a}}$ to~$\bb_0^i/\bb_1^i$ and~$\Res_{\mathrm{k}_{\E_{i,\so}}|\mathrm{k}_\so}$ is the Weil restriction. Recall that the \emph{parahoric group}~$\P^{\so}_-(\La_\E)$ is the pre-image of the set of~$\mathrm{k}_\so$-rational points of the neutral component of~\eqref{eqReductiveGroup}.}

%\rob{We note that~$(\G^\so)_\beta=(\G_\beta)^\so$, because the only~$\sigma$-orbit~$\J$ in~$\I$ with~$(\G_\J)_{\beta_\J}$ non-connected  singletons with~$\beta_\J=0$ ....  (this shows that~$(\G_\beta)^\circ=\prod_\J \G_J\times \prod(\G_\J^\circ)_{\beta_\J=0}$.... but how does one see that this equals $(\G^\so)_\beta$?????); we write~$\G^\so_\beta=(\G^\so)_\beta$.}
\shauny{We note also that~$(\G^\so)_\beta=(\G_\beta)^\so$, since there is at most a single~$i\in\I$ such that~$\b_i=0$; we may therefore unambiguously denote this group by~$\G^\so_\beta$.}
%\red{Define parahoric subgroups of~$\B_\b^\times$.}
%\subsection{Cuspidal types and cuspidal representations}
\begin{definition}
A skew semisimple stratum~$[\La,n,0,\b]$ is called \emph{cuspidal} if~\rob{$\G^{\so}_\b$} has compact centre and~$\P^{\so}_-(\La_\E)$ is a maximal parahoric subgroup in~\rob{$\G_\b^\so$}. 
\end{definition}

\shaun{
The property of being cuspidal depends only on the equivalence class of the stratum and we have the following stronger result. 
\begin{proposition}\label{propCuspidalStratum}
Suppose~$[\La,n,0,\b]$ and~$[\La,n,0,\b']$ are skew semisimple strata such that~$\Cc(\La,0,\b)=\Cc(\La,0,\b')$. Then~$[\La,n,0,\b]$ is cuspidal if and only if~$[\La,n,0,\b']$ is cuspidal.
% who define the same set of semisimple characters. Then the first stratum is cuspidal if and only if the second is. 
\end{proposition}
}

\shaun{%Before beginning the proof, we recall some results on the definition of~$\P^{\so}_-(\La_\E)$, so let~$[\La,n,0,\b]$ be a skew semisimple stratum with index set~$\I=\I_0$, and~$\E=\F[\b]=\bigoplus_{i\in\I}\E_i$. We write~$\bb_n=\bb_n(\La)$ for the intersection of the~$\o_\F$-lattice~$\aa_n=\aa_n(\La)$ with the centralizer~$\B=\B_\b$ of~$\b$, so that~$\bb_n=\bigoplus_{i\in\I}\bb^i_n$. The quotient~$\P_-(\La_\E)/\P^1_-(\La_\E)$ is the set of rational points of the reductive group (defined over~$\mathrm{k}_\so$)
%\begin{equation}\label{eqReductiveGroup}
%\prod_{i\in\I}\Res_{\mathrm{k}_{\E_{i,\so}}|\mathrm{k}_\so} (\mathbb{U}(\ov{\phantom{a}}^i)),
%\end{equation}
%where~$\mathbb{U}(\ov{\phantom{a}}^i)$ is the reductive group defined by the anti-involution which is the restriction of~$\ov{\phantom{a}}$ to~$\bb_0^i/\bb_1^i$ and~$\Res_{\mathrm{k}_{\E_{i,\so}}|\mathrm{k}_\so}$ is the Weil restriction. The group~$\P^{\so}_-(\La_\E)$ is the pre-image of the set of~$\mathrm{k}_\so$-rational points of the neutral component of~\eqref{eqReductiveGroup}. 
We need the following straightforward result on algebraic groups.
\begin{lemma}\label{lemFiniteGpofLieType}
Let~$\mathrm{k}$ be a finite field of odd characteristic~$p$, with involution~$\ov{\phantom{a}}$ on~$\mathrm{k}$ with fixed point set~$\mathrm{k}_\so$. Let~$\mathbb{H}$ be an algebraic group defined over~$\mathrm{k}_\so$ and~$\mathrm{k}_\so$-isomorphic to a symplectic, a unitary or an orthogonal group over~$\mathrm{k}_\so$. Suppose~$\mathbb{H}(\mathrm{k}_\so)$ is isomorphic to~$\O(1,1)(\mathrm{k})$ as an abstract group. Then~$\mathrm{k}=\mathrm{k}_\so$ and~$\mathbb{H}$ is~$\mathrm{k}$-isomorphic to~$\O(1,1)$. 
\end{lemma}
\begin{proof}
Put~$q=|\mathrm{k}|$ and~$q_\so=|\mathrm{k}_\so|$% \bob{and assume that the conclusion of the lemma does not hold}
. If~$\mathbb{H}$ is isotropic \bob{and not~$\mathrm{k}_\so$-isomorphic to~$\O(1,1)$} then the cardinality of~$\mathbb{H}(\mathrm{k}_\so)$ is divisible by~$p$, because the latter contains a unipotent group of cardinality~$p$, so we are left with the cases~$\U(1)$,~$\O(2)$\bob{,~$\O(1)$ and~$\O(1,1)$}. \shauny{But~$\U(1)(\mathrm{k}/\mathrm{k}_\so)$ is cyclic while~$\O(1,1)(\mathrm{k})$ is not; and the groups~$\O(2)(\bob{\mathrm{k}_\so})$ and~$\O(1)(\bob{\mathrm{k}_\so})$ have cardinalities~$2(\bob{q_\so}+1)$ and~$2$ which both differ from the cardinality~$2(q-1)$ of~$\O(1,1)(\mathrm{k})$. Thus~$\mathbb{H}$ is~$\mathrm{k}_\so$-isomorphic to~$\O(1,1)$ and, comparing cardinalities,~$\mathrm{k}=\mathrm{k}_\so$.}
\end{proof}
}

\shaun{
For the proof of Proposition~\ref{propCuspidalStratum} we need to recall some more of the data attached to a semisimple stratum~$[\La,n,0,\b]$, in particular the~$\o_\F$-order~$\JJ(\b,\La)$ (see~\cite[Section 9.1]{SkSt} and~\cite[3.1.8]{BK93}). It is the additive group generated by the intersection of~$\aa_0$ with the~$\tG$-intertwining of any semisimple character~$\t\in\Cc(\La,0,\b)$; it contains~$\bb_1$ and, writing~$\JJ^1(\b,\La)$ for its intersection with~$\aa_1$, we have a canonical isomorphism~$\JJ(\b,\La)/\JJ^1(\b,\La)\simeq\bb_0(\La)/\bb_1(\La)$. We have a chain of compact open subgroups of~$\G$
 \[
 \J^+_-(\b,\La)\supseteq\J_-(\b,\La)\supseteq\J^\so_-(\b,\La)\supseteq\J^1_-(\b,\La)\supseteq\H^1_-(\b,\La),
 \]
 with the first two defined as the intersection of~$\JJ(\b,\La)$ with~$\G$ and~$\G^\so$ respectively,~$\J^1_-(\b,\La)$ its intersection with~$\P^1_-(\La)$, and~$\J^\so_-(\b,\La)=\P^{\so}_-(\La_\E)\J^1_-(\b,\La)$; this is the inverse image in~$\J^+_-(\b,\La)$ of the connected component of
 \[
 \J^+_-(\b,\La)/\J^1_-(\b,\La) \simeq \P_-(\La_\E)/\P^1_-(\La_\E).
 \]
 }
\shaun{
Proposition~\ref{propCuspidalStratum} is an immediate consequence of the following lemma, in which we consider the reduced Bruhat--Tits building~$\BBred(\G_\b)$ (the product of the buildings~$\BBred(\G^i_{\b_i})$) with its weak simplicial structure, i.e. the facets of~$\BBred(\G^i_{\b_i})$ are the intersection of the facets of~$\BBred(\tG^i_{\b_i})$ with~$\BBred(\G^i_{\b_i})$. 
}

\shaun{
\begin{lemma}\label{lemCuspStratum}
Let~$[\La,n,0,\b]$ and~$[\La,n,0,\b']$ be skew semisimple strata such that~$\Cc(\La,0,\b)=\Cc(\La,0,\b')$. %Then we have
\begin{enumerate}\setlength\itemsep{5pt}
\item\label{lemCuspStratumi} $\JJ(\b,\La)=\JJ(\b',\La)$ and we have a canonical isomorphism~$\P_-(\La_\E)/\P^1_-(\La_\E)\simeq\P_-(\La_{\E'})/\P^1_-(\La_{\E'})$. 
\item\label{lemCuspStratumii} $\G_\b^\so$ has compact centre if and only if~$\G_{\b'}^\so$ has compact centre.  
\item\label{lemCuspStratumiii} Suppose~$\G_\b^\so$ has compact centre; then~$\La_\E$ corresponds to a vertex of~$\BBred(\G_\b)$ if and only if~$\La_{\E'}$ corresponds to a vertex of~$\BBred(\G_{\b'})$.
\item\label{lemCuspStratumiv} Suppose~$\G_\b^\so$ has compact centre and~$\La_\E$ corresponds to a vertex of~$\BBred(\G_\b)$; then~$\P^\so_-(\La_\E)$ is a maximal parahoric subgroup  of~$\G_\b$ if and only if~$\P^\so_-(\La_{\E'})$ is a maximal parahoric subgroup of~$\G_{\b'}$.
  \item\label{lemCuspStratumv} Let~$\z:\I\rightarrow\I'$ be the matching from~$(\t,\b)$ to~$(\t,\b')$ for some semisimple character~$\t\in\Cc^\Sigma(\La,0,\b)$. Then~$\mathrm{k}_{\E_i}$ and~$\mathrm{k}_{\E'_{\z(i)}}$ coincide in~$\aa_0/\aa_1$ and the canonical map 
%\[
%\prod_{i\in\I}\P_-(\La^i_{\E_i})/\P^1_-(\La^i_{\E_i})\simeq\prod_{j\in\I'}\P_-(\La^{j}_{\E'_{j}})/\P^1_-(\La^{j}_{\E'_{j}})
%\]
from~\ref{lemCuspStratumi} is a product of algebraic isomorphisms
\[
%f_i:
\P_-(\La^i_{\E_i})/\P^1_-(\La^i_{\E_i})\rightarrow\P_-(\La^{\z(i)}_{\E'_{\z(i)}})/\P^1_-(\La^{\z(i)}_{\E'_{\z(i)}})
\] 
defined over~$\mathrm{k}_{\E_{i,\so}}$.
\end{enumerate}
\end{lemma}
}

\begin{proof}
\shaun{
Assertion~\ref{lemCuspStratumi} follows immediately from the description of~$\JJ(\b,\La)$ as the additive group generated by the intersection of~$\aa_0$ with the~$\tG$-intertwining of any semisimple character in~$\Cc(\La,0,\b)$.
% directly from Remark~\ref{remJ}\ref{remJi}. 
The centre of~$\G_\b^\so$ is non-compact if and only if there is \rob{an} index~$i\in\I$ such that~$\b_i=0$,~$\dim_\F\V^i=2$ and the restriction~$h|_{\V^i}$ is isotropic and orthogonal. By Theorem~\ref{thm:MatchingForCharForG}, this is equivalent to~$\G_{\b'}^\so$ having non-compact centre, proving~\ref{lemCuspStratumii}. 
}

\shaun{
We now assume that the centre of~$\G_\b^\so$ is compact. By Proposition~\ref{prop:ConjIdempToEachOther} we can conjugate by some~$g\in\G$ which normalizes every character in~$\Cc(\La,0,\b)$ to reduce to the case where the splittings of~$\b$ and~$\b'$ coincide; thus we are in fact reduced to the case that~$\b$ and~$\b'$ are simple. We have the canonical maps
\[
\bb_0(\La)/\bb_1(\La)\hookrightarrow\aa_0(\La)/\aa_1(\La)\hookleftarrow\bb'_0(\La)/\bb'_1(\La)
\]
induced by the inclusions, which have the same image so we get an isomorphism
\[
\Psi:\bb_0(\La)/\bb_1(\La)\to\bb'_0(\La)/\bb'_1(\La). 
\]
The anti-involution of~$h$ restricts to the anti-involution of the form~$h_\b$ defining~$\G_\b$. The lattice sequence~$\La_\E$ corresponds to a vertex if and only if~$\bb_0(\La)/\bb_1(\La)$ has at most two central idempotents and they are fixed (not permuted) by the adjoint anti-involution on~$\aa_0/\aa_1$. Then~\ref{lemCuspStratumiii} follows because~$\Psi$ is an equivariant ring isomorphism. Now to prove assertion~\ref{lemCuspStratumiv}, suppose~$\La_\E$ and~$\La_{\E'}$ correspond to vertices. Then~$\P^\so_-(\La_\E)$ is not a maximal parahoric subgroup of~$\G_\b$ if and only if, for one of the central idempotents~$\fe$ of~$\bb_0(\La)/\bb_1(\La)$ the corresponding factor of~$\P_-(\La_\E)/\P^1_-(\La_\E)$ is given by the algebraic group~$\O(1,1)$ defined over~$\mathrm{k}_\E$; in that case the algebraic group defining the factor of~$\P_-(\La_{\E'})/\P^1_-(\La_{\E'})$ corresponding to~$\Psi(\fe)$ must be~$\O(1,1)$ by Lemma~\ref{lemFiniteGpofLieType}, since~$\mathrm{k}_\E=\mathrm{k}_{\E'}$ by Proposition~\ref{prop:SimpleDegrees}. % In particular~$\P^\so_-(\La_{\E'})$ would not be a maximal parahoric of~$\G_{\b'}$.
}

\shaun{
For the final assertion~\ref{lemCuspStratumv}, first the canonical embeddings of~$\mathrm{k}_\E$ and~$\mathrm{k}_{\E'}$ into~$\aa_0/\aa_1$ have the same image by~\cite[5.2]{BK94}. Further~$\P_-(\La_\E)/\P^1_-(\La_\E)$ is the set of rational points of the reductive group defined over~$\mathrm{k}_{\E_\so}$ defined by the anti-involution~$\ov{\phantom{a}}$ on~$\bb_0(\La)/\bb_1(\La)$. The map from~\ref{lemCuspStratumi} is~$\mathrm{k}_\E$-linear and preserves~$\ov{\phantom{a}}$, so it is an algebraic isomorphism defined over~$\mathrm{k}_{\E_\so}$.
}
\end{proof}

% Let~$[\La,n,0,\b]$ be a skew semisimple stratum.  In \cite[\S 3.2]{St05}, associated to~$[\La,n,0,\b]$ are compact open subgroups~$\J_-(\b,\La)\supseteq \J^1_-(\b,\La)$ of~$\G^\so$, \orange{the intersections of~$\JJ(\b,\La)$ and~$\JJ^1(\b,\La)$ with~$\G^\so$}.  Moreover,~$\J^1_-(\b,\La)\supseteq \H^1_-(\b,\La)$.  The quotient~$\J_-(\b,\La)/\J^1_-(\b,\La)\simeq \P_-(\La_\E)/\P_-^1(\La_\E)$ is a finite reductive group over~$\mathrm{k}_\so$, and we put~$\J^{\so}_-(\b,\La)$ the preimage of the~$\mathrm{k}_\so$-points of its connected component in~$\J_-(\b,\La)$.  

\shaun{
Now let~$[\La,n,0,\b]$ be a cuspidal skew semisimple stratum and let~$\t_-\in\Cc_{-}(\La,0,\b)$ be a skew semisimple character.   By \daniel{\cite[Corollary 3.29]{St05}}, there exists a unique irreducible representation~$\eta$ of~$\J^1_-(\b,\La)$ containing~$\theta_-$. The representation~$\eta$ extends to~$\J_-(\b,\La)$ and we call such an extension~$\k$ a \emph{$\b$-extension} if it extends further to~$\J^+_-(\b,\La)$ and its restriction to a pro-$p$-Sylow subgroup of~$\J_-(\b,\La)$ is intertwined by all of~$\I_\G(\t_-)$ (see\daniel{~\cite[after Theorem 4.1]{St08} and}~\cite[\S6]{RKSS}, where it is the extensions to~$\J^+_-(\b,\La)$ which are called~$\b$-extensions).
}
%and certain extensions satisfy a compatibility property when restricted to a pro-$p$-Sylow subgroup of~$\J_-(\b,\La)$, called~\emph{$\b$-extensions} \cite[\S 6]{RKSS}.  \orange{Note: A~$\b$-extension in~\emph{loc.cit.} and~\cite[Definition 4.5]{St08} is a certain extension of~$\eta$ to~$\J^+_-(\b,\La)$. Here we consider precisely the restriction of those to~$\J_-(\b,\La)$ and call these~$\b$-extensions of~$\eta$.}

\begin{definition}
A \emph{cuspidal type} for~$\G^\so$ is a pair~$(\J,\l)$ such that there exist a cuspidal skew semisimple stratum~$[\La,n,0,\b]$ and~$\theta_-\in\Cc_-(\La,0,\b)$, such that~$\J=\J_-(\b,\La)$ and~$\l=\kappa\otimes\tau$, with $\kappa$ a~$\b$-extension of the unique irreducible representation of~$\J^1_-(\b,\La)$ containing~$\theta_-$, and~$\tau$ an irreducible representation of $\J_-(\b,\La)/\J^1_-(\b,\La)$ with cuspidal restriction to~$\J^\so_-(\b,\La)/\J^1_-(\b,\La)$.
%\orange{We also say that the cuspidal type as above is~\emph{parametrized} by the cuspidal stratum~$[\La,n,0,\b]$.}
\end{definition}

%\orange{Analogously to cuspidal strata a cuspidal type can still be re-parametrized  as follows:

\begin{proposition}\label{propReparametrizationCuspType}
\shaun{
Let~$(\J,\l)$ be a cuspidal type for~$\G^\so$ defined via a cuspidal stratum~$[\La,n,0,\b]$, with~$\l=\k\otimes\tau$. Suppose~$[\La,n,0,\b']$ is a skew semisimple stratum such that~$\Cc(\La,0,\b)=\Cc(\La,0,\b')$. Then~$[\La,n,0,\b']$ is cuspidal,~$\J_-(\b',\La)=\J$,~$\kappa$ is a~$\b'$-extension and the restriction of~$\tau$ to~$\P^\so_-(\La_{\E'})$ is cuspidal. 
}
%Let~$(\J,\l=\kappa\otimes\tau)$ be a cuspidal type for~$\G^\so$ parametrized by a cuspidal stratum~$[\La,n,0,\b]$. Suppose~$[\La,n,0,\b']$ is a skew-semisimple stratum such that~$\Cc(\La,0,\b)=\Cc(\La,0,\b')$. Then~$[\La,n,0,\b']$ is cuspidal,~$\J_-(\b',\La)=\J$,~$\kappa$ is a~$\b'$-extension and the restriction of~$\tau$ to~$\P^\so_-(\La_{\E'})$ is cuspidal. 
\end{proposition}

\begin{proof}
\shaun{The stratum~$[\La,n,0,\b']$ is cuspidal by Proposition~\ref{propCuspidalStratum}, while~$\J_-(\b',\La)=\J$ by Lemma~\ref{lemCuspStratum}\ref{lemCuspStratumi}. Let~$\t_-\in\Cc_-(\La,0,\b)=\Cc_-(\La,0,\b')$ be a skew semisimple character contained in~$\l$; then the characterization of~$\b$-extensions in terms of the intertwining of~$\t_-$ implies that~$\k$ is also a~$\b'$-extension.
%Now~$\kappa$ is also a~$\b'$-extension because~$\La_{\E'}$ corresponds to a vertex in~$\BB(\G_{\b'})$, by Lemma~\ref{lemCuspStratum}\ref{lemCuspStratumiii}, and for vertexes the requirement on the~$\b'$-extension is just that~$\kappa$ is extendable to~$\J_-^+(\b',\La)$ and that the restriction to a pro-$p$-Sylow subgroup of~$\J$ is intertwined by~$\I(\t_-)$. Both properties only depend on~$\kappa$ and~$\La$. 
%The representation~$\l$ is contained in a cuspidal irreducible represention on~$G^\so$ by Theorem~\ref{thmRKSScusp} and therefore the representation~$\tau$ is cuspidal with respect to~$\La_{\E'}$ by Proposition~\cite[4.4]{MiSt}.
Moreover, the restriction of~$\tau$ to~$\P^\so_-(\La_{\E'})$ is cuspidal by Lemma~\ref{lemCuspStratum}\ref{lemCuspStratumv}.
}
\end{proof}

%
%Let~$\tau$ be an irreducible representation of~$\J_-(\b,\La)/\J^1_-(\b,\La)$ with cuspidal restriction to~$\J^{\so}_-(\b,\La)/\J^1_-(\b,\La)$.  Put~$\J=\J_-(\b,\La)$ and~$\l=\kappa\otimes\tau$.  A pair~$(\J,\l)$, constructed in this way is called a \emph{cuspidal type} for~$\G^\so$.  

The main result of \cite{RKSS}, generalizing the main result of \cite{St08} in the case~$\CC=\mathbb{C}$ (see also \cite[Appendix A]{MiSt} for a correction to the definition of cuspidal type given in \cite{St08}), can be stated as follows.

\begin{theorem}[{\cite[Theorems 12.1,  12.2]{RKSS}}]\label{thmRKSSexhaust}
Let~$(\J,\l)$ be a cuspidal type for~$\G^\so$.  Then the representation~$\ind_\J^{\G^\so}(\l)$ is irreducible and cuspidal.
Moreover, every irreducible cuspidal representation of~$\G^\so$ appears this way.
\end{theorem}

Thus, it remains to determine when cuspidal types~$(\J,\l)$ and~$(\J',\l')$ induce isomorphic cuspidal representations.  Notice that conjugate cuspidal types induce equivalent representations and if~$\ind_\J^{\G^\so}(\l)\simeq\ind_{\J'}^{\G^\so}(\l')$ then~$(\J,\l)$ and~$(\J',\l')$ intertwine in~$\G^\so$.  Hence the following result completes the classification in terms of conjugacy classes of cuspidal types.

%\subsection{Intertwining implies conjugacy}
%
\begin{theorem}\label{thmIIC}
Let~$(\J,\l)$ and~$(\J,\l')$ be cuspidal types for~$\G^\so$ which intertwine in~$\G^\so$. Then they are conjugate in~$\G^\so$.
%Cuspidal types intertwine in~$\G^\so$ if and only if they are conjugate in~$\G^\so$.
\end{theorem}

% \orange{For the proof of Theorem~\ref{IIC} we need to be able to re-parametrize a cuspidal type. For this purpose we 
% prove the following lemma. 
% }
% 
% \orange{
% \begin{lemma}\label{lemCuspTypeReparametrization}
%  Suppose~$[\La,n,0,\b]$ and~$[\La_,n,0,\b']$ are skew-semisimple strata. 
% \end{lemma}
% }

\begin{proof}
%Let~$\t_-\in\Cc_-(\La,0,\b)$ and~$\t'_-\in\Cc_-(\La',0,\b')$ be the underlying skew semisimple characters. 
As~$\l$ and~$\l'$ intertwine by an element of~$\G^\so$, the underlying skew semisimple characters~$\t_-\in\Cc_-(\La,0,\b)$ and~$\t'_-\in\Cc_-(\La',0,\b')$ intertwine by the same element, and we denote by~$\zeta:\I\rightarrow\I'$ the matching from~$(\t_-,\b)$ to~$(\t'_-,\b')$.
%for the intertwining Glauberman lifts of~$\t_-$ and~$\t'_-$ from Theorem~\ref{thm:MatchingForChar}. 
%and we denote~$\zeta:\I\rightarrow\I'$ the matching 
%from~$(\t_-,\b)$ to~$(\t'_-,\b')$\red{... we never say before matching between semisimple characters}. 
By Corollary~\ref{cor:diagonalSemisimple} we can replace the underlying strata so that there is an element~$g\in\G$ such that~$\b_i^{g^{-1}}=\b'_{\zeta(i)}$, for all~$i\in\I$, and by Proposition~\ref{propReparametrizationCuspType} these new strata are also cuspidal.
Now~$\t_-$ intertwines with~$\tau_{g\La,\La',\b'}(\t'_-)$ by an element of~$\G^\so$ by Corollaries~\ref{cor:IntertwiningEquivalenceRelG} and~\ref{cor:IntertwiningEquivalenceRelGo}. Thus, by~\cite[Theorem 10.3]{SkSt}, Theorem~\ref{thmSOIntertwiningConjugacy} and Remark~\ref{remIICforCharsSOvertex}, there is an element~$g_1\in\G^\so$ such that
\[
\tau_{g\La,\La',\b'}(\t'_-)=\t_-^{g_1^{-1}}.
\]
Then, using Proposition~\ref{propReparametrizationCuspType} again, we can assume that~$\t'_-$ is the transfer of~$\t_-$ to~$\La'$ with respect to~$\b'$, and the result now follows from~\cite[Theorem 12.3]{RKSS}.
\end{proof}

\begin{remark}
\shaun{
The proof of Proposition~\ref{propReparametrizationCuspType} could have been given without using Proposition~\ref{propCuspidalStratum}, using~\cite[Proposition~4.4]{MiSt} instead (and its analogue when~$\CC$ has positive characteristic~$\ell\ne p$); however, this uses the full strength of the exhaustion proof in~\cite[Section~7]{St08},~\cite[Appendix~A]{MiSt} and~\cite[Theorem~12.2]{RKSS}, while our approach here is more direct, and independent of the proof of exhaustion.   
}
\end{remark}

%% file: Endo-parameters.tex
%%%%%%%%%%%%%%%%%%%%%%%%%%%%%%%%%%%%

%%%%%%%%%%%%%%%%%%%%%%%%%%%%%%%%%%%%
\section{Endo-parameters}\label{secParametrizationGL}
In this section we are only interested in semisimple characters which are defined on~$\H^1(\b,\La)$ -- we call such semisimple characters \emph{\full}. We parametrize the set of intertwining-classes of (self-dual) semisimple characters for~$\tG$ (and for~$\G$) by arithmetic parameters which we call \emph{endo-parameters}. The restriction to \full\ semisimple characters is natural because every smooth representation of~$\G^\so$ contains a \full\ self-dual semisimple character \cite[Proposition 8.5]{Finitude}, though more refined arithmetic information could be sought by looking more generally (cf.~\cite{BH17}). %%I mean MR3664814

%%%%%%%%%%%%%%%%%%%%%%%%%%%%%%%%%%%%
\subsection{Endo-parameters for~$\tG$}
Here we parametrize~$\tG$-intertwining classes for \full\ semisimple characters for~$\tG$. First we state the definition of \full\ precisely. 

\begin{definition}
\begin{enumerate}\setlength\itemsep{5pt}
\item A semisimple character~$\t$ is called~\emph{\full}~\bob{if there exists a semisimple stratum~$[\La,n,0,\b]$ such that~$\t\in\Cc(\La,0,\b)$.}
\item A pss-character is called~\emph{\full}\ if it is supported on a semisimple pair of the form~$(0,\b)$.
\item An endo-class of pss-characters is called~\emph{\full}\ if it consists of~\full\ pss-characters; note that this includes the \full\ zero endo-class~$\mathbf{0}$.
\end{enumerate}
\end{definition}

%\shaun{Note also that any \gre{\full} semisimple character, although it is the realization of different pss-characters, determines a \gre{\full} endo-class. We will therefore say that two \gre{\full} semisimple characters are \emph{endo-equivalent} if they determine the same \gre{\full} endo-class.}
\shauny{Note that any~\full\ semisimple character~$\t$ is a realization of different pss-characters: for example, if~$\t$ is a realization of a pss-character supported on a semisimple pair~$(0,\b)$, then it is also a realization (on the same lattice sequence) of a pss-character supported on~$(0,\b+\varpi_\F)$. Moreover,~$\t$ can be a realization of different pss-characters supported on the same semisimple pair: for example, if~$\t,\t'\in\Cc(\La,0,\b)$ are as in Remark~\ref{remark:CyclicPermutationSemisimpleChar}, then $\t$ is conjugate to~$\t'$ by an element of~$\P(\La)$, by Theorem~\ref{thmIntImplConjSelfDual}\ref{thmIntImplConjSelfDual.i}; then the pss-characters~$\Th$ and~$\Th'$, supported on~$(0,\b)$, with realizations~$\t$ and~$\t'$ at~$(\V,\id_{\F[\b]},\La,0)$, respectively, both attain~$\t$.}

%\daniel{
%A full semisimple character~$\t$ can be parametrized by different semisimple strata and in general even by non-equivalent ones. In this way~$\t$ can be realized by pss-characters supported on different semisimple pairs. Even more: If~$\Th$ and~$\Th'$ are pss-characters supported on~$(0,\b)$ and attain~$\t$, then it doesn't imply that~$\Th$ and~$\Th'$ coincide. Take for 
%example~$\t,\t'\in\Cc(\La,0,\b)$ in Remark~\ref{remark:CyclicPermutationSemisimpleChar} (for~$n=2$). $\t$ is conjugate to~$\t'$ by an element of~$\P(\La)$, see Theorem~\ref{thmIntImplConjSelfDual}\ref{thmIntImplConjSelfDual.i}. Now~$\Th$ and~$\Th'$ realizing~$\t$ and~$\t'$ at~$(\V,\id_{\F[\b]},\La,0)$, respectively, attain~$\t$.}

\daniel{Nonetheless, any~\full\ semisimple character determines a~\full\ endo-class.} 
\shaun{We will therefore say that two \gre{\full} semisimple characters are \emph{endo-equivalent} if they determine the same \gre{\full} endo-class.}

Since every ps-character is a pss-character, we can also make the following definition.
\begin{definition}
We let~$\Ee$ denote the set of all~\full\ endo-classes of ps-characters, and let~$\Ee^{\fin}$ denote the set of finite subsets of~$\Ee$. 
\end{definition}

\shaun{We know that we can decompose pss-characters into ps-characters via Lemma~\ref{lemma:pssdecompositions}. This extends to give us a bijection between \full\ semisimple endo-classes and~$\Ee^{\fin}$.}

\begin{proposition}\label{propSemisimpleEndoclassesAsTupleOfSimpleEndo}
\shaun{
The map~$\Ff$ from the set of all \full\ semisimple endo-classes to~$\Ee^{\fin}$, defined by
\[
\Ff([\Th])=\{[\Th_i]\mid i\in\I\}
\]
is well-defined and bijective \daniel{where the~$\Th_i$ are defined in Lemma~\ref{lemma:pssdecompositions}.}
}
\end{proposition}

For the proof (of surjectivity) we need the following \rob{lemmas}.

\begin{lemma}\label{lemSquareExtensionOfASemisimpleChar}
\shaun{
Let~$[\La,n,r,\b]$ \daniel{and~$[\La,n,r+1,\tilde{\g}]$ be semisimple strata split by~$\V={\V'}\oplus {\V''}$ such that~$[\La,n,r+1,\b]$ is equivalent to~$[\La,n,r+1,\tilde{\g}]$.} Put~${\La'}=\La\cap{\V'}$ and~${\b'}=\b|_{\V'}$, and similarly for~$\La'',\b''$. Given~$\tilde{\t}\in\daniel{\Cc(\La,r+1,\tilde{\g})}$ and an extension~${\t'}\in\Cc({\La'},r,{\b'})$  of~$\tilde{\t}|_{\H^{r+2}({\b'},{\La'})}$ there is a semisimple character~$\t\in\Cc(\La,r,\b)$ such that the restriction to~$\H^{r+1}({\b'},{\La'})$ is~${\t'}$ and the restriction to~$\H^{r+2}(\b,\La)$ is~$\tilde{\t}$.
}
\end{lemma}

\begin{proof}
\shaun{
By induction on the number of simple blocks of~$[{\La''},{n''},r,{\b''}]$, we are reduced to the case where~$[{\La''},{n''},r,{\b''}]$ is simple. The proof proceeds by induction along the critical exponent~$k_0=k_0(\b,\La)$. If~$\b=0$ then we take~$\t$ the trivial character of~$\P^{r+1}(\La)$. Thus we can assume~$\b\neq0$; in particular~$-k_0\leq n$. Write~$\V'=\bigoplus_{i\in\I'}\V'^i$ for the splitting of~$[{\La'},{n'},r,{\b'}]$; then the splitting of~$[\La,n,r,\b]$ is either~$\V=\V''\oplus\bigoplus_{i\in\I'}\V'^i$ (if~$\b''$ is a simple block of~$\b$) or~$\V=\bigoplus_{i\in\I'}\V^i$, where
\begin{equation}\label{eqn:secondcase}
\V^i = \begin{cases} \V'^i\oplus\V'' &\text{ if }i=i_0, \\ 
 \V'^i &\text{ otherwise,} \end{cases}
\end{equation}
for some (unique) value~$i_0\in\I'$.
}

\shaun{
First we consider the case where~$r<\lfloor\frac{-k_0}{2}\rfloor$. If~$\b''$ is a simple block of~$\b$ then~\cite[Lemma~3.15]{St05} says that there exists~$\t\in\Cc(\La,r,\b)$ whose simple block restrictions are an extension of~$\tilde{\t}|_{\H^{r+2}(\b'',\La'')}$ to~$\H^{r+1}(\b'',\La'')$, and the simple block restrictions of~$\t'$; since a semisimple character is determined by its simple block restrictions, we are done.
}

\shaun{
Otherwise we are in the second case~\eqref{eqn:secondcase} above, and~$\b'_{i_0}$ and~$\b''$ have the same minimal polynomial over$~\F$. If we denote by~$\t''\in\Cc(\La'',r,\b'')$ the transfer of the simple block restriction~$\t'_{i_0}$, then the transfer of~$\t'_{i_0}$ to~$\Cc(\La'^{i_0}\oplus\La'',r,\b'_i+\b'')$ restricts to~$\t'_{i_0}\otimes\t''$ on~$\H^{r+1}(\b'_{i_0},\La'^{i_0})\times\H^{r+1}(\b'',\La'')$. Applying~\cite[Lemma~3.15]{St05} as above, we find a semisimple character as required.
}

\shaun{
Now suppose~$r\geq\lfloor\frac{-k_0}{2}\rfloor$. Then we take a semisimple stratum~$[\La,n,-k_0,\g]$ equivalent to~$[\La,n,-k_0,\b]$ which is split under the decomposition~$\V''\oplus\bigoplus_{i\in\I'}\V'^i$. We write~$\g'$ for~$\g|_{\V'}$. Since~$k_0(\g,\La)<k_0(\b,\La)$, we can apply the inductive hypothesis to find a common extension~$\t_\g\in\Cc(\La,r,\g)$ of~$\t'\psi_{\g'-\b'}$ and~$\tilde{\t}\psi_{\g-\b}$ to~$\H^{r+1}(\b,\La)=\H^{r+1}(\g,\La)$. Then the character~$\t=\t_\g\psi_{\b-\g}$ is as required.}
\end{proof}

\begin{lemma}\label{lem:changesecondonly}
\shaun{
Let~$[\La',n',r,\b']$ be a semisimple stratum in~$\V'$ and let~$[\La'',n'',r,\b'']$ be a simple stratum in~$\V''$ such that~$e(\La')=e(\La'')$. Set~$\V=\V'\oplus\V''$,~$\La=\La'\oplus\La''$ and~$n=\max\{n',n''\}$. 
%If~$[\La,n,r,\b'\oplus\b'']$ is equivalent to a semisimple stratum in~$\V$ 
Then there exists~$\tilde\b''\in\End_\F(\V'')$ such that~$[\La'',n'',r,\tilde\b'']$ is simple and equivalent to~$[\La'',n'',r,\b'']$, and~$[\La,n,r,\b'\oplus\tilde\b'']$ is semisimple.
}
\end{lemma}

\begin{proof} 
%%%%%%%%%%%%%%%%%%%%%%%%%%%%%%%%%%%%
%%%%%%%%%%%%%%%%%%%%%%%%%%%%%%%%%%%%
\shaun{Suppose~$[\La,n,r,\b]$ is not semisimple. Then, using notation as in the previous proof, there is an index~$i_0\in\I'$ %($\I'$ the index set for~$\b'$) 
such that the stratum~$[\La'^{i_0}\oplus\La'',\max\{n_{i_0},n''\},r,\b'_{i_0}\oplus\b'']$ is equivalent to a simple stratum. We only need to find~$\tilde{\b}''$ such that~$[\La'^{i_0}\oplus\La'',\max\{n_{i_0},n''\},r,\b'_{i_0}\oplus\tilde{\b}'']$ is simple and we may therefore assume~$|\I'|=1$, that is~$\b'=\b_i$. Then~$\F[\b']/\F$ has the same ramification index and inertia degree as~$\F[\b'']/\F$. (This follows from~\cite[Theorem~6.16]{SkSt} and~\cite[Theorem~2.4.1]{BK93} when the lattice sequences are strict, and in general via a~$\dag$-construction.) Thus there is an embedding~$\vphi:\F[\b']\to\End_\F\V''$ such that~$\F[\vphi(\b')]^\times$ normalizes~$\La''$.}

\shaun{We set~$m'=\dim_\F\V'$ and~$m''=\dim_\F\V''$ and identify~$\V'^{\oplus m''}$ with~$\V''^{\oplus m'}$. Then the strata
\[
[\La''\oplus\La''^{\oplus m'},n,r,\b''\oplus \vphi(\b')^{\oplus m'}]
\quad
\text{ and }
\quad
[\La''\oplus\La'^{\oplus m''},n,r,\b''\oplus \b'^{\oplus m''}]
\]
intertwine, while the latter is equivalent to a simple stratum by~\cite[Theorem~6.16]{SkSt}. Since the first is certainly equivalent to a semisimple stratum, the matching of~\cite[Proposition~7.1]{SkSt} implies that it must be equivalent to a simple stratum. Therefore
\[
[\La''\oplus\La'',n,r,\b''\oplus\vphi(\b')]
\]
is equivalent to a simple stratum (\cite[Theorem~6.16]{SkSt} again) so that~$[\La'',n,r,\b'']$ and~$[\La'',n,r,\vphi(\b')]$ intertwine. Then~\cite[Theorem~8.1]{SkSt} implies that there exists~$g\in\P(\La'')$ such that~$[\La'',n,r,g\vphi(\b')g^{-1}]$ is equivalent to~$[\La'',n,r,\b'']$ and the element~$\tilde{\b}'':=g\vphi(\b')g^{-1}$ satisfies the assertion. 
}
%%%%%%%%%%%%%%%%%%%%%%%%%%%%%%%%%%%%
%%%%%%%%%%%%%%%%%%%%%%%%%%%%%%%%%%%%
\end{proof}

\begin{lemma}\label{lemJoiningSemisimpleCharacters}
\shaun{
Let~$\t'\in\Cc(\La',r,\b')$ and~$\t''\in\Cc(\La'',r,\b'')$ be semisimple characters in~$\Aut_\F(\V')$ and~$\Aut_\F(\V'')$ respectively, and suppose that~$e(\La')=e(\La'')$. Set~$\V=\V'\oplus\V''$,~$\La=\La'\oplus\La''$ and~$n=\max\{n',n''\}$. Then there exist a semisimple stratum~$[\La'',n'',r,\tilde\b'']$, such that~$\Cc(\La'',r,\tilde\b'')=\Cc(\La'',r,\b'')$ and~$[\La,n,r,\b'\oplus\tilde\b'']$ is semisimple, and~$\t\in\Cc(\La,r,\b'\oplus\tilde\b'')$ such that~$\t|_{\H^{r+1}(\b',\La')}=\t'$ and~$\t|_{\H^{r+1}(\b'',\La'')}=\t''$.
}
\end{lemma}

\begin{proof}
\shaun{
We begin by reducing to the case that~$\t''$ is a \emph{simple} character. Indeed, the general case proceeds from this case by induction on the number of simple blocks of~$[\La'',n'',r,\b'']$ as follows: Suppose this stratum has splitting~$\V''=\bigoplus_{i\in\I}\V''^i$ and~$\I=\J\cup\{i_0\}$ (disjoint union). Applying the inductive hypothesis to~$\t'$ and~$\t''_\J$, we have a semisimple stratum~$[\La''^\J,n''_\J,r,\tilde\b''_\J]$, such that~$\Cc(\La''^\J,r,\tilde\b''_\J)=\Cc(\La''^\J,r,\b''_\J)$ and~$[\La'\oplus\La''^\J,n_\J,r,\b'\oplus\tilde\b''_\J]$ is semisimple, and~$\t_\J\in\Cc(\La'\oplus\La''^\J,r,\b'\oplus\tilde\b''_\J)$ with~$\t_\J|_{\H^{r+1}(\b',\La')}=\t'$ and~$\t_\J|_{\H^{r+1}(\b''_\J,\La''^\J)}=\t''_\J$. Then the simple case applied to~$\t_\J$ and the simple block restriction~$\t''_{i_0}$ give a (semi)\-simple stratum~$[\La''^{i_0},n''_{i_0},r,\tilde\b''_{i_0}]$ and semisimple character~$\t\in\Cc(\La,r,\b'\oplus\tilde\b''_\J\oplus\tilde\b''_{i_0})$ such that~$\t|_{\H^{r+1}(\b'\oplus\tilde\b''_\J,\La'\oplus\La''^\J)}=\t_\J$ and~$\t|_{\H^{r+1}(\b''_{i_0},\La''^{i_0})}=\t''_{i_0}$. Putting~$\tilde\b''=\tilde\b''_\J\oplus\tilde\b''_{i_0}$, we see that~$\t|_{\H^{r+1}(\b'',\La'')}$ is a semisimple character in~$\Cc(\La'',r,\tilde\b'')$ whose restriction to any~$\H^{r+1}(\b''_i,\La''^i)$ is~$\t''_i$, for~$i\in\I$. It follows from Corollary~\ref{corCDeltaPsia} that~$\t|_{\H^{r+1}(\b'',\La'')}=\t''$ and~$\Cc(\La'',r,\tilde\b'')=\Cc(\La'',r,\b'')$, as required. 
}

\shaun{So now we assume that~$\t''$ is simple, and prove the statement by induction on~$r$. If~$r=n$ then we take~$\t$ to be the trivial character. If~$r<n$ then let~$[\La',n',r+1,\g']$ be a semisimple stratum which is equivalent to~$[\La',n',r+1,\b']$ and split with respect to the splitting of~$\b'$, and let~$[\La'',n'',r+1,\g'']$ be a simple stratum equivalent to~$[\La'',n'',r+1,\b'']$. By the inductive hypothesis, there are a simple stratum~$[\La'',n'',r+1,\tilde\g'']$ and a character~$\t_\g\in\Cc(\La,r+1,\g'\oplus \tilde\g'')$ with restrictions~$\t'|_{\H^{r+2}(\b',\La')}$ and~$\t''|_{\H^{r+2}(\b'',\La'')}$. By the translation principle~\cite[Theorem~9.16]{SkSt}, there is a simple stratum~$[\La'',n'',r,\hat\b'']$ such that~$\Cc(\La'',r,\b'')=\Cc(\La'',r,\hat\b'')$ and~$[\La'',n'',r+1,\hat\b'']$ is equivalent to~$[\La'',n'',r+1,\tilde\g'']$. Moreover, Lemma~\ref{lem:changesecondonly} implies that we can replace~$[\La'',n'',r,\hat\b'']$ by an equivalent simple stratum so that~$[\La,n,r,\b'\oplus\hat\b'']$ is itself semisimple.
}

\shaun{
Now Lemma~\ref{lemSquareExtensionOfASemisimpleChar} provides a semisimple character~$\tilde\t\in\Cc(\La,r,\b'\oplus\hat\b'')$ with restrictions~$\t'$ on~$\H^{r+1}(\b',\La')$ and~$\t_\g$ on~$\H^{r+2}(\b'\oplus\hat\b'',\La)$. Thus there is an element~$a\in\mathfrak{a}''_{-1-r}$ such that~$\t''=\tilde\t\psi_a|_{\H^{r+1}(\b'',\La'')}$; \orange{moreover,~$[\La'',n'',r,\hat\b''+a]$ is equivalent to a simple stratum} \gre{by Lemma \ref{lemCDeltaPsia}\ref{lemCDeltaPsiai}}. Applying Lemma~\ref{lem:changesecondonly} again, there is a simple stratum~$[\La'',n'',r,\tilde\b'']$ equivalent to~$[\La'',n'',r,\hat\b''+a]$ such that~$[\La,n,r,\b'\oplus\tilde\b'']$ is semisimple. Finally, setting~$\t=\tilde\t\psi_a\in\Cc(\La,r,\b'\oplus\tilde\b'')$, we are done.
}
\end{proof}

\begin{proof}[Proof of Proposition~\ref{propSemisimpleEndoclassesAsTupleOfSimpleEndo}]
\shaun{
Let~$\Th$ be a \full\ pss-character supported on~$(0,\b)$, with index set~$\I$, and write~$[\Th]$ for the endo-class of~$\Th$. Then Theorem~\ref{thm:Endopss}\ref{thm:Endopss-i} shows that~$\Ff([\Th])$ is well-defined and~%
%$\#\I=\#\Ff([\Th])$
\orange{$|\I|=|\Ff([\Th])|$}, and also that the map~$\Ff$ is injective. For surjectivity, suppose we are given a finite set~$\{[\Th_i]:i\in\I\}$ of \full\ simple ps-characters and, for each~$i\in\I$, choose a realization~$\t_i$ of~$\Th_i$. Then Lemma~\ref{lemJoiningSemisimpleCharacters} and induction on~%
%$\#\I$
\orange{$|\I|$} give a \full\ semisimple character~$\t$ which has simple character restrictions~$\t_i$. The corresponding pss-character~$\Th$ has simple block restrictions which must be~$\Th_i$, by Lemma~\ref{lemma:pssdecompositions}, and~$\Ff([\Th])=\{[\Th_i]:i\in\I\}$, as required.
}
\end{proof}

Recall that the \emph{degree} of a \bob{full} simple character~$\t\in\Cc(\La,0,\b)$ is defined to be~$[\F[\b]:\F]$, and is independent of intertwining \bob{between full simple characters} and transfer, and that the degree of a simple endo-class~$c\in\Ee$ is defined to be the common degree of the values of the ps-characters in~$c$, which we denote by~$\deg(c)$.

\begin{definition}
An \emph{endo-parameter} is a function~$\f$ from the set $\Ee$ to the set~$\NN_0$ of non-negative integers, with finite support. We define the degree of an endo-parameter~$\f$ by
\[
\deg(\f):=\sum_{c\in\Ee} \deg(c) \f(c).
\]
\end{definition}

\rob{
Given a~\full\ semisimple character~$\t\in\Cc(\La,0,\b)$, let~$\Th$ be the pss-character supported on~$(0,\b)$ with~$\Th(\V,\varphi_{\b},\La,0)=\t$, where~$\varphi_\b$ denotes the canonical embedding of~$\F[\b]$ as usual, and~$c_i$ be the endo-classes of its simple block restrictions; we define the endo-parameter~$\f_\t$ to be the map with support~$\Ff([\Th])$, and such that~$\f_\t(c_i):=\frac{\dim_\F\V^i}{\deg(c_i)}$, for~$i\in\I$.
}

\begin{theorem}\label{thmClassifyConjClassGL}
\rob{There is a canonical bijection from the set of intertwining classes of \full\ semisimple characters for~$\tG=\GL_\F(\V)$ to the set of endo-parameters~$\f$ of degree~$\dim_\F(\V)$, defined by mapping the intertwining class of a \full\ semisimple character~$\t$ to the endo-parameter~$\f_\t$.}
\end{theorem}

\begin{proof}
\shaun{
Suppose~$\t\in\Cc(\La,0,\b)$ and~$\t'\in\Cc(\La',0,\b')$ are semisimple characters with index sets~$\I,\I'$ respectively. By changing lattice sequences in their affine class, we may and do assume~$e(\La)=e(\La')$. Let~$\Th$ be the pss-character supported on~$(0,\b)$ with value~$\t$ \daniel{at~$(\V,\varphi_\b,\La,0)$}, and similarly~$\Th'$ the pss-character supported on~$(0,\b')$ with value~$\t'$ \daniel{at~$(\V,\varphi_{\b'},\La',0)$}.
}

\shaun{
First we prove that the map described is well-defined (on intertwining classes). Suppose~$\t$ and~$\t'$ intertwine in~$\tG$. Then Theorem~\ref{thm:MatchingForChar} gives a matching~$\z:\I\to\I'$ between index sets, and~$\z_{\Th',\Th}=\z$ by Theorem~\ref{thm:Endopss}\ref{thm:Endopss-iib}. Since~$\dim_\F\V^i=\dim_\F\V'^{\z(i)}$ (from Theorem~\ref{thm:MatchingForChar} again), it follows that~$\f_\t=\f_{\t'}$.
}

\shaun{
Conversely, suppose that~$\f_\t=\f_{\t'}$. Then, by comparing the support, we have~$\Ff([\Th])=\Ff([\Th'])$, so~$\Th\approx\Th'$ by Proposition~\ref{propSemisimpleEndoclassesAsTupleOfSimpleEndo} and we have a matching~$\z_{\Th',\Th}:\I\to\I'$, by Theorem~\ref{thm:Endopss}\ref{thm:Endopss-i}. In particular, the simple components~$[\Th_i]=[\Th'_{\zeta(i)}]$ have the same degree so it follows from~$\f_\t=\f_{\t'}$ that~$\dim_\F\V^i=\dim_\F\V^{\zeta(i)}$, for all~$i\in\I$. Finally Theorem~\ref{thm:Endopss}\ref{thm:Endopss-iic} implies~that~$\t$ and~$\t'$ intertwine. Thus the map is injective.
}

\shaun{
Finally we prove surjectivity, so let~$\f$ be an endo-parameter of degree~$\dim_\F(\V)$. By Proposition~\ref{propSemisimpleEndoclassesAsTupleOfSimpleEndo}, there is a \full\ pss-character~$\Th$, supported on some~$(0,\b)$, such that~$\supp(\f)=\Ff([\Th])=\{[\Th_i]\mid i\in\I\}$; further, each~$\Th_i$ is supported on the simple pair~$(0,\b_i)$. For~$i\in\I$, we choose~$(\V^i,\vphi_i,\La^i,0)\in\Qq(0,\b_i)$ with~$\dim_\F \V^i=\deg([\Th_i]) \f([\Th_i])$. Replacing the~$\La^i$ in their affine classes, we may and do assume that~$e(\La^i)$ is independent of~$i$. Since~$\f$ has degree~$\dim_\F(\V)$, there is an isomorphism between~$\bigoplus_{i\in\I}\V^i$ and~$\V$; replacing the~$\V^i$ by their images, we may assume this isomorphism is an equality. Then~$(\V,\bigoplus_{i\in\I}\vphi_i,\bigoplus_{i\in\I}\La^i,0)\in\Qq(0,\b)$ and~$\f=\f_{\t}$ for~$\t=\Th(\V,\bigoplus_{i\in\I}\vphi_i,\bigoplus_{i\in\I}\La^i,0)$.
}
\end{proof}

%%% Include/improve or delete? Or just say that~$\f$ is a section of~$\bigsqcup\limits_{c\in\Ee} \mathbb{Z}_{\ge 0} \to \Ee$ with finite support?
\ignore{\begin{remark}
One could also describe endo-parameters as follows. For~$\Th$ a \full\ ps-character supported on~$(0,\b)$, define~$\operatorname{vec}_\Th$ to be the category of isomorphism classes of finite-dimensional~$\F[\b]$-vector spaces. For~$c$ an endo-class, set~$\operatorname{vec}_c=\lim\limits_{\xrightarrow[{\!\!\Th\in c\!\!}]{}}\operatorname{vec}_\Th$. \red{(Not sure this is quite right.)} Then we have a natural map~$\bigsqcup\limits_{c\in\Ee} \operatorname{vec}_c \to \Ee$ and an endo-parameter is a section of this map with finite support. 
\end{remark}}
%%% Include or delete?

%%%%%%%%%%%%%%%%%%%%%%%%%%%%%%%%%%%%
\subsection{The classical involution on~$\Ee$}\label{subsecTheClassicalInvolutionOnE}
\shaun{
We return to the classical setting so that we have an extension~$\F/\F_{\so}$ of degree at most two, whose \rob{Galois} group is generated by~$x\mapsto\ov x$. We fix an algebraic closure~$\Falg$ of~$\F$ and denote by~$\PsF$ be the set of \full\ ps-characters supported on simple pairs~$(0,\b)$ such that~$\F[\b]$ is contained in~$\Falg$. 
}

\shaun{
We choose an automorphism~$f$ of~$\Falg$ extending the map~$x\mapsto\ov x$ on~$\F$ and a sign~$\e$. We are going to define a map~$\Th\mapsto\Th^f$ on~$\PsF$ which, \emph{a priori}, depends on many choices we will make. In the end we will find that it is in fact independent of these choices and, moreover, it induces an involution on the set~$\Ee$ of endo-classes over~$\F$ \emph{which is independent of the choices of~$f$,~$\e$ and~$\Falg$.}
}

\begin{definition}\label{deffOperationonPsCharacters}
\shaun{Given~$\Th\in\PsF$ supported on \rob{a} simple pair~$(0,\b)$, we set~$\E=\F[\b]$ and \emph{choose} the following:
\begin{itemize}\setlength\itemsep{5pt}
\item a finite-dimensional~$\F$-vector space~$\V$ such that~$[\E:\F]$ divides~$\dim_\F\V$;
\item an~$\F$-linear field embedding~$\vphi:\E\to\A=\End_\F(\V)$;
\daniel{\item a hyperbolic~$\e$-hermitian space~$(\V_h,h)$ over~$\F/\F_\so$ with a complete polarization~$\V_h=\V\oplus\V^{\#}$;}
%see the before Remark~\ref{remark:hyperbolicplane};}
% \item a hyperbolic~$\e$-hermitian space~$(\V_h,h)$ over~$\F/\F_\so$ such that~$\V$ is a maximal \daniel{totally} isotropic subspace of~$\V_h$, and a \daniel{maximal totally isotropic} complement~$\V^\#$ for~$\V$ in~$\V_h$; 
\item an~$\o_\E$-lattice sequence~$\La$ in~$\V$.
\end{itemize}
\daniel{In particular, the adjoint anti-involution of~$h$ defines a map from~$\End_\F(\V)$ to~$\End_\F(\V^\#)$.}
Then~$(0,-f(\b))$ is a simple pair and we set~$\E^\#=f(\E)$, define the lattice sequence~$\La^\#$ in~$\V^\#$ by
\[
\La^\#(r)=\{v\in\V^\# \mid h(v,\La(1-r))\subseteq \p_\F \},
\]
and denote by~$\vphi^\#:\E^\#\into\End_\F(\V^\#)$ the embedding~$\vphi^\#(x)=\ov{\vphi(f^{-1}(x))}$, for~$x\in\E^\#$, where~$\ov{\phantom{a}}$ here denotes the adjoint anti-involution on~$\V_h$. Then~$(\V^\#,\vphi^\#,\La^\#,0)\in\Qq(0,-f(\b))$ and we define~$\Th^f$ to be the unique \full\ ps-character supported on~$(0,-f(\b))$ such that
\begin{equation}\label{eqn:eqDefClassicalInv}
  \Th^f(\V^\#,\vphi^\#,\La^\#,0)(y)= \(\Th(\V,\vphi,\La,0)(\ov{y})\)^{-1},\qquad\text{for }y\in\H^1(\vphi^\#(-f(\b)),\La^\#).
\end{equation}
}
\end{definition}

\shaun{
Note that we are defining~$\Th^f$ by its single value at~$(\V^\#,\vphi^\#,\La^\#,0)\in\Qq(0,-f(\b))$, and its value at other elements of~$\Qq(0,-f(\b))$ is then given by transfer; it is for this reason that the definition appears to depend on the choices of~$\V,\vphi,(\V_h,h),\V^\#,\La$. Note also that the value~$\Th^f(\V^\#,\vphi^\#,\La^\#,0)$ is in fact independent of the automorphism~$f$: the group~$\H^1(\vphi^\#(-f(\b)),\La^\#)=\H^1(-\ov{\vphi(\b)},\La^\#)$ does not depend on~$f$ and the formula~\eqref{eqn:eqDefClassicalInv} is clearly independent of~$f$. However, the ps-character~$\Th^f$ does depend on~$f$ since the simple pair~$(0,-f(\b))$ on which it is supported does.
}

\begin{proposition}\label{propWellDefClassInv}
 Let~$\Th,\Th'$ be elements of~$\PsF$. 
 \begin{enumerate}\setlength\itemsep{5pt}
  \item\label{propWellDefClassInvi} $\Th^f$ is well-defined, i.e. is independent of the choices made. 
  \item\label{propWellDefClassInvii} $(\Th^f)^f$ is endo-equivalent to~$\Th$.
  \item\label{propWellDefClassInviii} If~$\Th$ is endo-equivalent to~$\Th'$ then~$\Th^f$ is endo-equivalent to~$\Th'^f$. 
  \item\label{propWellDefClassInviv} The endo-class of~$\Th^f$ does not depend on~$f$ or~$\e$.
 \end{enumerate}
\end{proposition}

\begin{proof}
\shaun{
We start with assertion~\ref{propWellDefClassInvi}. Suppose we have chosen~$(\V_i,\vphi_i,(\V_{h_i},h_i),\V_i^\#,\La_i)$, for~$i=1,2$, as in Definition~\ref{deffOperationonPsCharacters}. In order to show that these give rise to the same ps-character~$\Th^f$, we need to show that the characters~$\t^\#_i$ given by~$y\mapsto \(\Th(\V_i,\vphi_i,\La_i,0)(\ov{y})\)^{-1}$ are \daniel{related by transfer} (from~$(\La_1^\#,\vphi_1^\#)$ to~$(\La_2^\#,\vphi_2^\#)$). Since the self-dual~$\dag$-construction commutes with transfer, we can perform such a construction to reduce to the case where~$(\V_{h_1},h_1)$ and~$(\V_{h_2},h_2)$ are isometric; conjugating then by a suitable isometry, we can assume
 \[
 h_1=h_2=:h,\qquad \V_{h_1}=\V_{h_2}=:\V_h, \qquad \V^\#_1=\V^\#_2=:\V^\#.
\]
Now, $\Th(\V,\vphi_1,\La_1,0)$ and~$\Th(\V,\vphi_2,\La_2,0)$ are intertwined by an element~$g\in\Aut_\F(\V)$ which conjugates~$\vphi_1$ to~$\vphi_2$, because these characters are \daniel{related by transfer} (from~$(\La_1,\vphi_1)$ to~$(\La_2,\vphi_2)$). Thus~$\bar{g}^{-1}$ intertwines~$\t^\#_1$ with~$\t^\#_2$ and conjugates~$\vphi^\#_1$ to~$\vphi^\#_2$, i.e. they are transfers as required.
}

\shaun{
Now assertion~\ref{propWellDefClassInvii} follows immediately because, if the tuple~$(\V,\vphi,(\V_h,h),\V^\#,\La)$ is used to define~$\Th^f$ then~$(\V^\#,\vphi^\#,(\V_h,h),\V,\La^\#)$ can be used to define~$(\Th^f)^f$ so we see that~$\Th(\V,\vphi,\La,0)=(\Th^f)^f(\V,\vphi^{\#\#},\La,0)$; thus the intersection of the images of~$\Th$ and~$(\Th^f)^f$ is non-empty and they are endo-equivalent. Similarly, if~$\Th$ and~$\Th'$ are endo-equivalent then they have a common \daniel{value} $\Th(\V,\vphi,\La,0)=\Th'(\V,\vphi',\La,0)$ \daniel{by Lemma~\ref{lem:changetoconjlatt}, Theorem~\ref{thm:EndoEquivMeanspairwiseIntforAllrealizations} and Theorem~\ref{thmIntImplConjSelfDual}\ref{thmIntImplConjSelfDual.i}}; using this (together with a choice of~$\V_h,\V^\#$), it follows immediately from~\eqref{eqn:eqDefClassicalInv} that~$\Th^f,\Th'^f$ have a common realization so are endo-equivalent, and~\ref{propWellDefClassInviii} follows.
}

\shaun{
We are left with assertion~\ref{propWellDefClassInviv}. That~$[\Th^f]$ is independent of~$f$ is clear from the formula~\eqref{eqn:eqDefClassicalInv} (which is independent of~$f$). To see that~$[\Th^f]$ is independent of~$\e$, we can replace~$h$ \gre{by the twist~${(\vphi(\b)-\ov{\vphi(\b)})}^*(h)$} (which is a~$-\e$-hermitian form over~$\F/\F_\so$) in the construction; this replaces~$\La^\#$ by a translate, which affects nothing, so that the formula~\eqref{eqn:eqDefClassicalInv} remains the same. Thus the image of the resulting ps-character has a non-trivial intersection with the image of $\Th^f$, and they are endo-equivalent.
}
\end{proof}

The last proposition allows us to define an involution on~$\Ee$. 

\begin{definition}\label{defClassicalInv}
%We define an involution~$\ov{\phantom{a}}$ on~$\Ee$ by~$\ov{[\Th]}:=[\Th^f]$, for~$\Th\in\PsF$. 
\orange{We define an action of~$\Sigma$ on~$\Ee$ by~$\s([\Th]):=[\Th^f]$, for~$\Th\in\PsF$.}
\end{definition}

\shaun{
Note that it also follows from Proposition~\ref{propWellDefClassInv} that this involution (defined on the set~$\Ee$, which does not depend on the choice of~$\Falg$) does not in fact depend on the choice of algebraic closure.
}

%%%%%%%%%%%%%%%%%%%%%%%%%%%%%%%%%%%%
\subsection{Orbits and self-dual endo-classes}\label{subsecOrbitsAndSelfDualEndoClasses}
We fix~$\e$,~$\F/\F_\so$ and~$\Falg$ as in the previous subsection. We want to compare orbits of~$\Ee$ with self-dual endo-classes. We therefore introduce the notion of \full\ self-dual endo-classes. We call a self-dual pss-character or endo-class~\emph{\full} if the corresponding lift is \full~\bob{ and we call a self-dual semisimple 
character~\emph{full} if it is contained in~$\Cc_-(\La,0,\b)$ for some 
self-dual semisimple stratum~$[\La,n,0,\b]$. } Let~$(0,\b)$ be a self-dual semisimple pair. Recall that, attached to~$\b$ is an index set~$\I$ together with an action of~$\s$, and we choose a set~$\I_0\cup \I_+$ of representatives for the orbits of~$\s$, where~$\I_0$ is the set of~$\s$-fixed points,~$\I_+$ is a section through the orbits of length two, and we put~$\I_-:=\s(\I_+)$. Analogously to the previous subsection, we denote by~$\PsFcl$ the set of \full\ self-dual pss-characters~$\Th_-$ such that the attached semisimple pair~$(0,\b)$ satisfies 
\[
%\#
|(\I_0\cup\I_+)|=1,
\]
and such that~$\E$ is a subset of~$\oplus_{i\in\I}\Falg$; in the case~$\I_+\ne\emptyset$, we will usually write~$\I_+=\{+1\}$ and~$\I_-=\{-1\}$, and write~$\Thipm$ for the simple block restrictions of the lift~$\Th$ of~$\Th_-$. We denote by~$\Ee_-$ the set of endo-classes of elements of~$\PsFcl$. 

\begin{lemma}\label{lemElementarySelfdualECandOrbits}
Suppose~$\Th$ is a lift of a \full\ self-dual pss-character~$\Th_-\in\PsFcl$ supported on~$(0,\b)$. 
%In the case of~$\#\I=2$ let~$\Thip$ and~$\Thim$ be the block restrictions of~$\Th$. Then we have:
\begin{enumerate}\setlength\itemsep{5pt}
 \item\label{lemElementarySelfdualECandOrbitsii} If~$\Th$ is not simple then~$[\Thip]\neq [\Thim]=\s([\Thip])$.
 \item\label{lemElementarySelfdualECandOrbitsi} If~$\Th$ is simple then~$[\Th]=\s([\Th])$. 
\end{enumerate}
\end{lemma}

\begin{proof}
\shaun{
We first prove~\ref{lemElementarySelfdualECandOrbitsii}. Take~$((\V,h),\vphi,\La,0)\in\Qq_-(0,\b)$ 
%. We write~$\I=\{\pm 1\}$ for the index set. As usual
so that~$\V$,~$\b$ and~$\vphi$ decompose as 
\[
\V^{{1}}\oplus\V^{-1}, \qquad \b=\b_{{1}}+\b_{-1},\qquad \vphi=\vphi_{{1}}\oplus\vphi_{-1}.
\] 
Following Definition~\ref{deffOperationonPsCharacters} using the data~$(\V^{{1}},\vphi_{{1}},(\V,h),\V^{-1},\La^{{1}})$ and \orange{taking~$f:\Falg\to\Falg$ %to be an extension of the map~$\E_{{1}}\to\E_{-1}$ with
to be an extension of the map~$x\mapsto\ov x$ on~$\F$ such that~\gre{$f(\b_{{1}})=-\b_{-1}$},} we compute 
\begin{align*}
 \Thip^f(\V^{-1},\vphi_{-1},\La^{-1},0)(y)&=\Thip^f((\V^{{1}})^\#,\vphi_{{1}}^\#,(\La^{{1}})^\#,0)(y) \\
 &=(\Thip(\V^{{1}},\vphi_{{1}},\La^{{1}},0)(\bar{y}))^{-1} \\
 &=\Thim(\V^{-1},\vphi_{-1},\La^{-1},0)(y).
\end{align*}
Thus~$\Thip^f$ and~$\Thim$ have a common value. They are already \full\ and have the same degree so~$[\Thim]=\s([\Thip])$. The inequality of~\ref{lemElementarySelfdualECandOrbitsii} is a consequence of~\ref{lemma:pssdecompositions}\ref{lemma:pssdecompositions-iii}.
}

\shaun{
The proof of~\ref{lemElementarySelfdualECandOrbitsi} is similar. We choose~$((\V,h),\vphi,\La,0)\in\Qq_-(0,\b)$ and~$f$ an extension of the generator of~$\Gal(\E/\E_\so)$. 
%to extend \[\E\to\E,\ \b\mapsto -\b.\] 
We define an~$\e$-hermitian form~$\tilde{h}$ on~$\tilde\V=\V\oplus\V$ by
\[
\tilde{h}=\begin{pmatrix} 0 & h \\ h & 0\end{pmatrix},\quad\text{i.e. }
\tilde{h}(v_1+v_2,w_1+w_2):=h(v_1,w_2)+h(v_2,w_1),\qquad\text{ for~$(v_1,v_2),(w_1,w_2)\in\V\oplus\V$.}
\]
Then~$\La\oplus\La$ is self-dual with respect to~$\tilde{h}$. Applying the definition with~$(\V\oplus 0,\vphi\oplus 0,(\tilde\V,\tilde{h}),0\oplus\V,\La)$, we obtain: 
\begin{align*}
\Th^f(0\oplus\V,0\oplus\vphi,0\oplus\La ,0)(0\oplus y)&=\Th^f((\V\oplus 0)^\#,(\vphi\oplus 0)^\#,(\La\oplus 0)^\#,0)(0\oplus y)\\
&=(\Th(\V\oplus 0,\vphi\oplus 0,\La\oplus 0,0)(\bar{y}\oplus 0))^{-1}\\
&=(\Th(\V,\vphi,\La,0)(\bar{y}))^{-1}\\
&=\Th(\V,\vphi,\La,0)(y)\\ 
&=\Th(0\oplus\V,0\oplus\vphi,0\oplus\La ,0)(0\oplus y).
\end{align*}
The \bob{third} and the final equality follow from the fact that~$\Th$ respects transfers, \rob{while} the fourth follows from self-duality.
 }
\end{proof}

\ignore{
\begin{proof}
 We start with the proof of~\ref{lemElementarySelfdualECandOrbitsi}. We choose~$((\V,h),\vphi,\La,0)\in\Qq_-(0,\b)$. We define 
 the~$\e$-hermitian form~$\tilde{h}$ on~$\V\oplus\V$ via
 \[\tilde{h}=\left(\begin{array}{cc} 0 & h \\ h & 0\end{array}\right),\text{i.e. }\tilde{h}(v_1+v_2,w_1+w_2):=h(v_1,w_2)+h(v_2,w_1),\]
 for~$(v_1,v_2),(w_1,w_2)\in\V\oplus\V$. Then~$\La\oplus\La$ is self-dual with respect to~$\tilde{h}$. We consider two isomorphisms:
 \[t_u:\V\to\V\oplus 0,\ t_u(v):=(v,0)\ t_l:\V\to 0\oplus\V,\ t_l(v):=(0,v).\]
 Then
 \begin{eqnarray*}
  ^{t_l^{-1}}((^{t_u}\Th(\V,\vphi,\La,0))^{\sigma_{\tilde{h}}}) &=& \Th(\V,\vphi,\La,0)^{\sigma_h}\\
  &=& \Th(\V,\vphi,\La,0).
 \end{eqnarray*}
This finishes the proof of~\ref{lemElementarySelfdualECandOrbitsi}.
The proof of~\ref{lemElementarySelfdualECandOrbitsii} is similar to the proof of~\ref{lemElementarySelfdualECandOrbitsi}. 
Note that~$\Thip$ is not endo-equivalent to~$\Thim$ because otherwise~$\Th$ would be endo-equivalent to a ps-characters which is absurd. 
\end{proof}
}

Using the notation of Lemma~\ref{lemElementarySelfdualECandOrbits} we define a map~$\Phi$ from~$\Ee_-$ to~$\Ee/\Sigma$, the set of orbits of simple endo-classes, by
\begin{equation}\label{eqDefPhi}
\Phi([\Th_-]):=\left\{\begin{array}{ll}
                             \{[\Th]\}, & \text{if }\Th\text{ is simple}\\
                             \{[\Thip],[\Thim]\}, & \text{if }\Th\text{ is not simple}\\
                            \end{array}\right. .
\end{equation}
This map is well-defined and injective by Theorem~\ref{thm:Endopss} and Theorem~\ref{thmEndoSemisimplev4}. We now state the converse of Lemma~\ref{lemElementarySelfdualECandOrbits}.

\begin{theorem}\label{thmConverseSDECandOrbits}
The map~$\Phi$ is surjective.
\end{theorem}

The key  idea for the proof of the theorem is enclosed in the following lemma: 

\shaun{
\begin{lemma}\label{lemNotSkewInvSemiSimpleCharAreLifts}
Let~$(\V,h)$ be an~$\e$-hermitian space over~$\F/\F_\so$ with a \daniel{complete} polarization
\begin{equation}\label{eqSplitting}
%\V=\V^1\oplus\V^2
\V=\V^{1}\oplus\V^{-1}
\end{equation}
%be a splitting of~$\V$ such that the corresponding idempotents~$1_1$ and~$1_2$ are swapped by~$\bar{(\ )}$.  
Suppose~$[\La,n,r,\b]$ is a semisimple stratum split by~\eqref{eqSplitting} such that~$\La$ is self-dual and  
\[
\Cc(\La,r,\b)=\Cc(\La,r,-\bar{\b}),
\]
and such that~$[\La^{1},n,r,\b_{1}]$ is simple. Then there exists a self-dual semisimple stratum~$[\La,n,r,\b']$ split by~\eqref{eqSplitting} such that
\[
\Cc(\La,r,\b)=\Cc(\La,r,\b').
\]
\end{lemma}
}

\begin{proof}
\shaun{
We prove the statement by induction along~$r$. % (in fact along~$n-r$).
If~$n=r$ then we choose~$\b'=0$ so suppose now~$r<n$. We choose a semisimple approximation~$[\La,n,r+1,\g]$ of~$[\La,n,r+1,\b]$ split by~\eqref{eqSplitting}.
%(see Proposition~\ref{prop:semiapprox}) split by~\eqref{eqSplitting} (see~\cite[6.16]{SkSt}). 
Then by the induction hypothesis there is a self-dual semisimple stratum~$[\La,n,r+1,\g']$ split by~\eqref{eqSplitting} with the same set of semisimple characters as~$[\La,n,r+1,\g]$. By~\cite[Theorem 9.16]{SkSt} there is a simple stratum~$[\La^{1},n,r,\b'_{1}]$ such that~$[\La^{1},n,r+1,\b'_{1}]$ is equivalent to~$[\La^{1},n,r+1,\g'_{1}]$ and
\[
\Cc(\La^{1},r,\b_{1})=\Cc(\La^{1},r,\b'_{1}).
\]
We choose~$\t\in\Cc^\Sigma(\La,r,\b)$ and~$\t_0\in\Cc^\Sigma(\La,r,\g')$ which coincide on~$\H^{r+2}(\b,\La)=\H^{r+2}(\g',\La)$.
%such that~$\t|_{\H^{r+2}(\b,\La)}=\t_0|_{\H^{r+2}(\b,\La)}$. 
Then there exists an element~$a_{1}\in\mathfrak{a}^{1}_{-r-1}$ such that 
\[
\t_{1}=\psi_{a_1+\b'_1-\g'_1}\t_{0,1}.
\]
\shauny{Lemma~\ref{lemCDeltaPsia}\ref{lemCDeltaPsiai} applied to~$\psi_{\b'_1-\g'_1}\t_{0,1}$ (with~$|K|=1$) implies that~$[\La^{1},n,r,\b'_1+a_1]$ is equivalent to a semisimple stratum~$[\La^{1},n,r,\b''_{1}]$, and Lemma~\ref{lemCDeltaPsia}\ref{lemCDeltaPsiaii} then implies
\[
\Cc(\La^{1},r,\b_1)=\Cc(\La^{1},r,\b''_1).
\] 
In particular~$[\La^{1},n,r,\b''_{1}]$  is simple by Theorem~\ref{thm:MatchingForChar} and therefore~$[\La^{1},n,r,-\ov{\b''_{1}}]$ is simple by duality. We put~$\b'':=\b_{1}''-\ov{\b_{1}''}$. Then:
\begin{itemize} 
\item if~$[\La,n,r,\b'']$ is not equivalent to a simple stratum, then it is semisimple (and already self-dual);
\item if~$[\La,n,r,\b'']$ is equivalent to a simple stratum, then it is equivalent to  a self-dual simple stratum split by~\eqref{eqSplitting}, by~\cite[1.10]{St00}. 
\end{itemize} 
In any case~$[\La,n,r,\b'']$ is equivalent to a self-dual semisimple stratum~$[\La,n,r,\b''']$ split by~\eqref{eqSplitting}. Further we have:
\[
\t_{1}=\psi_{\b'''_{1}-\g'_{1}}\t_{0,1},\qquad \t_{-1}=\psi_{-\ov{\b'''_{1}}-\g'_{-1}}\t_{0,-1}.
\]
Applying Corollary~\ref{corCDeltaPsia} to the pair~$\t\in\Cc(\La,r,\b)$ and~$\psi_{\b'''-\g'}\t_{0}\in\Cc(\La,r,\b''')$, we obtain the desired equality 
\[
\Cc(\La,r,\b)=\Cc(\La,r,\b''').
\] 
} 
%  Lemma~\ref{lemCDeltaPsia} implies that~$[\La^{1},n,r,\b'_1+a_1]$ is equivalent to a simple stratum~$[\La^{1},n,r,\b''_{1}]$. Thus we have 
% \[
% \t_{1}=\psi_{\b''_{1}-\g'_{1}}\t_{0,1},\qquad \t_{-1}=\psi_{-\ov{\b''_{1}}-\g'_{-1}}\t_{0,-1}.
% \]
% We put~$\b'':=\b''_{1}-\ov{\b''_{1}}$ and then: %apply Lemma~\ref{lemCDeltaPsia} a second time to obtain:
% \begin{itemize}\setlength\itemsep{5pt}
% \item $[\La,n,r,\b'']$ is equivalent to a semisimple stratum split by~\eqref{eqSplitting} \gre{by~\cite[Theorem~6.16]{SkSt}}; and
% \item $\t=\psi_{\b''-\g'}\t_0$ by Corollary~\ref{corCDeltaPsia}.
% \end{itemize}
% In fact we have two cases: if~$[\La,n,r,\b'']$ is already semisimple then the assertion of the Lemma is proven; otherwise (in which case it must be equivalent to a simple stratum) we can replace~$[\La,n,r,\b'']$ by an equivalent self-dual simple stratum split by~\eqref{eqSplitting}, by~\cite[1.10]{St00}, and we have finished the proof.
}
\end{proof}

\begin{proof}[Proof of Theorem~\ref{thmConverseSDECandOrbits}]
\shaun{
Let~$[\Th]$ be a non-zero \full\ simple endo-class and choose~$(\V,\vphi,(\V_h,h),\V^\#,\La)$ as in Definition~\ref{deffOperationonPsCharacters}.
Let~$\t=\Th(\V,\vphi,\La,0)$, a simple character in~$\Cc(\La,0,\vphi(\b))$, and define~$\t^\#\in\Cc(\La^\#,0,-\ov{\vphi(\b)})$ by
\[
\t^\#(y):=(\t(\bar{y}))^{-1}.
\]
Setting~$n=-v_\La(\vphi(\b))$, the strata~$[\La,n,0,\vphi(\b)]$ and~\rob{$[\La^\#,n,0,-\ov{\vphi(\b)}]$} are simple, and therefore~$[\La\oplus\La^\#,n,0,\vphi(\b)\oplus (-\ov{\vphi(\b)})]$ is equivalent to a semisimple stratum split by~$\V\oplus\V^\#$. Lemma \ref{lemJoiningSemisimpleCharacters} provides us with a semisimple stratum~$[\La\oplus\La^\#,n,0,\tilde{\b}]$ split by~$\V\oplus\V^\#$ and a semisimple character~$\tilde{\t}\in\Cc(\La\oplus\La^\#,n,0,\tilde{\b})$ with restrictions~$\t$ and~$\t^\#$. Note that~$[\La,n,0,\tilde{\b}|_\V]$ and~$[\La^\#,n,0,\tilde{\b}|_{\V^\#}]$ are simple strata: they are certainly semisimple and, since~$\Cc(\La,0,\tilde{\b}|_\V)=\Cc(\La,0,\vphi(\b))$, the matching of Theorem~\ref{thm:MatchingForChar} implies that~$[\La,n,0,\tilde{\b}|_\V]$ is simple. We have~$\tilde{\t}=\tilde{\t}^\sigma$ by Corollary~\ref{corCDeltaPsia}, and therefore by Lemma~\ref{lemNotSkewInvSemiSimpleCharAreLifts} we can choose~$[\La\oplus\La^\#,n,0,\tilde{\b}]$ to be self-dual. Let~$\tilde{\Th}_-$ be the self-dual \bob{pss}-character supported on~$(0,\tilde{\b})$ whose lift~$\tilde{\Th}$ takes value~$\tilde{\t}$ at~$(\V_h,\daniel{\varphi_{\tilde{\b}}},\La\oplus\La^\#,0)$. 
Then
\[
\Phi([\tilde{\Th}_-])=\{\s([\Th]),[\Th]\}.
\]
}
\end{proof}

\shauny{Let us illustrate the two cases which occur at the end of (the proof of) Theorem~\ref{thmConverseSDECandOrbits}. If the lift~$\tilde{\Th}$ is simple then~$\tilde{\Th}$ and~$\Th$ are endo-equivalent and we have
\[
[\Th]=[\tilde{\Th}]=\sigma[\tilde{\Th}]=\sigma([\Th])
\]
so that~$\Phi([\tilde{\Th}_-])=\{[\Th]\}$. Otherwise, the lift~$\tilde{\Th}$ is not simple, the endo-classes~$[\Th]$ and~$\sigma([\Th])$ are the non-endo-equivalent simple block restrictions of~$[\tilde{\Th}]$, and~$\Phi([\tilde{\Th}_-])$ consists of two elements. 
}

%%%%%%%%%%%%%%%%%%%%%%%%%%%%%%%%%%%%
\subsection{Endo-parameters for~$(h,\G)$}\label{subsec:endoG}

We now fix~$\F/\F_\so$ and~$\e$, and our~$\e$-hermitian space $(\V,h)$ over~$\F$.  In this section, we parametrize the~$\G=\U(\V,h)$- and~$\G^\so$-intertwining classes of \full\ self-dual semisimple characters (i.e. supported on a self-dual semisimple pair of the form~$(0,\b)$).  %As in the non-self-dual setting, we call these self-dual semisimple characters \full.

One invariant of an intertwining class of skew semisimple characters comes from the theory of concordant pairs: if two skew semisimple characters~$\t_-\in\Cc_-(\La,0,\b)$ and~$\t'_-\in\Cc_-(\La,0,\b')$ intertwine by an element of~$\G$ then the underlying simple block strata are concordant (see Theorem~\ref{thm:MatchingForCharForG}). We now encode this into invariants in the following way. 

\begin{definition}\label{enDefWittType}
Consider the class of pairs~$(\b,\ft)$ where~$(0,\b)$ is a self-dual simple pair and~$\ft$ is an element of~$\Ww_\e(\E/\E_\so)$, where~$\E=\F[\b]$ as usual. Two pairs~$(\b,\ft)$ and~$(\b',\ft')$ are \emph{equivalent} if 
\begin{enumerate}\setlength\itemsep{5pt}
 \item\label{enDefWittTypei} $(\F[\b],\b)$ and~$(\F[\b'],\b')$ are similar self-dual extensions (see Definition~\ref{def:similar}); and
 \item\label{enDefWittTypeii} $\sw_{\e,\b',\b}(\ft)=\ft'$. 
 \end{enumerate}
This is clearly an equivalence relation. We call the equivalence classes \emph{Witt types} and we denote the Witt type associated to a pair~$(\b,\ft)$ by~$[\b,\ft]$. 
%The set of Witt types is denoted by~$\WT_{\e}(\F/\F_\so).$
\end{definition}

\begin{remarks}
\begin{enumerate}\setlength\itemsep{5pt}
\item Note that the diagrams~\eqref{eqDiagConcOdd} and~\eqref{eqDiagConcEven} commute by similarity of $(\F[\b],\b)$ and~$(\F[\b'],\b')$ \orange{and Corollary~\ref{cor:similardiagramscommute}}. Therefore, in the non-symplectic case, the condition~\ref{enDefWittTypeii} is equivalent to the anisotropic dimensions of~$\ft$ and~$\ft'$ having the same parity, together with~$\lambda^*_{\b}(\ft)=\lambda^*_{\b'}(\ft')$. 
\item Given a self-dual field extension~$(\E,\b)$, the set of all possible Witt types~$[\b,\ft]$ is in bijection with~$\Ww_\e(\E/\E_\so)$.
\end{enumerate}
\end{remarks}

\shaun{We are going to attach Witt types to elements of~$\Ee_-$. Throughout this section we identify~$\Ee_-$ with~$\Ee/\Sigma$, a consequence of Theorem~\ref{thmConverseSDECandOrbits} and~\eqref{eqDefPhi}; we will usually use~$\fo\in\Ee/\Sigma$ and write~$[\Th_-]=\Phi^{-1}(\fo)$ for the corresponding element of~$\Ee_-$. We will write~$\deg(\fo)=\deg(\Th_-)$ so that, if~$\Th_\sfo$ is any pss-character whose endo-class is in the orbit~$\fo$, we have~$\deg(\fo)=|\fo|\deg(\Th_\sfo)$.
}

Not all Witt types are suitable for a given~$\fo\in\Ee/\Sigma$. We therefore define %~$[\Th_-]\in\Ee_-$. We therefore define
\[
\WT(\fo):=\begin{cases}
\{[\b,\ft] \mid \text{$\exists\Th_-$ supported on~$(0,\b)$ with~$\Phi([\Th_-])=\fo$, $\ft\in\Ww_\e(\F[\b]/\F[\b]_\so)$}\}, 
        & \text{ if }|\fo|=1, \\
\{[0,{\bs 0}]\}, 
       & \text{ if }|\fo|=2.
\end{cases}
\]
In the non-simple case~$|\fo|=2$ there is only one Witt type because all realizations of a corresponding~$\Th_-$ use a hyperbolic space over~$\F$.
\begin{remark} 
If~$\fo\in\Ee/\Sigma$ is an orbit of cardinality one and we choose a self-dual ps-character~$\Th_-$ supported on the self-dual pair~$(0,\b)$ with~$\Phi([\Th_-])=\fo$, then the map
\[
\Ww_\e(\F[\b]/\F[\b]_\so)\to \WT(\fo),\qquad \ft\mapsto [\b,\ft]
\]
is a bijection. Indeed, if~$\Th'_-$ is another self-dual ps-character, supported on the self-dual pair~$(0,\b')$ and with~$\Phi([\Th'_-])=\fo$ then, since~$\Th_-$ and~$\Th'_-$ are endo-equivalent, Corollary~\ref{cor:endosimilar} implies that the self-dual extensions~$(\F[\b],\b)$ and~$(\F[\b'],\b')$ are similar.
\end{remark}
\shauny{If we have self-dual semisimple characters~$\t_-\in\Cc_-(\La,0,\b)$ and~$\t'_-\in\Cc_-(\La',0,\b')$ which intertwine in~$\G$ with matching~$\zeta:\I \rightarrow\I'$ then, for~$i\in\I$, the dimensions~$\dim_{\E_i}\V^i$ and~$\dim_{\E'_{\zeta(i)}}\V^{\zeta(i)}$ must coincide and, for~$i\in\I_0$, the forms~$h_{\b_i}$ and~$h_{\b_{\zeta(i)}}$ must have the same Witt index. We need invariants taking this into account.}
% Any two realizations of a self-dual pss-character which intertwine certainly need to live on vector spaces of the same Witt index so we need a numerical invariant taking this into account. 

We define the set of possible~\emph{endo-parameters for~\shauny{$\fo\in\Ee/\Sigma$}} to be
\[
\EP(\fo):=\WT(\fo)\times\NN_0.
\]
%We collect some invariants for the elements of this set.
%
\begin{definition}
%Let~$[\Th_-]\in\Ee_-$ and let~$\fz=([\a,\ft],n)$ be an endo-parameter for~$[\Th_-]\leftrightarrow\fo_{[\Th]}$. We define
%\begin{itemize}\setlength\itemsep{5pt}
%\item $\diman(\fz):=\diman(\ft)$, the~\emph{anisotropic dimension of~$\fz$}, 
%\item $\deg(\fz):=(2n+\diman(\fz))\deg(\fo_{[\Th]})$, the~\emph{degree of~$\fz$},
%\item $\herm_{\F/\F_\so}(\fz):=\lambda_{\a}^*(\ft)$. 
%\end{itemize}
Let~$\fo\in\Ee/\Sigma$ and let~$\fz=([\a,\ft],n)$ be an endo-parameter for~$\fo$. We define
\begin{itemize}\setlength\itemsep{5pt}
\item $\diman(\fz):=\diman(\ft)$, the~\emph{anisotropic dimension of~$\fz$}, 
\item $\deg(\fz):=(2n+\diman(\fz))\frac{\deg(\sfo)}{|\sfo|}$, the~\emph{degree of~$\fz$},
\item $\herm_{\F/\F_\so}(\fz):=\lambda_{\a}^*(\ft)$, an element of the Witt group~$\Ww_\e(\F/\F_\so)$.
\end{itemize}
\end{definition}

\shaun{
We are now able to define the parameters for the classification of intertwining classes of self-dual semisimple characters.
\begin{definition}\label{defSDEndoParameters}
A~\emph{self-dual endo-parameter} (with respect to~$(\F/\F_\so,\e)$) is a section~$\f_-$ of 
\[
\bigsqcup_{\fo\in\Ee/\Sigma}\EP(\fo)\to\Ee/\Sigma
\]
with finite support~$\supp(\f_-)=\{\fo\in\Ee/\Sigma \mid \deg(\f_-(\fo))\ne 0\}$. For~$\f_-$ a self-dual endo-parameter, we define its degree by
\[
\deg(\f_-) = \sum_{\fo\in\Ee/\Sigma}\deg(\f_-(\fo))\in\NN_0
\]
and also set
\[
\herm_{\F/\F_\so}(\f_-)=\sum_{\fo\in\Ee/\Sigma}\herm_{\F/\F_\so}(\f_-(\fo)) \in \Ww_\e(\F/\F_\so).
\]
% \begin{enumerate}\setlength\itemsep{5pt}
% \item We call the sum~$\sum_{o\in\Ee/\Sigma}\deg(f_-(o))$ the degree of~$f_-$ and the set 
% \[\{o\in\Ee/\Sigma\mid\ \deg(o)\neq 0\}\]
% the support of~$f_-$.
% \item $f_-$ is called a~\emph{self-dual endo-parameter} (with respect to~$(\F/\F_\so,\e)$) if its degree is finite.  
% \end{enumerate}
\end{definition}
Notice then that the endo-parameters for~$\fo$ are just the endo-parameters with support contained in the singleton~$\{\fo\}$.
\begin{remark}\label{remSDEPtoGLEP}
A self-dual endo-parameter~$\f_-$ defines a~$\GL$-endo-parameter~$\f:\Ee\to\NN_0$ %(Definition~\ref{defEndoParameter}) which satisfies
by setting
\[
\f([\Th])=\frac{\deg (\f_-(\fo))}{\deg(\fo)} ,\qquad [\Th]\in \fo \in\Ee/\Sigma.
\]
We have~$\deg (\f)=\deg (\f_-)$. 
\end{remark}
}

\shaun{
Recall that Proposition~\ref{propSemisimpleEndoclassesAsTupleOfSimpleEndo} gives us a canonical bijection~$\Ff$ from the set of all \full\ semisimple endo-classes to~$\Ee^{\fin}$. We write~$(\Ee/\Sigma)^{\fin}$ for the set of all finite subsets of~$\Ee/\Sigma$, and we have:
\begin{proposition}\label{propSemisimpleEndoclassesAsTupleOfSimpleEndoG}
There is a canonical bijection~$\Ff_-$ from the set of all \full\ self-dual endo-classes to~$(\Ee/\Sigma)^{\fin}$, defined by mapping a \full\ self-dual endo-class~$[\Th_-]$ with lift~$[\Th]$ to~$\Ff([\Th])/\Sigma$, the set of orbits of elements 
of~$\Ff([\Th])$.
\end{proposition}
}

The proof mimics that of Proposition~\ref{propSemisimpleEndoclassesAsTupleOfSimpleEndo}. 
%Using a similar proof, or by the Glauberman correspondence, we get the following self-dual version of Lemma~\ref{lemJoiningSemisimpleCharacters}.
First we have an analogue of Lemma~\ref{lemSquareExtensionOfASemisimpleChar}:

\begin{lemma}\label{lemSquareExtensionOfASemisimpleCharG}
\shaun{\daniel{Let~$[\La,n,r,\b]$ and~$[\La,n,r+1,\g]$ be self-dual semisimple strata which are split by~$\V=\V'\operp \V''$ and such that~$[\La,n,r+1,\b]$ is equivalent to~$[\La,n,r+1,\g]$.} Put~$\La'=\La\cap \V'$ and~$\b'=\b|_{\V'}$, and similarly for~$\La'',\b''$. Given~$\tilde{\t}\in\daniel{\Cc^\Sigma(\La,r+1,\g)}$ and an extension~$\t'\in\Cc^\Sigma(\La',r,\b')$  of~$\tilde{\t}|_{\H^{r+2}(\b',\La')}$, there is a semisimple character~$\t\in\Cc^\Sigma(\La,r,\b)$ such that the restriction to~$\H^{r+1}(\b',\La')$ is~$\t'$ and the restriction to~$\H^{r+2}(\b,\La)$ is~$\tilde{\t}$.}
%  %Suppose that a semisimple stratum~$[\La,n,r,\b]$ is split by~$\V=\V^1\oplus \V^2$. 
%Let~$[\La,n,r,\b]$ be a self-dual semisimple stratum which is split by~$\V=\V^1\operp \V^2$, put~$\La^i=\La\cap \V^i$ and~$\b_i:=\b|_{\V^i}$, for~$i=1,2$.
%%\red{and $\b$}.  
% Given~$\tilde{\t}\in\Cc^\Sigma(\La,r+1,\b)$ and an extension~$\t_1\in\Cc^\Sigma(\La^1,r,\b_1)$  of~$\tilde{\t}|_{\H^{r+2}(\b_1,\La^1)}$ there is a semisimple character~$\t\in\Cc^\Sigma(\La,r,\b)$ such that the restriction to~$\H^{r+1}(\b_1,\La^1)$ is~$\t_1$ and the restriction to~$\H^{r+2}(\b,\La)$ is~$\tilde{\t}$. 
\end{lemma}

\shaun{We could prove this in a similar way to Lemma~\ref{lemSquareExtensionOfASemisimpleChar}, but the group action of~$\Sigma$ provides a significant simplification.}

\begin{proof}
%  The proof is line by line similar to the proof of Lemma~\ref{lemSquareExtensionOfASemisimpleChar}, with one difference:
%  The stratum for~$\b_2$ in the induction step is elementary, i.e. the corresponding index set has one~$\sigma$-orbit, but we allow
%  the stratum to be non-simple.
\shaun{By the definition of semisimple character, in particular see~\cite[Definition~9.5]{SkSt}, we have that for any three characters~$\t_0,\t_1,\t_2\in\Cc(\La,r,\b)$ the character~$\t_0\t_1\t_2^{-1}$ is also an element of $\Cc(\La,r,\b)$. Thus for~$\t'_1\in\Cc(\La',r,\b')$ and~$\tilde{\t}_1\in\Cc(\La,r+1,\g)$ coinciding on~$\H^{r+2}(\b',\La')$, since there always is an extension to~$\Cc(\La,r,\b)$ by Lemma~\ref{lemSquareExtensionOfASemisimpleChar}, the number of such extensions does not depend on the choice of~$(\t'_1,\tilde{\t}_1)$ \daniel{and therefore it is a divisor of the cardinality of~$\Cc(\La,r,\b)$}. The cardinality of~$\Cc(\La,r,\b)$ is odd, and therefore the number of extensions of~$(\t',\tilde{\t})$ is odd; thus~$\Sigma$ has a fixed point which extends~$\t'$ and~$\tilde{\t}$.
}
\end{proof}

The analogue of Lemma~\ref{lemJoiningSemisimpleCharacters} needs more subtle modifications. 

\begin{lemma}\label{lemJoiningSemisimpleCharactersG}
%  Let~$\t_1\in\Cc^\Sigma(\La^1,r,\b_1)$ and~$\t_2\in\Cc^\Sigma(\La^2,r,\b_2)$ be semisimple characters, then there 
%  exists a self-dual semisimple 
%  stratum~$[\La^1\oplus\La^2,\max(n_1,n_2),r,\b'_1\oplus \b'_2]$ on~$\V^1\operp\V^2$
%  and~$\t\in\Cc^\Sigma(\La^1\oplus\La^2,r,\b'_1\oplus\b'_2)$ such 
%  that~$\t\mid_{\H^{r+1}(\b'_1,\La^1)}=\t_1$ and~$\t\mid_{\H^{r+1}(\b_2,\La^2)}=\t_2$.
\shaun{Let~$(\V',h')$ and~$(\V'',h'')$ be~$\e$-hermitian spaces and let~$\t'\in\Cc^{\Sigma'}(\La',r,\b')$ and~$ \t''\in\Cc^{\Sigma''}(\La'',r,\b'')$ be characters for self-dual semisimple strata in~$\V'$ and~$\V''$ respectively, such that~$e(\La')=e(\La'')$. Set~$\V=\V'\operp\V''$,~$\La=\La'\oplus\La''$ and~$n=\max\{n',n''\}$. Then there exist self-dual semisimple strata~$[\La',n',r,\tilde{\b}']$ and~$[\La'',n'',r,\tilde{\b}'']$ in~$\V'$ and~$\V''$ respectively, such that~$[\La,n,r,\tilde{\b}'\oplus\tilde{\b}'']$ is self-dual semisimple and such that there is a character~$\t\in\Cc^\Sigma(\La,r,\tilde{\b}'\oplus\tilde{\b}'')$ with restrictions~$\t|_{\H^{r+1}(\tilde{\b}',\La')}=\t'$ and~$\t|_{\H^{r+1}(\tilde{\b}'',\La'')}=\t''$. Moreover, we can choose~$\tilde{\b}'$ and~$\tilde{\b}''$ with the same associated splittings as
$\b'$ and~$\b''$, respectively.}
\end{lemma}

\shaun{
Note that, contrary to the case in Lemma~\ref{lemJoiningSemisimpleCharacters}, we do not claim that we can take~$\tilde{\b}'=\b'$. We cannot give a proof mutatis mutandis to that of~\emph{loc.cit.}, because we would need to join non-skew self-dual characters. Instead, the proof uses the full strength of the translation principle~\ref{thmTranslationPrincipleG}.
 }

\begin{proof}
\shaun{We proceed by induction along~$r$. If~$r=n$ then~$\b'$ and~$\b''$ vanish, and we put~$\tilde{\b}'$ and~$\tilde{\b}''$ as zero, so suppose~$r<n$. We choose approximations~$[\La',n',r+1,\g']$ and~$[\La'',n'',r+1,\g'']$ for the strata with~$\b'$ and~$\b''$, respectively. By the induction hypothesis we can find self-dual semisimple strata~$[\La',n',r+1,\tilde{\g}']$ and~$[\La'',n'',r+1,\tilde{\g}'']$, with the same associated splittings as~\rob{$\g'$} and~\rob{$\g''$} respectively, such that their direct sum is semisimple and such that there is a character~$\t_\g\in\Cc^\Sigma(\La,\rob{r+1},\tilde{\g})$ %($\tilde{\g}=\tilde{\g}'\oplus\tilde{\g}''$) 
with restrictions~$\t'|_{\H^{r+2}(\b',\La')}$ and~$\t''|_{\H^{r+2}(\b'',\La'')}$.
}

\shaun{
Using Theorem~\ref{thmTranslationPrincipleG} we find a self-dual semisimple stratum~$[\La',n',r,\d']$ and~$u'\in\P^1_-(\La')$ which normalizes every element of~$\Cc(\La',r+1,\tilde\g')$, such that~$\Cc(\La',r,\d')=\Cc(\La',r,\b')$ and~$[\La',n',r+1,\d']$ is equivalent to~$[\La',n',r+1,\tilde\g']$, and~\bob{$\tilde\g'$} respects the splitting of~\bob{$u\d'u^{-1}$}. Moreover, replacing~$\tilde\g'$ by~$u^{-1}\tilde\g'u$, we can assume that~$\bob{\tilde\g'}$ respects the splitting of~\bob{$\d'$}; that is,~$[\La',n',r,\d']$ has an approximation given by~$\tilde\g'$. Similarly, we have a self-dual semisimple stratum~$[\La'',n'',r,\d'']$  with approximation~$\tilde{\g}''$ such that~$\Cc(\La'',r,\d'')=\Cc(\La'',r,\b'')$.
}

\shaun{
The stratum~$[\La,n,r,\d'\oplus\d'']$ is equivalent to a semisimple stratum and therefore, by~\cite[Theorem~6.16]{SkSt}, to a self-dual semisimple stratum respecting the splittings of~$\d',\d''$ and~$\V'\operp\V''$; thus we may assume~$[\La,n,r,\d]$ is self-dual semisimple, where~$\d:=\d'\oplus\d''$. We apply Lemma~\ref{lemSquareExtensionOfASemisimpleCharG} to find a character~$\t_\d\in\Cc^\Sigma(\La,r,\d)$ with restrictions~$\t',\ \t_\g$. 
\ignore{
By Theorem~\ref{thmTranslationPrincipleG} we can assume without  loss of generality that there are self-dual semisimple strata~$[\La',n',r,\d']$ and~$[\La'',n'',r,\d'']$  with approximations given by~$\tilde{\g}'$ and~$\tilde{\g}''$, respectively, such that
\[
\Cc(\La',r,\d')=\Cc(\La',r,\b'),\ \Cc(\La'',r,\d'')=\Cc(\La'',r,\b'').
\]
}
%\orange{
%The stratum~$[\La,n,r,\d'\oplus\d'']$ is equivalent to a semisimple stratum and therefore by~\cite[6.16]{SkSt} to a self-dual semisimple stratum respecting the splittings of~$\d',\d''$ and~$\V'\oplus\V''$. We do not introduce new notation and assume without loss of generality~$[\La,n,r,\d]$ to be self-dual semisimple,~$\d:=\d'\oplus\d''$. We apply Lemma~\ref{lemSquareExtensionOfASemisimpleCharG} to find a character~$\t_\d\in\Cc^\Sigma(\La,r,\d)$ with restrictions~$\t',\ \t_\g$. 
Then there is a skew~$a''\in\mathfrak{a}_{-(r+1)}''$ split by the splitting of~$\d''$ such that
 \[
 \t''=\t_\d\psi_{a''}\ \text{on }\H^{r+1}(\b'',\La'').
 \]
The stratum~$[\La,n,r,\d+a'']$ is equivalent to a semisimple stratum by Lemma~\ref{lemCDeltaPsia} and therefore to a self-dual semisimple stratum~$[\La,n,r,\tilde{\b}]$ split by the splittings of~$\d',\d''$ and~$\V'\operp\V''$. %This is a consequence of~\cite[6.16]{SkSt}. 
Thus~$\tilde{\b}=\tilde{\b}'\oplus\tilde{\b}'' $ and~$\t=\t_\d\psi_{a''}$ satisfy the first part of the lemma. Finally Proposition~\ref{prop:ConjIdempToEachOther}\ref{prop:ConjIdempToEachOther-ii} states that we can conjugate~$\tilde{\b}'$ and~$\tilde{\b}''$ to the splittings of~$\b'$ and~$\b''$, which completes the proof.  
}
\end{proof}

\begin{proof}[Proof of Proposition~\ref{propSemisimpleEndoclassesAsTupleOfSimpleEndoG}]
%\orange{At first the injectivity of the map. 
\shaun{Let~$\Th_-,\Th'_-$ be self-dual pss-characters, with lifts~$\Th,\Th'$ and index sets~$\I=\I_0\cup\I_-\cup\I_+$ and~$\I'$ respectively. If~$[\Th_-]$ and~$[\Th'_-]$ are mapped to the same set, then~$\Ff([\Th])/\Sigma=\Ff_-([\Th_-]) = \Ff_-([\Th'_-]) =\Ff([\Th'])/\Sigma$. It follows that~$\Ff([\Th])$ and~$\Ff([\Th'])$ coincide because both sets are union of orbits; indeed,~$\s([\Th_i])=[\Th_i]$ for~$i\in\I_0$ and~$\s([\Th_i])=[\Th_{-i}]$ for~$i\in\I_+$, by Lemma~\ref{lemma:pssdecompositions}\ref{lemma:pssdecompositions-iia} and then Lemma~\ref{lemElementarySelfdualECandOrbits} applied to~$\Th_{\{\pm i\}}$. Thus~$\Th$ and~$\Th'$ are endo-equivalent by Proposition~\ref{propSemisimpleEndoclassesAsTupleOfSimpleEndo} and therefore~$\Th_-,\Th'_-$ are endo-equivalent by Theorem~\ref{thmEndoSemisimplev4}. The surjectivity follows inductively as in the case of Proposition~\ref{propSemisimpleEndoclassesAsTupleOfSimpleEndo}, using Lemma~\ref{lemJoiningSemisimpleCharactersG}.}
\end{proof}

We are finally ready to define endo-parameters for~$(h,\G)$; note that these depend not only on the group~$\G$ but also on the isometry class of the form~$h$.

\shaun{
\begin{definition}
We denote by~$\EP(h,\G)$ the set of those self-dual endo-parameters with~$\herm_{\F/\F_0}(\f_-)=[h]$ and~$\deg(\f_-)=\dim_\F\V$, and we call it \emph{the set of endo-parameters for~$(h,\G)$.} 
\end{definition} 
}

\rob{%The bijection \gre{depends on~$h$ (not just the isometry class of~$h$) and} is Given as follows: 
Let~$\t_-\in\Cc_-(\La,0,\b)$ be a self-dual semisimple character, let~$\Th_-$ be the self-dual pss-character supported on~$(0,\b)$ with value~$\t_-$ \daniel{at~$((\V,h),\varphi_\b,\La,0)$}, and let~$\Th$ be its lift, with index set~$\I=\I_0\cup\I_+\cup\I_-$. We have
% The map~$\Ff_-$ 
%from Proposition~\ref{propSemisimpleEndoclassesAsTupleOfSimpleEndoG} attaches to~$[\Th_-]$ a set of orbits 
\begin{equation*}
\Ff_-([\Th_-])=\{\fo_i\mid  i\in\I_0\cup\I_+\}\subseteq \Ee/\Sigma,
\end{equation*}
where~$\fo_i$ is the orbit associated to the block restriction~$[\Th_i]$. For each~$i\in\I_0$, we have the~$\e$-hermitian form~$h_{\id_{\E_i}}$ given by Lemma~\ref{lem:hphi}. We attach to~$\t_-$ the endo-parameter with support~$\Ff_-([\Th_-])$ given by 
\[
\f_{\t_-}(\fo_i):=\begin{cases}
(\left[\b_i,[h_{\id_{\E_i}}]\right],m_i), & i\in\I_0,\\
 ([0,{\bs 0}],m_i), & i\in\I_+,
\end{cases}
\]
where~$m_i$ denotes the Witt index of~$h_{\id_{\E_i}}$ when~$i\in\I_0$, and~$m_i=\dim_{\E_i}(\V^i)$ for~$i\in\I_+$. Note that the map~$\t_-\mapsto \f_{\t_-}$ depends on~$h$ (not just the isometry class of~$h$)\shauny{, which is why we have included~$h$ in the notation~$\EP(h,\G)$}.
%The endo-parameter~$\f_{\t_-}$ only depends on the~$\U(\V,h)$-intertwining class of~$\t_-$.  
}

\begin{theorem}\label{thmClassifyConjClassG}
\rob{There is a canonical bijection from the set of intertwining classes of~full~self-dual semisimple characters for~$\G=\U(\V,h)$ to the set of self-dual endo-parameters~$\EP(h,\G)$, defined by mapping the intertwining class of a full self-dual semisimple character~$\t_-$ to the self-dual endo-parameter~$\f_{\t_-}$.} 
% the set of self-dual endo-parameters~$\f_-$ which satisfy~$\deg(\f_-)=\dim_\F\V$ and~$\herm_{\F/\F_0}(\f_-) =[h].$
\end{theorem}

\shauny{Before starting the proof, we give an example illustrating the dependence on~$h$. Suppose that~$-1$ is not a square in~$\F$, that~$(\V,h)$ is a symplectic space, and that~$\vphi:\E\hookrightarrow \End_\F(\V)$ is a self-dual embedding of a self-dual extension~$(\E,\b)$ such that~$\dim_\E\V$ is odd. Consider a self-dual simple character~$\t_-\in\Cc_-(\La,0,\varphi(\b))$. 
%The~$\G$-intertwining class for~$\t_-$ does not depend on~$h$: that is, if~$h'$ is another~$\e$-hermitian form on~$\V$ such that~$\G=\U(\V,h)=\U(V,h')$ then . 
Note that we have~$\G=\U(\V,h)=\U(\V,-h)$ so that we also have~$\EP(h,\G)=\EP(-h,\G)$. However,~$[h_{\varphi}]\neq [-h_{\varphi}]=[(-h)_{\varphi}]$, because, $-1\not\in\N_{\E/\E_o}(\E)$. Therefore, the~$\G$-intertwining class of~$\t_-$ is attached to a  different self-dual endo-parameter in~$\EP(h,\G)$ than in~$\EP(-h,\G)$.}

\shauny{The reason for this phenomenon is that the notion of self-dual endo-parameter is equivariant with respect to isometries: in the above example, an isometry~$g$ from~$(\V,h)$ to~$(\V,-h)$ maps the intertwining class of~$\t_-$ to the intertwining class of~${}^g\t_-$; and~$\t_-$ is not intertwined with~${}^g\t_-$ by any element of~$\G$, only by an isometry from~$h$ to~$-h$.}

% \begin{theorem}\label{thmTranslationPrincipleForSelfdualCharacters}
% Let~$[\La,n,r+1,\g]$ and~$[\La,n,r+1,\g']$ be self-dual-semisimple strata with the same associated splitting such that 
% \[
% \Cc(\La,r+1,\g)=\Cc(\La,r+1,\g').
% \]
% Let~$[\La,n,r,\b]$ be a self-dual -semisimple stratum, with splitting~$\V=\bigoplus_{i\in I}\V^i$, such that~$[\La,n,r+1,\b]$ is equivalent to~$[\La,n,r+1,\g]$ and~$\g$ is an element of~$\prod_{i\in I} A^{i,i}$. Then, there exists a self-dual semisimple stratum~$[\La,n,r,\b']$, with splitting~$\V=\bigoplus_{i'\in I'}\V'^{i'}$, such that~$[\La,n,r+1,\b']$ is equivalent to~$[\La,n,r+1,\g']$, with~$\g'\in \prod_{i'\in I'}A^{i'i'}$ and
% \[
% \Cc(\La,r,\b)=\Cc(\La,r,\b').
% \]
% \end{theorem}

\begin{proof}
\shaun{
We need to show that the map described is well-defined, injective and surjective. Let~$\t_-\in\Cc_-(\La,0,\b)$ and~$\t'_-\in\Cc_-(\La',0,\b')$ be self-dual semisimple characters in~$\G$ with corresponding pss-characters~$\Th_-$,~$\Th'_-$ and lifts~$\t,\t',\Th,\Th'$. Denote by~$\I,\I'$ the corresponding index sets and decompose them~$\I=\I_0\cup\I_+\cup\I_-$ as usual, similarly for~$\I'$. For~$i\in\I$, we put~$\E_i=\F[\b_i]$ as usual,~\bob{$[\Th_i]$} for the simple endo-classes corresponding to~\bob{$\Th_i$}, and~$\fo_i$ for the orbit of~$[\Th_i]$ in~$\Ee/\Sigma$, with similar notation~$\E'_{i'}$,~$\Th'_{i'}$ and~$\fo'_{i'}$, for~$i'\in\I'$.
}

\shaun{
Suppose first that~$\t_-$ and~$\t'_-$ intertwine by an element of~$\G=~\U(\V,h)$, with matching~$\z:\I\to\I'$. In particular, the self-dual pss-characters~$\Th_-,\Th'_-$ are endo-equivalent so~$\Ff_-([\Th_-])=\Ff_-([\Th'_-])$ and the endo-parameters~$\f_{\t_-},\f_{\t'_-}$ have the same support. Then, for all~$i\in\I$, we have:
\begin{itemize}\setlength\itemsep{5pt}
\item $\z$ commutes with~$\s$, by Remark~\ref{rem:matching}, and $\Th_i$ is endo-equivalent to~$\Th'_{\z(i)}$ so~$\fo_i=\fo'_{\z(i)}$;
\item $\dim_\F\V^i=\dim_\F\V'^{\z(i)}$;
\item if~$i\in\I_0$ then the self-dual field extensions~$(\E_i,\b_i)$ and~$(\E'_{\z(i)},\b'_{\z(i)})$ are similar, by Corollary~\ref{cor:endosimilar};
\item the pairs~$(\b_i,\id_{\E_i})$ and~$(\b'_{\z(i)},\id_{\E'_{\z(i)}})$ are~$(h_i,h_{\z(i)})$-concordant, by Theorem~\ref{thm:MatchingForCharForG}.
\end{itemize}
The final two points imply that the pairs~$(\b_i,[h_{\id_{\E_i}}])$ and~$(\b'_{\z(i)},[h_{\id_{\E'_{\z(i)}}}])$ are equivalent, for~$i\in\I_0$; then~$(\V^i,h_{\id_{\E_i}})$ and~$(\V'^{\z(i)},h_{\id_{\E'_{\z(i)}}})$ have the same anisotropic dimension and the same dimension so also the same Witt index. This \rob{implies}~$\f_{\t_-}(\fo_i)=\f_{\t'_-}(\fo_i)$, for~$i\in\I_0$, while the same is true for~$i\in\I_+$ simply by the comparison of dimensions. Thus~$\f_{\t_-}=\f_{\t'_-}$, as required.
}

\shaun{Conversely, suppose~$\f_{\t_-}=\f_{\t'_-}$ and let~$\f:\Ee\to\NN_0$ be the corresponding~$\GL$-endo-parameter (see Remark~\ref{remSDEPtoGLEP}). Then~$\f_\t=\f=\f_{\t'}$ by Theorem~\ref{thmClassifyConjClassGL} and, applying the same theorem again, the characters~$\t,\t'$ intertwine by an element of~$\Aut_\F(\V)$ with a matching~$\z$; thus~$\Th_i$ and~$\Th'_{\z(i)}$ are endo-equivalent and~$\fo_i=\fo_{\z(i)}$ for all indices~$i\in\I$. The fact that~$\f_{\t_-}(\fo_i)=\f_{\t'_-}(\fo_i)$ says that~$\dim_\F\V^i=\dim_\F\V'^{\z(i)}$, for~$i\in\I$, and that the pairs~$(\b_i,\id_{\E_i})$ and~$(\b'_{\z(i)},\id_{\E'_{\z(i)}})$  are~$(h_i,h'_{\z(i)})$-concordant, for~$i\in\I_0$. Therefore~$\t_-$ and~$\t'_-$ intertwine in~$\G$, by Theorem~\ref{thmEndoSemisimplev4}.}

\shaun{It remains to prove surjectivity. Let~$\f_-$ be a self-dual endo-parameter for~$(h,\G)$ so that we
%such that~$\herm(\f_-)=[h]$. We 
need to construct a semisimple character~$\t_-$ for~$\G$ satisfying~$\f_{\t_-}=\f_-$. In fact it is enough to study endo-parameters supported on a single orbit in~$\Ee/\Sigma$, since Lemma~\ref{lemJoiningSemisimpleCharactersG} then allows us to construct self-dual semisimple characters in general. So we suppose~$\supp(\f_-)=\{\fo\}$ and consider the two cases for the cardinality of the orbit separately.
%So suppose that~$\f_-$ is only supported at one orbit~$o$. We treat the two cardinality cases for the orbit separately. 
}

\shaun{
Suppose first that~$|\fo|=1$ and~$\f_-(\fo)=([\b,\ft],m)$. Let~$\Th_-$ be a self-dual ps-character supported \rob{on}~$(0,\b)$ such that~$\Phi([\Th_-])=\fo$, and set~$\E=\F[\b]$. We choose~$(\tilde{\V},\tilde{h})$, an~$\e$-hermitian space over~$\E/\E_\so$, such that
$[\tilde{h}]=\ft$ and with Witt index~$m$, and a self-dual~$\o_{\E}$-lattice sequence~$\La$ in~$\tilde{\V}$. Then~$\l^*_\b(\tilde{h})$ is isometric to~$h$, as~$\l^*_\b(\ft)=[h]$ and~$\dim_\F\tilde{\V}=\deg(\f_-)=\dim_\F\V$. We identify~$\l^*_\b(\tilde{h})$ with~$h$ via an isometry and then the character
\[
\t_-=\Th_-((\V,h),\id_{\E},\La,0)
\]
satisfies~$\f_{\t_-}=\f_-$.
}

\ignore{
Case~$|o|=1$,~$\f(o)=([\b,\ft],m)$. Let~$\Th_-$ be a self-dual ps-character supported~$(0,\b)$ such that~$\Phi([\Th_-])=o$. Choose~$(\tilde{\V},\tilde{h})$, an~$\e$-hermitian space over~$\F[\b]/\F[\b]_0$, such that
$[\tilde{h}]=\ft$ and with Witt rank~$m$. Further we choose a self-dual~$o_{\F[\b]}$-lattice sequence~$\La$ in~$\tilde{\V}$. Then~$\lambda^*_\b(\tilde{h})$ is isometric to~$h$, as~$\lambda^*_\b(\ft)=[h]$ and
\[\dim_\F\tilde{\V}=\deg(\f_-)=\dim_\F\V.\]
We identify~$\lambda^*_\b(\tilde{h})$ with~$h$ via an isometry and the character
\[\t_-=\Th_-((\V,h),\id_{\E_i},\La,0)\]
satisfies~$\f_{\t_-}=\f_-$.
}

\shaun{Finally, suppose~$|\fo|=2$ and~$\f_-(\fo)=([0,\bs 0],m)$. Let~$\Th_-$ be a self-dual pss-character supported at some pair $(0,\b_{1}\oplus\b_{-1})$ such that~$\Phi([\Th_-])=\fo$, and write~$\E_{1}=\F[\b_{1}]$ and~$\E_{-1}=\F[\b_{-1}]$. We take an~$m$-dimensional~$\E_{1}$-vector space~$\W$, which we consider as an~$\F$-vector space and identify with a maximal \daniel{totally} isotropic space of~$\V$ as part of a \daniel{complete polarization}~$(\W,\W^\#)$. We choose an~$\o_{\E_{1}}$-lattice sequence~$\La$ in~$\W$ so that the stratum~$[\La,n,0,\b_{1}]$ is simple, for an appropriate integer~$n$. Since~$\b_{-1}$ and~$-\ov{\b_{1}}$ have the same minimal polynomial over~$\F$, there is an embedding~$\vphi_{-1}:\E_{-1}\to\End_\F(\W^\#)$ which maps~$\b_{-1}$ to~$-\ov{\b_{1}}$. Then the character
\[
\t_-=\Th_-((\V,h),\id_{\E_1}\oplus\vphi_{-1},\La\oplus\La^\#,0)
\]
satisfies~$\f_{\t_-}=\f_-$.
}
\ignore{
Case~$|o|=2$: Let~$\Th_-$ be a self-dual pss-character supported at some pair $(0,\b_1\oplus\b_2)$ such that~$\Phi([\Th_-])=o$.  We take an~$\F$-vector space~$\W$ of dimension~$\frac{\deg(\f_-)}{2}$ and 
equip it with an~$\F[\b_1]$-vector space structure, i.e. an embedding~$\vphi_1:\F[\b_1]\to\End_\F\W$. We identify~$\W$ with a maximal \daniel{totally} isotropic space of~$\V$ as part  of a Lagrangian~$(\W,\W^\#)$. The stratum~$[\La,n,0,\vphi_1(\b_1)]$ (for appropriate~$n$) is a simple stratum
for any~$o_{\E_1}$-lattice sequence~$\La$. The elements~$\b_2$ and~$-\bar{\b}_1$ have the same minimal polynomial over~$\F$, and we denote with~$\vphi_2:\F[\b_2]\to\End_\F(\W^\#)$ the embedding which maps~$\b_2$ 
to~$-\ov{\vphi(\b_1)}$. Then the character
\[\t_-:=\Th_-((\V,h),\vphi_1\times\vphi_2,\La\oplus\La^\#,0)\]
satisfies~$\f_{\t_-}=\f_-$. Which finishes the proof. 
}
% The proof is completely analogous to the proof of Theorem~\ref{thmClassifyConjClassGL}; we just need to keep track of Witt towers, using 
% Theorem~\ref{thm:MatchingForCharForG}.
%, while the multiplicities~$f_-^1(\{[\Th]\})$,~$\Th\in\Ee_o$, have the meaning of the~$\F[\b]$-Witt index if~$\Th$ is the lift of a self-dual ps-character~$\Th_-$ supported on~$(0,\b)$.
\end{proof}

\shauny{Let~$\t\in\Cc(\La,0,\b)$ be the lift of a self-dual semisimple character~$\t_-\in\Cc(\La,0,\b)$. %, and let~$\fCc$ be the~$\tG$-intertwining class of~$\t$. 
Then Theorem~\ref{thmClassifyConjClassG}, combined with Theorem~\ref{thmClassifyConjClassGL}, provides a way to count the number of~$\G$-intertwining classes of self-dual semisimple characters whose lift is in the~$\tG$-intertwining class of~$\t$, by counting the number of endo-parameters~$\f_-$ which give the same~$\GL$-endo-parameter (see Remark~\ref{remSDEPtoGLEP}). We denote this number by~$\Nn(\t_-,\tG,\G)$. Write
\[
n_0:=\begin{cases}
0, &\text{ if~$\G$ is symplectic and~$\b$ has no zero component;}\\
2, &\text{ if~$\G$ is orthogonal,~$\b$ has zero component~$\b_{i_0}$ with~$\dim_\F\V^{i_0}\leq 2$ and~$\diman\V^{i_0}\le 1$;}\\ %is}\\[0pt] &\text{ \qquad\ either~$2$-dimensional isotropic or~$1$-dimensional;}\\
1, &\text{ otherwise.}\\
\end{cases}
\]
Then we obtain
%If~$\G$ is symplectic or unitary then we obtain
\[
\Nn(\t_-,\tG,\G)=
%\left\{\begin{array}{ll}
                     2^{|\I_0|-n_0} .%, &\text{ if~$\G$ is symplectic;}\\
%                     2^{\#\I_0-1}, &\text{ if~$\G$ is unitary.}
%                    \end{array}\right..
\]
Note that the reason for the difference in the symplectic case lies in Proposition~\ref{prop:TiGandGIntertwiningSameNonSympl} and the remark following it: concordance is implied by intertwining in~$\tG$ in the non-symplectic case, but not in the symplectic case. The difference in the orthogonal case comes from the fact that, when~$\V^{i_0}$ is small, there is no orthogonal space~$(\V',h')$ such that~$\dim_\F\V^{i_0}=\dim_\F\V'$ and~$[h_{i_0}]-[h']$ is the maximal element of the Witt group~$\Ww_1(\F/\F_0)$.%For the orthogonal case we obtain 
%\[
%n(\t_-,\tG,\G)=\left\{\begin{array}{ll}
%                     2^{\#\I_0-2}, &\text{ if~$n_0=1$ and~$\dim_\F\V^{i_0}\leq 2$ and~$\diman h_{i_0}\leq 1$}\\
%                     2^{\#\I_0-1}, &\text{ else.}\\
%                    \end{array}\right.
%\]
}
%%%%%%%%%%%%%%%%%%%%%%%%%%%%%%%%%%%%
\subsection{Special orthogonal groups}\label{sect:sporthendoparameter}

We conclude with the parametrization of the~$\G^\so$-intertwining classes of \full\ self-dual semisimple characters in the orthogonal case ($\sigma=1$,~$\e=1$).  % For that let us fix a symplectic form~$h_{symp}$ on~$\V$ if the~$\F$-dimension is even. We take a system of representatives~$\mathcal{R}$ of~$\F^\times/(\pm 1+\mathfrak{p}_\F)$.
\shaun{The partition of the set \rob{of } \full\ self-dual semisimple characters \rob{for~$\G$} into~$\G^\so$-intertwining classes is in general finer than the partition into~$\G$-intertwining classes so it is necessary to augment the self-dual endo-parameters. This will of course only happen when~$\V$ is even-dimensional \shauny{ (since there is an element of determinant~$-1$ in the centre in the odd-dimensional case)}, and will indeed only occur when the zero endo-class is not in the support of the endo-parameter\daniel{, by Theorem 
~\ref{thmSOintertwining}\ref{thmSOintertwining.i}.}}

%\shaun{Let~$\f_-$ be an endo-parameter for~$(h,\G)$. The difficulty will only occur, when the zero-class is not in the support of~$\f_-$. We need at first a new equivalence relation on the set of symplectic forms on~$\V$, if~$\V$ is even dimensional.
\begin{definition}
\shaun{%Suppose~$\V$ is even-dimensional over~$\F$. 
Two symplectic forms over~$\F$ on an even-dimensional~$\F$-vector space~$\V$ are~\emph{$(1+\p_\F)$-equivalent} if they are isometric by an automorphism of determinant in~$1+\p_\F$. We write%~$[\![h']\!]$
~\rob{$[h']_{1+\p_\F}$} for the~$(1+\p_\F)$-equivalence class of a symplectic form~$h'$ on~$\V$.}
% We write~$[h']_\sim$ for the equivalence class of a symplectic form~$h'$ on~$\V$. 
\end{definition}

\shaun{
Now let~$\f_-$ be an endo-parameter for~$(h,\G)$, and let~$\Th_-$ be a self-dual pss-character supported on some pair~$(0,\b)$ such that~$\Ff_-([\Th_-])$ is the support of~$\f_-$. As usual, let~\rob{$\I=\I_+\cup\I_0\cup\I_-$} be the corresponding index set, put~$\E=\F[\b]$, and write~$\Ff_-([\Th_-])=\{\fo_i\mid  i\in\I_0\cup\I_+\}$. 
\begin{definition}
For~$\vphi:\E\to \End_\F\V$ a self-dual embedding with corresponding decomposition~$\V=\bigoplus_{i\in\I}\V^i$, denote by~$\f_{\vphi,-}$ the endo-parameter for~$(h,\G)$ with support~$\Ff_-([\Th_-])$ such that
\[
\f_{\vphi,-}(\fo_i):=\begin{cases}
(\left[\b_i,[h_{i,\vphi_i}]\right],m_i), & i\in\I_0,\\
 ([0,{\bs 0}],m_i), & i\in\I_+,
\end{cases}
\]
where~$m_i$ denotes the Witt index of~$h_{i,\vphi_i}$ when~$i\in\I_0$, and~$m_i=\dim_{\E_i}(\V^i)$ for~$i\in\I_+$. We say that~$\vphi$ is \emph{adapted to~$\f_-$} if~~$\f_-=\f_{\vphi,-}$.
%$\dim_\F(\V^i)=\frac{\deg(\f_-(\fo_i))}{|\fo_i|}$, for all~$i\in\I$ and~$[h_{\vphi_i}]$ is the Witt tower in~$\f_-(\fo_i)$ for all~$i\in\I_0$.  
%A self-dual embedding~$\vphi:\E=\F[\b]\to \End_\F\V$ is called \emph{adapted to~$\f_-$} if~$\dim_\F(\V^i)=\frac{\deg(\f_-(\fo_i))}{|\fo_i|}$, for all~$i\in\I$ and~$[h_{\vphi_i}]$ is the Witt tower in~$\f_-(\fo_i)$ for all~$i\in\I_0$.  
\end{definition}}

\shaun{
We are interested in the following classes of symplectic forms if the support of~$\f_-$ does not contain the zero endo-class: 
\[
\mcH(\f_-):=\left\{[\vphi(\b)^*h]_{1+\p_\F}\mid \vphi\text{ an embedding adapted to }\f_-\right\}.
\]
If the support of~$\f_-$ contains the zero endo-class, then we formally just put~$\mcH(\f_-):=\{0\}$.  We need to prove that~$\mcH(\f_-)$ is well-defined -- that is, independent of the choice of self-dual pss-character~$\Th_-$.
\begin{lemma}
% Suppose that~$\f_-$ is does not have the zero-class in its support. Then 
The definition of~$\mcH(\f_-)$ does not depend on the choice of~$\Th_-$.  
\end{lemma}
\begin{proof}
If the support of~$\f_-$ contains the zero endo-class then there is nothing to prove, so we suppose otherwise. Let~$\Th'_-$ be a self-dual pss-character endo-equivalent to~$\Th_-$ and supported on~$(0,\b')$ and take embeddings~$\vphi,\vphi'$, of~$\E,\E'$ respectively, which are adapted to~$\f_-$. Take any self-dual~$\o_{\vphi(\E)}$ (respectively~$\o_{\vphi'(\E')}$)-lattice sequence~$\La$ (respectively~$\La'$) in~$\V$. The characters
 \[
 \t_-:=\Th_-((\V,h),\vphi,\La,0)\quad\text{ and }\quad \t'_-:=\Th'_-((\V,h),\vphi',\La',0)
 \]
intertwine in~$\G$ by Theorem~\ref{thmClassifyConjClassG}, because~$\f_{\t_-}=\f_{\t'_-}=\f_-$. Thus~$\vphi(\b)^*h$  is isometric to~$\vphi'(\b')^*h$ by an element of determinant congruent to~$\pm 1\pmod{\p_\F}$ (and both are possible), by %Theorem~\ref{thmSOintertwining}\ref{thmSOintertwining.ii}.
\daniel{Lemma~\ref{propSOintertwiningInvertibleBeta}.}
\end{proof}
It follows moreover from the proof that~$\mcH(\f_-)$ has cardinality two whenever it is non-trivial. We are now able to define the endo-parameters for~$(h,\G^\so)$. 
\begin{definition}
The set
\[
\EP(h,\G^\so):=\{(\f_-,\fy)\mid \f_-\in\EP(h,\G),\ \fy\in\mcH(\f_-)\}
\]
is called the set of endo-parameters for~$(h,\G^\so)$. 
\end{definition}
}
%$\EP(h,\G^\so)$ parametrizes the~$\G^\so$-intertwining classes of self-dual semisimple characters: 
\rob{Let~$\t_-\in\Cc_-(\La,0,\b)$ be a self-dual semisimple character.  We attach to~$\t_-$ the pair~$(\f_{\t_-},\fy_{\t_-})\in \EP(h,\G^\so)$, where
\[
\fy_{\t_-}:=\begin{cases}
                    [\b^*h]_{1+\p_\F}, & \text{if }\b\text{ has no zero component},\\
                    0, & \text{otherwise}.
\end{cases}
 \]
%This attachment does only depend on the intertwining class of~$\t_-$.}
}

\begin{corollary}\label{CorollaryendoparamSO}
%Suppose~$\sigma=\id$ and~$\epsilon=1$. 
\rob{There is a canonical bijection from the set of~$\G^\so$-intertwining classes of \full\ self-dual semisimple characters to~$\EP(h,\G^\so)$, defined by mapping the~$\G^\so$-intertwining class of a \full\ self-dual semisimple character~$\t_-$ to the endo-parameter~$(\f_{\t_-},\fy_{\t_-})$.}
\end{corollary}

\shaun{
\begin{proof}
%The well-definedness and injectivity of the map 
It follows immediately from Theorems~\ref{thmClassifyConjClassG} and~\ref{thmSOintertwining} that the map is well-defined and injective. For surjectivity, we only need to consider a pair~$(\f_-,\fy)\in\EP(h,\G^\so)$ such that the zero endo-class is not contained in the support of~$\f_-$. % (since the other case follows directly from Theorem~\ref{thmClassifyConjClassG}). 
By Theorem~\ref{thmClassifyConjClassG} there is some~$\t_-\in\Cc_-(\La,0,\b)$ such that~$\f_{\t_-}=\f_-$. We are done, if~$\fy_{\t_-}=\fy$; if not, then we conjugate~$\t_-$ by any~$g\in\G\setminus\G^\so$ to find~$\presuper{g}\t_-\in\Cc_-(\La,0,\presuper{g}\b)$ while~$[(\presuper{g}\b)^*h]_{1+\p_\F} =\fy$, because~$\mcH(\f_-)$ consists of only two elements.  
\end{proof}
}

%% file: Endo-one-appendix.tex
\appendix
\section{From skew to self-dual semisimple characters}\label{sec:appendix}
%%%%%%%%%%%%%%%%%%%%%%%%%%%%%%%%%%%%

\orange{In this appendix, we generalize several previous results which are needed in the main body of the paper, in particular to extend them from the case of \emph{skew} semisimple characters \daniel{\cite{SkSt}} (where the index set is pointwise fixed by the involution) to the case of \emph{self-dual} semisimple characters.}
%\red{We might want these in a different order???}

\shauny{Every cuspidal representation of~$\G^\so$ contains a skew semisimple character by~\cite[Theorem~5.1]{St05} so, for cuspidal representations, it suffices to only consider skew semisimple characters; thus many results were originally only proved in the skew case. However, to consider the category of all smooth representations of~$\G^\so$, it is necessary also to consider the broader class of self-dual semisimple characters: indeed, every smooth representation contains a full self-dual semisimple character by~\cite[Proposition 8.5]{Finitude}, and every self-dual semisimple character is contained in some irreducible representation; moreover, if an irreducible representation contains two \bob{full} self-dual semisimple characters then these characters are endo-equivalent. In particular, skew semisimple characters do not suffice to study the category of all smooth representations. For this reason, in the main text, we need to consider non-skew self-dual semisimple characters and pss-characters.}

\shaun{\shauny{In this highly technical appendix, we extend the results of~\cite{SkSt} we need to the self-dual case, relying heavily on (and assuming knowledge from)~\cite{SkSt}}. We will use the notation introduced in the main body of the paper freely, in particular that in Section~\ref{secSemisimpleChars}, though not, of course, any results from the main paper.}

%\section{Intertwining semisimple self-dual characters}
%%%%%%%%%%%%%%%%%%%%%%%%%%%%%%%%%%%%
\subsection{The translation principle for self-dual semisimple characters}
%%%%%%%%%%%%%%%%%%%%%%%%%%%%%%%%%%%%
%In this appendix
We first generalize the translation principle of the second and third authors~\cite[Theorem~9.26]{SkSt}. %, from skew semisimple to self-dual semisimple characters. We use the notation of Section \ref{secSemisimpleChars} freely.
Let~$[\La,n,r,\b]$ be a non-null semisimple stratum and put~$k_0=k_0(\b,\La)$.  We write~$\mm_r(\b,\La)$ for the~$\o_\F$-lattice
\[
\mm_r(\b,\La)=\nn_{-r}(\b,\La)\cap \aa_{-(r+k_0)}(\La),
\]
where~$\nn_{-r}(\b,\La)=\{a\in\aa\mid \b a-a\b\in\aa_{-r}\}$. The pro-$p$ subgroup~$1+\mm_r(\b,\La)$ of~$\bob{\P^{-k_0-r}(\La)}$ normalizes the equivalence class of the stratum, and every character in~$\Cc(\La,r,\b)$; indeed, the group
\[
\SS_r(\b,\La)=1+\mm_r(\b,\La)+\JJ^{\left\lfloor(-k_0+1)/2\right\rfloor}(\b,\La),
\]
also normalizes every character in~$\Cc(\La,r,\b)$, by~\cite[Lemma~3.16]{St05}.

%For a stratum~$\Delta=[\La,n,r,\b]$ we write~$\mm(\Delta)$
%or~$\mm_\b$ 
%for the set~$\nn_{-r}(\b,\La)\cap \aa_{-(r+k_0(\b,\La))}$,
%where the set~$\nn_{-r}$ is defined as in section~\ref{secPotentialsimplecharacters} with arbitrary~$\b$. The pro-p-subgroup~$(1+\mm(\Delta))$ of~$P_{-k_0-r}(\La)$ normalizes the stratum as already~$1+\nn_{-r}$ does and therefore in the case of a semisimple stratum it also normalizes every character 
%in~$\Cc(\La,m,\b)$. 

\begin{theorem}\label{thmTranslationPrincipleG}
Let~$[\La,n,r+1,\g]$ and~$[\La,n,r+1,\g']$ be self-dual semisimple strata with the same associated splitting~$\V=\bigoplus_{j\in\J}\V^j$ such that 
\[
\Cc(\La,r+1,\g)=\Cc(\La,r+1,\g').
\]
Let~$[\La,n,r,\b]$ be a self-dual semisimple stratum, with associated splitting~$\V=\bigoplus_{i\in \I}\V^i$, such that~$[\La,n,r+1,\b]$ is equivalent to~$[\La,n,r+1,\g]$ and~$\g$ is an element of~$\prod_{i\in\I} \A^{i}$. Then, there exist a self-dual semisimple stratum~$[\La,n,r,\b']$, with splitting~$\V=\bigoplus_{i'\in \I'}\V^{i'}$ and an element~$u$ of~$(1+\mm_{r+1}(\g,\La))\cap \prod_{j\in\J}\A^{j}\cap \G$, such that~$[\La,n,r+1,\b']$ is equivalent to~$[\La,n,r+1,\g']$, with~$u\g'u^{-1}\in \prod_{i'\in\I'}\A^{i'}$ and
\[
\Cc(\La,r,\b)=\Cc(\La,r,\b').
\]
\end{theorem}

\orange{The proof will take the next few subsections.}

\subsection{Idempotents and self-dual minimal strata}
Let~$(\kk_r)_{r\ge 0}$ be a decreasing sequence of~$\s$-invariant~$\o_\F$-lattices in~$\A$ such that~$\kk_r\kk_s\subseteq\kk_{r+s}$, for all~$r,s\ge 0$, and~$\bigcap_{r\ge 1}\kk_r = \{0\}$. 

\begin{lemma}[{\cite[Lemma~7.13]{SkSt}}]\label{lemFindIdempotents0}
Suppose that there is an element~$\a$ of~$\kk_0$ which satisfies~$\a^2-\a\in \kk_1$. Then there is an idempotent~$\tilde{\a}\in \a + \kk_1$. Moreover, if~$\ov{\a}=\a$ then we can choose~$\tilde{\a}$ such that~$\ov{\tilde\a}=\tilde\a$. 
%Let~$\La,\La'$ be~$o_\F$-lattice sequences in~$A$ and suppose there is an element~$\a$ of~$\aa_0\cap\aa'_0$ which satisfies~$\a^2-\a\in \aa_r\cap\aa'_r$. Then there is an idempotent~$\tilde{\a}\in \aa_0\cap\aa'_0$ such that~$\tilde{\a}-\a\in \aa_r\cap\aa'_r$. Moreover, if~$\La$ and~$\La'$ are self-dual and~$\s(\a)=\a$ then we can choose~$\tilde{\a}$ such that~$\s(\tilde\a)=\tilde\a$. 
\end{lemma}

For the self-dual setting we also consider idempotents~$e\in\A$ which satisfy~$\ov{e}e=0$. 

\begin{lemma}\label{lemFindIdempotents+}
Suppose that there is an element~$\a$ of~$\kk_0$ which satisfies~$\a^2- \a\in \kk_1$ and%~$\ov{\a}\a\equiv\a\ov{\a}\equiv 0\pmod{\kk_1}$. 
~$\ov{\a}\a,\a\ov{\a}\in{\kk_1}$. Then there is a idempotent~$\tilde{\a} \in \a +\kk_1$ such that~$\ov{\tilde{\a}}\tilde{\a}=\tilde{\a}\ov{\tilde{\a}}=0$.
\end{lemma}

\begin{proof}
Lemma~\ref{lemFindIdempotents0} provides a symmetric idempotent~$e$ congruent to~$\a+\ov{\a}$ modulo~$\kk_{1}$.
Then the element~$\a'=e\(\frac{1+\a-\ov{\a}}{2}\)e$ satisfies~$\a'+\ov{\a'}=e$.  We follow now the idea of the proof of Lemma~\cite[Lemma~7.13]{SkSt}. It is easy to check that~$\a'':=3\a'^2-2\a'^3$ also satisfies~$\a''+\ov{\a''}=e$ and the result follows in the same way as in~\emph{loc.~cit.}. 
\end{proof}

\begin{corollary}\label{corFindSumOfIdempotents}
 %Let~$(\kk_r)_{r\ge 0}$ be as in Lemma~\ref{lemFindIdempotents+}.  
 Suppose that~$\a_1,\ldots,\a_l$ are elements of~$\kk_0$ such that%
% ~$\a_i^2\equiv\a_i$ and~$\a_i\a_j\equiv 0$ modulo~$\kk_1$ for all~$i,j$ with~$i\neq j$. 
~$\a_i^2-\a_i,\a_i\a_j\in\kk_1$, for all~$i,j$ with~$i\neq j$. 
%Suppose further that~$\sum_i\a_i\equiv 1$ modulo~$\kk_1$ and that
Suppose further that~$\sum_i\a_i - 1\in\kk_1$ and that there is an action of~$\s$ on~$\I=\{1,\ldots,l\}$ such that%~$\ov{\a}_i$ is congruent to~$\a_{\s(i)}$ for all indexes~$i\in \I$.  
~$\ov{\a_i}-\a_{\s(i)}\in\kk_1$, for all indices~$i\in \I$. Then there are idempotents~$\tilde{\a}_i\in\a_i+\kk_1$ which are pairwise orthogonal and such that~$\sum_i\tilde{\a}_i=1$ and~$\ov{\tilde{\a}_i}=\tilde{\a}_{\s(i)}$ for all~$i\in\I$.
 \end{corollary}

\begin{proof}
This follows from \rob{Lemmas}~\ref{lemFindIdempotents0} and~\ref{lemFindIdempotents+}, \emph{cf.}~\cite[Corollary~7.14]{SkSt}.
\end{proof}

Recall that, whenever we have a splitting~$\V=\bigoplus_{i\in\I}\V^i$, we have the associated idempotents~$\mfe_i$ with image~$\V^i$ and kernel~$\bigoplus_{j\ne i}\V^j$.

\begin{corollary}\label{corConjSplittingEquivStrata}
Let~$[\La,n,r,\b]$ and~$[\La,n,r,\b']$ be equivalent self-dual semisimple strata. Suppose that~$[\La,n,r,\b]$ is split by~$\V=\oplus_{k\in\K}\V^k$ and suppose that the set of idempotents of this splitting is invariant under~$\ov{\phantom{a}}$. Then there is an element~$u\in(1+\mm_r(\b,\La))\cap\G$ such that~$u\b'u^{-1}\in\prod_{k\in\K}\A^{k}$.
\end{corollary}

\begin{proof}
 The action of~$\ov{\phantom{a}}$ on the idempotents gives a sum~$1=\sum_{k\in \K_0}\mfe^k+\sum_{k\in \K_+}(\mfe^k+\ov{\mfe}^k)$ and by~\cite[Lemma~9.25]{SkSt} there is an element~$g$ of~$(1+\mm_r(\b,\La))\cap\G$ such that~$g\b'g^{-1}$ is split by~$\bigoplus_{[i]\in \bob{\K}/\Sigma}V^{[i]}$, where~$\bob{\K}/\Sigma$ denotes the set of~$\Sigma$-orbits in~$\bob{\K}$. Thus we only need to consider the case where~$\K$ is one orbit with two elements. In this case~$\K=\{+,-\}$, and we take \orange{idempotents~$\mfe'_+ \in \mfe_+ + \mm_r(\b,\La)$ and~$\mfe'_-$ commuting with~$\b'$ such that~$\ov{\mfe'_+}=\mfe'_-$ and~$\mfe'_++\mfe'_-=1$, which exist by Corollary~\ref{corFindSumOfIdempotents} and comparison of the descriptions of the intertwining of~$[\La,n,r,\b]$ in terms of~$\b$ and~$\b'$.} This case follows now from~\cite[Lemma~9.15]{SkSt} which provides an element~$g=(g_+,g_-)$ of~$(\End_\F(\V^+,\V'^+)\times\End_\F(\V^-,\V'^-))\cap(1+\mm_r(\b,\La))$ such that~$g\b'g^{-1}\in\prod_{k\in\K}\A^{k}$, and we take the element~$u=(g_+,\ov{g}^{-1}_+)\in(1+\mm_r(\b,\La))\cap\G$.  
\end{proof}

%For the next proposition we need an action of~$\s$ on the set of polynomials~$\phi$ over the residue field of~$\F$. We define~$\s(\phi)$ in applying~$\ov{\phantom{a}}$ to all the coefficients.  \gre{I think we do not need the statement with the polynomials. So we can skip it. But we need the remaining part of the proposition.}

\begin{proposition}[{cf.~\cite[Lemma~7.19]{SkSt}}]\label{propSelfDualMinimalStrata}
Let~$[\La,n,n-1,\a]$ be a self-dual stratum which is equivalent to a semisimple stratum. Then it is equivalent to a self-dual semisimple stratum.
%Suppose that~$[\La,n,n-1,\a]$ is a self-dual stratum which is equivalent to a semisimple stratum and let~$\phi$ be the characteristic polynomial of the stratum with primary factors~$\phi_i$,~$i\in\I'$.
% We define~$\eta$ to be~$(-1)^{(2n+ee_0)/ge_0}$, where~$e(\F/\F_0)$,~$e=e(\La|\o_\F)$ and~$g=(e,n)$, the greatest common divisor of~$e$ and~$n$.  Then~$[\La,n,n-1,\a]$  is equivalent to a self-dual semisimple stratum~$[\La,n,n-1,\b']$, say with associated splitting~$\V=\oplus_{i\in\I'}V'^i$, and the action of~$\s$ on the index set~$\I'$ satisfies
% \begin{equation}\label{eqConMinStrEqSelfDual}
% \s(\phi_i)(-X)=\eta^{\deg(\phi_i)}\phi_{\s(i)}(\eta X),
%  \end{equation}
% for all indexes~$i \in \I'$.
\end{proposition}

\begin{proof}
The stratum~$[\La,n,n-1,\a]$ is equivalent to a semisimple stratum~$[\La,n,n-1,\b]$ with associated splitting~$\V=\bigoplus_{i\in\I}\V^i$ and idempotents~$\mfe_i$, for~$i\in \I$. The skew-symmetry of~$\a$ implies that the strata~$[\La,n,n-1,\b]$ and~$[\La,n,n-1,-\ov{\b}]$ are equivalent and thus by~\cite[Lemma~7.17]{SkSt} the idempotents are permuted by~$\ov{\phantom{a}}$ modulo~$\aa_1$; this defines an action of~$\s$ on~$\I$. %such that~\eqref{eqConMinStrEqSelfDual} is satisfied. 
Corollary~\ref{corFindSumOfIdempotents} provides pairwise orthogonal idempotents~$\mfe'_i$ congruent to~$\mfe_i$ modulo~$\aa_1$ which sum to~$1$ and satisfy~$\ov{\mfe}'_i=\mfe'_{\s(i)}$. The map~$g=\sum_i \mfe'_i\mfe_i\in \bob{\P^1(\La)} $ conjugates~$[\La,n,n-1,\b]$ to a semisimple stratum which is split by~$\V=\bigoplus_{i\in\I}\V'^i$, where~$\V'^i=\im(\mfe'_i)$. For the indices~$i$ fixed by~$\s$ we put~$\b'_i=g\b_ig^{-1}$ and then the stratum~$[\La,n,n-1,\b_i']$ is equivalent to a self-dual stratum and to a simple stratum, so to a self-dual simple stratum by~\cite[Proposition~1.10]{St00}; thus we may assume it is itself self-dual simple. For the remaining indices we take a section~$\I_+$ through the non-singleton orbits and define~$\b'_i=g\b_ig^{-1}$ and~$\b'_{\s(i)}=-\ov{\b'_i}$ for all~$i\in \I_+$. Then setting~$\b'=\sum_{i\in\I}\b'_i$, we have found a self-dual semisimple stratum~$[\La,n,n-1,\b']$ equivalent to~$[\La,n,n-1,\a]$.
\end{proof}

%%%%%%%%%%%%%%%%%%%%%%%%%%%%%%%%%%%%
\subsection{Equal sets of semisimple characters}
%%%%%%%%%%%%%%%%%%%%%%%%%%%%%%%%%%%%

\begin{lemma}[{cf. \cite[Lemma~9.13]{SkSt}}]\label{lemCDeltaPsia}
Suppose that~$\V=\bigoplus_{k\in\K}\V^k$ is a splitting which refines the associated splitting of a semisimple \daniel{stratum}~$[\La,n,r,\b]$; denote by~$\mfe_k$ the idempotents of the decomposition and~$\b_k=\mfe_k\b_k\mfe_k$. %We write~$A^k$ for~$\End_\F(V^k)$.
Let~$\t\in\Cc(\La,r,\b)$ with restrictions~$\t_k\in\Cc(\La^k,r,\b_k)$ and, for~$k\in\K$, let~$a_k\in\aa_{-r-1}\cap\A^{k}$ be such that~$\t_k\psi_{a_k}\in\Cc(\La^k,r,\b_k)$. Put~$a=\sum_{k\in\K}a_k$.
% Suppose further that~$\t\in\Cc(\La,r,\b)$ and~$a\in\aa_{-r-1}\cap\prod_{k\in\K}\A^{k}$ such that~$\t_k\psi_{a_k}$, where $a_k=\mfe_ka\mfe_k$,  is a semisimple  character in~$\Cc(\La^k,r,\b_k)$.
\begin{enumerate} 
\item\label{lemCDeltaPsiai}  $[\La,n,r,\b+a]$ is equivalent to a semisimple stratum~$[\La,n,r,\b']$ which is split by~$\V=\bigoplus_{k\in\K}\V^k$, and the sets~$\Cc(\La,r,\b')$ and~$\psi_a\Cc(\La,r,\b)$ coincide. 
\item\label{lemCDeltaPsiaii} Suppose~$[\La,n,r,\b'']$ is a semisimple stratum whose associated splitting is refined by~$\V=\bigoplus_{k\in\K}\V^k$, such that~$\H^{r+2}(\b,\La)=\H^{r+2}(\b'',\La)$ and such that there is a semisimple character~$\t''\in\Cc(\La,r,\b'')$ with~$\t''_k=\t_k\psi_{a_k}$ for all~$k\in\K$. Then~$\t\psi_a=\t''$ and~$\Cc(\La,r,\b'')=\Cc(\La,r,\b')$. 
\end{enumerate} 
\end{lemma}

\begin{proof}
Although the statement is slightly different, the proof is the same as that of~\cite[Lemma~9.13]{SkSt}.
% The proof is the same as in~\cite[9.11]{SkSt} (The statement is slightly different, but the proof is the same.)
% One just shows that $[\La^k,n,r,\b_k+a_k]$ is equivalent to a simple stratum for every~$k$, using derived strata. \red{May be asked to include the proof I don't know whether it is worth it before sending it to the referee.}\gre{Let's Shaun decide it.}
\end{proof}

\begin{corollary}[{\cite[Corollary~9.14]{SkSt}}]\label{corCDeltaPsia}
Suppose that~$\V=\bigoplus_k\V^k$ is a splitting which refines the associated splittings of two semisimple strata~$[\La,n,r,\b]$ and~$[\La,n,r,\b']$, and suppose that there are characters~$\t\in\Cc(\La,r,\b)$ and~$\t'\in\Cc(\La,r,\b')$ such that~$\t_k$ and~$\t'_k$ coincide, for all~$k$. Then~$\t=\t'$ and~$\Cc(\La,r,\b)=\Cc(\La,r,\b')$.
\end{corollary}

The following result shows that, if~$\Cc(\La,r,\b)\cap\Cc(\La,r,\b')$ is non-empty (or, equivalently, these sets are equal) then there is an element of~$\SS_r(\b,\La)$ which maps the splitting of~$\b$ to the splitting of~$\b'$.

\begin{proposition}[{\cite[Proposition~9.9(iv)]{SkSt}, cf.~\cite[9.23(iii)]{SkSt}}]\label{prop:ConjIdempToEachOther}
Let~$[\La,n,r,\b]$ and~$[\La,n,r,\b']$ be semisimple strata with associated splittings~$\V=\bigoplus_{i\in\I}\V^i$ and~$\V=\bigoplus_{i'\in\I'}\V'^{i'}$ respectively, with corresponding idempotents~$\mfe_i$ and~$\mfe'_{i'}$. Suppose that~$\Cc(\La,r,\b)=\Cc(\La,r,\b')$, and let~$\tau:\I\to\I'$ be the bijection of~\cite[Proposition~9.9]{SkSt}, such that~$\mfe_i\equiv \mfe'_{\tau(i)}\pmod{\aa_1(\La)}$, for all~$i\in\I$.
\begin{enumerate}
\item There is an element~$g\in \SS_r(\b,\La)$  such that~$g\mfe_i g^{-1}=\mfe'_{\tau(i)}$.% where~$\tau$ is the bijection from $\I$ to~$\I'$ such that~$\mfe_i\equiv \mfe'_{\tau(i)}$ modulo~$\aa_1(\La)$. 
\label{prop:ConjIdempToEachOther-i}
\item If~$[\La,n,r,\b]$ and~$[\La,n,r,\b']$ are also self-dual then there exists~$g\in \SS_r(\b,\La)\cap\G$  such that~$g\mfe_i g^{-1}=\mfe'_{\tau(i)}$. \label{prop:ConjIdempToEachOther-ii}
\end{enumerate}
\end{proposition}

Notice that the element~$g$ in Proposition~\ref{prop:ConjIdempToEachOther} normalizes every element of~$\Cc(\La,r,\b)$. 
%
%
% \blue{
% \begin{proposition}\label{prop:ConjIdempToEachOtherG}
% Suppose~$[\La,n,r,\b]$ and~$[\La,n,r,\b']$ are self-dual semisimple strata and that~$\Cc_-(\La,r,\b)=\Cc_-(\La,r,\b')$, say with associate splittings~$(\mfe_i)_{i\in \I}$ and~$(\mfe_{i'})_{i'\in 
% \I'}$. \gre{maybe we should just say splittings instead of associated splittings in the whole paper} Then there is an element of~$g\in \SS_r(\b,\La)\cap\G$  such that~$g\mfe_ig^{-1}=\mfe'_{\tau(i)}$ 
%  where~$\tau$ is the bijection from $\I$ to~$\I'$ such that~$\mfe_i\equiv \mfe'_{\tau(i)}$ modulo~$\aa_1(\La)$. 
%  \end{proposition}
%  }
% 
We denote the normalizer of a character~$\t\in \Cc(\La,r,\b)$ by~$\nn(\t)$. Note that all elements of~$\Cc(\La,r,\b)$ have the same normalizer, because they have the same set of intertwining elements.
 
\begin{proof}
Part~\ref{prop:ConjIdempToEachOther-i} is given by~\cite[Proposition~9.9(iv)]{SkSt} so we prove~\ref{prop:ConjIdempToEachOther-ii}. Take a decomposition~$\I=\I_0\cup \I_+\cup \I_-$ as usual, which gives a decomposition into idempotents~$1=\mfe_0+\mfe_++\mfe_-$, and the same for~$\I'=\I'_0\cup\I'_+\cup\I'_-$ with~$\I'_+$ chosen to coincide with~$\tau(\I_+)$. Then~$\mfe_0\equiv \mfe'_0$,~$\mfe_+\equiv \mfe'_+$ and~$\mfe_-\equiv \mfe'_-$ modulo~$\SS_r(\b,\La)-1$ by~\cite[Proposition~9.9(iv)]{SkSt}. By~\cite[Proposition~9.23(iii)]{SkSt} there is an element~$g\in \bob{\P^1_-(\La)}\cap\nn(\t)$ which sends $\V_0$ to~$\V'_0$. Thus by Corollary~\ref{corCDeltaPsia} we only have to prove the proposition for the cases where~$\I_0$ or~$\I_+$ is empty. The case where~$\I_+$ is empty is~\cite[Proposition~9.23(iii)]{SkSt} so let us assume that~$\I_0$ is empty. By~\cite[Proposition~9.9(iv)]{SkSt} there is an element~$g=(g_+,g_-)\in \SS_r(\b,\La)$ which maps~$\V^i$ to~$\V'^{\tau(i)}$ for all~$i\in\I$; then $u=(g_+,\ov{g}_+^{-1})\in \bob{\P^1_-(\La)}$ also maps~$\V^i$ to~$\V'^{\tau(i)}$.

Take a character~$\t\in\Cc^\Sigma(\La,r,\b)=\Cc^\Sigma(\La,r,\b')$, so that~$\t^{u^{-1}}_{i}=\t_{\tau(i)}$ for all~$i\in \I_+$; since~$u\in\G$ and~$\t$ is self-dual, this equality holds for all~$i\in\I$. Then Corollary~\ref{corCDeltaPsia} implies that the sets~$\Cc(\La,r,u\b u^{-1})$ and~$\Cc(\La,r,\b')$ are the same and indeed that~$\presuper{u}\t=\t$. Since~$\P^1_-(\La)\cap\nn(\t)=(\SS_r(\b,\La)\cap\G)\P^1_-(\La_\E)$ and~$\P^1_-(\La_\E)$ commutes with~$\b$, this finishes the proof.
\end{proof}

%%%%%%%%%%%%%%%%%%%%%%%%%%%%%%%%%%%%
\subsection{Proof of the translation principle}
%%%%%%%%%%%%%%%%%%%%%%%%%%%%%%%%%%%%
Here we prove Theorem~\ref{thmTranslationPrincipleG}, granted that we have already the Theorem for the skew case (\cite[Theorem~9.26]{SkSt}) and the~$\tG$-case (\cite[Theorem~9.16]{SkSt}). Let~$\J=\J_0\cup \J_+\cup \J_-$ be a partition with respect to the action of~$\s$ as usual, and write~$\J_{+-}=\J_+\cup \J_-$. 
\begin{enumerate}
\item\label{thmTranslationPrincipleG.i} First we assume that~$\J_0$ is empty. By~\cite[Theorem~9.16]{SkSt} (the~$\tG$-case) there is a semisimple stratum~$[\La^{\J_+},n,r,\b'_{\J_+}]$ such that~$\Cc(\La^{\J_+},r,\b_{\J_+})=\Cc(\La^{\J_+},r,\b'_{\J_+})$ and such that~$\g'_{\J_+}$ satisfies the desired conjugation property. Setting~$\b'=\b'_{\J_+}-\ov{\b'_{\J_+}}$, we deduce that~$[\La,n,r,\b']$ is a self-dual stratum whose set of semisimple characters coincides on the blocks of~$\V^{\J_+}$ and~$\V^{\J_-}$ with the corresponding restrictions of characters in~$\Cc(\La,r,\b)$, and such that~$\g'=\g'_{\J_+\cup \J_-}$ satisfies the desired conjugation property. Take~$\t\in \Cc^\Sigma(\La,r,\b)$ and an extension~$\t'\in\Cc^\Sigma(\La,r,\b')$ of~$\t|_{\H^{r+2}(\b,\La)}$. By Proposition~\ref{prop:ConjIdempToEachOther}\ref{prop:ConjIdempToEachOther-i}, \gre{conjugating~$\b$ with an element of~$\SS_r(\b,\La)\cap\G$},  we can assume that~$\b$ and~$\b'$ have the same associated splitting. Take a skew-symmetric~$a\in\aa_{-r-1}\cap\prod_{i\in\I} \A^{i}$ such that~$\t=\t'\psi_a$. Then by Lemma~\ref{lemCDeltaPsia} the stratum~$[\La,n,r,\b'+a]$ is equivalent to a self-dual semisimple stratum~$[\La,n,r,\b'']$ with the same associated splitting as~$\b'$ and such that~$\Cc(\La,r,\b'')=\Cc(\La,r,\b)$.

\item Now we reduce to the case where~$\J$ is a singleton, so suppose we have proven the theorem in that case. By~\ref{thmTranslationPrincipleG.i} and the singleton case we find~$[\La,n,r,\b']$ such that~$\Cc(\La,r,\b')$ coincides with~$\Cc(\La,r,\b)$ on every simple block for~$j\in \J_0$ and on the block corresponding to~$\J_{+-}$ (and the conjugation property is satisfied). Using Proposition~\ref{prop:ConjIdempToEachOther} we can assume that~$\b$ and~$\b'$ have the same associated splitting and we finish the proof using Lemma~\ref{lemCDeltaPsia} in the same manner as at the end of~\ref{thmTranslationPrincipleG.i}.

\item Finally we prove the case where~$\J$ is a singleton. We follow the step~(iv) of the proof of~\cite[Theorem~9.16]{SkSt}. Note that, by Corollary~\ref{corConjSplittingEquivStrata}, we are free to replace~$[\La,n,r+1,\g']$ by any equivalent stratum. Thus, by~\cite[Proposition~9.24]{SkSt}, we can assume that~$\Cc(\La,r,\g)=\Cc(\La,r,\g')$. Take tame corestrictions~$s$ and~$s'$, for~$\g$ and~$\g'$ respectively, which commute with the adjoint anti-involution and which satisfy the assertions of~\cite[Lemma~5.2]{BK94}; in particular~$s(x)\equiv s'(x) \pmod{\aa_l}$, for all~$x\in\aa_{l-1}$ and all integers~$l$. The stratum~$[\La,r,r+1,s(\b-\g)]$ is equivalent to a semisimple stratum, by~\cite[Corollary~6.15]{SkSt}, and as in step~(iv) of the proof of~\cite[Theorem~9.16]{SkSt} it follows that~$[\La,r,r+1,s'(\b-\g)]$ is equivalent to a semisimple stratum; further,~$s'(\b-\g)$ is skew-symmetric and, by Proposition~\ref{propSelfDualMinimalStrata}, this stratum is equivalent to a self-dual semisimple stratum, say with associated splitting~$\V=\bigoplus_{i''\in\I''}\V^{i''}$ and corresponding idempotents~$\mfe_{i''}$. Thus~$[\La,n,r,\g'+\sum_{i''\in\I''}\mfe_{i''}(\b-\g)\mfe_{i''}]$ is equivalent to a self-dual semisimple stratum~$[\La,n,r,\b'']$ with associated splitting~$\V=\bigoplus_{i''\in\I''}\V^{i''}$ by~\cite[Corollary~6.15]{SkSt} and~\cite[Proposition~1.10]{St00}. Finally, by~\cite[Proposition~7.6]{SkSt} there is an element~$u\in (1+\mm_{r+1}(\g',\La))\cap\G$ such that~$\b':=u\b''u^{-1}$ is congruent to~$\g'+\b-\g$ modulo~$\aa_{-r}$. This element~$\b'$ is as required. %completes the proof.
\end{enumerate}

\subsection{Asymmetric statements}
%%%%%%%%%%%%%%%%%%%%%%%%%%%%%%%%%%%%
%In this appendix, w
We now prove some asymmetric versions of results already in the literature.

\begin{proposition}\label{prop:IntertwiningOfTransfers}
Let~$[\La,n,r,\b]$ and~$[\La',n,r,\b]$ be semisimple strata with~$e(\La)=e(\La')$.%=e$.
\begin{enumerate}
\item  Let~$\t\in\Cc(\La,r,\b)$ and%~$\t'\in\Cc(\La',r,\b)$ be such that
~$\t'=\tau_{\La',\La,\b}(\t)$. Then
\[
\I_{\tG}(\t,\t')=\mathrm{S}_{r}(\b,\La')
\tG_\b
%\B_{\b}^\times 
\mathrm{S}_r(\b,\La).
\]
\label{prop:IntertwiningOfTransfersPart1}
\item Suppose~$[\La,n,r,\b]$ and~$[\La',n,r,\b]$ are self-dual and let~$\t_-\in\Cc_-(\La,r,\b)$ and%~$\t'\in\Cc_-(\La',r,\b)$ be such that
~$\t'_-=\tau_{\La',\La,\b}(\t_-)$. Then
\[
\I_\G(\t_-,\t'_-)=(\SS_{r}(\b,\La')\cap \G)
%(\B_{\b}^\times \cap \G)
\G_\b
(\SS_r(\b,\La)\cap \G).
\]
\end{enumerate}
\end{proposition}

\begin{proof}
%See~\cite[Theorem 3.22]{St05} for the case~$\La=\La'$. 
The proof is analogous to that of~\cite[Theorems~4.9, 4.10]{RKSS}. 
\begin{enumerate}
\item Let us at first assume that both lattice sequences are block-wise regular strict. % principal lattice chains of the same block size. 
There is an element~$g\in\tG_\b$ such that~$g\La$ is equal to $\La'$ and the conjugation with~$g$ realizes the transfer from~$\Cc(\La,r,\b)$  to~$\Cc(\La',r,\b)$. Thus we can reduce to the case where~$\t$ is equal to~$\t'$ which follows from~\cite[Theorem~3.22]{St05} (see also~\cite[Proposition~9.8]{SkSt} and the paragraph following it).

We now consider the general case. Applying the $\dag$-construction, 
%of \cite[\S 3.1]{RKSS}
we obtain semisimple characters~$\t^\dag\in\Cc(\La^\dag,r,\b^\dag)$ and~$\t'^\dag\in\Cc(\La'^\dag,r,\b^\dag)$, where~$\La^\dag$ and~$\La'^\dag$ are strict and regular of the same block size. %in a direct sum~$\V^\dag$ of~$e$-copies of $\V$ (where $e=e(\La)=e(\La')$). %\gre{We have to include a notion of~$\dag$-construction for~$\tG$, because in the self-dual~$\dag$-construction we have used the dual of a lattice.} We set~$\tG^\dag=\Aut_\F(\V^\dag)$.  
From the first case, we have the formula~\ref{prop:IntertwiningOfTransfersPart1} for~$\I_{\tG^\dag}(\t^\dag,\t'^\dag)$. Moreover, exactly as in the proof of~\cite[Lemma~4.6]{RKSS}, we have the simple intersection property
\[
\SS_{r}(\b^\dag,\La'^\dag)x \SS_{r}(\b^\dag,\La^\dag)\cap \tG_{\b^\dag}= (\SS_{r}(\b^\dag,\La'^\dag)\cap \tG_{\b^\dag}) x (\SS_{r}(\b^\dag,\La^\dag)\cap \tG_{\b^\dag}),
\]
for all~$x\in  \tG_{\b^\dag}$. As in~\cite[Corollary~4.14]{St01}, it follows from~\cite[Theorem~4.10]{RKSS} % ~$\Gamma=\{\pm 1\}^{e}$, 
that the intertwining formula behaves well under intersection with the Levi group~$\M^\dag$ attached to the~$\dag$-construction, i.e.
\[
\I_{\tG^\dag}(\t^\dag,\t'^\dag)\cap \M^\dag=(\SS_{r}(\b^\dag,\La'^\dag)\cap \M^\dag)(\tG_{\b^\dag} \cap \M^\dag)(\SS_{r}(\b^\dag,\La^\dag)\cap \M^\dag).
\]
%by~\cite[Theorem~2.7]{RKSS}, using the group~$\Gamma=\{\pm 1\}^{e}$, if we have the intersection property
%\[\SS_{r}(\b^\dag,\La'^\dag)x \SS_{r}(\b^\dag,\La^\dag)\cap \B_{\b^\dag}^\times= (\SS_{r}(\b^\dag,\La'^\dag)\cap \B_{\b^\dag}^\times) x (\SS_{r}(\b^\dag,\La^\dag)\cap \B_{\b^\dag}^\times),\] for all~$x\in  \B_{\b^\dag}^\times$. The proof of the intersection property follows mutatis mutandis 
%%
%\todo{Explain that this means going right back to BK 1.6.1, through all the intervening references.}
%%
%to the proof of~\cite[Lemma~4.6]{RKSS}. 
Finally, we restrict to the first block of~$\M^\dag$ to obtain the desired description of~$\I_{\tG}(\t,\t')$. 
\item This follows from~\ref{prop:IntertwiningOfTransfersPart1} and a standard cohomology argument~\cite[Theorem~2.12]{RKSS}, as in the proof of~\cite[Theorem~4.10]{RKSS}. 
\end{enumerate}
%\red{This proof is fine as it is similar to the reference, I don't know if we should include more details or not as it is quick.}
%\red{Are strict, regular, principal defined in our preliminaries?}
\end{proof}

The proofs of the following two \rob{lemmas} are \emph{mutatis mutandis} to the proofs of~\cite[Proposition~9.17]{SkSt} and~\cite[Proposition~9.27]{SkSt} respectively, except that one uses Proposition~\ref{prop:IntertwiningOfTransfers} instead of~\cite[Propositions~9.8,9.22]{SkSt}. %(which are~\cite[3.22,3.27]{St05}).  

\begin{lemma}[{cf.~\cite[Proposition~9.17]{SkSt}}]\label{lem:DerivedCharacters}
Suppose~$m<q-1$ and let~$[\La,q,m,\b]$ and~$[\La',q,m,\b']$ be semisimple strata with~$e(\La)=e(\La')$, and with splitting~$\V=\bigoplus_{i\in\I}\V^i$ and~$\V=\bigoplus_{i\in\I'}\V'^{i}$. Suppose that~$[\La,q,m+1,\g]$ and~$[\La',q,m+1,\g]$ are non-null simple strata equivalent to~$[\La,q,m+1,\b]$ and~$[\La',q,m+1,\b']$ respectively, and that~$\g$ lies in both~$\bigoplus_{i\in\I}\A^i$ and~$\bigoplus_{i'\in\I'}\A^{i'}$. % commutes with all idempotents of the decomposition of~$\b$ and that~$\g$ does so with the idempotents for~$\b'$. 
Let~$\t_0\in\Cc(\La,m,\g)$ and set~$\t'_0=\tau_{\La',\La,\g}(\t_0)\in\Cc(\La',m,\g)$. Let~$\t\in \Cc(\La,m,\b)$ and~$\t'\in \Cc(\La',m,\b')$ be semisimple characters which satisfy
\[\bob{\t}=\t_0\psi_{\b-\g+c}\text{ and } \t'=\t'_0\psi_{\b'-\g},
\]
for some~$c\in\aa_{-(m+1)}$. Let~$s_\g$ be a tame corestriction with respect to~$\g$. Then we have:
%agree on restriction to~$H^{m+2}(\g,\La)$, so that we can write~$\t'=\t_0\psi_{\b'-\g}$ and~$\t=\t_0\psi_{\b-\g+c}$, for some~$\t_0\in \C(\La,m,\g)$ and~$c\in \aa_{-(m+1)}$. Let~$s_\g$ be a tame corestriction with respect to~$\g$.
\begin{enumerate} 
\item For any~$g\in \I_{\tG}(\t,\t')$ there are elements~$x\in \SS_{m+1}(\g,\La')$ and~$y\in \SS_{m+1}(\g,\La)$ and~$g'\in\tG_\g$ such that~$g=xg'y$; moreover,~$g'$ intertwines~$[\La,m+1,m,s_\g(\b-\g+c)]$ with~$[\La',m+1,m,s_\g(\b'-\g)]$. \label{lem:DerivedCharacters.i}
\item For any~$g'\in \tG_\g$ which intertwines~$[\La,m+1,m,s_\g(\b-\g+c)]$ with~$[\La',m+1,m,s_\g(\b'-\g)]$, 
there are elements~$x\in 1+\mm_{m+1}(\g,\La')$ and~$y\in 1+\mm_{m+1}(\g,\La)$ such that~$xg'y$ intertwines~$\t$ with~$\t'$. \label{lem:DerivedCharacters.ii}
\end{enumerate}
\end{lemma}

% \begin{proof}
%  The proof is mutatis mutandis to the proof of \it{loc. cit.}, with the exeption to use Proposition~\ref{prop:IntertwiningOfTransfers}. 
% \end{proof}
% 

%Finally, we have the same result for~$\G$.

\begin{lemma}[{cf.~\cite[Proposition~9.27]{SkSt}}]\label{lem:DerivedCharactersForG}
In the situation of Lemma~\ref{lem:DerivedCharacters}, suppose additionally that all strata are self-dual, all semisimple characters are self-dual,~$c\in\aa_{-(m+1)}^-$, and~$s_\g$ commutes with the adjoint anti-involution. 
\begin{enumerate} 
\item For any~$g\in \I_\G(\t,\t')$ there are elements~$x\in \SS_{m+1}(\g,\La')\cap\G$ and~$y\in \SS_{m+1}(\g,\La)\cap\G$ 
and~$g'\in \G_\g$ such that~$g=xg'y$; moreover,~$g'$ intertwines~$[\La,m+1,m,s_\g(\b-\g+c)]$ with~$[\La',m+1,m,s_\g(\b'-\g)]$. \label{lem:DerivedCharactersforG.i}
\item For any~$g'\in \G_\g$ which intertwines~$[\La,m+1,m,s_\g(\b-\g+c)]$ with~$[\La',m+1,m,s_\g(\b'-\g)]$, 
there are elements~$x\in (1+\mm_{m+1}(\g,\La'))\cap\G$ and~$y\in(1+\mm_{m+1}(\g,\La))\cap\G$ such that~$xg'y$ intertwines~$\t$ with~$\t'$. 
\label{lem:DerivedCharactersforG.ii}
\end{enumerate}
\end{lemma}

\subsection{Intertwining and conjugacy for self-dual semisimple characters}
%%%%%%%%%%%%%%%%%%%%%%%%%%%%%%%%%%%%
In this final subsection, we prove an intertwining implies conjugacy theorem for self-dual semisimple characters, which generalizes a result of the second and third authors~\cite[10.2,10.3]{SkSt} for skew semisimple characters. Let~$[\La,n,r,\b]$ and~$[\La,n,r,\b']$ be self-dual semisimple strata in~$\A$.

\begin{theorem}\label{thm:IntImplConjSelfDual}
Let~$\t\in\Cc_-(\La,r,\b)$ and~$\t'\in\Cc_-(\La,r,\b')$ be self-dual semisimple characters which intertwine in~$\G$ and such that the matching~$\zeta:\I\rightarrow \I'$ of~\cite[Theorem~10.1]{SkSt} satisfies 
\[
\La_j^i/\La_{j+1}^i\cong \La^{\zeta(i)}_j/\La^{\zeta(i)}_{j+1}, 
\]
for all indices~$i\in \I$ and all integers~$j$. Then there is an element of $\bob{\P_-(\La)}\cap\prod_{i\in\I}\Hom_\F(\V^i,\V'^{\z(i)})$ which conjugates~$\t$ to~$\t'$.  
\end{theorem}

\begin{proof}
The involution~$\s$ acts on the index sets~$\I$ and~$\I'$ and this action commutes with the map~$\zeta$ by the matching theorem~\cite[Theorem~10.1]{SkSt}. We write~$\I=\I_0\cup\I_+\cup\I_-$ as usual, and similarly for~$\I'$. We deduce that~$\zeta$ sends~$\I_0$ to~$\I'_0$ and~$\I_{+}\cup\I_-$ to~$\I'_{+}\cup\I'_-$. We abbreviate~$\V_0=\V^{\I_0}$ so that~$\V_0^\perp=\V^{\I_{+}\cup\I_-}$, and similarly~$\V'_0$. 

The hyperbolic spaces~$\V_0^\perp$ and~$\V'^\perp_0$ are isometric since they have the same dimension, so~$\V_0$ and~$\V'_0$ are isometric. Take an isometry~$g$ of~$(\V,h)$ which sends~$\V_0$ to~$\V'_0$ and~$\V^\perp_0$ to~$\V'^\perp_0$. By~\cite[Proposition~5.2]{SkodField} we can modify~$g$ such that~$g$ is an element of~$\bob{\P_-(\La)}$. Conjugating~$\t$ by~$g$, we may assume without loss of generality that~$\V_0$ and~$\V'_0$ coincide. 

We show next that there is an element of~$\G\cap (\Aut_\F(\V_0)\times\Aut_\F(\V_0^\perp))$ which intertwines~$\t$ with~$\t'$. By Theorem~\cite[Theorem~10.2]{SkSt} there is an element~$\tilde{g}\in \P(\La) \cap \prod_{i\in\I} \Hom_\F(\V^i,\V'^{\z(i)})$ which conjugates~$\t$ to~$\t'$. Taking the intertwining formula of Proposition~\ref{prop:IntertwiningOfTransfers} and conjugating back with~$\tilde{g}$ we obtain %that~$\I_{\tG}(\t,\t')$ coincides with~$\SS_r(\b',\La)\B^\times_{\b'}\tilde{g}\SS_r(\b,\La)$, which is a subset of
\[
\I_{\tG}(\t,\t') = \SS_r(\b',\La)\tG_{\b'}\tilde{g}\SS_r(\b,\La) \subseteq
\SS_r(\b',\La)(\Aut_\F(\V_0)\times\Aut_\F(\V_0^\perp))\SS_r(\b,\La).
\] 
By a standard cohomology argument, as in~\cite[Corollary~4.14]{St01}, we see that %the intersection of~$\I_{\tG}(\t,\t')$ with~$\G$ is contained in
\[
\I_{\G}(\t,\t')\subseteq (\SS_r(\b',\La)\cap \G)((\Aut_\F(\V_0)\times\Aut_\F(\V_0^\perp))\cap \G)(\SS_r(\b,\La)\cap \G)
\]
and thus we obtain that the restrictions of~$\t$ and~$\t'$ on~$\V_0$ and on~$\V_0^\perp$ intertwine by an element of~$\U(\V_0)$
and~$\U(\V_0^\perp)$ respectively. Thus, by Corollary~\ref{corCDeltaPsia}  we can restrict to the cases where~$\I=\I_0$ or~$\I=\I_{+-}$. The first case is precisely~\cite[Theorem~10.3]{SkSt} and the second case is an easy exercise using Theorem~\cite[Theorem~10.2]{SkSt} (for~$\V_+=\V^{\I_+}$ and~$\V'_+=\V^{\I'_+}$) and Corollary~\ref{corCDeltaPsia}. 
% 
%  The proof is mutatis mutandis the proof of \cite[Theorem 10.3]{SkSt} following a strata induction, see \cite[after Remark 7.2]{SkSt}, if we know the result for 
%  minimal strata. Thus, we suppose that we have two strata~$[\La,n,n-1,\b]$ and~$[\La',n,n-1,\b']$ which intertwine over some element of~$\G$. We only need to consider the case that the first stratum is 
%  a self-dual non-skew stratum whose index set~$I$ has cardinality two. By the matching theorem Theorem~\ref{thmMatchingForCharForG} the index set~$I'$ of the second stratum has also two elements interchanged by the adjoint involution. By \cite[Theorem 10.2]{SkSt} there is an isomorphism~$g$ from~$\V^1$ to~$\V^{\zeta(1)}$ such 
%  that~$g\La^1$ is equal to~$\La'^{\zeta(1)}$ and~$[\La'^{\zeta(1)},n,n-1,g\b_1 g^{-1}]$ is equivalent to~$[\La'^{\zeta(1)},n,n-1,\b'_{\zeta(1)}]$. 
%  The self-duality of the strata imply the equivalence of~$[\La'^{\zeta(2)},n,n-1,\s(g)^{-1}\b_2 \s(g)]$ and~$[\La'^{\zeta(2)},n,n-1,\b'_{\zeta(2)}]$.
\end{proof}

%% file: KurinczukSkodlerackStevens_Endoparameters.bbl
\begin{thebibliography}{10}

\bibitem{AKMSS}
U.K. Anandavardhanan, Robert Kurinczuk, Nadir Matringe, Vincent S\'{e}cherre,
  and Shaun Stevens.
\newblock Galois self-dual cuspidal types and {A}sai local factors.
\newblock {\em J. Eur. Math. Soc.}, to appear.

\bibitem{Arthur}
James Arthur.
\newblock {\em The endoscopic classification of representations}, volume~61 of
  {\em American Mathematical Society Colloquium Publications}.
\newblock American Mathematical Society, Providence, RI, 2013.
\newblock Orthogonal and symplectic groups.

\bibitem{Blondel}
Corinne Blondel.
\newblock {${\rm SP}(2N)$}-covers for self-contragredient supercuspidal
  representations of {${\rm GL}(N)$}.
\newblock {\em Ann. Sci. \'Ecole Norm. Sup. (4)}, 37(4):533--558, 2004.

\bibitem{BlHeSt}
Corinne Blondel, Guy Henniart, and Shaun Stevens.
\newblock Jordan blocks of cuspidal representations of symplectic groups.
\newblock {\em Algebra Number Theory}, 12(10):2327--2386, 2018.

\bibitem{BSS}
Paul Broussous, Vincent S{\'e}cherre, and Shaun Stevens.
\newblock Smooth representations of {${\rm GL}_m(D)$} {V}: {E}ndo-classes.
\newblock {\em Doc. Math.}, 17:23--77, 2012.

\bibitem{BS09}
Paul Broussous and Shaun Stevens.
\newblock Buildings of classical groups and centralizers of {L}ie algebra
  elements.
\newblock {\em J. Lie Theory}, 19(1):55--78, 2009.

\bibitem{bushnellFroehlich:85}
Colin~J. Bushnell and A.~Fr{\"o}hlich.
\newblock Non-abelian congruence gauss sums and~{$p$}-adic simple algebras.
\newblock {\em Proc. Lond. Math. Soc. (3)}, 50:207--264, 1985.

\bibitem{BH96}
Colin~J. Bushnell and Guy Henniart.
\newblock Local tame lifting for {${\rm GL}(N)$}. {I}. {S}imple characters.
\newblock {\em Inst. Hautes \'Etudes Sci. Publ. Math.}, (83):105--233, 1996.

\bibitem{BHIntertwiningSimple}
Colin~J. Bushnell and Guy Henniart.
\newblock Intertwining of simple characters in {${\rm GL}(n)$}.
\newblock {\em Int. Math. Res. Not. IMRN}, (17):3977--3987, 2013.

\bibitem{BHEffective}
Colin~J. Bushnell and Guy Henniart.
\newblock To an effective local {L}anglands correspondence.
\newblock {\em Mem. Amer. Math. Soc.}, 231(1087):v+88, 2014.

\bibitem{BH17}
Colin~J. Bushnell and Guy Henniart.
\newblock Higher ramification and the local {L}anglands correspondence.
\newblock {\em Ann. of Math. (2)}, 185(3):919--955, 2017.

\bibitem{BK93}
Colin~J. Bushnell and Philip~C. Kutzko.
\newblock {\em The admissible dual of {${\rm GL}(N)$} via compact open
  subgroups}, volume 129 of {\em Annals of Mathematics Studies}.
\newblock Princeton University Press, Princeton, NJ, 1993.

\bibitem{BK94}
Colin~J. Bushnell and Philip~C. Kutzko.
\newblock Simple types in {${\rm GL}(N)$}: computing conjugacy classes.
\newblock In {\em Representation theory and analysis on homogeneous spaces
  ({N}ew {B}runswick, {NJ}, 1993)}, volume 177 of {\em Contemp. Math.}, pages
  107--135. Amer. Math. Soc., Providence, RI, 1994.

\bibitem{BK99}
Colin~J. Bushnell and Philip~C. Kutzko.
\newblock Semisimple types in {${\rm GL}_n$}.
\newblock {\em Compositio Math.}, 119(1):53--97, 1999.

\bibitem{Chinello}
Gianmarco Chinello.
\newblock Blocks of the category of smooth {$\ell$}-modular representations of
  {${\rm GL}(n,F)$} and its inner forms: reduction to level 0.
\newblock {\em Algebra Number Theory}, 12(7):1675--1713, 2018.

\bibitem{Finitude}
Jean-Fran\c{c}ois Dat.
\newblock Finitude pour les repr\'{e}sentations lisses de groupes
  {$p$}-adiques.
\newblock {\em J. Inst. Math. Jussieu}, 8(2):261--333, 2009.

\bibitem{Datfunctoriality}
Jean-Fran\c{c}ois Dat.
\newblock A functoriality principle for blocks of {$p$}-adic linear groups.
\newblock In {\em Around {L}anglands correspondences}, volume 691 of {\em
  Contemp. Math.}, pages 103--131. Amer. Math. Soc., Providence, RI, 2017.

\bibitem{DHKM}
Jean-Fran\c{c}ois Dat, David Helm, Robert Kurinczuk, and Gil Moss.
\newblock Moduli of langlands parameters.
\newblock {\em in preparation}, 2019.

\bibitem{Dotto}
Andrea Dotto.
\newblock The inertial {J}acquet-{L}anglands correspondence.
\newblock {\em arXiv:1707.00635}, 2018.

\bibitem{Fintzen}
Jessica Fintzen.
\newblock Types for tame $p$-adic groups.
\newblock {\em arXiv:1810.04198}, 2018.

\bibitem{MR3709003}
Radhika Ganapathy and Sandeep Varma.
\newblock On the local {L}anglands correspondence for split classical groups
  over local function fields.
\newblock {\em J. Inst. Math. Jussieu}, 16(5):987--1074, 2017.

\bibitem{HM}
Jeffrey Hakim and Fiona Murnaghan.
\newblock Distinguished tame supercuspidal representations.
\newblock {\em Int. Math. Res. Pap. IMRP}, (2):Art. ID rpn005, 166, 2008.

\bibitem{HelmBC}
David Helm.
\newblock The {B}ernstein center of the category of smooth {$W(k)[{\rm
  GL}_n(F)]$}-modules.
\newblock {\em Forum Math. Sigma}, 4:e11, 98, 2016.

\bibitem{Howe}
Roger~E. Howe.
\newblock Tamely ramified supercuspidal representations of {${\rm Gl}_{n}$}.
\newblock {\em Pacific J. Math.}, 73(2):437--460, 1977.

\bibitem{KMSW}
Tasho Kaletha, Alberto Minguez, Sug~Woo Shin, and Paul-James White.
\newblock Endoscopic classification of representations: Inner forms of unitary
  groups.
\newblock {\em arXiv:1409.3731}, 2014.

\bibitem{RKSS}
Robert Kurinczuk and Shaun Stevens.
\newblock Cuspidal {$\ell$}-modular representations of {$p$}-adic classical
  groups.
\newblock {\em J. Reine Angew. Math.}, 764:23--69, 2020.

\bibitem{Lanard}
Thomas Lanard.
\newblock Sur les {$\ell$}-blocs de niveau z\'{e}ro des groupes {$p$}-adiques.
\newblock {\em Compos. Math.}, 154(7):1473--1507, 2018.

\bibitem{MStype}
Alberto M\'{\i}nguez and Vincent S\'{e}cherre.
\newblock Types modulo {$\ell$} pour les formes int\'{e}rieures de {${\rm
  GL}_n$} sur un corps local non archim\'{e}dien.
\newblock {\em Proc. Lond. Math. Soc. (3)}, 109(4):823--891, 2014.
\newblock With an appendix by Vincent S\'{e}cherre et Shaun Stevens.

\bibitem{MiSt}
Michitaka Miyauchi and Shaun Stevens.
\newblock Semisimple types for {$p$}-adic classical groups.
\newblock {\em Math. Ann.}, 358(1-2):257--288, 2014.

\bibitem{Mok}
Chung~Pang Mok.
\newblock Endoscopic classification of representations of quasi-split unitary
  groups.
\newblock {\em Mem. Amer. Math. Soc.}, 235(1108):vi+248, 2015.

\bibitem{Scharlau}
Winfried Scharlau.
\newblock {\em Quadratic and {H}ermitian forms}, volume 270 of {\em Grundlehren
  der Mathematischen Wissenschaften [Fundamental Principles of Mathematical
  Sciences]}.
\newblock Springer-Verlag, Berlin, 1985.

\bibitem{secherreI}
Vincent S{\'e}cherre.
\newblock Repr\'esentations lisses de {${\rm GL}(m,D)$}. {I}. {C}aract\`eres
  simples.
\newblock {\em Bull. Soc. Math. France}, 132(3):327--396, 2004.

\bibitem{Secherre}
Vincent S\'{e}cherre.
\newblock Supercuspidal representations of {${\rm GL}_n(\rm F)$} distinguished
  by a {G}alois involution.
\newblock {\em Algebra Number Theory}, 13(7):1677--1733, 2019.

\bibitem{SeStSupercuspidals}
Vincent S{\'e}cherre and Shaun Stevens.
\newblock Repr\'esentations lisses de {${\rm GL}_m(D)$}. {IV}.
  {R}epr\'esentations supercuspidales.
\newblock {\em J. Inst. Math. Jussieu}, 7(3):527--574, 2008.

\bibitem{secherreStevensVI:10}
Vincent S{\'e}cherre and Shaun Stevens.
\newblock Smooth representations of {$GL_m(D)$} {VI}: semisimple types.
\newblock {\em Int. Math. Res. Not. IMRN}, (13):2994--3039, 2012.

\bibitem{SeStblock}
Vincent S\'{e}cherre and Shaun Stevens.
\newblock Block decomposition of the category of {$\ell$}-modular smooth
  representations of {${\rm GL}_n(\rm F)$} and its inner forms.
\newblock {\em Ann. Sci. \'{E}c. Norm. Sup\'{e}r. (4)}, 49(3):669--709, 2016.

\bibitem{SeStJL}
Vincent S\'{e}cherre and Shaun Stevens.
\newblock Towards an explicit local {J}acquet-{L}anglands correspondence beyond
  the cuspidal case.
\newblock {\em Compos. Math.}, 2019.

\bibitem{SkodField}
Daniel Skodlerack.
\newblock Field embeddings which are conjugate under a p-adic classical group.
\newblock {\em Manuscripta Mathematica}, pages 1--25, 2013.

\bibitem{SkSt}
Daniel Skodlerack and Shaun Stevens.
\newblock Intertwining semisimple characters for $p$-adic classical groups.
\newblock {\em Nagoya Math. J.}, 238:137--205, 2020.

\bibitem{St01}
Shaun Stevens.
\newblock Double coset decompositions and intertwining.
\newblock {\em manuscripta mathematica}, 106(3):349--364, 2001.

\bibitem{St00}
Shaun Stevens.
\newblock Intertwining and supercuspidal types for {$p$}-adic classical groups.
\newblock {\em Proc. London Math. Soc. (3)}, 83(1):120--140, 2001.

\bibitem{St05}
Shaun Stevens.
\newblock Semisimple characters for {$p$}-adic classical groups.
\newblock {\em Duke Math. J.}, 127(1):123--173, 2005.

\bibitem{St08}
Shaun Stevens.
\newblock The supercuspidal representations of {$p$}-adic classical groups.
\newblock {\em Invent. Math.}, 172(2):289--352, 2008.

\bibitem{Vig96}
Marie-France Vign\'{e}ras.
\newblock {\em Repr\'{e}sentations {$l$}-modulaires d'un groupe r\'{e}ductif
  {$p$}-adique avec {$l\ne p$}}, volume 137 of {\em Progress in Mathematics}.
\newblock Birkh\"{a}user Boston, Inc., Boston, MA, 1996.

\bibitem{Yutame}
Jiu-Kang Yu.
\newblock Construction of tame supercuspidal representations.
\newblock {\em J. Amer. Math. Soc.}, 14(3):579--622 (electronic), 2001.

\end{thebibliography}
